%% file: Thesis.tex
\documentclass[oldfontcommands]{ucbThesis}  

\input{macros}
\usepackage{verbatim}  
\usepackage{vector}  

\usepackage[style = alphabetic, url=false, doi =false]{biblatex}
\hypersetup{urlcolor=blue, colorlinks=true}  
\begin{document}


\title  {Factorization Algebras for Bulk-Boundary Systems}
\author {Eugene S Rabinovich}
\degreesemester{Summer}
\degreeyear{2021}
\degree{Doctor of Philosophy}
\chair{Professor Peter Teichner}
\othermembers{Professor Nicolai Reshetikhin \\
  Professor Ori Ganor}
\numberofmembers{3}
\field{Mathematics}

\maketitle
\copyrightpage

\include{abstract}







\begin{frontmatter}

\begin{dedication}
\null\vfil
\begin{center}
This dissertation is dedicated to the memory of my grandmother, Esya Solomonovna Gurevich, who remains my role model in all aspects of life.
\end{center}
\vfil\null
\end{dedication}


\tableofcontents
\clearpage
\listoffigures
\clearpage

\begin{acknowledgements}
Long before this dissertation was even a whisper in the air, my family was encouraging my studies in physics and mathematics.
My particular gratitude goes to my uncle Oscar, whose love of physics and pedagogy is contagious. 
Thanks to Oscar, physics and mathematics are not simply subjects of study for me---they are lenses through which the world acquires profound beauty and mystery.
I thank my parents---Edward and Marina---for teaching me to value learning, reading, and thinking; and my sister Emily for being my first pupil, and a patient one at that!

I have had the great fortune of learning from many very skilled teachers in my life, and two from before my time at Berkeley are of particular relevance to my doctoral studies.
Sean Yee introduced me to concepts of higher mathematics---proof and abstract definitions in algebra and topology---long before I had to sit an exam on them.
His enthusiasm for pedagogy and lived commitment to sharing mathematics with others serve as models for me to this day.
At Duke, Ronen Plesser taught me the vast majority of what I know about quantum field theory.
I continue to be inspired by his generosity with his time and his commitment to sharing physics with a broad audience.

This dissertation is most directly the product of my six years at UC Berkeley, and so the list of people to thank from this period of my studies is much longer.
First, my thanks go to my advisor, Peter Teichner, who: first taught me what a factorization algebra is; guided me through difficult moments in my dissertation work; gave me many opportunities to speak about this work in his student seminar; and connected me to the wide network of scholars of which I am now fortunate to be a part.
My unofficial advisor, Owen Gwilliam, has played a similarly crucial role in the development of this dissertation.
I would like to thank him for: his constant belief in my skills as a mathematician; the opportunities to travel (to Canada and China, in particular) and to share my work which he secured for me; the innumerably many things he has taught me which are directly relevant to this dissertation; and his enthusiasm as a collaborator.
I owe most of what I know about factorization algebras and quantum field theory \`a la Costello to Owen and to Brian Williams, both of whom guided me as I learned about these subjects in detail.

I thank Owen and Brian Williams also for collaboration on the work (ar$\chi$iv: 2001.07888) that is essentially Chapter \ref{chap: freequantum}, and for granting their permission to reproduce that paper in this dissertation.

An earlier form of ideas from Chapters \ref{chap: classical} and \ref{chap: interactingquantum} was developed over the course of discussions with Benjamin Albert. 
I would like to sincerely thank Benjamin for these discussions.

I would also like to thank: Dylan Butson, Nicola Capacci, Kevin Costello, Fil Dul, Si Li, Matthias Ludewig, Aaron Mazel-Gee, Nima Moshayedi, Laura Murray-Wells, Denis Nardin, Pavel Safronov, Claudia Scheimbauer, and Philsang Yoo for discussions which have made this dissertation better.
In particular, I have learned a lot from the work of and discussions with Kevin, Si, and Pavel.
I warmly thank Minghao Wang for significant help in proving Lemma \ref{lem: propdefined}.

I would like to thank Paul Vojta for maintaining the ucbthesis \LaTeX class, which made it easy to format this thesis to conform with the university's requirements.
I have been supported for the majority of this project by the NSF Graduate Research Fellowship, under grant number DGE 1752814.
I would also like to thank the Max Planck Institut f\"ur Mathematik and the Perimeter Institute for Theoretical Physics for their hospitality during my stays there.

My sincere gratitude extends to my committee members, Peter Teichner, Nicolai Reshetikhin, and Ori Ganor, for agreeing to review this work.

I would like to thank my former officemate Mike Lindsey for jumping all the major hurdles of doctoral student life before me, and then supporting me with his acquired wisdom as I followed in his tracks.

Finally, I would like to thank my partner, Katie Bondy, for her steadfast support, even through difficult times.
That support has made an enormous difference in my life.

\end{acknowledgements}

\end{frontmatter}

\pagestyle{headings}  


\input{Chapters/Chapter1} 

\input{Chapters/Chapter2} 

\input{Chapters/Chapter3} 

\input{Chapters/Chapter4} 

\input{Chapters/Chapter5} 




\addtocontents{toc}{\vspace{2em}} 

\appendix 

\input{Appendices/AppendixA}	

\addtocontents{toc}{\vspace{2em}}  

\backmatter

\label{Bibliography}
\printbibliography

\printindex[notation]
\end{document}

%% file: macros.tex
\usepackage{mathpple}
\usepackage{mathrsfs}
\usepackage{amsmath, amscd, amsthm,amssymb, thmtools, amsfonts, verbatim, breqn}
\usepackage[mathcal]{eucal}
\usepackage{relsize}
\usepackage{tikz-cd}
\usepackage[cmtip, arrow, all]{xy}
\usepackage{pb-diagram ,pb-xy}
\usepackage{graphicx}
\usepackage{subcaption}
\usepackage{epstopdf}
\usepgfmodule{shapes}
\usepackage{color}
\usepackage{ifthen}
\usetikzlibrary{patterns, arrows}

\renewcommand{\epsilon}{\varepsilon}
\newcommand{\xto}{\xrightarrow}

\newcommand{\br}{\overline}

\newcommand{\CC}{\mathbb C}
\newcommand{\N}{\mathbb N}

\newcommand{\Z}{\mathbb Z}

\newcommand{\mbb}{\mathbb}

\newcommand{\ip}[1][\cdot,\cdot]{\left\langle #1 \right\rangle}

\newcommand{\R}{\mbb R}
\renewcommand{\d}{\mathrm{d}}
\newcommand{\DVS}{\mathrm{DVS}}
\newcommand{\CVS}{\mathrm{CVS}}

\newcommand{\dbar}{\br{\partial}}

\DeclareMathOperator{\Aut}{Aut} 
 
\DeclareMathOperator{\Sym}{Sym} \DeclareMathOperator{\Hom}{Hom}

\def\cC{\mathcal C}\def\cD{\mathcal D}
\def\cF{\mathcal F}\def\cG{\mathcal G}
\def\cK{\mathcal K}
\def\cO{\mathcal O}\def\cP{\mathcal P}
\def\cR{\mathcal R}\def\cS{\mathcal S}\def\cT{\mathcal T}
\def\cU{\mathcal U}

\def\AA{\mathbb A}\def\BB{\mathbb B}\def\CC{\mathbb C}
\def\HH{\mathbb H}
\def\II{\mathbb I}\def\KK{\mathbb K}\def\LL{\mathbb L}
\def\NN{\mathbb N}\def\PP{\mathbb P}
\def\RR{\mathbb R}

\def\ZZ{\mathbb Z}

\def\sA{\mathscr A}\def\sB{\mathscr B}\def\sC{\mathscr C}
\def\sE{\mathscr E}\def\sF{\mathscr F}\def\sG{\mathscr G}
\def\sJ{\mathscr J}\def\sK{\mathscr K}\def\sL{\mathscr L}
\def\sO{\mathscr O}

\def\sV{\mathscr V}\def\sW{\mathscr W}

\def\fA{\mathfrak A}
\def\fJ(E){\mathfrak E}
\def\fJ{\mathfrak J}

\def\fS{\mathfrak S}
\def\fU{\mathfrak U}\def\fV{\mathfrak V}
\def\fg{\mathfrak g}\def\fh{\mathfrak h}

\declaretheoremstyle[
spaceabove=7pt, spacebelow=7pt,
headfont=\normalfont\bfseries,
notefont=\mdseries, notebraces={(}{)},
bodyfont=\normalfont,
postheadspace=5pt,
headpunct = .
]{thm}

\declaretheoremstyle[
spaceabove=7pt, spacebelow=7pt,
headfont=\normalfont\bfseries,
notefont=\mdseries, notebraces={(}{)},
bodyfont=\normalfont,
postheadspace=10pt,
headpunct = .
]{def}

\declaretheoremstyle[
spaceabove=4pt, spacebelow=7pt,
headfont=\itshape,
postheadspace=5pt,
headpunct = :,
postheadspace = 3pt, qed = $\lozenge$
]{rem}

\declaretheorem[numbered = no, style = thm, name = Theorem]{utheorem}

\declaretheorem[numbered = yes, parent = section, style = thm]{theorem}

\declaretheorem[numbered = no, style = thm]{conjecture}
\declaretheorem[numbered = no, style = thm, name = Corollary]{ucorollary}
\declaretheorem[numbered = no, style = thm, name = Corollary A]{corA}

\declaretheorem[numbered = no, style = thm, name = Program]{Program}

\declaretheorem[sibling = theorem, style = thm, name = Theorem/Definition]{thm-def}
\declaretheorem[sibling = theorem, style = thm]{proposition}
\declaretheorem[sibling = theorem, style = thm]{lemma}
\declaretheorem[sibling = theorem, style = thm]{notation}

\declaretheorem[sibling = theorem, style = thm, name = Construction]{constr}
\declaretheorem[sibling = theorem, style = thm, name = Definition-Lemma]{deflem}
\numberwithin{equation}{section}

\declaretheorem[sibling = theorem, style = def]{definition}

\declaretheorem[sibling = theorem, style = rem]{remark}

\declaretheorem[sibling = theorem, style = rem]{example}

\newcommand{\cinfty}{C^{\infty}}
\newcommand\Obcl[1][~]{\ifthenelse{ \equal{#1}{~}} {
  \operatorname{Obs}^{cl}
}{
  \operatorname{Obs}^{cl}(#1)
}}
\newcommand\Obq[1][~]{\ifthenelse{ \equal{#1}{~}} {
  \operatorname{Obs}^{q}
}{
  \operatorname{Obs}^{q}(#1)
}}

\DeclareMathOperator{\tr}{tr}

\DeclareMathOperator{\id}{id}

\DeclareMathOperator{\im}{im}

\DeclareMathOperator{\supp}{supp}

\DeclareMathOperator{\sgn}{sgn}

\newcommand{\sEb}{\sE_\partial}
\newcommand{\sAb}{\sA_\partial}
\newcommand{\sBb}{\sB_\partial}

\newcommand{\Eb}{E_\partial}
\newcommand{\Ab}{A_\partial}
\newcommand{\Bb}{B_\partial}

\newcommand{\Qbgf}{Q_{\partial}^{GF}}
\newcommand{\diff}{\ell_1}
\newcommand{\Qb}{\ell_{1,\partial}}
\newcommand{\pertdiff}{\ell_{cross}}

\newcommand{\condfields}{\sE_\sL}
\newcommand{\condfieldsterm}{$\sL$-conditioned }
\newcommand{\condfieldscs}{\sE_{\sL,c}}

\newcommand{\Oloc}{\sO_{loc}(\condfields)}
\newcommand{\Olocred}{\sO_{loc,red}(\condfields)}
\newcommand{\bdyM}{\partial M}

\newcommand{\fullOloc}{\mathrm{DR}_{M,\bdyM}(E;L)}

\newcommand{\nocrossdiff}{\mathring \sE}
\newcommand{\condnocrossdiff}{\nocrossdiff_\sL}
\newcommand{\diffonnocrossdiff}{Q}
\newcommand{\QL}{Q_{\sL}}
\newcommand{\QLprime}{Q_{\sL'}}
\newcommand{\diffonnocrossdiffbdy}{\diffonnocrossdiff_\partial}
\newcommand{\ultrafields}{ \sE_{\widehat \sL}}
\newcommand{\ultrafieldsterm}{ultra-conditioned }
\newcommand{\Mdbl}{M_{DBL}}
\newcommand{\Udbl}{U_{DBL}}
\newcommand{\Edbl}{E_{DBL}}
\newcommand{\Adbl}{A_{DBL}}
\newcommand{\Bdbl}{B_{DBL}}

\newcommand{\involution}{\sigma}
\newcommand{\sEdbl}{\sE_{DBL}}
\newcommand{\sEdblcs}{\sE_{DBL,c}}
\newcommand{\sAdbl}{\sA_{DBL}}
\newcommand{\sBdbl}{\sB_{DBL}}

\newcommand{\tp}{\sV_{1,\cdots,k}}
\newcommand{\tpcs}{(\sV_{1,\cdots,k})_{\cK_1\times\cdots\times \cK_k}}
\newcommand{\tpbc}{(\sV_{1,\cdots,k})_{L_1,\cdots, L_k}}
\newcommand{\tpbccs}{(\sV_{1,\cdots,k})_{L_1,\cdots, L_k,\cK_1\times\cdots \times \cK_k}}

\newcommand{\ps}{\sO_P(\condfieldscs)}
\newcommand{\pertfunctionals}{\sO^+_{P}(\condfieldscs)}
\newcommand{\pertfunctionalspsmcs}{\sO^+_{P,sm}(\condfieldscs)}

\newcommand{\conddelsmoothdistr}{\widetilde{\sE}^!_\sL}
\newcommand{\prop}{P_\Phi^\Psi}
\newcommand{\rgflow}[3]{W\left(P({#1},{#2}),#3\right)}
\newcommand{\clrgflow}[3]{W_0\left(P({#1},{#2}),#3\right)}
\newcommand{\rgdefault}{\rgflow{\Phi}{\Psi}{I}}
\newcommand{\rgobs}[2]{W_{#1}^{#2}}
\newcommand{\Dens}{\mathrm{Dens}}
\newcommand{\tubnhd}{T}

\renewcommand{\bf}[1]{\mathbf{#1}}
\DeclareMathOperator{\Sing}{Sing}
\DeclareMathOperator{\Reg}{Reg}
\newcommand{\densM}{Dens_M}

\newcommand{\densbdyM}{Dens_{\bdyM}}
\DeclareMathOperator{\cone}{cone}
\newcommand{\delsmoothdistr}{\Omega^{n,n-1}_{M,tw}}
\newcommand{\delsmoothdistrE}{\widetilde {\sE}^!}

\def\d{{\rm d}}

\def\dt{\d t}
\def\del{\partial}
\def\delbar{{\overline{\partial}}}

\def\GF{{\rm GF}}

\def\RRge{\RR_{\geq 0}}
\def\RRgt{\RR_{> 0}}

\def\Cur{{\rm Cur}}

\def\hotimes{{\widehat{\otimes}}}

\def\im{{\rm im}}

\def\bu{{\bullet}}
\def\Obs{{\rm Obs}}
\def\q{{\rm q}}
\def\cl{{\rm cl}}
\newcommand{\Mbdy}{M_\del}
\newcommand{\dirforms}{\Omega^\bullet_{\RRge,c,D}}
\newcommand{\tpbccslf}{(\sV_{1,\cdots,k})_{W_1,\cdots, W_k,c}}
\newcommand{\dmdr}{\textrm{DR}^\bullet_M(D_M)}
\newcommand{\dmdrbdy}{\textrm{DR}^\bullet_{\bdyM}(D_M)}
\newcommand{\dmdrrel}{\textrm{DR}^\bullet_{M,\bdyM}(D_M)}
\newcommand{\dpush}{\iota_D}
\newcommand{\dpull}{\iota^D}
\newcommand{\drrel}{\textrm{DR}(M,\bdyM)}

\newcommand{\tdownup}{t_{a}^{\ \ a}}
\newcommand{\tupdown}{t_{\ \ a}^a}
\newcommand{\Mob}{\mathrm{M\ddot o b}}
\newcommand{\innerhom}[2]{\underline{CVS}\left( #1, #2\right)}
\newcommand{\innerhomsym}[3]{\underline{CVS}\left( #1, #2\right)_{S_#3}}

\usetikzlibrary{decorations.pathmorphing}
\usetikzlibrary{decorations.markings}
\pgfarrowsdeclare{chevron}{chevron}{...}
{
\pgfsetdash{}{0pt} 
\pgfsetlinewidth{.5pt}
\pgfpathmoveto{\pgfpoint{-1pt}{0pt}}
\pgfpathlineto{\pgfpoint{-2pt}{2pt}}
\pgfpathlineto{\pgfpoint{0pt}{2pt}}
\pgfpathlineto{\pgfpoint{1pt}{0pt}}
\pgfpathlineto{\pgfpoint{0pt}{-2pt}}
\pgfpathlineto{\pgfpoint{-2pt}{-2pt}}
\pgfpathclose
\pgfusepathqstroke
}
\pgfarrowsdeclare{newdiamond}{newdiamond}{...}
{
\pgfsetdash{}{0pt} 
\pgfsetlinewidth{.5pt}
\pgfpathmoveto{\pgfpoint{0pt}{0pt}}
\pgfpathlineto{\pgfpoint{-1pt}{2pt}}
\pgfpathlineto{\pgfpoint{1pt}{2pt}}
\pgfpathlineto{\pgfpoint{0pt}{0pt}}
\pgfpathlineto{\pgfpoint{1pt}{-2pt}}
\pgfpathlineto{\pgfpoint{-1pt}{-2pt}}
\pgfpathclose
\pgfusepathqstroke
}

\tikzset{
    vector/.style={decorate, decoration={snake}, draw},
	provector/.style={decorate, decoration={snake,amplitude=2.5pt}, draw},
	antivector/.style={decorate, decoration={snake,amplitude=-2.5pt}, draw},
    fermion/.style={draw=black, postaction={decorate},
        decoration={markings,mark=at position .55 with {\arrow{>}}}},
    Ahalfedge/.style={draw=black, postaction={decorate},
        decoration={markings,mark=at position .25 with {\arrow{>}}, mark = at position 1 with {\arrow[draw=black,xshift=1.5pt]{Circle[length=3pt]}}}},
    Bhalfedge/.style={draw=black, postaction={decorate},
        decoration={markings,mark=at position .80 with {\arrow[draw=black]{>}}, mark = at position 0 with {\arrow[draw=black,xshift=1.5pt]{Circle[length=3pt]}}}},
 ABedge/.style={draw=black, postaction={decorate},
        decoration={markings,mark=at position .50 with {\arrow[draw=black]{chevron}}}},
ABfulledge/.style={draw=black, postaction={decorate},
        decoration={markings,mark = at position .50 with {\arrow[draw=black]{chevron}}, mark = at position .3 with {\arrow[draw=black]{>}}, mark = at position .75 with {\arrow[draw=black]{>}}, mark = at position .0 with {\arrow[draw=black,xshift=1.5pt]{Circle[length=3pt]}}, mark = at position 1 with {\arrow[draw=black,xshift=1.5pt]{Circle[length=3pt]}}}},
BBfulledge/.style={draw=black, postaction={decorate},
        decoration={markings,mark = at position .50 with {\arrow[draw=black]{newdiamond}}, mark = at position .3 with {\arrow[draw=black]{>}}, mark = at position .75 with {\arrow[draw=black]{<}}, mark = at position .0 with {\arrow[draw=black,xshift=1.5pt]{Circle[length=3pt]}}, mark = at position 1 with {\arrow[draw=black,xshift=1.5pt]{Circle[length=3pt]}}}},
 BBedge/.style={draw=black, postaction={decorate},
        decoration={markings,mark=at position .50 with {\arrow[draw=black]{newdiamond}}}},
    fermionbar/.style={draw=black, postaction={decorate},
        decoration={markings,mark=at position .55 with {\arrow{<}}}},
    fermionnoarrow/.style={draw=black},
    gluon/.style={decorate, draw=black,
        decoration={coil,amplitude=4pt, segment length=5pt}},
    scalar/.style={dashed,draw=black, postaction={decorate},
        decoration={markings,mark=at position .55 with {\arrow[draw=black]{>}}}},
    scalarbar/.style={dashed,draw=black, postaction={decorate},
        decoration={markings,mark=at position .55 with {\arrow[draw=black]{<}}}},
    scalarnoarrow/.style={dashed,draw=black},
    electron/.style={draw=black, postaction={decorate}/,
        decoration={markings,mark=at position .55 with {\arrow[draw=black]{>}}}},
	bigvector/.style={decorate, decoration={snake,amplitude=4pt}, draw},
	    ray/.style={draw=black, postaction={decorate},
        decoration={markings,mark=at position 1 with {\arrow[draw=black]{>}}, mark = at position 0 with {}}}
}

%% file: abstract.tex
\begin{abstract}
    
Costello and Gwilliam have given both 1) a general definition of perturbative quantum gauge theory on a manifold $M$ and 2) a construction of a factorization algebra $\Obq$ of quantum observables assigned to every quantum gauge theory.
In this dissertation, we extend these constructions to a certain general class of field theories on manifolds with boundary. 

\end{abstract}

%% file: Chapters/Chapter1.tex
\chapter{Introduction}
\label{chap: intro}
Quantum field theory (QFT) has been an enormously successful theoretical physical framework for explaining the behavior of the elementary particles of the universe. 
Just one of its many successes has been the correct prediction of the existence of the Higgs boson, whose discovery at the Large Hadron Collider was announced on July 4, 2012 \autocite{higgsdiscovery}.
Techniques of quantum field theory have also been used in other disciplines of theoretical physics, including string theory, quantum gravity, condensed matter theory, hydrodynamics, and statistical physics. 
Despite its clear utility in theoretical physics, QFT remains befuddling to mathematicians. 
A number of mathematical axiomatizations of QFT have been proposed, all with complementary virtues and drawbacks.
We will mention only the two approaches most relevant to this dissertation.
We note, however, that in so doing, we pass over a vast quantity of rich mathematics.

The first such framework, inspired by physicists' study of conformal field theory and topological field theory, is \emph{functorial field theory.}
Functorial field theory has the advantage of being described by a number of axioms natural from both the physical and mathematical perspective. 
The main drawback of the functorial field theory formalism, however, is that it remains somewhat distant from the traditional way physicists describe field theories---in terms of fields, action functionals, and symmetries. 
This has made it difficult to produce examples of ``non-trivial'' functorial field theories, although in this respect see \autocite{kandel}.

An alternative---and complementary---approach has been developed by Kevin Costello and Owen Gwilliam \autocite{cost, CG1, CG2}.
In this framework, one starts with a perturbative Lagrangian classical Batalin-Vilkovisky (BV) field theory.
Such a field theory can be described in mathematical terms familiar from physics: in terms of a space of fields and a classical action satisfying the so-called classical master equation.
Costello \autocite{cost} gives a precise definition of the notion of a perturbative quantization of such a theory.
More generally, a gauge theory may not possess a consistent quantization, and Costello also describes a homological-algebraic method to compute the obstruction theory of such quantizations.

One of the advantages of Costello's definition is that every quantum gauge theory in the sense of \autocite{cost} on a spacetime manifold $M$ gives rise to an object known as a \emph{factorization algebra} $\Obq$ on $M$ of quantum observables \autocite{CG1, CG2}.
A factorization algebra is a cosheaf-like local-to-global object on $M$. Special classes of factorization algebras give rise to more widely known objects. For example, the $(\infty, 1)$-category of locally constant factorization algebras on $\R^n$ has been shown by Lurie (see Theorem 5.5.4.10 of \autocite{higheralgebra}) to be equivalent to that of $E_n$ algebras, the latter objects having a deep importance in topology.
Locally constant factorization algebras are the sort of object which one expects to encode the observables of a topological quantum field theory.
In a similar vein, factorization algebras on $\CC$ of a holomorphic nature are meant to encode the observables of a chiral conformal field theory.
In \autocite{CG1}, Costello and Gwilliam show that such factorization algebras on $\CC$ give rise to vertex algebras, the latter objects having a deep importance in representation theory.

In this dissertation, we extend the work of Costello and Gwilliam to the case that $M$ has a boundary. Strictly speaking, we will only extend the constructions of Costello and Gwilliam for a certain broad class of theories, which we call ``topological normal to the boundary,'' following Butson and Yoo \autocite{butsonyoo}. 
We will use the abbreviation TNBFT for such field theories. 
We also study a particular class of boundary conditions for TNBFTs, and we will use the term ``bulk-boundary system'' to refer to a TNBFT together with such a boundary condition. 
The term ``bulk-boundary system'' deserves to apply to a much more general class of systems.
However, as already witnessed by the abbreviation TNBFT, appending qualifications to the term would create an unwieldy terminology; we will therefore use the overly broad terminology, hoping the reader will keep in mind that the particular class of bulk-boundary systems we consider is somewhat restricted.
For precise definitions, we refer the reader to Chapter \ref{chap: classical}.

We note that this is by no means the first mathematical work dealing with the perturbative quantization of field theories on manifolds with boundary.
In particular, the series of works produced by Cattaneo, Mn\"ev, Reshetikhin, and others (see \autocite{CMRquantumgaugetheories}, as well as the review \autocite{CMRreview}, and references therein) also aims to address this problem.
These works have served as valuable guidance in the preparation of these techniques.
The main distinction which can be made between the present work and that of Cattaneo Mn\"ev, and Reshetikhin concerns the final output of the formalism.
Here, the output of our constructions is a factorization algebra; in the work of Cattaneo, Mn\"ev, and Reshetikhin, the output is a functorial field theory.
Work of Dwyer-Stolz-Teichner \autocite{DST} has shown that factorization algebras (functorially) give rise to functorial field theories, so our approach and that of Cattaneo, Mn\"ev, and Reshetikhin are likely to be intimately related.
It is an important goal of our future work to relate the two approaches.

\section{Historical Aside}
Before we describe the approach of Costello and Gwilliam in more detail, we would like to make an important note.
Many of the ideas presented in \autocite{cost} are not due originally to Costello himself. 
The book \autocite{cost} is more a novel and highly original synthesis of physical ideas in a mathematical language than a new physical concept in and of itself.
However, because the work of Costello and Gwilliam is the work that we are attempting to generalize here, the references \autocite{cost, CG1, CG2} will remain the main source of background for this dissertation.
In this section, we attempt to document some of the main lines of influence that led to the Costello-Gwilliam techniques.

Renormalization in physics is nearing its centenary. It was pioneered by Hans Kramers, Hans Bethe, Julian Schwinger, Richard Feynman, Shin'ichiro Tomonaga, and Freeman Dyson in 1947-49 in their study of quantum electrodynamics. 
Even then, it was understood that the ``bare'' or ``scale zero'' interaction needed to be cured so that it would provide physically meaningful answers.
The renormalization group method used by Costello is more directly derived from techniques of Bogoliubov, Kadanoff, and Wilson; these were introduced in the `70s to study second-order phase transitions in statistical physics.

The techniques with which we treat gauge theories here are newer (dating to the `70s or so), but also have a rich history.
The first of these techniques is due to Ludwig Fadeev and Victor Popov.
These two physicists introduced the notion of ``ghost fields,'' which today are understood as serving the function served by the degree $+1$ elements $\fg^\vee$ of $C^\bullet(\fg, V)$ for a Lie algebra $\fg$ and $\fg$-module $V$.
Namely, the ghost fields are introduced to pass from the strict invariants $V^\fg$ to the \emph{derived} invariants $C^\bullet(\fg, V)$.
(In the physics context, $V$ is the space of classical observables and $\fg$ is the Lie algebra of infinitesimal gauge symmetries.)
Later, Becchi-Rouet-Stora and independently Tyutin developed what is now known as the BRST formalism.
In mathematical terms, one might say that they rediscovered the Chevalley-Eilenberg differential on $C^\bullet(\fg, V)$.
(This is certainly an oversimplification, at the very least since the BRST formalism involves also several important insights about gauge fixing.)
This discovery helped explain to physicists the reason for the appearance of the ``ghost fields'' $\fg$ in non-zero cohomological degree.
The BRST formalism was generalized even further by Igor Batalin and Grigori Vilkovisky into what is now known as the BV formalism.
This is the technique for quantum field theory which most directly forms the basis for the work of Costello and Gwilliam, and by extension for the present work.
Though it was not understood in precisely this way at the time, today we understand the classical BV formalism as the statement that a classical gauge theory is concretely axiomatized in the language of $(-1)$-shifted symplectic stacks. 
This latter language was only recently articulated by Pantev, To\"en, Vaqui\'e, and Vezzosi \autocite{ptvv}, with antecedents in the work of Alexandrov, Kontsevich, Schwarz, and Zaboronsky \autocite{AKSZ}; however, it is a concise mathematical way to describe the underlying mathematics of the (classical aspects of the) construction of Batalin and Vilkovisky.
As for the quantum aspects of the BV formalism, these are best understood as providing a homological theory of integration, and more precisely of integration over Lagrangians in the $(-1)$-shifted stacks described by the classical BV formalism.
This perspective has been elaborated by Edward Witten and Albert Schwarz. 
For an exposition of this perspective on the quantum BV formalism in a simple example, we recommend \autocite{GJF}.

The last major element of the approach of Costello and Gwilliam is the notion of a factorization algebra. 
This notion has much newer antecedents, notably the theory of $E_n$ algebras (see \autocite{cohenthesis} for what we believe is the first discussion of an operad-like structure on the spaces of embeddings of little disks) and the theory of chiral algebras as developed by Beilinson and Drinfeld in \autocite{BD}.
The working definition of factorization algebras in the language of Costello and Gwilliam very much resembles these other notions.

\section{Statement of Main Results}
\label{sec: mainresults}
In this section, we state the main results of this dissertation.
The statement of the results requires the introduction of terminology which has not been formally introduced yet; we direct the reader to the remainder of the introduction for an elaboration of some of this terminology.

The first main result of this dissertation is the definition classical and quantum bulk-boundary systems (cf. Definitions \ref{def: clblkbdysystem} and \ref{def: QTNBFT}).
The motivation for these definitions is discussed at length in Chapters \ref{chap: classical} and \ref{chap: interactingquantum}, respectively.

Given these definitions, the following Theorem summarizes the most important results of the dissertation (cf. Theorems \ref{thm: classobsformFA}, \ref{thm: P0classobsFA}, and \ref{thm: quantumFA}).

\begin{utheorem}
For each classical bulk-boundary system $(\sE,\sL)$ on a spacetime manifold $M$, there exists a $P_0$ factorization algebra $\Obcl_{\sE,\sL}$ of classical observables for $(\sE,\sL)$.
Similarly, for a quantum bulk-boundary system, there exists a factorization algebra
$\Obq_{\sE,\sL}$ of quantum observables for $(\sE,\sL)$.
\end{utheorem}

In Chapter \ref{chap: examples}, we investigate a simple interacting example: BF theory with $B$ boundary condition on the manifold $\RR_{\geq 0}$.
BF theory requires as its input a Lie algebra $\fg$, which we require to be unimodular, for simplicity.
We find (cf. Theorems \ref{thm: 1dbfquantization} and \ref{thm: 1dbfobs}):

\begin{utheorem}
Let $\fg$ be a unimodular Lie algebra.
The classical bulk-boundary system specified by BF theory for $\fg$ with $B$ boundary condition on $\RR_{\geq 0}$ admits a quantization.
Further, there are the following quasi-isomorphisms
\begin{align*}
    C_\bullet(\fg)&\to \Obq_{\sE,\sL}(\RR_{\geq 0})\\
    C^\bullet(\fg, C_\bullet(\fg)) &\to \Obq_{\sE,\sL}(\RR_{>0})
\end{align*}
of cochain complexes describing the boundary and bulk observables, respectively.
\end{utheorem}
We refer the reader to the corresponding theorem statements in Chapter \ref{chap: examples} for a more precise statement of the above theorem.

When we are given a \emph{free} bulk-boundary system, a great simplification occurs and one may describe a much more elegant model for the factorization algebra of quantum observables of a free bulk-boundary system.
We explore the consequences of this simplification in Chapter \ref{chap: freequantum}.
The main example that we explore in this context is abelian Chern-Simons theory with the chiral Wess-Zumino-Witten boundary condition (cf. Section \ref{sec: cswzw}):

\begin{utheorem}
Consider the abelian Chern-Simons/chiral Wess-Zumino-Witten bulk-boundary system for the abelian Lie algebra $\fA$ on $\Sigma\times \RR_{\geq 0}$, where $\Sigma$ is a Riemann surface.
Let $\pi: \Sigma \times \RR_{\geq 0}$ denote the projection $\Sigma\times\RR_{\geq 0}\to \Sigma$.
Then, there is an equivalence
\[
\mathrm{KM}\to \pi_*\Obq_{\sE,\sL}
\]
of factorization algebras on $\Sigma$, where $\mathrm{KM}$ is the Kac-Moody factorization algebra for the abelian Lie algebra $\fA$.
\end{utheorem}

The significance of this result is the following.
The BV formalism traditionally describes the theory of (--1)-shifted symplectic geometry.
However, one may adapt the BV formalism to also describe the theory of shifted Poisson geometry.
Further, a boundary condition for a classical BV theory possesses the structure of a Poisson BV theory \autocite{butsonyoo}.
One may naturally be interested in the construction of factorization algebras of observables for Poisson BV theories.
However, there is no general procedure for the quantization of such theories.
Nevertheless, for the case of the chiral WZW boundary condition discussed above, there is an ansatz for the quantum observables of the boundary condition, namely the factorization algebra $\mathrm{KM}$.
The above theorem shows that the quantization of the bulk-boundary CS/WZW system n $\Sigma\times\RR_{\geq 0}$ induces the expected quantization on the boundary $\Sigma$.
In this dissertation, we have constructed a factorization algebra of quantum observables for any quantum bulk-boundary system.
One may therefore understand the results of this dissertation as shedding insight into the quantization of Poisson BV theories corresponding to boundary conditions.
We elaborate on this perspective in Section \ref{subsec: centers}.

\section{A sketch of the methods of Costello and Gwilliam}
\label{sec: CGoverview}
Let us describe the approach of Costello and Gwilliam in more detail, since it is the basis for the work presented here.
We choose a manifold $M$ to serve as the "background" spacetime for our field theories, and assume for simplicity that $M$ is closed.
A perturbative \emph{classical} field theory is described by a space of fields $\sE$ which is the space/sheaf of (global) sections of a vector bundle $E\to M$ over a closed spacetime manifold, and a local action functional $S: \sE \to \RR$. 
We will not be precise about what we mean by ``local'' here, though we will do so in Section \ref{sec: localfcnls}.
The equations of motion for the field theory are the PDEs imposed on $\sE$ via a variational calculus procedure for $S$.
For example, we may choose $M$ to be Riemannian, and $E$ to be the trivial bundle, so that $\sE=\cinfty(M)$.
We may further choose 
\begin{equation}
S(\varphi) = \frac{1}{2}\int_M \varphi (\Delta_g+m^2) \varphi \mathrm{dVol_g} +\frac{1}{6}\int_M \varphi^3 \mathrm{dVol}_g,
\end{equation}
where $m\in \RR^\times$ is the ``mass'' of the theory and $\Delta_g$ is the Laplacian on $M$, chosen by convention to have non-negative eigenvalues.
For the remainder of this introduction, we will consider only this example, though we note that one can generalize this example dramatically.

In quantum field theory, one wishes to make sense of path integrals of the form 
\begin{equation}
Z(\phi_0) = \hbar \log\left( \int_{\sE}e^{-S(\varphi+\phi_0)/\hbar} [\mathrm{D}\varphi]\right);
\end{equation}
if $M$ is zero-dimensional, then these are simply finite-dimensional integrals. 
In such a situation, one can make formal sense of $Z$ as a power series in $\hbar$ and the parameters appearing in the action functional $S$, even if the actual integral defining $Z$ diverges.
Further, the combinatorics of this formal power series is encoded in Feynman diagrams.
However, if $M$ has positive dimension greater than one, the path integral ``measure'' $[D\varphi]$ has no known definition.
Moreover, if one tries to generalize in the most na\"ive way possible the power series obtained via Feynman diagram calculus in the case that $M$ is zero dimensional, one quickly runs into divergences arising from the subtleties of functional analysis.

One of these functional analytic subtleties is the fact that the inverse to $\Delta_g+m^2$ is distributional in nature.
Define
\begin{equation}
    P:=(\Delta_g+m^2)^{-1}.
\end{equation}
The operator $P$ can be understood as a distributional function on $M\times M$.
Because of the distributional nature of $P$, the putative values assigned to Feynman diagrams would involve multiplication of distributions, a well-known analytical difficulty.

The key to the approach elaborated by Costello is to note that the kinetic term $\Delta_g+m^2$ induces a one-parameter semi-group $K_L= \exp{-(\Delta_g+m^2)L}$ of trace-class (in fact, smoothing) operators which approximate the identity for $L>0$ (and for $L=0$ give the identity).
The propagator can be expressed in terms of $K_L$ by the following equation:
\begin{equation}
P = \int_0^\infty K_L \d L.
\end{equation}
More generally, define
\begin{equation}
P(\epsilon,L) = \int_{\epsilon}^L K_L \d L
\end{equation}
with $0\leq \epsilon< L\leq \infty$.
For $\epsilon>0$, $P(\epsilon, L)$ is a smoothing operator, which will be crucial in the sequel.
Given $\Phi\in \cinfty(M\times M)$ (for example, one could take $\Phi=P(\epsilon,L)$ for $\epsilon>0$) and $I\in \widehat\Sym (\sE^\vee)[[\hbar]]$ which is cubic in $\sE$ modulo $\hbar$, let 
\begin{equation}
W(\Phi, I)
\end{equation}
denote the element of $\widehat\Sym(\sE^\vee)[[\hbar]]$ obtained by the following (roughly sketched) assignment:
\begin{equation}
W(\Phi, I) = \sum_{\Gamma} \frac{1}{|\Aut(\Gamma)|} \hbar^{\beta_1(\Gamma)}w_{\Gamma}(\Phi, I).
\end{equation}
Here, the sum is over connected graphs $\Gamma$ which are allowed to have ``internal'' (($\geq$3)-valent) and ``external'' (univalent) vertices. 
Each $w_{\Gamma}(\Phi, I)$ is an element of $\widehat\Sym(\sE^\vee)$ obtained by ``placing'' $I$ at each internal vertex, $\Phi$ on each internal edge, and the functional inputs $\phi_1,\ldots, \phi_k$ on each external vertex/edge.
Crucially, because $\Phi\in \cinfty(M\times M)$, $w_{\Gamma}(\Phi,I)$ induces no functional analytic difficulties.
One can show the following important property satisfied by the assignment $W(\Phi,I)$:
\begin{equation}
\tag{$\dagger$}
W\left( \Phi, W(\Psi, I)\right) = W(\Phi+\Psi, I)
\end{equation}
(cf. the Equation immediately preceding the beginning of Section 2.3.5 of \autocite{cost}).
Now, choose $I = \frac{1}{6}\int_M \phi^3 \textrm{dVol}_g$, and pretend that one could define
\begin{equation}
I[L]\overset{!}{=} W(P(0,L), I).
\end{equation}
Here, the exclamation point reminds us that this definition is mathematically ill-defined.
Suppose that $\epsilon<L$, and note that $P(0,L) = P(0,\epsilon)+P(\epsilon, L),$ so in particular $P(0,L)-P(0,\epsilon)$ is a smoothing operator.
Using, $(\dagger)$, therefore, we would find that 
\begin{equation}
\tag{$\ddagger$}
I[\epsilon] \overset{!}{=} W(P(\epsilon,L), I[L]).
\end{equation}
Again, because $I[\epsilon]$ and $I[L]$ are ill-defined, the above equation is ill-defined.
However, if $I[L]$ \emph{were} well-defined, $I[\epsilon]$ would be determined from $I[L]$ by ($\ddagger)$, which is in fact a well-defined equation in such a situation.
Equation ($\ddagger$) is called the ``homotopy renormalization group flow equation.''
The reason for the adjective ``homotopy'' will become a bit clearer in Chapter \ref{chap: interactingquantum}.

In the approach of Costello, one throws away the pseudo-definition of $I[L]$ and instead \emph{defines} a quantum field theory to be a collection $\{I[L]\}_{L> 0}$
of functionals satisfying the equation $(\ddagger)$.
The original (local) functional $I$ would be the limit $\lim_{L\to 0} I[L]$ if it existed.
However, Equation $(\ddagger)$ shows that this limit does not, in general, exist, since this limit would require the evaluation of Feynman diagrams with $I[L]$ at the vertices and the distribution $P(0,L)$ at the edges.

From this perspective, the fundamental physical object is the collection $\{I[L]\}$ and not the ``scale-zero interaction'' $I$.
How does this fact reconcile itself with the way physicists understand field theory?
In the physicists' perspective, one attempts to make the pseudo-definition of $I[L]$ work and, upon finding that the result diverges, ``subtracts away'' the divergence to yield a ``physically sensible'' $I[L]$.

To a mathematician, this procedure may seem meaningless, but Costello explains it thus. 
The ``true'' or ``canonical'' object of study is the collection $\{I[L]\}$.
This is the object that, for example, most naturally encodes the information of the operator product expansion.
Costello shows that there is a \textbf{non-canonical} (!!) bijection $\psi$ between the space of field theories in his sense (collections $\{I[L]\}$ related by exact RG flow) and the set of local (``scale-zero'') action functionals $I$.
On the other hand, a physicist would say that a quantum field theory is defined not just by the action functional $I$ but also by the ``renormalization scheme'' one chooses, i.e. by the choice one makes about how to turn the physically meaningless divergences arising from $I$ into physically sensible answers.
(This is not how most quantum field theory textbooks explain the notion of renormalization, at least not at first, but it is, in the author's experience, how physicists actually understand renormalization in practice.)
In other words, physicists would also say that the collection $\{I[L]\}$ is the canonical physical object, and the non-canonical choice $I$ is canceled by another non-canonical choice (the choice of bijection $\psi$).

The assignment of the field theory $\{I[L]\}$ to a local action functional goes as follows.
First, one studies the small-$\epsilon$ asymptotics of 
\begin{equation}
W(P(\epsilon, L), I);
\end{equation}
one finds that, for small $\epsilon$, one can expand $W(P(\epsilon,L),I)$ as a sum of terms like $g(\epsilon) \Psi$ where $g(\epsilon)$ is a smooth function of $\epsilon\in (0,1)$ and $\Psi$ is a smoothly $L$-dependent (non-local, in general) functional.
(Cf. Theorem 9.3.1 of Chapter 2 of \autocite{cost}.)
Then, one makes sense (this is where the arbitrary choice comes in) of the ``purely singular'' part of the functions $g(\epsilon)$ as $\epsilon\to 0$, as follows.
The space of functions in $\cinfty((0,1))$ which admit an $\epsilon\to 0$ limit is certainly a subspace $V$ of $\cinfty((0,1))$.
Costello calls a choice of complement to $V$ in $\cinfty((0,1))$ a \emph{renormalization scheme}.
Such a choice allows one to make sense of which functions of $\epsilon$ are ``purely'' divergent. 
For example, while it is indisputable that the function $\sin(\epsilon)$ is non-singular, both $\log(2\epsilon)$ and $\log(\epsilon)$ are divergent as $\epsilon \to 0$.
But the two functions differ by a non-singular function, so they cannot both be ``purely divergent''.
A renormalization scheme is the choice that allows one to say a statement like ``the singular part as $\epsilon\to 0$ of $\log(2\epsilon)$ is $\log(\epsilon)$ and the non-singular part is $\log(2)$''.
The upshot is that this choice is implicitly what allows the physicist to ``subtract away the divergences of field theory.''

The discussion above describes the procedure for quantizing classical field theories without gauge symmetries. 
However, most of the theories physicists study in practice are indeed gauge theories.
Moreover, physicists are often concerned with the essential question of whether or not the procedure of renormalization violates the classical symmetries.
In \autocite{cost}, Costello uses the Batalin-Vilkovisky (BV) formalism to provide mathematical descriptions of these issues.
Namely, the gauge invariance of the classical theory is encoded in the ``classical master equation'' for the classical interaction $I$.
After renormalization, one obtains a family $\{I[L\}$ of interactions.
At each scale $L$, one may write down the ``scale $L$ quantum master equation,'' which is meant to describe the gauge-invariance of the quantum interaction $I[L]$.
The main insight of Costello is that if $I[L]$ satisfies the scale $L$ quantum master equation, then $I[\epsilon] = W(P(\epsilon,L),I[L])$ satisfies the scale $\epsilon$ quantum master equation.
In other words, the quantum master equations at different scales are compatible with the exact RG flow equation.

Putting these observations together, Costello defines a quantum gauge theory as a collection $\{I[L]\}$ of interactions which are related to each other by the exact RG flow, satisfy the quantum master equation at one (and therefore all) scales, and which satisfy a certain asymptotic locality condition as $L\to 0$.
We will elaborate on this locality condition in the body of this dissertation.
Suffice it to say that this condition is designed to capture the locality of the classical ``scale-0'' interaction $I$.

Given this definition of a quantum field theory, Costello and Gwilliam \autocite{CG2} show that a quantum gauge theory on a manifold $M$ gives rise to a factorization algebra on $M$.
A factorization algebra assigns a cochain complex $\Obq_{\sE,\sL}(U)$ to every open subset $U\subset M$ together with ``multiplication maps'' 
\begin{equation}
\Obq(U_1)\otimes \cdots \otimes \Obq(U_k)\to \Obq(V)
\end{equation}
for every inclusion of a collection of disjoint open subsets $U_1, \ldots U_k $ into a larger open subset $V$ of $M$.
To extract the underlying cochain complex (respectively, vector space) from a locally constant (respectively, holomorphic) factorization algebra, one takes $\Obq(D)$ for $D$ the unit disk in $\RR^n$ (respectively, $H^0(D)$, for $D$ the formal unit disk in $\CC$).
The multiplication maps of the factorization algebra then provide the additional structure maps arising in the corresponding $E_n$ algebra (respectively, vertex algebra).
A key feature of the factorization algebras in both cases is that $\Obq(D)$ doesn't depend on the size and center of the disk $D$; this is what allows one to turn the structure maps of the factorization algebra, which allow multiplications only for disjoint disks, into a structure solely on the cochain complex $\Obq(D)$.
In the case of locally constant factorization algebras, the resulting $E_n$ algebra encodes all of the homotopy theoretic information contained in the locally constant factorization algebra which produced it.
A theorem of a similar sort is not known for holomorphic factorization algebras, but for a partial reconstruction result, see \autocite{bruegmann}.
For more general field theories, the resulting factorization algebras are not known to reduce to familiar objects.
Instead, a general factorization algebra encodes an enormous amount of information: a cochain complex for each open subset $U$ of $M$ and multiplication maps for all possible inclusions of disjoint open subsets into a larger one.
So, factorization algebras are in general very rich objects, and indeed any quantum field theory on a manifold $M$ in the sense of Costello gives rise to one.

\section{Bulk-boundary systems as BV theories}
\label{sec: bbs}
In this section, we summarize the main theoretical elements that underlie our extension of the work of Costello and Gwilliam.
For $\bdyM=\emptyset$, the Batalin-Vilkovisky formalism is naturally articulated in the language of $(-1)$-shifted symplectic geometry.
The space of fields $\sE$ is the space of sections of some $\ZZ$-graded vector bundle on $M$, and this infinite-dimensional space is endowed with a sort of $(-1)$-shifted symplectic structure.
The infinite-dimensional nature of $\sE$ poses some problems for quantization, however, and Costello resolves these problems in a homotopically coherent way.
The essence of the approach we take here is to maintain as much of this articulation as possible when $\bdyM\neq \emptyset$.
To be able to do this, however, we need to impose boundary conditions on the fields $\sE$.
Let us consider an example to see why this is the case. Let $M=S^1$. Let $(V,\omega)$ be a symplectic vector space.
\emph{Topological mechanics} with values in $V$ is the theory on $M$ whose space of fields $\sE$ is  
\begin{equation}
\left(\Omega^\bullet_M(M) \otimes V,d \right)
\end{equation}
and whose degree --1 (symplectic) pairing $\ip$ is
\begin{equation}
\ip[\alpha,\beta] = \int_M \omega (\alpha \wedge \beta).
\end{equation}
In the language of action functionals, these data determine the action 
\begin{equation}
S(\alpha) = \frac{1}{2}\int_M \omega(\alpha\wedge d_{dR}\alpha).
\end{equation}
The main condition for this structure to indeed be a symplectic structure on $\sE$ is the equation
\begin{equation}
\label{eq: invpairing}
\ip[d_{dR}\alpha, \beta]+(-1)^{|\alpha|}\ip[\alpha, d_{dR}\beta]=0.
\end{equation}
Equation \ref{eq: invpairing} describes the compatibility of the de Rham differential with the pairing $\ip$. Its derivation uses Stokes's theorem and the fact that $\bdyM = \emptyset$.
Equation \eqref{eq: invpairing} has a geometric interpretation.
Thinking of $\sE$ as a graded manifold, $\ip$ as a closed two-form on this manifold, and $d$ as a degree +1 vector field $X$ on this manifold, Equation \eqref{eq: invpairing} can be interpreted as the equation
\begin{equation}
L_X\ip = 0,
\end{equation}
which is to say that $X$ is a symplectic vector field.
Alternatively, Equation \ref{eq: invpairing} is a key ingredient in showing that $S$ satisfies the so-called ``classical master equation.''

The manifold $S^1$ is the simplest example of a closed one-manifold.
Let us now consider what happens when we choose instead $M=[0,1]$, the simplest compact one-manifold with a non-empty boundary.
Note that Equation \ref{eq: invpairing} no longer holds in this case. Instead, letting $\iota_0, \iota_1$ denote the inclusions $\star \hookrightarrow[0,1]$ determined by the respective endpoints of the interval $[0,1]$, we find that \begin{equation}
\label{eq: noninvariantpairing}
\ip[d_{dR}\alpha, \beta]+(-1)^{|\alpha|}\ip[\alpha, d_{dR}\beta]=\omega(\iota_1^* \alpha, \iota_1^*\beta)-\omega(\iota_0^* \alpha, \iota_0^*\beta).
\end{equation}
The situation may be summarized as follows.
On the boundary $\bdyM$, one obtains a symplectic vector space $V\oplus (-V)$. (Here, $-V$ denotes $V$ with the symplectic structure $-\omega$.)
Pullback to the boundary determines a map $\rho: \sE \to V\oplus (-V)$.
Equation \ref{eq: noninvariantpairing} then states that $\ip$ determines a (derived) isotropic structure on $\rho$. 
In fact, one can show further that $\ip$ determines a (derived) Lagrangian structure on $\rho$ (cf. Lemma \ref{lem: fieldslagrangian}).
If one chooses a Lagrangian $L\subset V\oplus (-V)$, one may form the homotopy pullback 
\begin{equation}
\sE_L := \sE\times^h_{V\oplus(- V)} L;
\end{equation}
by general considerations of derived symplectic geometry \autocite{ptvv}, $\sE_L$ is a (--1)-shifted symplectic space, and hence an object amenable to the techniques of BV quantization. 
In our context, we will see below that the strict pullback provides a model for the homotopy pullback. One may therefore take $\sE_L$ to be the space of $V$-valued forms on $[0,1]$ whose pullback to the boundary lies in $L$.
For fields $\alpha,\beta\in \sE_L$, it follows directly that Equation \ref{eq: noninvariantpairing} reduces to Equation \ref{eq: invpairing}.
We will refer to $L$ as a ``boundary condition'' for obvious reasons.

We have therefore seen that a field theory on $M$ together with a boundary condition determine the data of a classical BV-like theory. With a few modest generalizations, all of the examples we will consider will fall into this form; more precisely, we will consider field theories which are ``topological normal to the boundary'' (see Definition \ref{def: tnbft}).
In a very rough sense, then, one may simply replace all appearances of $\sE$ (the space of fields of the theory) in the formalism of Costello and Gwilliam with $\sE_L$ (the space of fields with the boundary condition imposed) to obtain the results of this dissertation.
Because $\sE_L$ is infinite-dimensional, however, the story is not as straightforward as this characterization may seem.
In fact, the main work of this dissertation is to address the difficulties that arise precisely because of this property of $\sE_L$.
Unfortunately, these aspects of the work are somewhat technical.
We invite the interested reader to dive into the body of the text to see further details on this approach.

\section{Factorization algebras: Definitions and Examples}

\input{Chapters/FAdefexamples}

\section{Bulk-boundary systems and centers of factorization algebras}
\label{subsec: centers}

Perhaps the archetypal example of a bulk-boundary system in mathematics is the Swiss cheese algebra given by the pair $(HH^\bullet(A), A)$, where $A$ is an associative algebra and $HH^\bullet(A)$ is its Hochschild cochain complex.
The latter object is an $E_2$ algebra, according to the Deligne conjecture. The former object may be understood as an $E_1$ algebra.
The Swiss cheese algebra $(HH^\bullet(A), A)$ is the terminal object in the category of Swiss cheese algebras whose underlying $E_1$ algebra is $A$ \autocite{Thomas_2016}.

The physical interpretation of this pair is as follows. The algebra $A$ encodes the algebra of observables of a classical or quantum mechanical system. 
Meanwhile, the $E_2$ algebra of Hochschild cochains encodes the universal two-dimensional topological field theory admitting $A$ as a boundary condition. 
This perspective is implicit in the work of Kontsevich on the deformation quantization of Poisson manifolds \autocite{KontPSM}, which studies the case that $A$ is the algebra of smooth functions on a Poisson manifold $X$. 
Cattaneo and Felder \autocite{CFPSM} gave an interpretation of the construction of Kontsevich in the language of two-dimensional quantum field theory.
Namely, to a Poisson manifold $X$ one may associate a two-dimensional topological quantum theory on any oriented two-dimensional manifold and a canonical boundary condition for this theory.
Quantization of the associated bulk-boundary system provides a deformation quantization of the algebra of functions on $X$ whose semi-classical approximation is dictated by the Poisson bracket for $X$.
The theory is known as the Poisson sigma model.

Recently, Butson and Yoo \autocite{butsonyoo} have proposed a more general framework in which to understand the Poisson sigma model.
In this framework, one understands the Poisson manifold $X$ as defining a Batalin-Vilkovisky-type field theory on $\RR$ which is ``degenerate'' in nature (we prefer to call such theories \textbf{Poisson BV theories}).
One might call this theory ``Poisson topological mechanics.''
Poisson topological mechanics is ``degenerate'' because it is described by a Poisson target $X$ as opposed to a symplectic target as is the case for ``ordinary'' topological mechanics.
Using the language of derived geometry, Poisson BV theories are described by $(-1)$-shifted \emph{Poisson} stacks as opposed to $(-1)$-shifted \emph{symplectic} stacks.
In terms familiar to physicists, Poisson BV theories are one class of ``non-Lagrangian'' field theories.
Butson and Yoo assign---to each Poisson BV theory $\cT$ on a manifold $M_0$---a(n ``ordinary'') BV theory on $M_0\times \RR_{\geq 0}$ known as the ``universal bulk theory'' of $\cT$. 
The universal bulk theory of $\cT$ also possesses a canonical boundary condition associated to $\cT$.
In the case of Poisson topological mechanics, the universal bulk theory is the Poisson sigma model.
One of the insights implicit in Kontsevich's work was that one may quantize Poisson topological mechanics by quantizing the universal bulk-boundary system on the upper half-plane $\HH^2$ associated to it.
The work of Butson and Yoo systematized this intuition for more general Poisson BV theories.
The advantage of this perspective is that one needs to study only a theory described in the usual BV formalism, albeit on a manifold with boundary.
Moreover, this bulk-boundary system is of precisely the sort amenable to the techniques of this dissertation.
One might therefore understand the results of this dissertation as a continuation of the following program which was initiated by Butson and Yoo:

\begin{Program}
Develop a framework for the study of the quantization of Poisson BV theories generalizing that of Costello and Gwilliam for ordinary BV theories.
\end{Program}

This, perhaps, is the most important application of the techniques presented herein. One may, on the one hand, view these techniques as an extension of the methods of Costello and Gwilliam to a certain class of bulk-boundary systems on manifolds with boundary. 
On the other hand, one may view them as an extension of the techniques of Costello and Gwilliam to arbitrary (perturbative) Poisson BV theories on manifolds with boundary.

\section{Comparison to Related Work}
We are aware of a number of other works discussing the quantization of gauge theories on manifolds with boundary using the techniques of homological algebra. Among them are \autocite{CMRquantumgaugetheories} (discussed above), \autocite{mnevschiavinawernli}, \autocite{equivariantbv}, \autocite{edgemodesyangmills}. 
Our approach has a number of features in common with these references. 
For example, all of the mentioned references use the Batalin-Vilkovisky formalism and study what happens when the classical master equation fails to be satisfied. 
In \autocite{mnevschiavinawernli}, the possibility of local operators which are not simply integrals over all of spacetime is introduced. 
In \autocite{edgemodesyangmills}, the authors impose boundary conditions via homotopy pullbacks. 
These ideas are present also in this work; however, by restricting to the class of TNBFTs, we enable a discussion of renormalization and the construction of factorization algebras of observables.

\section{A Brief Primer on Differentiable Vector Spaces}
\label{sec: DVS}

As we have remarked in Section \ref{sec: review} (see Remark \ref{rmk: DVSforFAs}), we would like our factorization algebras to take values in a symmetric monoidal category of chain complexes of vector spaces.
Typically, the vector spaces we consider are infinite-dimensional, and endowed with topologies; the natural tensor products in this context have universal properties in categories of topological vector spaces where the morphisms are continuous linear maps.
For example, if $E\to M$ is a vector bundle, then the space of sections $\cinfty(M;E)$ possesses the so-called Whitney topology.
There is a natural tensor product on a subcategory of the category of all topological vector spaces, called the completed projective tensor product and denoted $\hotimes_\pi$, with the property that for $F\to N$ another vector bundle, the equation 
\begin{equation}
\label{eq: cptp}
\cinfty(M; E)\hotimes_\pi \cinfty(N; F)\cong \cinfty(M\times N; E\boxtimes F)
\end{equation}
holds, where $E\boxtimes F$ is the external tensor product of vector bundles (see Section \ref{sec: notation}).
Such an equation simply does not hold if one takes the algebraic tensor product of the two spaces of sections.

Based on this discussion and the definition of a factorization algebra, we seek a category $\cC$ of vector spaces such that
\begin{enumerate}
\item $\cC$ is an abelian category,
\item $\cC$ is endowed with a tensor product in which equations like Equation \eqref{eq: cptp} hold.
\end{enumerate}
One standard category meant to encode the theory of infinite-dimensional vector spaces---namely, that of locally convex topological vector spaces---is infamously non-abelian. 
Moreover, we would like Equation \eqref{eq: cptp} to hold even if $M$ and $N$ are non-compact, and one takes compactly-supported sections throughout.
This, too, fails to obtain in the category of locally convex topological vector spaces.

The approach we follow here is to consider the category of differentiable vector spaces $\DVS$, as discussed in Appendix B of \autocite{CG1}.
The category $\DVS$ has the following properties:
\begin{enumerate}
\item It is abelian.
\item It has a multi-category structure.
\item It possesses an enrichment over itself, i.e. one can form a differentiable vector space of linear maps between two differentiable vector spaces.
\end{enumerate}

Very briefly, a differentiable vector space $V$ is a sheaf of vector spaces $\cinfty(\cdot, V)$ (we will use both notations interchangeably to represent $V$) on the site of smooth manifolds.
Moreover $\cinfty(\cdot, V)$ must be endowed with the structure of a module over the sheaf of rings $\cinfty$ and a flat connection
\begin{equation}
\nabla: \cinfty(\cdot, V) \to \cinfty(\cdot, V)\otimes_{\cinfty} \Omega^1(\cdot).
\end{equation}
For any DVS $V$, the vector space $\cinfty(\{*\}, V)$ encodes the underlying vector space of $V$, and $\cinfty(M,V)$ encodes the space of smooth maps from $M$ to $V$.
For example, given a vector bundle $E\to M$, the differentiable vector space $\cinfty(M; E)$ is given by the assignment
\begin{equation}
\cinfty(X,\cinfty(M;E)):= \cinfty(X\times M; \pi_2^* E),
\end{equation}
where $\pi_2: X\times M\to M$ is the natural projection.
(Note that we use the notation $\cinfty(M;E)$ to denote global sections of $E\to M$, and the notation $\cinfty(X,V)$ to denote smooth maps from $X$ to $V$.)
We refer the reader to the aforementioned Appendix of \autocite{CG1} for further details.

The category $\DVS$ has all of the categorical properties that we desired in a category of infinite-dimensional vector spaces.
Nevertheless, it is somewhat difficult to work with differentiable vector spaces directly.
Hence, Costello and Gwilliam study another category of infinite-dimensional vector spaces, namely the category of \emph{convenient} vector spaces $\CVS$.
The category $\CVS$ possesses the following properties:
\begin{enumerate}
    \item The objects in $\CVS$ are vector spaces endowed with additional structure (namely a \emph{bornology}), and morphisms are spaces of bounded linear maps.
    \item The category $\CVS$ possesses all limits and colimits.
    \item The category $\CVS$ is closed symmetric monoidal. The symmetric monoidal structure on $\CVS$ is denoted by the symbol $\hotimes_\beta$.
    \item Vector spaces of smooth, compactly-supported, distributional, and compactly-supported distributional sections of a vector bundle $E\to M$ can be described by objects in $\CVS$.
    \item There is a functor $dif: \CVS\to \DVS$ which preserves all limits. Moreover, the functor preserves inner hom objects, and the symmetric monoidal product $\hotimes_\beta$ represents the multi-category structure on~$\DVS$.
    \item Equation \eqref{eq: cptp} holds, with $\hotimes_\beta$ replacing $\hotimes_\pi$, both for spaces of sections of a bundle $E\to M$ and spaces of compactly-supported sections of $E$.
    \item The functor $dif$ embeds $\CVS$ as a full subcategory of $\DVS$.
\end{enumerate}
Again, we refer the reader to Appendix B of \autocite{CG1} for further details.
It is far easier to work directly with convenient vector spaces than with differentiable vector spaces.
However, the category $\CVS$ is not abelian, so one may not use classical homological algebra to study chain complexes of convenient vector spaces.
Further, the functor $\CVS\to \DVS$ does not preserve colimits, so in the computation of homology groups, it matters whether a complex 
\[
\cdots \to V_i \to V_{i+1}\to \cdots
\]
of convenient vector spaces is considered to be a complex of convenient or differentiable vector spaces via the functor~$dif$.
We adopt the following conventions, which we follow implicitly throughout the dissertation.
First, nearly all of our chain complexes of differentiable vector spaces will arise from chain complexes of convenient vector spaces.
To verify that a map of complexes $dif(f):dif(A^\bullet) \to dif(B^\bullet)$ is a quasi-isomorphism, we never compute the cohomology groups of the underlying complexes of convenient vector spaces $A^\bullet, B^\bullet$ directly.
Instead, we allow ourselves only to use some combination of the following tools:
\begin{enumerate}
    \item Construction of a homotopy inverse for $f$, in which case both $f$ and $dif(f)$ are quasi-isomorphisms by functoriality.
    \item Use of standard homological-algebraic techniques in $\DVS$, such as spectral sequences and the snake lemma.
\end{enumerate}
This allows us to avoid computing cohomology groups directly in $\DVS$, and instead to perform explicit computations in $\CVS$, doing so in a way that is ``kosher.''

\section{Notations and Conventions}
\label{sec: notation}
In this section, we collect some conventions and notation which we use repeatedly in the body of this dissertation.
\begin{itemize}
    \item $M$ will denote the fixed manifold with boundary on which our bulk-boundary systems are formulated, with $\bdyM$ its boundary, and $\iota: \bdyM\to M$ the inclusion. We will often fix a tubular neighborhood $\tubnhd$ of $\bdyM$ inside $M$.
    The letter $n$ will stand for the dimension of $M$.
    \item We will often use Latin letters in regular font to denote vector bundles on $M$ (e.g. $V,E,L,\ldots$). 
    We use the corresponding script letters (e.g. $\sV,\sE,\sL,\ldots$) to denote the sheaves of sections of the corresponding bundles.
    We will append the subscript $_c$ to denote the cosheaf of compactly-supported sections (e.g. $\sV_c,\sE_c,\sL_c$, etc.).
    Very rarely, we will use the script letter to denote the space of global sections of the corresponding bundle.
    We will usually be explicit when this is the case.
    \item If $E\to M$ is a vector bundle and $x\in M$, then $E_x$ denotes the fiber of $E$ over~$x$.
    \item We introduce the following notation for standard bundles on $M$:
        \begin{itemize}
            \item $\underline \RR$ denotes the trivial $\RR$ bundle on $M$, with sheaf of sections $\cinfty_M$ (hence, we force ourselves into the somewhat redundant notation $\cinfty_M(M)$ to denote the space of global sections of $\underline \RR$).
            \item $TM$ and $T^*M$ denote the tangent and cotangent bundles, respectively, with sheaves of sections $Vect_M$ and $\Omega^1_M$.
            \item More generally, $\Lambda^k(T^*M)$ denotes the bundle of differential $k$ forms on $M$, with sheaf of sections $\Omega^k_M$.
            \item $\Dens_M$ and $\mathfrak o$ denote the bundle of densities and the orientation line bundle on $M$, respectively. 
            That is, $\Dens_M$ is the associated bundle associated to the rank one representation of $GL_n$ given by the absolute value of the determinant, and $\mathfrak o$ by the sign of the determinant.
            We denote by $\Omega^n_{M,tw}$ the sheaf of sections of $\Dens_M$ and by $\Omega^0_{M,tw}$ the sheaf of sections of $\mathfrak o$.
            \item More generally, we denote by $\Omega^k_{M,tw}$ the sheaf of sections of the bundle $\Lambda^k(T^*M)\otimes \mathfrak o$ of twisted $k$-forms.
        \end{itemize}
        \item If $E\to M$ is a vector bundle, then $E^\vee$ denotes the fiberwise dual vector bundle, and $E^!$ denotes the bundle~$E^\vee\otimes \Dens_M$.
    \item We use the symbols $\Upsilon$ and $\digamma$ repeatedly when all other symbols seem to be taken. 
    They are intended to have a local scope, not exceeding a few results immediately preceding and succeeding their appearance.
    \item We use the symbol $C$ with extremely local scope, i.e. its referent changes frequently. The main function of this symbol is to represent some numerical factors which are irrelevant in the computation of a particular limit or integral.
    The referent of $C$ is often a product of a half-integral power of $\pi$ and a half-integral power of $4$ .
    \item We use the superscript $^\vee$ to denote the dual of a vector space or the fiberwise dual of a vector bundle as above.
    \item If $E\to M$ and $F\to N$ are two bundles, the symbol $E\boxtimes F$ denotes the bundle on $M\times N$ whose fiber at the point $(x,y)$ is $E_x\otimes F_y$.
    Alternatively, $E\boxtimes F = \pi_1^* E\otimes \pi_2^* F$, where the $\pi_i$ are the projections from $M\times N$ to its corresponding factors.
    \item If $V=\oplus_{i\in \ZZ}V^i$ is a $\ZZ$-graded vector space, then $V[k] = \oplus_{i\in \ZZ}V^{i+k}$, i.e. $V[k]$ is the graded vector space whose $i$-th graded summand is the $(i+k)$-th graded summand of $V$.
\end{itemize}

\section{Outline of the Dissertation}
In the remainder of this introduction, we give a detailed outline of the dissertation.

Chapter \ref{chap: classical} is based on the pre-print \autocite{classicalarxiv} and defines the classical bulk-boundary systems which we consider in this dissertation (Definitions \ref{def: tnbft}, \ref{def: bdycond}, and  \ref{def: clblkbdysystem}).
These definitions, except Definition \ref{def: clblkbdysystem}, are due to Butson and Yoo.
The main result of the Chapter is the construction in Section \ref{sec: FAs} of a $P_0$ factorization algebra of classical observables for a classical bulk-boundary system.
We actually construct \emph{two} equivalent factorization algebras: one which is easy to define and understand (Definition \ref{def: classFA} and Theorem \ref{thm: classobsformFA}), and one for which the $P_0$ structure is manifest (Theorem \ref{thm: P0classobsFA}).
We also expand on the derived geometric interpretation of bulk-boundary systems sketched in Section \ref{sec: bbs}.
In particular, we show that the canonical map from bulk fields to boundary fields is a Lagrangian map (Definition \ref{def: loclagrangianstr} and Lemma \ref{lem: fieldslagrangian}), and we show that the imposition of boundary conditions serves to form a homotopy intersection of two Lagrangians in the boundary fields (Lemma \ref{lem: tildefieldsmodel}).
In Section \ref{sec: localfcnls}, which is somewhat outside the main line of development of the rest of the Chapter, we study the theory of local functionals on manifolds with boundary.
This section does not appear in the pre-print \autocite{classicalarxiv}.
We give an interpretation of local functionals in terms of $D_M$-modules analogous to the discussion in Section 6 of Chapter 5 of \autocite{cost}.
The theory of local functionals is helpful in the study of the deformation theory of quantum bulk-boundary systems.
Finally, in Section \ref{sec: classexamples}, we study two simple bulk-boundary systems and their factorization algebras of observables, namely topological mechanics and one-dimensional BF theory.
We formulate these theories on $\RR_{\geq 0}$ and show that they can be described by factorization algebras of the form $\cF_{A,M}$ (see Section \ref{sec: review} for the definition of $\cF_{A,M}$) for a suitable algebra $A$ and right pointed $A$-module $M$.

In Chapter \ref{chap: freequantum}, we execute the quantization of \emph{free} bulk-boundary systems. Chapter \ref{chap: freequantum} is based on joint work with Owen Gwilliam and Brian Williams and has appeared as \autocite{GRW}.
For free theories, one may use a functional analytic result known as the Atiyah-Bott lemma to provide a more convenient model of the classical observables than given for general bulk-boundary systems in Chapter \ref{chap: classical}.
This more convenient model also exhibits a very natural deformation quantization. 
This discussion culminates in Theorems \ref{thm: obcl} and \ref{thm: freequantumFA}.
Butson and Yoo have shown that boundary conditions, in the generality which we consider in this dissertation, can be described by Poisson BV theories, and that these boundary conditions possess $P_0$ factorization algebras of classical observables on the boundary $\bdyM$.
In the case of free theories, there is a natural ansatz for a deformation quantization of this factorization algebra on $\bdyM$.
In the bulk, that is in the interior $\mathring M$,  Costello and Gwilliam have given a construction of a factorization algebra of quantum observables for a free theory.
The main theorem, Theorem \ref{thm: maingenlq}, of Chapter \ref{chap: freequantum} shows that the bulk-boundary quantum observables reproduce---in the precise way specified in the statement of the theorem---both the bulk quantum observables and the boundary quantum observables.
We then apply this theorem to a number of examples: topological mechanics, the Poisson sigma model with linear target, abelian Chern-Simons/chiral WZW theory, and higher-dimensional analogues of the CS/WZW system.

Chapter \ref{chap: interactingquantum} extends the results of the previous chapter to define general quantum bulk-boundary systems and their factorization algebras.
To extend the heat kernel methods outlined in Section \ref{sec: CGoverview}, we make use of a doubling procedure familiar from physics, which we introduce in Section \ref{sec: firstdbling}.
(In physics, the method is also known as the ``method of image charges'' or the ``method of mirror images.'')
This doubling procedure will require a further restriction on the bulk-boundary systems we consider, but the ``universal bulk theories'' of Poisson BV theories and almost all of the examples of interest to us survive this cut.
In Section \ref{sec: parametrices}, we define the appropriate generalizations of the regularized propagators $P(\epsilon, L)$ discussed in Section \ref{sec: CGoverview}.
Finally, in Section \ref{sec: qbbsdef} (Definition \ref{def: QTNBFT}), we define what we mean by a quantum bulk-boundary system.
We then proceed to study the renormalization and obstruction theory of quantum bulk-boundary systems, extending the results of Costello.
Finally, in Section \ref{sec: quantumFA} (Theorem \ref{thm: quantumFA} and the discussion immediately succeeding it), we construct the quantum observables of a quantum bulk-boundary system.

Chapter \ref{chap: examples} applies the general formalism of the preceding chapters to an example where the spacetime manifold is $\RR_{\geq 0}$.
The example is BF theory in one dimension.
This example has the following interpretation in terms of representation theory.
Given a Lie algebra $\fg$, we consider the 0-shifted symplectic formal moduli problem $T^* (B\fg)$. 

The cotangent fiber defines a Lagrangian in this 0-shifted symplectic formal moduli problem; the algebra of functions on this Lagrangian is $\Sym(\fg[1])$.
As expected from the theory of Butson and Yoo, this algebra has a natural $(+1)$-shifted Poisson bracket which, on linear elements $x,y\in \fg[1]$ is simply given by the Lie bracket:
\begin{equation}
\{x,y\} = [x,y].
\end{equation}
(Note: this Poisson bracket is \emph{not} compatible with the algebra structure on $\Sym(\fg[1])$.)
One may then understand the Chevalley-Eilenberg chains $C_\bullet(\fg)$ as a deformation quantization of the $P_0$ algebra $\Sym(\fg[1])$. 
We will see in a moment that this information encodes the information of one-dimensional BF theory at the boundary $\{0\}\subset \RR_{\geq 0}$.

The representation theory corresponding to the behavior of 1D BF theory in the bulk is also interesting.
First, we note that $C^\bullet(\fg, \Sym(\fg[1]))$---the algebra of functions on $T^*(B\fg)$---is naturally a $P_1$ algebra, i.e. a Poisson algebra, since it is the algebra of functions on a 0-shifted symplectic space.
There is a forgetful functor from the category of $P_1$ algebras to the category of (homotopy) $P_0$ algebras.
The natural $P_0$ deformation quantization of $\sO(T^*(B\fg))$ is then $C^\bullet(\fg, C_\bullet(\fg))$.

The main Theorems of Chapter \ref{chap: examples} (cf. Theorems \ref{thm: 1dbfquantization} and \ref{thm: 1dbfobs}) show that:
\begin{enumerate}
    \item There exists a quantum bulk-boundary system quantizing classical BF theory on $\RR_{\geq 0}$, and
    \item The cochain complex of boundary observables for this bulk-boundary system is equivalent to $C_\bullet(\fg)$ and the cochain complex of bulk observables is equivalent to $C^\bullet(\fg, C_\bullet(\fg))$.
\end{enumerate}
The bulk observables form a locally constant factorization algebra on $\RR_{>0}$ and hence one should be expect the structure of a dg associative algebra on $C^\bullet(\fg, C_\bullet(\fg))$.
This algebra should be a deformation quantization of the classical Poisson algebra $\sO(T^*(B\fg))$.
Kostant and Sternberg \autocite{KSBRS} have given a description of this algebra structure, and we expect that one may refine the results of Chapter \ref{chap: examples} to make a statement about the bulk \emph{factorization algebra}.
We leave this, however, to future work.

In section \ref{sec: DVS}, we have introduced some of the main functional-analytic ingredients used in our formalism.
We have also shown how to endow spaces of (compactly-supported) sections of a vector bundle $E\to M$ with the structure of a convenient or differentiable vector space.
Further, we have introduced a preferred tensor product $\hotimes_\beta$, and have obtained an explicit characterization of the completed bornological tensor product of two vector spaces arising as the spaces of sections of vector bundles.
We have also given a similar characterization of the internal hom in the category DVS.
However, we need to extend these characterizations to spaces of sections \emph{with boundary conditions imposed}.
That is the task of the Appendix.
In Section \ref{sec: tensorproduct}, we characterize the completed bornological tensor product of vector spaces of sections of vector bundles with boundary conditions imposed.
In Section \ref{sec: innerhom}, we do the same thing for the inner homs of such spaces of sections, specifically with target $\RR$.
Finally, we prove a lemma generalizing the Atiyah-Bott lemma.
This lemma allows the simple description of the quantum observables of a free theory in Chapter \ref{chap: freequantum}.

%% file: Chapters/FAdefexamples.tex
\label{sec: review}
In this section, we review basic definitions regarding factorization algebras and present some elementary examples.
None of the content of this chapter is original; our main reference throughout is \autocite{CG1}.

\subsection{Basic Definitions}

There is a notion of a \emph{prefactorization algebra}, which stands in the same relationship to the notion of factorization algebra that the notion of presheaf does to the notion of a sheaf.
A factorization algebra, then, will be a prefactorization algebra satisfying a descent condition, just as sheaves are preshesaves satisfying an additional descent condition.
We will follow this division here.

Let us begin with the definition of a prefactorization algebra. 
This definition makes sense with values in any symmetric monoidal (possibly $(\infty,1)$-)category, but since we have no reason to consider any other such category besides the category of cochain complexes of vector spaces, we will present the definition only for this standard $(\infty,1)$-category.

\begin{definition}
Let $M$ be a topological space. A \textbf{prefactorization algebra} on $M$ consists of a functor
\begin{equation}
\cF: \mathrm{Opens}(M) \to \mathrm{Ch}(\mathrm{Vect})
\end{equation}
together with maps
\begin{equation}
m_{U_1,\ldots,U_k}^V : \cF(U_1)\otimes \cdots \otimes \cF(U_k) \to \cF(V)
\end{equation}
for every collection $U_1,\ldots, U_k$ of pairwise disjoint open subsets of $M$ which include into $V$.
(Here $\mathrm{Ch}(\mathrm{Vect})$ is the symmetric monoidal category of chain complexes of vector spaces over a field, which we almost always take to be $\RR$ or $\CC$.)
These data are required to satisfy the following axioms:
\begin{enumerate}
    \item Given an inclusion $U\subset V$, the map $m_U^V$ coincides with the map $\cF(U\subset V)$.
    \item Let $\sigma\in \Sigma_k$ be a permutation of the $k$-element set. Using the symmetric monoidal structure of $\mathrm{Ch(Vect)}$, one obtains a map
    \begin{equation}
    \sigma^*: \cF(U_1)\otimes \cdots\otimes \cF(U_k) \to \cF(U_{\sigma(1)})\otimes \cdots \otimes \cF(U_{\sigma(k)}).
    \end{equation}
    We require that $m_{U_{\sigma(1)},\ldots,U_{\sigma(k)}}^V\circ \sigma^* = m_{U_1,\ldots, U_k}^V$, i.e. the structure maps $m_{U_1,\ldots, U_k}^V$ be compatible with the natural $\Sigma_k$ action on the tensor factors of their domains.
    \item The map $m_{U,V}^{U\sqcup V}$ is a quasi-isomorphism.
    \item Suppose $U_{i,1}\sqcup\ldots\sqcup U_{i,k_i}\subset V_i$ for $i=1,\ldots, r$ and $V_1\sqcup\cdots\sqcup V_r\subset W$. Then we require that the triangle
    \begin{equation}
    \begin{tikzcd}
    \otimes_{i=1}^r \otimes_{j=1}^{k_i} \cF(U_{i,j}) \ar[rrr,"\otimes_i m_{U_{i,1},\ldots, U_{i,k_i}}^{V_i}"] \ar[rrrd, "m_{U_{i,1},\ldots, U_{r, k_r}}^W"]&&& \otimes_{i=1}^r \cF(V_i)\ar[d,"m_{V_1,\ldots, V_r}^W"]\\
    &&& \cF(W)
    \end{tikzcd}
    \end{equation}
    commute. This is a natural associativity condition.
\end{enumerate}
\end{definition}

\begin{remark}
\label{rmk: DVSforFAs}
We are being imprecise about which class of vector spaces we consider. 
Since we are interested in quantum field theory, the vector spaces we study are in general infinite-dimensional, and we would ideally like to work with an abelian category of such vector spaces.
The standard approach to the study of infinite-dimensional vector spaces is the theory of locally convex topological vector spaces, which infamously does not furnish an abelian category.
Hence, we will implicitly work in the context of differentiable vector spaces (DVS, see Appendix B of \autocite{CG1}).
We will provide a brief overview of these subtleties later in this chapter.
\end{remark}

\begin{remark}
Note that a prefactorization algebra $\cF$ is, in particular, a precosheaf, but $\cF$ also possesses structure maps $m_{U_1,\ldots, U_k}^V$ for $k>1$.  
\end{remark}

\begin{definition}
Let $\cF, \cG$ be prefactorization algebras on a space $M$. A \textbf{morphism of prefactorization algebras} $\cF\to \cG$ is a natural transformation $\varphi$ of the underlying functors
\begin{equation}
\textrm{Opens}(M) \to \textrm{Ch}(\textrm{Vect})
\end{equation}
such that all squares of the form
\begin{equation}
\begin{tikzcd}
\cF(U_1)\otimes \cdots \otimes \cF(U_k) \ar[r,"m_{U_1,\ldots, U_k}^V"] \ar[d,"\bigotimes_i \varphi(U_i)"]& \cF(V)\ar[d,"\varphi(V)"]\\
\cG(U_1)\otimes \cdots \otimes \cG(U_k) \ar[r, "m_{U_1,\ldots, U_k}^V"] & \cG(V)
\end{tikzcd}
\end{equation}
commute.
\end{definition}

\begin{remark}
Note that the above condition for $k=1$ is simply the requirement that $\varphi$ indeed be a map of precosheaves.
\end{remark}

Just as the category of sheaves is a full subcategory of the category of presheaves, the category of factorization algebras is a full subcategory of the category of prefactorization algebras. Namely, it will be the category of prefactorization algebras satisfying \emph{Weiss codescent}:

\begin{definition}
Let $U\subset M$ be an open subset, $\cF$ a prefactorization algebra on $M$, and $\cU:=\{U_\alpha\}_{\alpha \in A}$ a cover of $U$.
\begin{enumerate}
    \item The cover $\cU$ is a \textbf{Weiss cover} of $U$ if for every finite subset $S\subset U$, there exists a $U_\alpha \in \cU$ such that $S\subset U_\alpha$. Equivalently, $\cU$ is a Weiss cover if (and only if) the collection 
    \begin{equation}
    \cU^k:=\{U_\alpha^k\}_{\alpha\in A}
    \end{equation}
    is a cover of $U^k$ for all $k=1,2,\ldots$.
    \item The \textbf{\v{C}ech complex} for the pair $\cF, \cU$ is the cochain complex
    \begin{equation}
    \check{C}(\cU, \cF):=\bigoplus_{n=0}^\infty \bigoplus_{\substack{U_{\alpha_1},\ldots U_{\alpha_n}\in \cU\\ \alpha_i\neq \alpha_j}} \cF(U_{\alpha_1}\cap \cdots \cap U_{\alpha_n})[k]
    \end{equation}
    together with the total differential induced from the usual \v{C}ech differential and the internal differentials on the cochain complexes $\cF(V)$ for $V$ any open subset of $M$.
    \item The prefactorization algebra $\cF$ satisfies \textbf{Weiss codescent on $U$} if for all Weiss covers $\cU$, the natural map
    \begin{equation}
    \check{C}(\cU, \cF) \to \cF(U)
    \end{equation}
    is a quasi-isomorphism.
    
    \item A prefactorization algebra satisfying Weiss codescent for all open $U$ is denoted a \textbf{factorization algebra.}
    \item A \textbf{morphism of factorization algebras} is a morphism of the underlying prefactorization algebras.
\end{enumerate}
\end{definition}

\begin{remark}
In a bit more detail, the \v{C}ech differential on $\check{C}(\cU,\cF)$ is induced from the alternating sum of the structure maps
\begin{equation}
\cF(U_{\alpha_1}\cap\cdots \cap U_{\alpha_n}) \to \cF(U_{\alpha_1}\cap \cdots \cap\widehat{U_{\alpha_i}} \cap \cdots \cap U_{\alpha_n}),
\end{equation}
where $\widehat{U_{\alpha_i}}$ denotes that $U_{\alpha_i}$ is to be omitted from the intersection.
\end{remark}

\subsection{Basic Examples and Constructions}

\begin{example}
Let $(A,\mu)$ be an associative algebra over $\RR$ with unit map 
\[
\eta: \RR\to A.
\]
There is a natural factorization algebra $\cF_A$ on $\RR$ given by the assignment
\begin{equation}
\cF_A(\emptyset) = \RR;
\end{equation}
moreover,
\begin{equation}
\cF_A(I) = A
\end{equation}
for any interval $I$, 
\begin{equation}
m_{\emptyset}^I =\eta
\end{equation}
for any interval $I$,
\begin{equation}
m_{I}^J = \cF_A(I\subset J) = \id
\end{equation}
for any inclusion of intervals, and 
\begin{equation}
m_{I_1,I_2}^J=\mu: A\otimes A \to A
\end{equation}
for any triple of intervals $I_1, I_2, J$ with $I_1\sqcup I_2\subset J$.
See the Example in Section 3.1.1 of \autocite{CG1} for further details.
\end{example}

\begin{example}
The above example can be extended from $\RR$ to $\RR_{\geq 0}$; in so doing, one must choose, in addition to an associative algebra $A$, a pointed right $A$-module $M$, with point $m\in M$.
With these choices in hand, we may define a factorization algebra $\cF_{M,A}$ on $\RR_{\geq 0}$ as follows. Set $\cF_{M,A}([0,a))=M$ for any $a>0$, and $\cF_{M,A}((a,b))$ for any $a,b\geq 0$.

If $I_1$ and $I_2$ are intervals of the same type (either both open or both including 0), then $m_{I_1}^{I_2}$ is simply the identity on $A$ or $M$ as the case may be.
If $I_1$ is an open interval and $I_2$ is an interval containing 0, then $m_{I_1}^{I_2}$ is the map $\phi_m:A\to M$ determined by $m$, namely $\phi_m(a)=ma$.
If $I_1$, $I_2$, and $I_3$ are all open intervals (with $I_1\sqcup I_2\subset I_3$, then the structure map $m_{I_1,I_2}^{I_3}$ is simply the multiplication on $A$, as for $\cF_A$ above.
If $I_1,I_2,$ and $I_3$ are such that both $I_1$ and $I_3$ contain $0$, then the multiplication map 
\begin{equation}
m_{I_1,I_2}^{I_3}: M\otimes A \to M
\end{equation}
is determined by the $A$-module structure on $M$.

Further, we assign $\cF_{M,A}(\emptyset)=\RR$. The maps $\RR\to M$ and $\RR\to A$ corresponding to $m_\emptyset^I$ for $I$ an interval containing zero or not, respectively, are just those determined by the point $m$ and the unit map of $A$, respectively.

The remaining factorization algebraic structures are determined by the axioms of a factorization algebra.
One may check that the associativity of $A$ and the axioms of a right-$A$-module structure on $M$ guarantee that these assignments are well-defined.
\end{example}

\noindent
\begin{figure}
\begin{tikzpicture}[line width=.2mm]
		\draw ++(1,0) node {}; 
		\draw[Circle->] +(2, 0)--+(10, 0);
		\draw +(9.85, 0.45) node {$\RR_{\geq 0}$}; 
		\draw +(4, 0) node {)}; 
		\draw +(2.08,0) node {[};
		\draw +(3, .45) node {$U_1$};
		\draw +(13.5,0) node (domain) {$\cF_{A,M}(U_1)\otimes \cF_{A,M}(U_2)=M\otimes A$};
		\draw +(7, 0) node {)}; 
		\draw +(6,.45) node {$U_2$};
		\draw +(5,0) node {(};
		\draw[Circle->, yshift = -3cm] +(2,0)--+(10, 0);
		\draw[yshift = -3 cm] +(9.85, 0.45) node {$\RR_{\geq 0}$}; 
		\draw[yshift = -3 cm] +(2.08,0) node {[};
		\draw[yshift = -3 cm] +(13.5, 0) node (m) {$\cF_{A,M}(V)=M$};
		\draw[yshift = -3 cm] +(7, 0) node {)}; 
		\draw[yshift = -3 cm] +(4.5, .45) node {$V$};
		\draw[->] (domain) -- (m)
		node[pos = 0.5, right] {$m_{U_1,U_2}^V$};

\end{tikzpicture}
\caption{A picture illustrating one structure map for $\cF_{A,M}$.}
\end{figure}
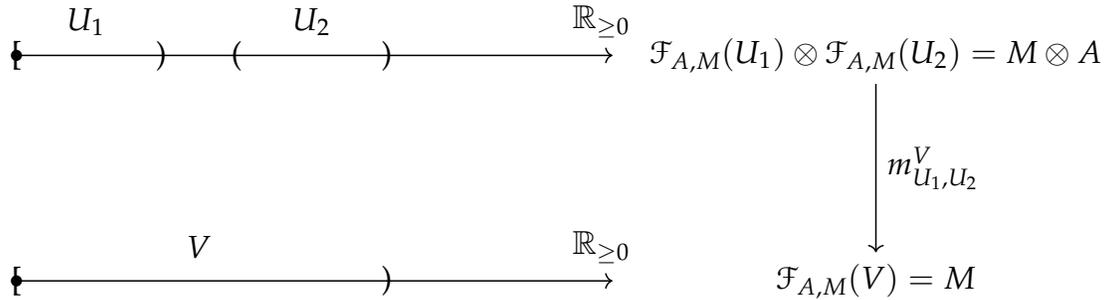

\begin{example}
\label{ex: factenvelope}
Let $\fg$ be a Lie algebra. The \textbf{$n$-factorization envelope of $\fg$} is the factorization algebra on $\RR^n$ which assigns, to an open subset $U\subset \RR^n$, the cochain complex
\begin{equation}
\cF_{\fg, n}(U):=C_\bullet(\Omega^\bullet_{\RR^n,c}(U)\otimes \fg);
\end{equation}
here, $C_\bullet$ is the Chevalley-Eilenberg chains functor. (There are some subtleties here because $\Omega^\bullet_{\RR^n,c}(U)$ is infinite-dimensional, i.e. the underlying graded vector space of $\cF_{\fg,n}(U)$ is given as a direct sum of \emph{completed} symmetric powers of $\Omega^\bullet_{\RR^n,c}(U)\otimes \fg$ rather than algebraic symmetric powers.)
The structure maps for $\cF_{\fg,n}$ are given as follows:
\begin{enumerate}
    \item The maps $m_U^V$ are induced from the extension by zero maps 
    \begin{equation}
    \Omega^\bullet_{\RR^n,c}(U)\to \Omega^\bullet_{\R^n,c}(V),
    \end{equation}
    and
    \item The structure maps $m_{U_1,\ldots, U_k}^V$ are the composites
    \begin{equation}
    \begin{tikzcd}
    \cF_{\fg, n}(U_1)\otimes \cdots \otimes \cF_{\fg,n}(U_k) \ar[rr,"m_{U_1}^V\otimes \ldots \otimes m_{U_k}^V"] && \cF_{\fg, n}(V)\otimes\cdots\otimes \cF_{\fg,n}(V)\ar[r, "\mu"]& \cF_{\fg, n}(V),
    \end{tikzcd}
    \end{equation}
    where the map $\mu$ is the multiplication on the underlying symmetric algebra.
\end{enumerate}
We note that $\mu$ is not a cochain map. 
This is true for the same reasons that the Chevalley-Eilenberg chains of a Lie algebra are not themselves a differential graded commutative algebra 
(though they do form a differential graded cocommutative coalgebra).
However, since the $U_i$ are disjoint, the $\Omega^\bullet_{\RR^n,c}(U_i)$ mutually commute in $\Omega^\bullet_{\RR^n,c}(V)$, and it follows that the composite $m_{U_1,\ldots, U_k}^V$ is in fact a cochain map.
\end{example}

There is a relationship between the two examples discussed above. Indeed, Proposition 3.4.1 of \autocite{CG1} shows that 
\begin{equation}
H^\bullet(\cF_{\fg,1})\cong \cF_{U(\fg)}.
\end{equation}
This is one justification of the use of the term ``$n$-factorization envelope'' to describe $\cF_{\fg, n}$.

We now describe a number of basic operations which produce new factorization algebras from existing factorization algebras.

\begin{definition}
Let $f:X\to Y$ be a map of topological spaces and $\cF$ a factorization algebra on $X$. The \textbf{pushforward} factorization algebra $f_*\cF$ is the factorization algebra on $Y$ whose value on an open set $U\subset Y$ is given by
\begin{equation}
f_* \cF(U) = \cF(f^{-1}(U)),
\end{equation}
with 
\begin{equation}
m_{U_1,\ldots, U_k}^V := m_{f^{-1}(U_1),\ldots, f^{-1}(U_k)}^{f^{-1}(V)}
\end{equation}
This definition extends the pushforward operation on precosheaves. 
\end{definition}

\begin{definition}
Let $\cF$ be a factorization algebra on a space $X$. Let $U$ be an open subset of $X$. The \textrm{restriction} factorization algebra $\cF\mid_U$ is defined via the composite (partially symmetric monoidal) functor
\begin{equation}
\begin{tikzcd}
\textrm{Opens}(U) \ar[r] &\textrm{Opens}(X) \ar[r,"\cF"]& \textrm{Ch}(\textrm{Vect}).
\end{tikzcd}
\end{equation}
\end{definition}

%% file: Chapters/Chapter2.tex
\chapter{Classical Bulk-Boundary Systems}
\label{chap: classical}
\section{Overview of Chapter}
In this chapter, we define and study factorization algebras for classical bulk-boundary systems. 
A preliminary version of this chapter, excluding Section \ref{sec: localfcnls}, has appeared on the preprint server ArXiv as \autocite{classicalarxiv}.
We will fix throughout a manifold $M$ \index[notation]{M@$M$} with boundary $\bdyM$ \index[notation]{Mb@$\bdyM$} and we will denote the inclusion $\bdyM \hookrightarrow M$ with the letter $\iota$ \index[notation]{iota@$\iota$}. T
he manifold $M$ serves as the spacetime background for the theory. 
As discussed above, the bulk-boundary systems we will consider arise from a certain class of theories, the TNBFTs (see Definition \ref{def: tnbft}, below).  

In \autocite{CG2}, Costello and Gwilliam construct factorization algebras of observables for classical and quantum field theories on manifolds without boundary. 
They also exhibit a $P_0$ ($+1$-shifted Poisson, i.e. the bracket has degree $+1$ as an operator) structure on the factorization algebra of classical observables. The main ingredient in their constructions at the classical level is the Batalin-Vilkovisky (BV) formalism \autocite{originalBV}, which has a geometric interpretation in terms of (--1)-shifted symplectic structures \autocite{geometryofBV}, \autocite{AKSZ}, \autocite{ptvv}. 

Our aim in the present work is to execute the same constructions for certain field theories on manifolds with boundary. In Chapter \ref{chap: freequantum}, based on joint work with Owen Gwilliam and Brian Williams, we discuss the quantization of the classical factorization algebras described here for \emph{free} bulk-boundary systems.
Finally, in Chapter \ref{chap: interactingquantum}, we perform the quantization of the factorization algebras for general, i.e. interacting, bulk-boundary systems.

In the version of the BV formalism of \autocite{cost}, \autocite{CG1}, and \autocite{CG2}, a classical perturbative BV theory on the manifold without boundary $\mathring{M}:=M\backslash \bdyM$ is specified by a $\Z$-graded vector bundle $E\to \mathring{M}$ (whose space of sections $\sE$ furnishes the space of fields for the theory) together with a $(-1)$-shifted fiberwise symplectic structure $\ip_{E}$ and a cohomological degree 0 local action functional $S$ satisfying the classical master equation $\{S,S\}=0$. 
One may easily extend $E$, $\ip_{E}$, and $S$ to $M$; 
however, the presence of the boundary $\bdyM$ makes the consideration of the classical master equation more subtle, since, in almost all cases, one uses integration by parts to verify that $\{S,S\}=0$ on $\mathring{M}$.
Put in a different way, the space of fields ceases to be (--1)-shifted symplectic once we pass from $\mathring{M}$ to $M$.  

Our solution is to impose boundary conditions on the fields (the space $\sE$) so that the boundary terms obstructing the classical master equation vanish when restricted to the fields with this boundary condition. 
We note that the imposition of boundary terms is not a novel idea.
What is notable about our approach, however, is that we impose boundary conditions in a way that is consistent with the interpretation of the BV formalism in terms of the geometry of shifted symplectic spaces.
Namely, we are careful to ensure that the imposition of the boundary conditions happens in a homotopically coherent way.
The payoff is two-fold. 
First, we find that the space of fields with the boundary condition imposed is indeed (--1)-shifted symplectic, using general facts about shifted symplectic geometry. 
Second, we guarantee that our constructions naturally take into account the gauge symmetry of the problem---there is no need to separately impose boundary conditions on the fields, ghosts, etc. 

Furthermore, we ask for our boundary conditions to be suitably local (see Definition \ref{def: bdycond}). 
The locality condition ensures that even after the boundary condition is imposed, the fields remain the space of global sections of a sheaf on $M$: this allows us to carry through the constructions of \autocite{CG2} almost without change, namely we can construct factorization algebras of observables by following \autocite{CG2} \emph{mutatis mutandis}.
 
Let us say a bit more about the class of field theories we consider. 
These are the theories which are so-called ``topological normal to the boundary'' (first defined in \autocite{butsonyoo}; see also Definition \ref{def: tnbft}). 
We will often use the acronym TNBFT to stand for ``field theory which is topological normal to the boundary.'' For a TNBFT we require that, roughly speaking, in a tubular neighborhood $\tubnhd\cong \bdyM\times [0,\epsilon)$ of the boundary $\bdyM$, the space of fields $\sE$ (the sheaf of sections of the bundle $E$ introduced above) admits a decomposition
\begin{equation}
\sEb\mid_{\tubnhd}\hotimes_\beta \Omega^{\bullet}_{[0,\epsilon)},
\end{equation}
where $\sEb$ is a 0-shifted symplectic space living on the boundary $\bdyM$, and arising, like $\sE$, from a $\Z$-graded bundle $\Eb\to \bdyM$.
We require all relevant structures on $\sE$ to decompose in a natural way with respect to this product structure.
A pair consisting of a TNBFT $\sE$ and a boundary condition for $\sE$ will be termed a \emph{classical bulk-boundary system}. Though this term properly applies to a much more general class of objects, we use it here to avoid bulky terminology.
The reason for restricting our attention to TNBFTs is that such theories have a prescribed simple behavior near the boundary $\bdyM$. This behavior makes it easier to verify the Weiss cosheaf condition for the factorization algebra of observables of a classical bulk-boundary system. 
This behavior also enables us to study heat kernel renormalization of bulk-boundary systems in Chapter \ref{chap: interactingquantum}.

Important examples of TNBFTs are of course the topological theories: BF theory, Chern-Simons theory, the Poisson sigma model, and topological mechanics. 
However, even for topological theories, we may consider boundary conditions which are not topological in nature. 
For example, one may study in Chern-Simons theory the chiral Wess-Zumino-Witten boundary condition, whose definition requires the choice of a complex structure on the boundary surface.
One may also study theories which have a dependence on arbitrary geometric structures on the boundary $\bdyM$, as long as their dependence in the direction normal to $\bdyM$ is topological. Mixed BF theory, Example \ref{ex: mixedbf} is one such example; it depends on the choice of a complex structure on the boundary of a three manifold of the form $N\times \R_{\geq 0}$, where $N$ is a surface.

As discussed in Chapter \ref{chap: intro}, the ``universal bulk theories'' of Poisson BV theories are naturally TNBFTs,
and so fall in the class of theories we consider.
For the sake of brevity, we do not discuss universal bulk theories at any length; however, the techniques presented here are designed to apply in particular to universal bulk theories (with their canonical boundary condition).

\section{Classical Bulk-Boundary Systems}
\label{sec: tnbfts}
In this section, we introduce the basic classical field-theoretic objects with which we work, namely we introduce the notion of a classical bulk-boundary system.
We use the Batalin-Vilkovisky (BV) formalism as the natural language for quantum field theory, which encodes both the equations of motion and the symmetries of a field theory in an intrinsically homotopically invariant fashion. 
The theories we consider are called ``topological normal to the boundary'' (TNBFT), a notion which we discuss in Section \ref{sec: classTNBFTs}.  
The definition we use here is due to Butson and Yoo \autocite{butsonyoo}. For a manifold with boundary $M$, a TNBFT is a field theory on  $\mathring M$ (the interior of $M$) which has a specified behavior of the field theory near the boundary. 
As the name suggests, the solutions to the equations of motion of a TNBFT are constant along the direction normal to the boundary $\bdyM$ in some tubular neighborhood $\tubnhd\cong \bdyM\times [0,\epsilon)$ of the boundary.

In Section \ref{sec: boundcnds}, we introduce boundary conditions for TNBFTs and discuss the homotopical interpretation of the resulting bulk-boundary systems.

\subsection{Classical Field Theories on Manifolds with Boundary}
\label{sec: classTNBFTs}

Before elaborating on the definition of a TNBFT, let us first recall what is meant by a perturbative classical field theory on $\mathring M$. The definition here follows \autocite{CG2}.

\begin{definition}
\label{def: clbv}
A \textbf{perturbative classical BV theory} on $\mathring M$ consists of
\begin{enumerate}
\item a $\Z$-graded vector bundle $E\to \mathring M$ \index[notation]{E@$E$},
\item sheaf maps 
\begin{equation}
\ell_k: (\sE[-1])^{\hotimes_\beta k} \to \sE[-1] 
\end{equation}
\index[notation]{ell@$\ell_k$}
of degree $2-k$,
\item and a degree $-1$ bundle map 
\begin{equation}
\ip_{loc} : E\otimes E \to \mathrm{Dens}_{\mathring M},
\end{equation}
\index[notation]{iploc@$\ip_{loc}$}
The objects $E,\ell_k, \ip_{loc}$ are required to satisfy the following properties:
\begin{itemize}
\item the maps $\ell_k$ turn $\sE[-1]$ into a sheaf of $L_\infty$ algebras;
\item the maps $\ell_k$ are polydifferential operators (we will call the pair $(\sE[-1],\ell_k)$ a \textbf{local $L_\infty$ algebra} on $\mathring M$;
\item the complex $(\sE,\ell_1)$ is elliptic;
\item the pairing $\ip_{loc}$  is fiberwise non-degenerate;
\item let us denote by $\ip$ \index[notation]{ip@$\ip$} the induced map of precosheaves 
\begin{equation}
\sE_c\otimes \sE_c \to \underline{\RR}
\end{equation}
given by using the fiberwise pairing $\ip_{loc}$, and then integrating the result. (Here, $\underline \RR$ is the constant pre-cosheaf assigning $\RR$ to each open subset.) We use the same notation for the pairing induced on $\sE_c[-1]$. We require that, endowed with this pairing, $\sE[-1]$ becomes a precosheaf of cyclic $L_\infty$ algebras (on $\sE_c[-1]$, the pairing has degree --3). 
\end{itemize}
\end{enumerate}
\end{definition}
We will often think of $\sE$ as a sheaf of 0-shifted symplectic formal moduli problems on $\mathring M$.
This is simply a way of recasting the axioms above in a more geometric language.
\begin{remark}
\label{rmk: classfcnl}
To a field theory in the above sense, we may associate the following element of 
\begin{equation}
\prod_{k\geq 1} \CVS\left(\sE_c^{\hotimes_\beta k}, \RR \right)^{S_k},
\end{equation}
(the space of functionals on $\sE$):
\begin{equation}
S(\varphi) = \sum_{k\geq 1}\frac{1}{(k+1)!} \ip[\varphi, \ell_k(\varphi,\ldots, \varphi)],
\end{equation}
which is usually denoted the \textbf{classical action functional} of the theory.
We also make the definition 
\begin{equation}
I(\varphi ) = \sum_{k\geq 2} \frac{1}{(k+1)!}\ip[\varphi, \ell_k(\varphi,\ldots, \varphi)]
\end{equation}
(i.e. $I$ remembers only the brackets of arity at least 2).
This is the \textbf{classical interaction functional} of the theory.
The cyclicity of the brackets $\ell_k$ with respect to the pairing $\ip$ guarantees that $S$ and $I$ are symmetric in their inputs.
\end{remark}

\begin{remark}
The ellipticity of the complex $(\sE, \diff)$ is not a necessary requirement from the standpoint of physics, and in fact, there is no need for this requirement in the definition of a classical BV theory. We include it here because the theory of elliptic PDE furnishes a wealth of tools which make it possible to develop a framework for renormalization. This latter point will be clearer in Chapter \ref{chap: interactingquantum}.
\end{remark}

We now specify the precise definition of a TNBFT. The following definition is adapted from Definitions 3.8 and 3.9 of \autocite{butsonyoo}.

\begin{definition}
\label{def: tnbft}
A \textbf{field theory  on $M$ which is topological normal to the boundary} is specified, as in Definition \ref{def: clbv}, by a $\Z$-graded bundle $E\to M$, a collection of sheaf maps $\ell_k$, and a bundle map $\ip_{loc}$. We also specify the following data:
\begin{enumerate}
\item a $\Z$-graded bundle $\Eb\to \bdyM$ \index[notation]{Eb@$\Eb$};
\item a collection of sheaf maps
\begin{equation}
\ell_{k,\partial} : (\sEb[-1])^{\otimes k}\to \sEb[-1];
\end{equation}
\index[notation]{ellb@$\ell_{k,\del}$}
\item and a degree 0 bundle map 
\begin{equation}
\ip_{loc,\partial}: \Eb \otimes \Eb \to \densbdyM.
\end{equation}
\index[notation]{iplocb@$\ip_{loc,\del}$}
\item  In some tubular neighborhood $\tubnhd\cong \bdyM\times [0,\epsilon)$ of $\bdyM$, an isomorphism
\begin{equation}
\phi :E\mid_{\tubnhd}\cong \Eb\boxtimes \Lambda^\bullet T^* [0,\epsilon).
\end{equation}

\end{enumerate}

We require the following to hold:
\begin{itemize}
\item The $k$-th graded summand $E^k$ of $E$ is zero for $|k|>>0$.
\item When $E$, $\ell_k$, and $\ip_{loc}$ are restricted to $\mathring M$, the resulting data satisfy the conditions to be a classical BV theory on $\mathring M$.
\item The data $(\sEb,\ell_{k,\partial}, \ip_{loc,\partial})$ satisfy all the requirements of Definition \ref{def: clbv} (with the degree $-1$ pairing $\ip_{loc}$ replaced by the degree 0 pairing $\ip_{loc,\partial}$).
\item The isomorphism $\phi$ respects all relevant structures. More precisely, we require
\begin{itemize}
\item Over $\tubnhd$, the induced isomorphism 
\begin{equation}
\varphi: \sE[-1]\mid_{\tubnhd} \cong \sEb[-1]\hotimes_\beta \Omega^\bullet_{[0,\epsilon)}
\end{equation}
is an isomorphism of (sheaves of) $L_\infty$ algebras. Here, the target of $\varphi$ has the $L_\infty$-algebra structure induced from its decomposition as a tensor product of an $L_\infty$ algebra with a commutative differential graded algebra.
\item Over $\tubnhd$, the fiberwise pairing $\ip_{loc}$ is identified with the tensor product of the pairing $\ip_{loc,\partial}$ and the wedge product pairing $\bigwedge$ on $\Lambda^\bullet T^* [0,\epsilon)$. 
\end{itemize}
\end{itemize}
\end{definition}

Before listing examples of TNBFTs, let us mention a few of their properties which follow directly from the definitions. First, however, we need to establish a definition:

\begin{definition}
Let $\iota: \bdyM \to M$ denote the inclusion. Given a TNBFT $(\sE,\sEb,\cdots)$, the \textbf{canonical submersion} is the composite sheaf map
\begin{equation}
\rho: \sE \to \iota_* \sEb 
\end{equation}
\index[notation]{rho@$\rho$}
which arises as the composite of 
\begin{enumerate}
\item The restriction
\begin{equation}
\sE(U) \to \sE(U\cap \tubnhd),
\end{equation}
\item the isomorphism
\begin{equation}
\sE(U\cap \tubnhd) \cong (\sEb\hotimes_\beta \Omega^\bullet_{[0,\epsilon)})(U\cap \tubnhd),
\end{equation}
\item
and the ``pullback to $t=0$'' map
\begin{equation}
(\sEb\hotimes_\beta \Omega^\bullet_{[0,\epsilon)})(U\cap \tubnhd) \to \sEb(U\cap \bdyM).
\end{equation}
\end{enumerate}
\end{definition}

\begin{remark}
One can extract a BV-BFV theory (\autocite{CMRquantumgaugetheories}) from a classical TNBFT, using $\sE$ as the bulk space of fields and $\sEb$ as the boundary fields. The assumptions in the definition guarantee that the general procedure of symplectic reduction in \autocite{CMRquantumgaugetheories} outputs $\sEb$ as the phase space for $\sE$.
\end{remark}

\begin{remark}
Occasionally, when we wish to emphasize the interpretation of $\sE$ (respectively, $\sEb$) as formal moduli spaces (as in \autocite{CG2}), we will refer to $\sE$ (respectively, $\sEb$) as a \textbf{local, (--1)-shifted (respectively, 0-shifted) symplectic formal moduli problem}.
\end{remark}

The following facts follow in a straightforward manner from the definitions.

\begin{proposition}
\label{prop: propsofTNBFTs}
\begin{enumerate}
\item The canonical surjection $\rho$, when thought of as a map
\begin{equation}
\sE[-1] \to \iota_* \sEb[-1],
\end{equation}
is a strict map of sheaves of $L_\infty$ algebras on $M$ (i.e. it strictly intertwines the operations $\ell_k$ and $\ell_{k,\partial}$).
\item The only failure of $(\sE_c[-1],\ell_k, \ip)$ to be a precosheaf of cyclic $L_\infty$ algebras on $M$ is encoded in the equation 

\begin{equation}
\label{eq: boundaryterm}
\ip[\ell_1 e_1, e_2]+(-1)^{|e_1|} \ip[e_1,\ell_1 e_2] = \ip[\rho e_1, \rho e_2]_{\partial},
\end{equation} 
where $e_1$ and $e_2$ are compactly-supported sections of $E\to M$ (the support of such sections may intersect $\bdyM$). 
\end{enumerate}
\end{proposition}

\begin{remark}
We avoid interpreting TNBFTs in terms of local action functionals solving the classical master equation because of the subtleties which arise in defining the Poisson bracket of local functionals on manifolds with boundary. These subtleties will disappear once we impose boundary conditions.
\end{remark}

Let us move on to discuss some examples of TNBFTs.

\begin{example}
\label{ex: toplmech}
Let $M=[a,b]$, and $V$ a symplectic vector space.  For $\sE$ we take $\Omega^\bullet_M\otimes V$ with the de Rham differential and the ``wedge-and-integrate'' pairing. It is straightforward to verify that $\sE$ is a TNBFT, with $\sEb= V\oplus V$. This theory is called \textbf{topological mechanics}.
\end{example}

\begin{example}
\label{ex: bf}
Let $M$ be an oriented $n$-manifold with boundary $\bdyM$; let $\fg$ be a finite-rank Lie algebra. \textbf{(Topological) BF theory on $M$} has space of fields
\begin{equation}
\sE = \left(\Omega^\bullet_{M}\otimes \fg[1]\right)\oplus \left(\Omega^\bullet_{M}\otimes \fg^\vee[n-2]\right). 
\end{equation}
The $L_\infty$ structure on $\sE[-1]$ is the natural such structure obtained from considering $\sE[-1]$ as the tensor product of the differential graded commutative algebra $\Omega^\bullet_M$ with the graded Lie algebra $\fg\oplus \fg[n-3]$.
The pairing $\ip_{loc}$ is given by wedge product of forms, followed by projection onto the top form degree.

The boundary data are given similarly by
\begin{equation}
\sEb = \left(\Omega^\bullet_{\bdyM}\otimes \fg[1]\right)\oplus \left(\Omega^\bullet_{\bdyM} \otimes \fg^\vee[n-2]\right)
\end{equation}
with analogous $L_\infty$ structure and pairing $\ip_{loc, \partial}$.

In general, if $\fg$ is an $L_\infty$ algebra, the same definitions can be made. Furthermore, if $M$ is not orientable, one can make the same definitions by replacing the $\fg^\vee$-valued forms with $\fg^\vee$-valued \emph{twisted} forms.
\end{example}

\begin{example}
\label{ex: mixedbf}
Let $\Sigma$ be a Riemann surface and $\fg$ a Lie algebra. \textbf{Mixed BF theory} is a theory on $\Sigma \times \R_{\geq 0}$ whose space of fields is 
\begin{equation}
\sE = \Omega^{0,\bullet}_{\Sigma}\hotimes_\beta \Omega^\bullet_{\R_{\geq 0}}\otimes \fg[1]\oplus\Omega^{1,\bullet}_{\Sigma}\hotimes_\beta \Omega^\bullet_{\R_{\geq 0}}\otimes \fg^\vee;
\end{equation}
the differential on the space of fields is $\bar \partial+ d_{dR, t}$ (where $t$ denotes the coordinate on $\R_{\geq 0}$); the only non-zero bracket of arity greater than 1 is the two-bracket, which arises from the wedge product of forms and the Lie algebra structure of $\fg$ (and its action on $\fg^\vee$).
One defines $\ip_{loc}$ in an analogous way to the corresponding object in topological BF theory.

The boundary data are given by the space of boundary fields
\begin{equation}
\sEb = \Omega^{0,\bullet}_\Sigma \otimes \fg[1]\oplus \Omega^{1,\bullet}_\Sigma \otimes \fg,
\end{equation}
with similar definitions for the  brackets and pairing.

Mixed BF theory may be similarly formulated on any manifold of the form $X\times N$, where $X$ is a complex manifold and $N$ is a smooth manifold with boundary.
\end{example}

\begin{example}
Let $M$ be an oriented 3-manifold with boundary, and $\fg$ a Lie algebra with a symmetric, non-degenerate pairing $\kappa$. \textbf{Chern-Simons theory} on $M$ is given by space of fields
\begin{equation}
\sE = \Omega^\bullet_M\otimes \fg[1];
\end{equation}
$\sE[-1]$ has the $L_\infty$ structure induced from considering it as the tensor product of the commutative differential graded algebra $\Omega^\bullet_M$ with the Lie algebra $\fg$. The pairing $\ip_{loc}$ uses the wedge product of forms and the invariant pairing $\fg$. 

The boundary data are encoded in the boundary fields
\begin{equation}
\sEb = \Omega^\bullet_\bdyM \otimes \fg[1],
\end{equation}
with brackets and pairing $\ip_{loc, \partial}$ defined similarly to the analogous structures on the bulk fields $\sE$.

We note that, on manifolds of the form $\Sigma \times \R_{\geq 0}$, Chern-Simons theory is a deformation of mixed BF theory (see \autocite{ACMV} or \autocite{GW} for details). 
\end{example}

\begin{example}
Let $N$ be a manifold without boundar, and let $(\sG, \ell_1, \ell_2)$ be an elliptic differential graded Lie algebra on $N$ (i.e. $\sG$ is a local differential graded Lie algebra such that the complex $(\sG,\ell_1)$ is elliptic).
$\sG$-BF theory is the theory on $N\times \RR_{\geq 0}$ with sheaf of fields
\begin{equation}
\sG\hotimes_\beta \Omega^\bullet_{\RR_{\geq 0}}[1]\oplus \sG^!\hotimes_\beta \Omega^\bullet_{\RR_\geq 0}[-1];
\end{equation}
the brackets and pairing on this sheaf of fields are defined by analogy with the preceding examples.
\end{example}

\begin{remark}
Many of the theories we consider are fully topological in nature (topological BF theory, Chern-Simons theory, the Poisson sigma model, and topological mechanics are all of this form). However, we will see that even for topological theories, we can choose a non-topological boundary condition. This is notably true for Chern-Simons theory.
\end{remark}

\begin{remark}
As noted in the introduction, one of the main sources of examples of TNBFTs is the class of so-called ``degenerate'' BV theories (\autocite{butsonyoo}). 
Every degenerate theory $\cT$ on a manifold $N$ gives rise to a (``non-degenerate'') TNBFT on $N\times \RR_{\geq 0}$, which Butson and Yoo call the ``universal bulk theory for $\cT$''. 
We refer the reader to \autocite{butsonyoo}, Definitions 2.34 and 3.18 for the definitions of these notions, since we do not use them explicitly at any length. 
We note only that BF theory, mixed BF theory, the Poisson sigma model, Chern-Simons theory, and $\sG$-BF theory (on spaces of the form $N\times \R_{\geq 0}$) arise in this way. 
\end{remark}

\subsection{Boundary conditions and homotopy pullbacks}
\label{sec: boundcnds}
In this section, we study boundary conditions for TNBFTs.
Our main goal is to provide a derived geometric interpretation for the imposition of boundary conditions, as discussed in Chapter \ref{chap: intro}.
\subsubsection{Boundary conditions}
Equation \eqref{eq: boundaryterm} tells us that the space of fields $\sE$ of a TNBFT on $M$ is \emph{not} $(-1)$-shifted symplectic. However, we do have a map $\rho: \sE(M)\to \sEb(\bdyM)$ and $\sEb(\bdyM)$ has a $0$-shifted symplectic structure. We will construct a Lagrangian structure on the map $\rho$, using Equation \eqref{eq: boundaryterm}. If we are given another Lagrangian $\sL(M)\to \sEb(M)$, then by Example 3.2 of \autocite{calaquedag}, we can expect that the homotopy fiber product \[
\sE(M)\times^h_{\sEb(\bdyM)}\sL(M)
\]
have a $(-1)$-shifted symplectic structure. We will call this second Lagrangian $\sL(M)$ a \emph{boundary condition}. The purpose of this section is to carry out this general philosophy more precisely, and in a sheaf-theoretic manner on $M$.

\begin{definition}
\label{def: loclagrangianstr}
Let $\sF$ be the sheaf of sections of a bundle $F\to M$ and suppose that $\sF[-1]$ is a local $L_\infty$-algebra on $M$ with brackets $\ell_i$. Suppose that $\sG$ is a 0-shifted symplectic local formal moduli problem on $\partial M$ (i.e. $\sG$ satisfies the axioms that $\sEb$ does in Definition \ref{def: tnbft}), and $\rho: \sF[-1]\to \iota_*\sG[-1]$ a map of sheaves with the following properties:
\begin{itemize}
\item $\rho$ intertwines the $L_\infty$ brackets (strictly), and
\item $\rho$ is given by the action of a differential operator acting on $\sF$, followed by evaluation at the boundary, followed by a differential operator 
\begin{equation}
\cinfty(\bdyM, F\mid_{\bdyM})\to \sG.
\end{equation}
\end{itemize}

Suppose 
\begin{equation}
h_F: F\otimes F \to \densM.
\end{equation}
is a bundle map. We write $h_{loc}$ for the induced map
\begin{equation}
h_{loc}: \sF\otimes \sF \to \Omega^n_{M,tw}
\end{equation}
and $h$ for the induced pairing on compactly-supported sections of $F$.
\begin{enumerate}
\item  The pairing $h$ is \textbf{a constant local isotropic structure on $\rho$}  precisely if
\begin{equation}
\label{eq: diff}
h(\ell_1 f_1, f_2)+(-1)^{|f_1|}h(f_1, \ell_1 f_2) = \ip[\rho f_1, \rho f_2]_{\sG},
\end{equation}

where $\ip_{\sG}$ is the symplectic pairing on $\sG$, and
\begin{equation}
\label{eq: higherbrackets}
h( \ell_k(f_1,\cdots, f_k), f_{k+1}) \pm h(f_1, \ell_k (f_2, \cdots, f_{k+1}))=0
\end{equation}
for $k>1$.
\item A constant local isotropic structure induces a map of complexes of sheaves
\begin{align}
\Psi:Cone(\rho) &\to \sF_{c}^\vee,\\
\Psi(f,g)(f')&=h(f,f')-\omega(g,\rho(f')).
\end{align}
We say that the isotropic structure is \textbf{Lagrangian} (i.e. that there is a constant local Lagrangian structure on $\rho$) if this map of complexes of sheaves is a quasi-isomorphism. Here, the symbol $\vee$ denotes the sheaf which, to an open $U\subset M$, assigns the strong continuous dual to $\sF_c(U)$. In other words, $\sF^\vee_c$ is the sheaf of distributional sections of $F^!$. 
\end{enumerate}
\end{definition}

\begin{remark}
Let us give a more geometric interpretation of the definition of isotropic structure. To this end, let $\fg$ and $\fh$ be $L_\infty$ algebras and $\rho: \fg\to \fh$ a strict $L_\infty$ map. The algebras $\fg$ and $\fh$ define formal moduli spaces $B\fg$ and $B\fh$. Let us suppose that $B\fh$ is 0-shifted symplectic with symplectic form $\ip_{\fh}$. Then, an isotropic structure on the map $\rho$ is an element $h\in \Omega^2_{cl}(B \fg)$ (the complex of closed two-forms on $B\fg$) such that 
\begin{equation}
\label{eq: isotropicdef}
Q^{TOT} h = \rho^*\left( \ip_{\fh}\right),
\end{equation}
where $Q^{TOT}$ is the total differential on $\Omega^2_{cl}(B\fg)$, which includes a Chevalley-Eilenberg term and a de Rham term. If one requires that both $\ip_{\fh}$ and $h$ be constant on $B\fh$ (i.e. are specified by elements of $\Lambda^2 (\fh[1])^\vee$), then one obtains precisely Equations \eqref{eq: diff} and \eqref{eq: higherbrackets} from Equation \eqref{eq: isotropicdef}.
\end{remark}

\begin{remark}
Definition \ref{def: loclagrangianstr} is an adaptation to the setting of local $L_\infty$ algebras of the corresponding definitions for derived Artin stacks given in Section 2.2 of \autocite{ptvv}. 
\end{remark}

\begin{lemma}
\label{lem: fieldslagrangian}
The pairing $\ip_{\sE}$ endows the map $\sE\to \sEb$ with a constant local Lagrangian structure.
\end{lemma}

\begin{proof}
Equation \eqref{eq: boundaryterm} implies Equation \eqref{eq: diff}, while the invariance of $\ip_\sE$ under the higher ($k>1$) brackets implies Equation \eqref{eq: higherbrackets}. Hence $\ip_\sE$ gives a constant local isotropic structure on $\rho$. We need only to check that the induced map of sheaves
\begin{equation}
\left( \sE[1]\oplus \iota_* \sEb, \diff+\Qb \pm \rho\right) \to \sE_c^\vee
\end{equation}
is a quasi-isomorphism of sheaves. We prove this as follows: first, we note that for an open $U\subset M$ which does not intersect $\partial M$, we are studying the map
\begin{equation}
\sE[1](U)\to \sE^!(U) \to \sE_c(U)^\vee,
\end{equation}
where the first map is induced from the isomorphism $E[1]\to E^!$ arising from $\ip_{E}$. The composite map is a quasi-isomorphism because the first map is an isomorphism and the second map is the quasi-isomorphism of the Atiyah-Bott lemma (see, e.g., Appendix D of \autocite{CG1}). 

Now, let $U\cong U'\times [0,\delta)$, where $U'$ is an open subset of $\partial M$ and $U\subset \tubnhd$. We will show that both 
\begin{equation}
(\sE(U)[1]\oplus \sEb(U'),Q_{cone})
\end{equation}
and $\sE_c^\vee$
are acyclic. 
Let's start with the latter complex. 
The following heuristic argument explains why one should expect the complex $\sE_c(U)$ to be acyclic.
Because the theory topological normal to the boundary, any $e\in \sE_c(U)$ is cohomologous to its evaluation at some~$t\in [0,\delta)$.
Because $e$ is compactly-supported, evaluation at a sufficiently large $t$ will give 0. 

To make the preceding argument precise, write $e=e_1+e_2dt$ for $e\in \sE_{c}(U)$, and consider the map 
\begin{align}
&K: \sE_{c}(U)[-1]\to \sE_{c}(U)\\
&K(e)(t)=(-1)^{|e_2|+1}\int_t^\delta e_2(s)ds;
\end{align}
one verifies that $K$ is a contracting homotopy for $\sE_{c}(U)$: 
\begin{align}
\diff K(e) &= e_2\wedge dt+(-1)^{|e_2|+1}\int_t^\delta \Qb e_2(s)ds\\
K\diff(e)(t)& = -e_1(\delta)+e_1(t)+(-1)^{|e_2|}\int_t^\delta \Qb e_2(s)ds,
\end{align}
so that $\diff K+K\diff=id$. 
Here, we have used the fact that $e_1(\delta)$ is 0 by the compact-support condition.
By duality we have constructed a contracting homotopy for $\sE_c^\vee$.

Next, consider $(\sE(U)[1]\oplus \sEb(U'), Q_{cone})$. There is a natural map $\phi: \sEb(U')\to \sE(U)\cong \sEb(U')\hotimes_\beta \Omega^\bullet_{[0,\epsilon)}([0,\delta))$ induced from the map $\CC\to \Omega^\bullet_{[0,\epsilon)}([0,\delta))$. Let $C$ denote the mapping cone for the identity map on $\sEb(U')$, and define the map
\begin{align}
\Phi: C&\to \cone(\rho)(U)\\
\Phi(\alpha_0,\alpha_1) &=(\phi(\alpha_0),\alpha_1).
\end{align}
It is straightforward to check that $\Phi$ is a quasi-isomorphism. Hence, $\cone(\rho)(U)$ is acyclic. 
We have shown that, for every point $x\in M$, and any neighborhood $V$ of $x$, we have a neighborhood $U\subset V$ of $x$ on which the sheaf map under study is a quasi-isomorphism. This map is therefore a quasi-isomorphism of complexes of sheaves. 
\end{proof}

We would now like to choose another Lagrangian $\sL\to \sEb$, so that the homotopy pullback $\sL\times^h_{\sEb}\sE$ is $(-1)$-shifted. To make this precise, we need to choose a model category in which to take the homotopy pullback, and we need to introduce an appropriate class of Lagrangians $\sL$ which we will study. The following definition is intended to fulfill the latter aim.

\begin{definition}
\label{def: bdycond}
Let $\sE$ be a TNBFT. A \textbf{local boundary condition} for $\sE$ is a subbundle $L\subset \Eb$ \index[notation]{L@$L$} endowed with brackets $\ell_{L,i}$ \index[notation]{ellL@$\ell_{L,i}$} making $L[-1]$ into a local $L_\infty$ algebra satisfying the following properties:
\begin{itemize}
\item The induced map $\sL\to \sEb$ of sheaves of sections on $\bdyM$ intertwines (strictly) the brackets,
\item $\ip_{\Eb}$ is identically zero on $L\otimes L$, and
\item there exists a vector bundle complement $L'\subset \Eb$ on which $\ip_{\Eb}$ is also zero.
\end{itemize}
Such data are considered a \emph{local} boundary condition because the map $\sL\to \sEb$ arises from a bundle map $L\subset \Eb$. Since we have no need for boundary conditions which are not of this form, we will use the term ``boundary condition'' when we mean ``local boundary condition.''
\end{definition}

\begin{example}
\label{ex: toplmechbc}
Recall from Example \ref{ex: toplmech} that a symplectic vector space $V$ provides a TNBFT on $\R_{\geq 0}$. It is straightforward to check that a Lagrangian subspace $L\subset V=\sEb$ gives a local boundary condition for this theory.
\end{example}

\begin{example}
\label{ex: bfbdycond}
One can define two boundary conditions for BF theory (Example \ref{ex: bf}).
Recall that the space of boundary fields is 
\[
\sEb = \Omega^\bullet_{\bdyM}\otimes \fg [1]\oplus \Omega^\bullet_{\bdyM,tw}\otimes \fg^\vee[n-2];
\]
we may take either of these two summands as a boundary condition. 
We will call the former boundary condition the $A$ condition and the latter the $B$ condition.
\end{example}

\begin{example}
For Chern-Simons theory on an oriented 3-manifold $M$, the sheaf of boundary fields is
\begin{equation}
\Omega^\bullet_{\bdyM}\otimes \fg[1];
\end{equation}
let us choose a complex structure on $\bdyM$. Then, it is straightforward to verify that $\Omega^{1,\bullet}_{\bdyM}\otimes \fg$ gives a local boundary condition, the \textbf{chiral Wess-Zumino-Witten boundary condition}. The Chern-Simons/chiral Wess-Zumino-Witten system for abelian $\fg$ is the central example of Chapter \ref{chap: freequantum}.
\end{example}

\subsubsection{The \condfieldsterm fields}
In this section, we define what we mean by a classical bulk-boundary system, and we give the promised homotopy-theoretic interpretation of the imposition of the boundary condition.
\begin{definition}
\label{def: clblkbdysystem}
Given a TNBFT $(\sE,\sEb,\ip)$ and a boundary condition $\sL\subset \sEb$, the \textbf{space of \condfieldsterm fields} $\condfields$ is the complex of sheaves 
\begin{equation}
\condfields(U):= \{e\in \sE(U)\mid \rho(e)\in (\iota_*\sL)(U)\}
\end{equation}
\index[notation]{EL@$\condfields$}
of fields in $\sE$ satisfying the boundary condition specified by $\sL$. In other words, $\condfields$ is the (strict) pullback $\sE\times_{\iota_*\sEb}(\iota_*\sL)$ taken in the category of presheaves of complexes on $M$.

A \textbf{classical bulk-boundary system} is a TNBFT together with a boundary condition.
\end{definition}

\begin{remark}
The term ``bulk-boundary system'' is appropriate for a much more general class of field-theoretic information. However, since we study only this specific type of bulk-boundary system, we omit qualifying adjectives from the terminology.
\end{remark}

\begin{lemma}
\label{ref: tildefieldsBV}
The brackets on $\sE[-1]$ descend to brackets on $\condfields[-1]$, and 
\begin{equation}
(\condfields,\ell_i,\ip_{\sE})
\end{equation}
forms a classical BV theory in the sense of Definition \ref{def: clbv}, except that $\condfields$ is not the sheaf of sections of a vector bundle on $M$.
\end{lemma}
\begin{proof}
The first statement follows from the fact that $\condfields[-1]$ is a pullback of sheaves of $L_\infty$-algebras on $M$. The only thing that remains to be verified is that the the pairing $\ip_{\sE}$ is cyclic with respect to the brackets $\ell_i$, once we restrict to $\condfields$. By our assumptions, the only failure of the brackets to be cyclic for $\sE$ is captured in Equation \eqref{eq: boundaryterm}. Upon restriction to $\condfields$, however, the boundary term in that equation vanishes.
\end{proof}

The previous lemma shows that the pullback 
\begin{equation}\sE\times_{\iota_*\sEb}(\iota_*\sL)\end{equation}
in the category of presheaves of shifted $L_\infty$ algebras on $M$ has a $(-1)$-shifted symplectic structure. In the rest of this section, we explain why this is not an accident. 

As we have noted in Lemma \ref{lem: fieldslagrangian}, the map $\sE\to \iota_* \sEb$ is a Lagrangian map. 
Moreover, the map $\sL\to \sEb$ is also Lagrangian by assumption. 
We should therefore expect the homotopy pullback $\sE\times^h_{\iota_*\sEb}(\iota_*\sL)$ (in an appropriate model category) to have a $(-1)$-shifted symplectic structure.
The space of $\sL$-conditioned fields $\condfields$ is \emph{a priori} only the strict pullback.
However, on Lemma \ref{lem: tildefieldsmodel} below, we will show that $\condfields$ is indeed a model for this homotopy pullback. 

Let us first, however, describe the model category in which we would take the homotopy fiber product. The sheaves $\sE[-1],\iota_*\sEb[-1],\iota_*\sL[-1]$ are presheaves of $L_\infty$-algebras on $M$. In \autocite{hinichsheaves} (Theorem 2.2.1), a model structure on the category of such objects is given. The weak equivalences in this model category are those which induce quasi-isomorphisms of complexes of sheaves after sheafification, and the fibrations $f:M\to N$ are the maps such that $f(U):M(U)\to N(U)$ is surjective (degree by degree) for each open $U$ and such that for any hypercover $V_\bullet\to U$, the corresponding diagram 
\begin{equation}
\begin{tikzcd}
M(U)\ar[r]\ar[d] & \check{C}(V_\bullet,M)\ar[d]\\
N(U)\ar[r]& \check{C}(V_\bullet,N)
\end{tikzcd}
\end{equation}
is a homotopy pullback. This information about the model category is enough to show the following lemma:

\begin{lemma}
\label{lem: tildefieldsmodel}
In the model category briefly described in the preceding paragraph, the sheaf of \condfieldsterm fields $\condfields[-1]$ is a model for the homotopy pullback
 \begin{equation}(
 \sE[-1])\times^h_{\iota_*\sEb[-1]}(\iota_*\sL[-1])
 \end{equation}
 of presheaves of $L_\infty$-algebras on $M$.
\end{lemma}
\begin{proof}
We first note the following: $\sE$ satisfies \v{C}ech descent for arbitrary covers in $M$, and $\sEb,\sL$ do the same on $\bdyM$, by Lemma B.7.6 of \autocite{CG1}. 
Because the \v{C}ech complex for $\iota_*\sL$ (resp. $\iota_*\sEb$) with cover $\{U_\alpha\}$ is identically the \v{C}ech complex for $\sL$ (resp. $\sEb$) with cover $\{U_\alpha\cap \bdyM\}$, $\iota_*\sL$ and $\iota_*\sEb$ satisfy \v{C}ech descent as presheaves on $M$. By Theorem 7.2.3.6 and Proposition 7.2.1.10 of \autocite{HTT}, these presheaves satisfy descent for arbitrary hypercovers. 
(Strictly speaking, the cited results are only proved for presheaves of simplicial sets on $M$; however, using the boundedness of $E$ stipulated in Definition \ref{def: tnbft}, we can shift all objects involved to be concentrated in non-positive degree and then use the Dold-Kan correspondence to show that hyperdescent and descent coincide for presheaves of (globally) bounded complexes.)

We claim that the map $\sE\to \iota_*\sEb$ is a fibration, whence the lemma follows immediately. It is manifest that the maps $\sE(U)\to \sEb(U\cap \bdyM)$ are surjective for every open $U\subset M$. So, it remains to check that the square  
\begin{equation}
\begin{tikzcd}
\sE(U)\ar[r]\ar[d] & \check{C}(V_\bullet,\sE)\ar[d]\\
\sEb(U)\ar[r]& \check{C}(V_\bullet,\sEb)
\end{tikzcd}
\end{equation}
is a homotopy pullback square of complexes for any hypercover $V_\bullet$. This is true because the above square is the outer square of the diagram
\begin{equation}
\begin{tikzcd}
\sE(U) \ar[rrr]\ar[ddd]\ar[rd,"\sim"]&&&\check{C}(V_\bullet,\sE)\ar[ddd]\ar[ld,"\sim"]\\
&\check{C}(V_\bullet,\sE)\ar[r,"\id"]\ar[d]&\check{C}(V_\bullet,\sE)\ar[d]&\\
&\check{C}(V_\bullet,\sEb)\ar[r,"\id"]&\check{C}(V_\bullet,\sEb)&\\
\sEb(U)\ar[rrr]\ar[ur,"\sim"]&&&\check{C}(V_\bullet,\sEb)\ar[lu,"\sim"]
\end{tikzcd};
\end{equation}
all the diagonal maps in the diagram are quasi-isomorphisms, and the inner square is clearly a homotopy pullback square.
\end{proof}

\input{Chapters/LocalFunctionals}

\section{The Factorization Algebras of Observables}
\label{sec: FAs}
In this section, given a bulk-boundary system $(\sE,\sL)$, we construct a factorization algebra $\Obcl_{\sE,\sL}$ of classical observables for the bulk-boundary system on $M$. $\Obcl_{\sE,\sL}$ will be constructed as ``functions'' on $\condfields$. $\Obcl_{\sE,\sL}$ has the advantage of being easy to define; however, it does not manifestly carry the $P_0$ (shifted Poisson) structure that one expects to find on the space of functions on a (--1)-shifted symplectic space. Hence, we construct also a $P_0$ factorization algebra $\widetilde{\Obcl_{\sE,\sL}}$ and a quasi-isomorphism $\widetilde{\Obcl_{\sE,\sL}}\to \Obcl_{\sE,\sL}$. We closely follow \autocite{CG2}.

\begin{definition}
\label{def: classFA}
Let $(\sE,\sL)$ be a bulk-boundary system. Define 
\index[notation]{Obscl@$\Obcl_{\sE,\sL}$}
\begin{equation}
\Obcl_{\sE,\sL}(U):= (\sO(\condfields), d_{CE}) = \prod_{k\geq 0} \underline{CVS}\left(\condfields^{\hotimes_\beta k}, \RR\right)_{S_k}
\end{equation}
where the symmetric algebra is taken with respect to the completed projective tensor product of nuclear topological vector spaces. Here, $d_{CE}$ denotes the Chevalley-Eilenberg differential constructed from the structure of $L_\infty$ algebra on $\condfields(U)[-1]$. Given disjoint open subsets $U_1,\ldots, U_k\subset M$ all contained in an open subset $V
\subset M$, define the 
\begin{equation}
m_{U_1,\cdots, U_k}^V:\Obcl_{\sE,\sL}(U_1)\hotimes_\beta \cdots \hotimes_\beta \Obcl_{\sE,\sL}(U_k)\to \Obcl_{\sE,\sL}(V)
\end{equation}
to be the composite
\begin{align}
\Obcl_{\sE,\sL}(U_1)\hotimes_\beta \cdots \hotimes_\beta \Obcl_{\sE,\sL}(U_k)&\to \widehat{\Sym}(\oplus_i\condfields^\vee(U_i))\cong \Obcl_{\sE,\sL}(U_1\sqcup \cdots \sqcup U_k)\nonumber\\
&\to \Obcl_{\sE,\sL}(V),
\end{align}
where the first map is a version of the natural isomorphism $\Sym(A\oplus B) \cong \Sym(A)\otimes \Sym(B)$, and the last map is induced from the extension of compactly-supported distributions $\condfields^\vee(\sqcup_i U_i)\to \condfields^\vee(V)$.
\end{definition}

\begin{remark}
The construction here is almost identical to that of \autocite{CG2}. We simply impose boundary conditions on the fields of interest.
\end{remark}

\begin{theorem}
\label{thm: classobsformFA}
The differentiable vector spaces $\Obcl_{\sE,\sL}(U)$, together with the structure maps $m_{U_1,\ldots,U_k}^V$, form a factorization algebra.
\end{theorem}

\begin{proof}
The proof is very similar to the analogous proof in \autocite{CG2}. The proof that the classical observables form a prefactorization algebra is identical to that in \autocite{CG2}. For the Weiss cosheaf condition, it suffices, as in \autocite{CG2}, to check the condition for free theories. A similar situation arises in \ref{chap: freequantum}, though there one uses $\Sym(\condfieldscs[1])$ for the classical observables. By the same arguments as in the corresponding proof in \autocite{CG2}, we need only to show that, given any Weiss cover $\fU=\{U_i\}_{i\in I}$ of an open subset $U\subset M$, the map 
\begin{equation}
\label{eq: cechmap}
\bigoplus_{n=0}^\infty \bigoplus_{i_1,\cdots, i_n}\underline{\CVS}\left( \condfields(U_{i_1}\cap\cdots\cap U_{i_n})^{\hotimes_\beta m},\RR\right)_{S_k} [n-1]\to \underline{\CVS} \left(\condfields(U)^{\hotimes_\beta m},\RR\right)_{S_k}
\end{equation}
is a quasi-isomorphism, where the left-hand side is endowed with the \v{C}ech differential and the internal differential induced by $\ell_1$ only.

We show in Section \ref{sec: innerhom} that there is a non-canonical splitting 
\begin{equation}
    \sE(U) \cong \condfields(U) \oplus \sL'(U\cap \bdyM).
\end{equation}
We may therefore understand the mapping spaces appearing in Equation \eqref{eq: cechmap} as quotients of corresponding spaces without boundary conditions imposed (in both categories of interest, $\CVS$ and $\DVS$; this was why it was important to actually produce this splitting).
More precisely, the kernel of the quotient map 
\begin{equation}
\label{eq: quotientoffunctionals}
\underline{\CVS} \left(\sE(U)^{\hotimes_\beta m},\RR\right)_{S_k}\to \underline{\CVS} \left(\condfields(U)^{\hotimes_\beta m},\RR\right)_{S_k}
\end{equation}
is, naturally, the space of functionals which vanish when all inputs satisfy the boundary condition.
The quoted results from the appendix show that one may actually understand the codomain of the above equation as a quotient in $\CVS$ and $\DVS$.
In Lemma A.5.8 of \autocite{CG1}, Costello and Gwilliam describe an explicit contracting homotopy for the total complex in Equation \eqref{eq: cechmap} where no boundary conditions are imposed.
This contracting homotopy involves only addition and (the transpose of) multiplication by compactly-supported smooth functions on $U^m$ constituting a partition of unity for the cover $\{U_i^m\}_{i\in I}$ of $U^m$.
It is obvious that such a contracting homotopy will preserve the kernel of the quotient map Equation \eqref{eq: quotientoffunctionals}; hence it will also furnish a contracting homotopy for the codomain in Equation \eqref{eq: quotientoffunctionals}.
\end{proof}

In \autocite{CG2}, Theorem 6.4.0.1, it is also shown that $\Obcl$ possesses a sub-factorization algebra $\widetilde{\Obcl}$ which has a $P_0$ structure. We have the same situation here. Let us first make a definition:

\begin{definition}
Let 
\[
I\in \underline{\CVS}\left( \condfields(U)^{\hotimes_\beta k}, \RR\right)_{S_k}
\]
with $k\geq 1$. 
$I$ induces a map
\begin{equation}
\underline{\CVS}\left( \condfields(U)^{\hotimes_\beta (k-1)}, \RR\right)_{S_{k-1}}\to \underline{\CVS}\left( \condfields(U), \RR\right)
\end{equation}
by the hom-tensor adjunction in $\CVS$.
We say that $I$ has \textbf{smooth first derivative} if this map has image in $\condfieldscs[1](U)$. 
We consider this condition to be vacuously satisfied when $k=0$. Given $J\in \sO(\condfields(U))$, we say that $J$ has smooth first derivative if each of its Taylor components does.
\end{definition} 

\begin{theorem}
\label{thm: P0classobsFA}
Let
\index[notation]{Obscltilde@$\widetilde{\Obcl_{\sE,\sL}}$}
$\widetilde{\Obcl_{\sE,\sL}}(U)$ denote the subspace of $\Obcl_{\sE,\sL}(U)$ consisting of functionals with smooth first derivative. 
\begin{enumerate}
\item $\widetilde{\Obcl_{\sE,\sL}}$ is a sub-factorization algebra of $\Obcl_{\sE,\sL}$. 
\item $\widetilde{\Obcl_{\sE,\sL}}$ posseses a Poisson bracket of degree $+1$.
\item The inclusion $\widetilde{\Obcl_{\sE,\sL}}\to \Obcl_{\sE,\sL}$ is a quasi-isomorphism.
\end{enumerate}
\end{theorem}

\begin{proof}
Again, the proof follows \autocite{CG2}.
The proof that $\widetilde{\Obcl_{\sE,\sL}}$ is a sub-prefactorization algebra of $\Obcl_{\sE,\sL}$ is identical to the one in \autocite{CG2}. One point requires comment, however: $\widetilde{\Obcl}$ is closed under the Chevalley-Eilenberg differential on $\Obcl_{\sE,\sL}$ because $(\condfieldscs[-1],\ell_l,\ip)$ is a precosheaf of cyclic $L_\infty$ algebras. Indeed, any functional with smooth first derivative is a sum of functionals of the form
\begin{equation}
I(e_1,\ldots,e_k) = \ip[D(e_1,\cdots, e_{k-1}), e_k],
\end{equation}  
where $D: \condfields^{\hotimes_\beta (k-1)} \to \condfieldscs[1]$ is a continuous map. One can check directly that applying the Chevalley-Eilenberg differential to such a functional gives another such functional.
It is important that $(\condfieldscs[-1],\ell_1,\ip)$ is a precosheaf of cyclic $L_\infty$ algebras because this allows one to ``integrate by parts,'' i.e. use the equality
\begin{equation}
\ip[D(e_1,\ldots,e_{k-1}),\ell_1 e_k]=\pm\ip[\ell_1 D(e_1,\ldots, e_{k-1}),e_k]
\end{equation}
and its analogues for the higher brackets $\ell_2,\ell_3, \ldots$.

The construction of the shifted Poisson bracket is also identical to the construction in \autocite{CG2}. The Weiss cosheaf condition will be satisfied once we prove that $\widetilde{\Obcl_{\sE,\sL}}\to \Obcl_{\sE,\sL}$ is a quasi-isomorphism. Hence, statement (3) of the theorem is the only one that remains to be proved.

The essential ingredient in the proof of statement (3) is the fact that the inclusion 
\begin{equation}
\condfieldscs(U)[1]\to \condfields^\vee(U)
\end{equation}
is a quasi-isomorphism with a homotopy inverse for certain open subsets $U$, cf. Proposition \ref{prop: atiyahbott}.

Just as in \autocite{CG2}, we may assume that the theory is free. We let 
\begin{equation}
\Sigma^k: \innerhom{\condfields(U)^{\hotimes_\beta k}}{\RR}\to \innerhomsym{\condfields(U)^{\hotimes_\beta k}}{\RR}{k}
\end{equation}
denote the symmetrization map. We let $\Gamma_n$ denote $(\Sigma^{k})^{-1}\widetilde{\Obcl_{\sE,\sL}}(U)$. 
For each $j=1,\ldots, k$ Let $\Upsilon_j$ denote the following map:
\begin{align}
    \Upsilon_j &: \innerhom{\condfields(U)^{\hotimes_\beta(k-1)}}{\condfieldscs(U)[1]}\to \innerhom{\condfields(U)^{\hotimes_\beta k}}{\RR}\\
    \Upsilon_j& (J) = \ip[J(e_1,\ldots, e_{j},e_{j+1},\ldots, e_k),e_j].
\end{align}
We will abuse notation and use $\Upsilon_j$ to denote also the image of the map described in the above equation.
Just as in \autocite{CG2}, we can identify
\begin{equation}
\Gamma_k = \bigcap_{j=1}^{k} \Upsilon_j.
\end{equation}
It suffices to show that the inclusion 
\begin{equation}
\label{eq: incl}
\Gamma_k \hookrightarrow \innerhom{\condfields(U)^{\hotimes_\beta k}}{\RR}
\end{equation}
is an equivalence, since symmetrization is an exact functor. More generally, let $\{U_i\}_{i=1}^k$ be open subsets of $M$, and define $\Gamma_{k,\{U_i\}}$ by a similar intersection.

We will show that the natural inclusion 
\begin{equation}
\label{eq: incltwo}
\Gamma_{k,\{U_i\}}\hookrightarrow \innerhom{\bigotimes_{i=1}^k\condfields(U_i)}{\RR}
\end{equation}
is a quasi-isomorphism possessing a homotopy inverse when each $U_i$ is either contained entirely in $M\backslash \bdyM$ or is of the form $U'_i\times [0,\delta_i)$ for $U'_i$ open in $\bdyM$ (and using our fixed tubular neighborhood of $\bdyM$). We will call such $U_i$ ``somewhat nice.'' Let us explain why this proves that the inclusion in Equation \eqref{eq: incl} is a quasi-isomorphism. It follows from  the results of Section \ref{sec: innerhom} that 
\begin{equation}
\innerhom{\bigotimes_{i=1,i\neq j}^{i=k} \condfields(U_i)} {\condfieldscs(U_j)}
\end{equation}
is a quotient of 
\begin{equation}
\innerhom{\bigotimes_{i=1,i\neq j}^{i=k} \sE(U_i)} {\condfieldscs(U_j)}
\end{equation}
Now, let $V=U^n$. We may cover $V$ by products of somewhat nice sets in $M$. Let $\fV=\{U_i\}_{i\in I}$ be such an open cover. By taking the dual statement to that of Lemma A.5.7 of \autocite{CG1}, we find a contracting homotopy for the mapping cone of the map
\begin{equation}
\check{C}(\fV, (\sE^\vee)^{\hotimes_\beta k})\to (\sE^\vee(V))^{\hotimes_\beta k}.
\end{equation} 
This contracting homotopy involves only multiplication by smooth, compactly-supported functions on $M^n=M\times \cdots \times M$. 
We have seen that this cochain homotopy descends to one for the map
\begin{equation}
\check{C}\left(\fV, \innerhom{\condfields^{\hotimes_\beta k}(\cdot)}{\RR} \right)\to \innerhom{\condfields(V)^{\hotimes_\beta k}}{\RR}.
\end{equation} 
Using the explicit characterization of the spaces in the intersection defining $\Gamma_{k,\{U_i\}}$, the cochain homotopy also descends to one for $\Gamma_{k,\{U_i\}}$.

 Hence, we have a commuting diagram
\begin{equation}
\begin{tikzcd}
\check{C}(\fV, \Gamma_{k,\cdot}) \ar[r,"\sim"]\ar[d]& \Gamma_{k}\ar[d]\\
\check{C}\left(\fV, \innerhom{\condfields^{\hotimes_\beta k}(\cdot)}{\RR}\right) \ar[r,"\sim"]& \innerhom{\condfields^{\hotimes_\beta k}(V)}{\RR}
\end{tikzcd}
\end{equation}
where the top and bottom maps are quasi-isomorphisms. 
Here, we abuse notation slightly and let $\Gamma_{n,V'}$ denote $\Gamma_{n,\{U_i\}}$ when $V'=U_1\times \cdots U_n$. 
We are interested in showing that the right-hand arrow in the diagram is a quasi-isomorphism. 
The finite intersection of any number of products of somewhat nice subsets is again a product of somewhat nice sets. 
Therefore, if the map of Equation \eqref{eq: incltwo} is a quasi-isomorphism for $U$ somewhat nice, the left-hand map in the above commuting diagram will be a quasi-isomorphism. 
It follows that the map of \eqref{eq: incl} will be a quasi-isomorphism, since that map is also the right-hand map in the commuting diagram.

Let us now proceed to show that the map of Equation \eqref{eq: incl} is a quasi-isomorphism. To prove this, it suffices---just as in the corresponding proof in \autocite{CG2}---to show that the map
\begin{equation}
\condfieldscs(U_j)[1]\to \innerhom{\condfields(U_j)}{\RR}
\end{equation}
is a quasi-isomorphism when $U\cap \bdyM=\emptyset$ or when $U\cong U' \times [0,\delta)$. For $U\cap \bdyM=\emptyset$, this is the Atiyah-Bott lemma (Appendix D of \autocite{CG1}). For $U\cong U'\times [0,\delta)$, this is shown in Proposition \ref{prop: atiyahbott}
\end{proof}

\section{Examples}
\label{sec: classexamples}
In this section, we study two examples, whose quantizations we study in detail in Chapters \ref{chap: freequantum} and \ref{chap: examples}.

\subsection{Topological Mechanics}
The goal of this section is to study the factorization algebra of topological mechanics. 
Recall from Examples \ref{ex: toplmech} and \ref{ex: toplmechbc} that a symplectic vector space $V$ and a Lagrangian $L\subset V$ define a free bulk-boundary system on $\R_{\geq 0}$ known as topological mechanics. 
The procedure of the previous section constructs a factorization algebra $\Obq_{V,L}$ on $\R_{\geq 0}$ for these choices. Our goal is to study this factorization algebra.

Given an associative algebra $A$ and a right $A$-module $M$, there is a factorization algebra $\cF_{A,M}$ on $\R_{\geq 0}$ which assigns $A$ to any open interval, and $M$ to any interval containing $0$ (see \S 3.3.1 of \autocite{CG1}) for details). 
The structure maps are given by the multiplication in $A$ and the right-module action of $A$ on $M$. 
We will see that the cohomology factorization algebra of topological mechanics is isomorphic to one of the form $\cF_{A,M}$, for appropriate $A$ and $M$. 

Let $\sO(V)=\Sym(V^\vee)$ denote the symmetric algebra of polynomial functions on $V$, and similarly for $\sO(L)$. The inclusion $L\to V$ induces a restriction of functions map $\sO(V)\to \sO(L)$ which defines a right $\sO(V)$-module structure on $\sO(L)$. 

We would like to say that $\Obcl_{V,L}$ is equivalent to $\cF_{\sO(V),\sO(L)}$; however, $\Obcl_{\sE,\sL}$ is defined in terms of a space of power series on $\condfields$, while $\sO(V)$ and $\sO(L)$ are polynomial algebras. 
To remedy this, one may also consider, for each $U$, the space $\Obcl_{V,L,poly}(U)$ consisting only of polynomial functions on $\condfields(U)$. It is easy to verify that $\Obcl_{V,L,poly}$ forms a sub-factorization algebra of $\Obcl_{V,L}$.

\begin{lemma}
The cohomology of the factorization algebra $\Obcl_{V,L,poly}$ is isomorphic to the factorization algebra $\cF_{\sO(V),\sO(L)}$. 
\end{lemma} 

\begin{proof}
Let us choose a Lagrangian complement $L'$ to $L$ in $V$.
The sheaf of $\sL$-conditioned fields is
\begin{equation}
 \condfields(U) = \Omega^\bullet_M(U)\otimes L\oplus \Omega^\bullet_{M,D}(U)\otimes L'.
\end{equation}
Here, $\Omega^\bullet_{M,D}(U)$ is $\Omega^\bullet_M(U)$ is $U$ does not contain $t=0$ and otherwise is the space of de Rham forms on $U$ whose pullback to $t=0$ vanishes.
Let $\cS_{pre}$ the presheaf of vector spaces on $M$ which assigns $V$ to any open set not containing 0 and $L$ to any open set which does contain 0.
Let $\cS$ denote the sheafification of $\cS$.
There is a natural map of sheaves $\cS\to \condfields$ and it is straightforward to show that this map is a quasi-isomorphism.
Hence, there is also a quasi-isomorphism
\begin{equation}
    \Obcl \to \sO(\cS).
\end{equation}
One quickly verifies that $\sO(\cS)=\cF_{\sO(V),\sO(L)}$.
This completes the proof.
\end{proof}

\subsection{BF Theory in One Dimension}

In this section, we study the factorization algebra of observables of one-dimensional BF theory with $B$ boundary condition (cf. Examples \ref{ex: bf} and \ref{ex: bfbdycond}).

Namely, we show:

\begin{proposition}
\label{prop: 1dbfclass}
Let $A$ denote the algebra
\[
C^\bullet(\fg, \Sym(\fg[1]))
\]
and 
$M$ the right $A$-module
\[
\Sym(\fg[1]).
\]
There is a quasi-isomorphism of factorization algebras
\begin{equation}
    \Upsilon \Obcl_{\sE,\sL}\to : \cF_{A,M}
\end{equation}
on $\RR_{\geq 0}$, where $\Obcl_{\sE,\sL}$ is the factorization algebra of observables for BF theory with $B$ boundary condition.
\end{proposition}

\begin{proof}
Let $\cS$ denote the sheafification of the following presheaf on $\RR_{\geq 0}$:
\begin{equation}
    \cS_{pre}(U) = \left\{
    \begin{array}{lr}
    \fg\ltimes \fg^\vee[-2]     & U\cap \{0\}=\emptyset  \\
    \fg^\vee[-2] & U \cap \{0\} = \{0\}     
    \end{array}
    \right.
\end{equation}
As we have written it, $\cS$ is naturally a sheaf of graded Lie algebras on $\RR_{\geq 0}$.
The map 
\begin{equation}
\label{eq: 1dbfclassquasiisooffields}
    \cS \to \condfields[-1]
\end{equation}
which includes elements $\fg\ltimes \fg^\vee[-1]$ as constant functions is a map of sheaves of dg Lie algebras; in fact, it is straightforward to show that it is a quasi-isomorphism.
The inverse quasi-isomorphism
\begin{equation}
    \condfields(U)[-1]\to \cS(U),
\end{equation}
for $U$ a connected open subset, is obtained by pulling back along the inclusion of a point $t\in U$ into $U$.
This does not define a map of sheaves of dg Lie algebras, however.
Let $C^\bullet(\cS)$ be the factorization algebra on $\RR_{\geq 0}$ which, to an open $U\subset \RR_{\geq 0}$, assigns the Chevalley-Eilenberg cochain complex of $\cS(U)$ (cf. Definition 3.6.1 of \autocite{CG1}, though the construction here is slightly different).
The map of Equation \ref{eq: 1dbfclassquasiisooffields} induces a quasi-isomorphism
\begin{equation}
    \Obcl_{\sE,\sL}\to C^\bullet(\cS);
\end{equation}
it is straightforward to verify that~$C^\bullet(\cS)\cong \cF_{A,M}$.
This completes the proof.
\end{proof}

%% file: Chapters/LocalFunctionals.tex
\section{Local Action Functionals and $D$-modules}
\label{sec: localfcnls}
In this section, we use the language of $D$-modules to describe the local action functionals in a TNBFT. Our approach extends the discussion of Section 6 of Chapter 5 of \cite{cost}. 
In so doing, we step somewhat outside the main line of development of this chapter.
We will construct a complex of local action functionals; this complex can be thought of as the complex of first-order deformations of the bulk-boundary system $(\sE, \ell_k, \cdots,\sL)$, i.e. as the tangent space to the space of classical bulk-boundary systems at the point defined by $(\sE, \ell_k, \cdots, \sL)$.  
The results of this section are not needed in the construction of a factorization algebra of classical observables, which we discuss in the next section; however, we present them here because they concern the classical theory and will be relevant in the obstruction theory for the quantization of a given classical theory.

\subsection{Local action functionals}

When $\partial M=\emptyset$, one defines local action functionals to be of the form 
\begin{equation}
\varphi \mapsto \int_M D_1 \varphi\cdots D_k \varphi d\mu,
\end{equation}
where the $D_i$ are differential operators $\sE\to \cinfty_M$ and $d\mu$ is a density on $M$. We will define these to be local action functionals in the case $\bdyM\neq \emptyset$ as well, but one finds that certain local functionals in this sense (if they are total derivatives) can be written as integrals of a similar nature, \emph{but over $\partial M$}. Hence, the distinction between a local action functional and a codimension-one operator becomes blurry.

Before we define local functionals, we make the following definition, which will be useful in the sequel---in this chapter and especially in Chapter \ref{chap: interactingquantum}.

\begin{definition}
Let $V$ be a locally convex topological vector space.
The space \textbf{power series} on $V$ is the following concenient vector space
\begin{equation}
    \sO(V):= \prod_{k=0}^\infty \underline{CVS}(V^{\hotimes_\beta k},\RR),
\end{equation}
where the subscript ``bdd'' means we take only the bounded linear maps of topological vector spaces, and the tensor product $\hotimes_\beta$ is the completed bornological tensor product (Section B.5.2 of \autocite{CG1}). 
\end{definition}

We may therefore define $\sO(\sE), \sO(\sE_c), \sO(\condfields),$ and $\sO(\condfieldscs)$.
\index[notation]{OE1@$\sO(\sE)$}\index[notation]{OE2@$\sO(\condfields)$}\index[notation]{OE3@$\sO(\sE_c)$}\index[notation]{OE4@$\sO(\condfieldscs)$}

\begin{definition}
\label{def: localfcnls}
The \textbf{space of local action functionals} for the field theory $\sE$ with boundary condition $\sL$ is the (linear) subspace of $\sO(\condfieldscs)$
spanned by functionals of the form
\begin{equation}
\label{eq: localfcnls}
I(\varphi)= \int_M D_1\varphi \cdots D_j\varphi d\mu.
\end{equation}
Here, $\condfieldscs$ denotes the subspace of those fields which have compact support and satisfy the required boundary condition; all $D_i$ are differential operators on $\sE$, and in particular, may depend on the normal derivatives of the field $\varphi$ near the boundary. 
\end{definition}

\begin{notation}
\label{not: lclfcnls}
Let us establish the following notation:
\begin{itemize}
\item We denote the space of local functionals by $\Oloc$.
\item We denote by the symbol $\Olocred$ the quotient of $\Oloc$ by the space of constant functionals.
\item We denote by the symbol $\sO_{loc}(\sE)$ the space of functionals in $\sO(\sE_c)$ of the form \eqref{eq: localfcnls}.
\item We denote by the symbol $\sO_{loc,red}(\sE)$ the quotient of $\sO_{loc}(\sE)$ by the space of constant functionals.
\end{itemize}
\end{notation}
\begin{remark}
\label{rem: loclfcnlsrmk}
\begin{enumerate}
\item Though $\Oloc$ is a subspace of $\sO(\condfieldscs)$ (and this latter space is endowed with a topology), we do not wish to view $\Oloc$ as a topological vector space. We simply remember the vector space structure underlying $\Oloc$.
\item Note that the same functional can be written in two different ways, i.e. the representation of a local functional $I(\phi)$ in the form $\int_M D_1\varphi \cdots D_j \varphi d\mu$ is not unique.  
\end{enumerate}
\end{remark}

We note that $\Oloc$ contains functionals of the form 
\begin{equation}
I(\varphi)=\int_{\partial M} D_{1}(\varphi)\mid_{\partial M}\cdots D_{k}(\varphi)\mid_{\partial M}\d\nu=-\int_M \frac{\del}{\del t}\left( f_0(t) D_1\phi \cdots D_k\phi\right) \d \nu\, \d t,
\end{equation}
where the $D_i$ are differential operators on $\sE$ and $d\nu$ is a density on $\bdyM$. 
Here, $f_0$ is a compactly supported function on $[0,\epsilon)$ which is $1$ at $t=0$.
Writing $I$ as on the right-hand side of the above equation makes clear that it is local.

We obtain a natural filtration of $\Oloc$ with $F^0\Oloc$ the space of such functionals and $F^1\Oloc = \Oloc$. This filtration will be useful to us in the sequel.

The Chevalley-Eilenberg differential on $\sO(\condfieldscs)$ preserves $\Oloc$, so that $\Oloc$ is a cochain complex. Moreover, $\Oloc$ evidently has the structure of a presheaf, since $\condfieldscs$ is a cosheaf. Thus, $\Oloc$ is a complex of presheaves. It is also easily seen to be a complex of sheaves on $M$.

\subsection{The case $\bdyM=\emptyset$}
\label{sec: localfcnlsnobdy}
Let us first briefly review the relationship between $D$-modules and local action functionals in the case $\bdyM=\emptyset$. Our discussion follows Section 6 of Chapter 5 of \cite{cost}.

\begin{notation}
Given a vector bundle $E$ over a manifold $M$, we let $J(\sE)$ denote the (infinite-rank) bundle over $M$ whose fiber at a point $x\in M$ is the space of formal germs of sections of $E$ at the point $x$. We let $\sJ(\sE)$ denote the sheaf of sections of $J(\sE)$. $\sJ(\sE)$ has a natural topology which it obtains as the inverse limit of the finite-rank bundles of $r$-jets (as $r\to \infty$). We let $\sJ(\sE)^\vee$ denote the $\cinfty_M$ module
\begin{equation}
\Hom_{\cinfty_M}(\sJ(\sE),\cinfty_M),
\end{equation}
where $\Hom_{\cinfty_M}$ denotes $\cinfty_M$-module homomorphisms which are continuous for the inverse limit topology on $\sJ(\sE)$.

We let 
\begin{equation}
\sO(\sJ(\sE))=\prod_{k\geq 0}\Sym^k_{\cinfty_M} \sJ(\sE)^\vee.
\end{equation} 
The symbol $\sO_{red}(\sJ(\sE))$ denotes a similar product over $k>0$.
\end{notation}
$J(\sE)$ carries a canonical connection and so $\sJ(\sE)$ carries the structure of a left $D_M$-module. This $D_M$ module structure extends to $\sJ(\sE)^\vee$ and $\sO(\sJ(\sE))$. The $L_\infty$ brackets on $\sO(\sE)$ can be used to endow $\sO(\sJ(\sE))$ with a Chevalley-Eilenberg differential. Let us recall also that the bundle of densities $\Dens_M$ carries a structure of right $D_M$-module, where vector fields act on densities by the formal adjoints of their actions on functions (equivalently, by Lie derivatives).
When $\bdyM\neq \emptyset$, we will see that these two $D_M$-module structures are no longer the same; in fact, one of these $D_M$-module structures will cease to be defined.

The following is Lemma 6.6.1 of \cite{cost}:
\begin{lemma}
Suppose $\bdyM=\emptyset$. There is a canonical map of sheaves
\begin{equation}
\Dens_M\otimes_{D_M}\sO(\sJ(\sE))\to \sO_{loc}(\sE),
\end{equation}
and it is an isomorphism.
\end{lemma}

Note that the $D_M$-module tensor product ensures that Lagrangian densities which are total derivatives give zero as functionals.

We would like to resolve $\Dens_M$ as a right $D_M$-module. 

\begin{definition}
The \textbf{$D_M$-de Rham complex} is the complex of sheaves
\begin{equation}
\dmdr:=\Omega^\bullet_{M,tw}\otimes_{\cinfty_M} D_M;
\end{equation}
this is the de Rham complex for the flat vector bundle $D_M\otimes \mathfrak o $, where $\mathfrak o$ is the orientation line bundle on $M$. The complex of sheaves $\dmdr$ has the canonical structure of a right $D_M$ module. 
\end{definition}

There is a canonical map $\dmdr[n] \to \Omega^n_{M,tw}$ induced from the right $D_M$-module action
\begin{equation}
\Omega^n_{M,tw}\otimes_{\cinfty_M} D_M \to \Omega^n_{M,tw}
\end{equation}
This map is a quasi-isomorphism, since this complex is locally the Koszul resolution of $\Omega^n_{M,tw}(U)$ with respect to the regular sequence $\{\partial_I\}_{i=1}^n$ of $D_M(U)$. 
Combining this information and Lemma 6.6.4 of Chapter 5 of \cite{cost}, we have 
\begin{lemma}
\label{lem: derhamloc}
Let $\bdyM=\emptyset$. The canonical map 
\begin{equation}
\dmdr[n]\otimes_{D_M}\sO_{red}(\sJ(\sE))\to \Olocred
\end{equation}
is a quasi-isomorphism.
\end{lemma}

\subsection{Pushforwards and pullbacks of $D$-modules}

Before generalizing to $\bdyM\neq \emptyset$, we review a number of constructions common in the formalism of $D$-modules. 
In the sequel, there will be a number of situations where we will be presented with a $D_M$-module but will want a $D_\bdyM$-module or vice versa;
this section is intended to give us the fluidity to move between the two notions.
We will discuss two functors, $\dpush$ and $\dpull$.
The functor $\dpush$ will take right $D_\bdyM$-modules to right $D_M$-modules, and $\dpull$ will do the opposite for \emph{left} $D$-modules.
We note that these techniques can be used in a more general situation, namely one where $\bdyM$, $M$, and $\iota$ are replaced with general smooth manifolds $X$ and $Y$ and a general smooth map $f:X\to Y$.
Though we do not explicitly make the constructions for the general setup, 
it will be clear how to make the necessary (very small) modifications for the general situation.

We make no claims of originality here.

In both constructions, we will make use of the following sheaf on $\bdyM$:
\begin{equation}
\label{eq: pullD}
 D_{\bdyM \to M}:=\cinfty_\bdyM \otimes_{\iota^{-1}(\cinfty_M)} \iota^{-1}(D_M);
\end{equation}
this sheaf possesses a manifest right $\iota^{-1}(D_M)$ action.
Let $V\to M$ momentarily denote the bundle on $M$ whose sheaf of sections is~$D_M$.
(Concretely, $V= J(\underline \RR)^\vee$.)
One observes that $D_{\bdyM\to M}$ is simply the sheaf of sections of~$\iota^*(V)$.
From this characterization of $D_{\bdyM \to M}$, one finds also a natural left $D_{\bdyM}$ action on $D_{\bdyM \to M}$, using the pullback connection on~$\iota^*(V)$.
In the language of Equation \ref{eq: pullD}, given $X\in \textrm{Vect}(\bdyM)$, $f\in \cinfty_\bdyM$, and $D\in \iota^{-1}(D_M)$, we define
\begin{equation}
\label{eq: pullbackconnection}
X\cdot \left( f\otimes D \right):= (Xf)\otimes D +f\otimes (\widehat X D),
\end{equation}
where $\widehat X$ is a vector field on $M$ such that 
\begin{equation}
\widehat X \mid_\bdyM = X
\end{equation}
in the space $\Gamma(\bdyM,\iota^* TM)$.
One checks readily that this definition does not depend on $\widehat X$ and respects the tensor product over $\iota^{-1}(\cinfty_M)$.

\subsubsection{Pushforward}
Let $P$ be a right $D_{\bdyM}$ module. We may define 
\begin{equation}
\label{eq: pushM}
\dpush P = \iota_*\left(P\otimes_{D_\bdyM} D_{\bdyM \to M}\right);
\end{equation}
the sheaf 
\begin{equation}
\iota_*(D_{\bdyM \to M})
\end{equation}
possesses a natural right $\iota_*\iota^{-1}D_M$ module structure, and hence a natural right $D_M$ structure via the map of sheaves of algebras $D_M\to \iota_*\iota^{-1}D_M$.
It follows that $\dpush P$ is a right $D_M$ module.

Note that, for $P=D_\bdyM$, one finds~$\dpush D_\bdyM \cong \iota_*(D_{\bdyM \to M})$.
\subsubsection{Pullback}

If $\sV$ is the sheaf of sections of a bundle $V\to M$, then to give a left $D_M$-module structure on $\sV$ is to give a flat connection on $V$.
We may form the pullback $\iota^* V$ on $\bdyM$; the bundle $\iota^*V$ has a natural ``pullback'' flat connection, so that its sheaf of sections is a $D_\bdyM$-module.
In other words, for $D_M$-modules arising in this way as the sheaf of sections of a vector bundle, we have discussed a construction of pullback to $\bdyM$.
Our aim in this section is to give the analogous construction for general left $D_M$-modules.

Hence, given a left $D$ module $P$, define 
\begin{equation}
\label{eq: pullM}
\dpull P := D_{\bdyM \to M} \otimes_{\iota^{-1}(D_M)} \iota^{-1}(P).
\end{equation}
If $P$ is the left $D_M$ module $D_M$, one finds that~$\dpull D_M \cong D_{\bdyM \to M}$.

\subsubsection{A useful compatibility between $\dpull$ and $\dpush$}
In this subsection, we note 

\begin{lemma}
\label{lemma: pullpush}
Let $P$ be a right $D_\bdyM$ module and $Q$ a left $D_M$-module. There are natural isomorphisms
\begin{equation}
\dpush P \otimes_{D_M} Q \cong \iota_* \left( P\otimes_{D_\bdyM} \dpull Q\right)
\end{equation}
and
\begin{equation}
\iota^{-1}\left( \dpush P\otimes_{D_M} Q\right) \cong P \otimes_{D_{\bdyM}} \dpull Q.
\end{equation}
of sheaves on $M$ and $\bdyM$, respectively.
\end{lemma}

\subsection{$\sO_{loc}(\sE)$ in the $D$-module language}
\label{sec: OlocE}
We now proceed to generalize the discussion of Section \ref{sec: localfcnlsnobdy} to the case that $\bdyM\neq \emptyset$. There are two new subtleties that arise in this situation. 
The first is that, when $\bdyM\neq\emptyset$, Lagrangian densities which are total derivatives no longer necessarily give zero as local action functionals.
The second is that we have imposed boundary conditions on the fields, so some functionals vanish because they depend on components of the boundary value of fields which are constrained to be zero.
In this section, we address the first subtlety.
In Section \ref{sec: OlocEL}, we address the second.
To the aim of understanding this first subtlety, consider the inclusion of $\cinfty_M$ modules $\Omega^{n}_{M,tw} \hookrightarrow \cD_M$ (where $\cD_M$ object is the sheaf of distributions on $M$). 
Whether or not $\bdyM = \emptyset$, $\cD_M$ has a $D_M$-module structure; however, only when $\bdyM=\emptyset$ is the subspace of smooth densities closed under this $D_M$ action.
To see this, let $\omega$ be a density on $M$, let $T_\omega$ denote the corresponding distribution on $M$, i.e.
\begin{equation}
T_\omega(f) = \int_M f\omega.
\end{equation}
Let $X$ be a vector field on $M$; then, we find upon integration by parts that
\begin{equation}
\label{eq: dmdens}
\left(T_{\omega}\cdot X\right) (f) = T_\omega(X\cdot f) = \int_M (Xf)\omega = -\int_M f L_X \omega + \int_\bdyM f\mid_{\bdyM} \iota^* i_X \omega.
\end{equation}
From the last term in this equation, it follows that $(T_\omega)\cdot X$ is no longer in general a smooth density on $M$ if $\bdyM\neq \emptyset$.

On the other hand, $\Omega^n_{M,tw}$ does possess a $D_M$-module structure, even when $\bdyM\neq \emptyset$, namely the one induced from the action of vector fields on smooth densities by Lie derivatives.
(In other words, this is the action obtained by keeping the first of the two terms on the right hand side of Equation \eqref{eq: dmdens}.)
We emphasize that this is \emph{not} the $D_M$-module structure relevant for the construction of local functionals.
To see this, note that if $D_1\cdot\ldots \cdot D_k \in \sO(\sJ(\sE))$ and $\mu\in \Omega^{n}_{M,tw}$, we would like to understand $\mu\otimes D_1\cdots \ldots \cdot D_k$ as the functional
\begin{equation}
\int_M \mu\, D_1\phi \,\cdot \ldots \cdot \, D_k \phi\,;
\end{equation}
in other words, densities appear in local functionals in their capacity as distributions on $M$, whether or not $\bdyM=\emptyset$.
Of course, when $\bdyM=\emptyset$, Equation \eqref{eq: dmdens} shows that this $D_M$-module structure coincides with the one obtained by viewing densities as distributions.

We therefore seek the minimal sub-$D_M$-module of $\cD_M$ containing $\Omega^n_{M,tw}$.
That is the purpose of the following definition.

\begin{definition}
The sheaf of \textbf{$\partial$-smooth densities} $\delsmoothdistr$ is the smallest right sub-$D_M$ module of $\cD_M$ containing the smooth densities. 
\end{definition}

Let us come to a better understanding of the sheaf~$\delsmoothdistr$. 
This sheaf certainly contains the sheaf~$\Omega^n_{M,tw}$.
By Equation \eqref{eq: dmdens}, we see that it also contains the sheaf~$\iota_*\left(\Omega^{n-1}_{\bdyM,tw}\right)$, and therefore also the $D_M$-module this sheaf generates in $\cD_M$.
This last $D_M$ module is characterized in the following Lemma:

\begin{lemma}
\label{lem: bdydens}
The right $D_M$-module generated by $\iota_*\left(\Omega^{n-1}_{\bdyM,tw}\right)$ inside of $\cD_M$ is isomorphic to the right $D_M$-module
\begin{equation}
F^0 \delsmoothdistr:= \dpush \Omega^{n-1}_{\bdyM,tw}
\end{equation}
(see Equation \ref{eq: pushM} for the notation~$\dpush$).
\end{lemma}

\begin{remark}
The notation $F^0\delsmoothdistr$ is designed to suggest that this submodule is the zeroth component of a filtration on $\delsmoothdistr$. We will see later that this is indeed the case.
\end{remark}

\begin{proof}
Let us first construct a map 
\begin{equation}
\Upsilon: F^0 \delsmoothdistr\to \delsmoothdistr.
\end{equation}
Let $U\subset M$, with $V=U\cap \bdyM$. Let $\nu\in \Omega^{n-1}_{\bdyM,tw}(V)$, $f\in \cinfty_{\bdyM}(V)$, $g\in \cinfty_{c,M}(U)$, and let $D$ be the germ of a differential operator on a neighborhood of $V$ in $M$. Define
\begin{equation}
\Upsilon(\nu\otimes f\otimes D)(g) =\int_\bdyM (Dg)\mid_{\bdyM}f\nu.
\end{equation}
The map $\Upsilon$ is manifestly a sheaf map, assuming that it is well-defined.
Let us show that it is well-defined, i.e. that the relations imposed by the tensor products in $F^0\delsmoothdistr$ go to zero in $\delsmoothdistr$.
To this end, suppose that $h$ is a germ of a smooth function in a neighborhood of $V$ in $M$. Then, it is straightforward to verify that
\begin{equation}
\Upsilon( \nu\otimes f(h\mid_{\bdyM})\otimes D- \nu\otimes f\otimes hD)(g) = 0.
\end{equation}
It is likewise easy to verify that $\Upsilon$ is linear over $\cinfty_\bdyM\subset D_{\bdyM}$ and respects the right $D_M$-module structures.
Finally, let $X\in Vect(\bdyM)$ and choose some extension $\widehat X\in Vect(M)$ as in the discussion preceding Equation \eqref{eq: pullbackconnection}.
Then,
\begin{align}
\Upsilon(\nu\cdot X\otimes &f\otimes D -\nu \otimes X\cdot f\otimes D-\nu\otimes f\otimes \widehat X D)(g)
\\&=
-\int_\bdyM (Dg)\mid_\bdyM f L_X\nu-\int_\bdyM (Dg)\mid_{\bdyM} L_X f\nu - \int_\bdyM (\widehat X Dg)\mid_{\bdyM} f\nu.
\end{align}
One checks that $(\widehat X Dg)\mid_{\bdyM} = X\cdot (Dg\mid_{\bdyM})$, so that we have
\begin{align}
\Upsilon(\nu\cdot X\otimes &f\otimes D -\nu \otimes X\cdot f\otimes D-\nu\otimes f\otimes \widehat X D)(g)
\\&= -\int_\bdyM L_X\left( (Dg)\mid_{\bdyM} f\nu\right) = 0.
\end{align}
We have therefore established that $\Upsilon$ is a well-defined sheaf map.

It remains to show that $\Upsilon$ is a monomorphism, and that the image of $F^0\delsmoothdistr$ under $\Upsilon$ is the smallest submodule of $D_M$ containing the densities on $\bdyM$.
Choosing a tubular neighborhood (and a normal coordinate $t$), one obtains an isomorphism $D_{\bdyM \to M}\cong D_\bdyM[\partial_t]$ (of sheaves).
Here, by $D_\bdyM[\partial_t]$, we mean the sheaf $D_\bdyM[\partial_t]\otimes \underline{\RR[\partial_t]}$, where $\underline{\RR[\partial_t]}$ is the locally constant sheaf on $\bdyM$ assigning the space of polynomials in the variable $\del_t$.
One therefore finds that 
\begin{equation}
\label{eq: noncanonicaliso1}
F^0\delsmoothdistr \cong \iota_*\left( \Omega^{n-1}_{\bdyM, tw}[\partial_t]\right).
\end{equation}
This is an isomorphism of $\cinfty_M$-modules (in any case, we have not even described a $D_M$-module structure on the right-hand side of the equation).
The composite of this isomorphism with $\Upsilon$ takes $\nu \partial_t^n$ to the distribution (on $M$)
\begin{equation}
\label{eq: codim1dens}
g\mapsto \int_\bdyM (\partial_t^n g)\mid_{t=0}\nu. 
\end{equation}
From this characterization of $\Upsilon$, it is clear that $\Upsilon$ is a monomorphism.

Finally, we verify that $F^0\delsmoothdistr$ is the smallest sub-$(D_M)$-module of $\cD_M$ containing the densities on $\bdyM$.
Here, again, the isomorphism of Equation \ref{eq: noncanonicaliso1} is useful. Any sub $D_M$-module of $\cD_M$ which contains distributions of the form
\begin{equation}
g\mapsto \int_{\bdyM} g\mid_{t=0}\nu
\end{equation}
must also contain distributions of the form 
\begin{equation}
g\mapsto \int_{\bdyM}(\partial^n_tg)\mid_{t=0}\nu,
\end{equation}
by $D_M$-closedness.
But, Equation \ref{eq: noncanonicaliso1} tells us that $F^0\delsmoothdistr$ is precisely the space of all such distributions.

We have therefore seen that any sub-$D_M$-module of $\cD_M$ which contains $\Omega^{n-1}_{tw, \bdyM}$ also contains $F^0\delsmoothdistr$; since $F^0\delsmoothdistr$ is a sub-$D_M$-module of $\cD_M$ (via $\Upsilon$), this completes the proof.
\end{proof}

As in Notation \ref{not: lclfcnls}, let $\sO_{loc}(\sE)$ denote the space of functions on $\sE_c$ which are of the same form as those which define $\Oloc$ (the only difference being we do not identify two such if they agree on $\condfieldscs$). 

\begin{lemma}
There is a natural map 
\begin{equation}
 \digamma: \delsmoothdistr\otimes_{D_M}\sO(\sJ(\sE))\to \sO_{loc}(\sE)
\end{equation}
of $\cinfty_M$-modules.
\end{lemma}

\begin{proof}
The map is constructed just as in the case $\bdyM=\emptyset$.
More precisely, one can identify $\sO(\sJ(\sE))$ with the space of polydifferential operators on $E$ with values in smooth functions on $M$.
In other words, $\sO(\sJ(\sE))$ consists of spaces of operators which take in some number of sections of $E$ and produce a function on $M$.
We write such an object as $D_1\cdot D_2\cdot \cdots \cdot D_k$.
We define
\begin{equation}
\digamma (\omega \otimes D_1\cdot \cdots \cdot D_k) = \omega\left( D_1\phi \cdots D_k\phi\right).
\end{equation}
Here, we are viewing $\omega$ as a distribution on $M$.
(The space $\delsmoothdistr$ consists of distributions of a particular form.)
The map $\digamma$ respects the $D_M$-module tensor product for the same reason that it did in the case $\bdyM=\emptyset$.
There, it was important to use the right $D_M$-module structure on $\Omega^n_{tw, M}$ which is induced from the embedding $\Omega^n_{tw,M}\hookrightarrow \cD_M$.
By defining the $\delsmoothdistr$ as we have---namely as a subspace of the distributions on $M$---we have guaranteed that the same argument applies here.
It remains only to check that when $\omega\in \delsmoothdistr$, 
\begin{equation}
\omega\left( D_1\phi \cdots D_k\phi\right)
\end{equation}
is a local action functional.
To this end, our description of $\delsmoothdistr$ in the preceding lemma will be useful.
The distribution $\omega$ contains summands of two forms. Summands of the first type are associated to a smooth density $\mu$ on $M$ and manifestly produce local action functionals.
Summands of the second type are associated to a smooth density $\nu$ on $\bdyM$ and a natural number $n$ and are of the form in Equation \ref{eq: codim1dens}.
Such summands produce functionals of the form
\begin{equation}
I_{\nu, n}(\phi)=\int_\bdyM \nu \partial_t^n\mid_{t=0} \left( D_1\phi\cdots D_k\phi\right).
\end{equation}
Note that such a description requires a choice of tubular neighborhood $\tubnhd\cong \bdyM\times [0,\epsilon)$ of $\bdyM$ in $M$.
We continue to use this choice as follows. 
Let $f_0$ be a compactly-supported function on $[0,\epsilon)$ which is 1 at $t=0$.
Together with $\nu$, such a function determines a density on $M$ which we will call $(f_0 dt)\wedge\nu$.
It is straightforward to check that the functional
\begin{align}
J_{\nu, n}:=\int_M \frac{\del}{\del t}\left( f_0 dt \wedge \nu \left(\frac{\del}{\del t}\right)^n\left(D_1\phi\cdots D_k \phi\right) \right)&= \int_M \frac{df_0}{dt}\wedge \nu \left( \frac{\del}{\del t}\right)^n \left(D_1\phi \cdots D_k \phi\right)\nonumber\\
&+\int_M f_0 dt \wedge \nu \left(\frac{\del}{\del t}\right)^{n+1}\left( D_1\phi \cdots D_k\phi\right)
\end{align}
coincides with $I_{\nu, n}$.
Since $J_{\nu, n}$ is manifestly local in our definition, the lemma follows.
\end{proof}

\begin{lemma}
The natural map 
\begin{equation}
 \digamma: \delsmoothdistr\otimes_{D_M}\sO(\sJ(\sE))\to \sO_{loc}(\sE)
\end{equation}
is an isomorphism of sheaves on $M$.
\end{lemma}
\begin{proof}
Both the domain and codomain of $\digamma$ are filtered as follows. $F^0\sO_{loc}(\sE)$ is the space of all action functionals which may be written as integrals over $\bdyM$, and $F^1\sO_{loc}(\sE):=\sO_{loc}(\sE)$.
Similarly, we set 
\begin{equation}
F^0\left(\delsmoothdistr\otimes_{D_M}\sO(\sJ(\sE))\right) = (F^0\delsmoothdistr)\otimes_{D_M}\sO(\sJ(\sE)),
\end{equation}
and 
\begin{equation}
F^1\left(\delsmoothdistr\otimes_{D_M}\sO(\sJ(\sE))\right)=\delsmoothdistr\otimes_{D_M}\sO(\sJ(\sE)).
\end{equation}
The map $\digamma$ manifestly preserves these filtrations.
The associated graded object for $\sO_{loc}(\sE)$ is the direct sum 
\begin{equation}
\textrm{LMB}\oplus F^0\sO_{loc}(\sE),
\end{equation}
where $\textrm{LMB}$ is the space of local functionals modulo those which arise as integrals over the boundary.
On the other hand, the associated graded object for the domain of $\digamma$ is 
\[(\Omega^n_{M,tw}\oplus F^0\Omega^{n,n-1}_{M,tw})\otimes_{D_M} \sO(\sJ(\sE)),\] 
where $\Omega^{n}_{M,tw}$ has the right $D_M$-module structure given by Lie derivative along vector fields.
The map $\digamma$ is an isomorphism on the $F^1/F^0$ components by the same argument as for the case $\bdyM=\emptyset$.
To complete the proof, therefore, we need to argue that the $\digamma$ is an isomorphism on the $F^0$ pieces.
To this end, we use Lemma \ref{lemma: pullpush} to find:
\begin{equation}
F^0 \Omega^{n,n-1}_{\bdyM,tw}\otimes_{D_M}\sO(\sJ(\sE))= \dpush \Omega^{n-1}_{\bdyM, tw} \otimes_{D_M} \sO(\sJ(\sE))\cong \iota_*\left( \Omega^{n-1}_{\bdyM,tw}\otimes_{D_{\bdyM}} \dpull \sO(\sJ(\sE))\right).
\end{equation}
And now, we proceed as in the case $\bdyM=\emptyset$; the $D_\bdyM$-module tensor product now tells us that total derivatives on $\bdyM$ are zero.
The only difference here is that we use $\dpull \sO(\sJ(\sE))$ instead of $\sO(\sJ(\sF))$ for some vector bundle $F\to \bdyM$.
However, this precisely matches the fact that $F^0\sO_{loc}(\sE)$ contains functionals which depend on the normal derivatives of fields at the boundary.
\end{proof}

In fact, we also have a corresponding statement for the \emph{derived} $D_M$ module tensor product:

\begin{lemma}
\label{lem: derivedisunderived}
Let $\sO_{red}(\sJ(\sE)):=\sO_{red}(\sJ(\sE))/\cinfty_M$; then $\sO_{red}$ is a flat left $D_M$-module and hence 
\begin{equation}
\delsmoothdistr\otimes_{D_M}\sO_{red}(\sJ(\sE))\simeq  \Omega_{M,tw}^{n,n-1}\otimes_{D_M}^\LL\sO_{red}(\sJ(\sE))
\end{equation}
and there is a quasi-isomorphism
\begin{equation}
\delsmoothdistr\otimes_{D_M}^\LL\sO(\sJ(\sE))\simeq \sO_{loc,red}(\sE)
\end{equation}
of sheaves (of $\cinfty_M$ modules).
\end{lemma}

\begin{proof}
The proof is exactly the same as the argument used in the proof of Lemma 6.6.2 of \cite{cost}.
\end{proof}

Since Lemma \ref{lem: derivedisunderived} holds, we can replace $\delsmoothdistr$ with a quasi-isomorphic $D_M$ module, and we would have a quasi-isomorphic model of $\sO_{loc,red}(\sE)$. Hence, we seek a resolution of $\delsmoothdistr$ as a right $D_M$-module.
This is what we proceed to do now.
As before, the flat connection on $D_M$ provides a differential on the graded sheaf
\begin{equation}
\dmdr:=\Omega^\bullet_{M,tw}\otimes_{\cinfty_M} D_M;
\end{equation}
a similar argument endows the graded sheaf (on $M$)
\begin{equation}
\dmdrbdy:=\iota_*\left(D_{\bdyM\to M}\right)
\end{equation}
with a differential, and a natural right $D_M$-module structure via the map $D_M\to \iota_*\iota^{-1} D_M$.
There is a map
\begin{equation}
\label{eq: pullbackdeRham}
\iota^*: \dmdr \to \dmdrbdy
\end{equation}
defined using the pullback of forms and the unit map $D_M \to \iota_*\iota^{-1}D_M$.

\begin{definition}
The \textbf{universal de Rham resolution of }$\delsmoothdistr$ is the shifted mapping cone
\begin{equation}
\dmdrrel := \cone(\iota^{*})[n-1]
\end{equation}
of the map $\iota^*$.
\end{definition}

The following lemma justifies our introduction of $\dmdrrel$.

\begin{lemma}
\label{lem: drres}
There is a natural quasi-isomorphism
\begin{equation}
\Phi:\dmdrrel \to \delsmoothdistr
\end{equation}
of right $D_M$-modules.
\end{lemma}

\begin{proof}
Given $\alpha \in \Omega^n_{tw,M}$, $\beta\in \Omega^{n-1}_{tw,\bdyM}$, $D_1,D_2\in D_M$, we let
\begin{equation}
\Phi(\alpha\otimes D_1, \beta\otimes D_2) = \alpha\cdot D_1 + \beta \cdot D_2;
\end{equation}
on the right-hand side, we understand $\alpha$ and $\beta$ as elements of $\delsmoothdistr$ and $\cdot$ denotes the right $D_M$-module action in this space. This is manifestly a map of right $D_M$-modules, and $\Phi$ respects the tensor product relations defining $\dmdr$ and $\dmdrbdy$. 
Let us check that $\Phi$ respects the differentials. Namely, we are required to check that if $\mu\in \Omega^{n-1}_{M,tw}, \nu\in \Omega^{n-2}_{M,tw}$, then 
\begin{equation}
\Phi(d\mu \otimes D_1+\mu \otimes \nabla D_1,-\iota^*\mu \otimes D_1+d\nu\otimes D_2+ \nu \otimes \nabla D_2) = 0.
\end{equation}
Because the differential and $\Phi$ are right $D_M$-linear (the connection $\nabla$ on $D_M$ is defined in terms of \emph{left} multiplication in $D_M$), it suffices to check this for $D_1=D_2=1$. Then, one verifies that the required equality is equivalent to the integration by parts formulas for $M$ and $\bdyM$. 

It remains to check that $\Phi$ is a quasi-isomorphism. Note that, as before, we obtain a natural filtration on $\dmdrrel$ and $\Phi$ preserves this filtration. On the level of associated graded sheaves, $Gr(\Phi)$ is the sum of the map
\begin{equation} 
\dmdr [n]\xrightarrow{\varepsilon} \Omega^{n}_{M,tw}
\end{equation} 
(where $\Omega^{n}_{M,tw}$ is endowed with the $D_M$-module structure by Lie derivatives of vector fields) and the map
\begin{equation}
\dmdrbdy\xrightarrow{\varepsilon_\partial} F^0 \delsmoothdistr.
\end{equation}
The map $\varepsilon$ is precisely the quasi-isomorphism constructed in the paragraph preceding Lemma \ref{lem: derhamloc}.
For the story on $\bdyM$, let us recall that 
\begin{equation}
F^0 \delsmoothdistr:=\iota_*\left(\Omega^{n-1}_{\bdyM,tw}\otimes_{D_{\bdyM}}\left( \cinfty_{\bdyM}\otimes_{\iota^{-1}(\cinfty_M)}\iota^{-1}(D_M)\right)\right)
\end{equation}
and 
\begin{equation}
\dmdrbdy:= \iota_*\left( \Omega^\bullet_{\bdyM,tw}\otimes_{\iota^{-1}(\cinfty_M)} \iota^{-1}(D_M)\right).
\end{equation}
Since the pushforward functor $\iota_*$ is exact, it suffices to check that the map 
\begin{equation}
\Phi_\del:=\Omega^\bullet_{\bdyM,tw}\otimes_{\iota^{-1}(\cinfty_M)} \iota^{-1}(D_M)\to 
\Omega^{n-1}_{\bdyM,tw}\otimes_{D_{\bdyM}}\left( \cinfty_{\bdyM}\otimes_{\iota^{-1}(\cinfty_M)}\iota^{-1}(D_M)\right)
\end{equation}
induced from $\Phi$ is a quasi-isomorphism.
Upon choosing a tubular neighborhood of the boundary, one may filter both the domain and codomain of $\Phi_\del$ by the order of the differential operator in the normal direction (it is not difficult to show that this filtration does not depend on this choice).
Moreover, one has isomorphisms
\begin{equation}
Gr\left(\Omega^\bullet_{\bdyM,tw}\otimes_{\iota^{-1}(\cinfty_M)} \iota^{-1}(D_M)\right) \cong \Omega^\bullet_{\bdyM,tw}\otimes_{\cinfty_\bdyM}D_{\bdyM}[\del_t],
\end{equation}
\begin{equation}
Gr\left( \Omega^{n-1}_{\bdyM,tw}\otimes_{D_{\bdyM}}\left( \cinfty_{\bdyM}\otimes_{\iota^{-1}(\cinfty_M)}\iota^{-1}(D_M)\right)\right) \cong \Omega^{n-1}_{\bdyM_tw}[\partial_t],
\end{equation}
with the na\"ive right $\iota^{-1}(D_M)$-module structures (i.e. one decomposes a vector field $X$ as $X=X_\del+X_\nu \partial_t$ near the boundary, and lets the normal part of the vector field act on the $\RR[\partial]$ factors, while the tangential part acts in the canonical way on $D_\bdyM$ and $\Omega^{n-1}_{\bdyM,tw}$, respectively).
Using this characterization, it is clear that $Gr(\Phi_\del)$ is simply the base change over $\RR[\partial_t]$ of the codimension-one version of $\epsilon$. Hence, $Gr(\Phi_\del)$, and therefore $Gr(\Phi)$, and therefore $\Phi$, is a quasi-isomorphism.
\end{proof}

The following Lemma is an immediate consequence of Lemma \ref{lem: drres}.

\begin{lemma}
\label{lem: locfcnlresolution}
There is a cannonical quasi-isomorphism
\begin{equation}
\dmdrrel\otimes_{D_M}\sO_{red}(\sJ(\sE))\to \sO_{loc,red}(\sE)
\end{equation}
of sheaves on $M$.
\end{lemma}

We describe one property of $\dmdrrel$ which will be useful in the sequel. 
\begin{lemma}
\label{lmm: dmdrrelflat}
The right $D_M$-module $\dmdrrel$ is a complex of flat $D_M$-modules.
\end{lemma}

\begin{proof}
As a right $D_M$-module, $\dmdr$ is locally a free $D_M$-module, hence is flat.

As a right $D_M$-module, $\dmdrbdy$ looks locally like a sum of modules of the form $\iota_*\iota^{-1}(D_M)$. 
The endofunctor (acting on the category of left $D_M$-modules) 
\begin{equation}
\iota_*\iota^{-1}(D_M)\otimes_{D_M}\cdot
\end{equation}
is naturally isomorphic to the endofunctor $\iota_*\iota^{-1}$.
This functor is exact, since $\iota_*$ and $\iota^{-1}$ are. (Here, we use the fact that $\bdyM$ is a closed submanifold of $M$.)

The lemma follows.
\end{proof}

\subsection{Imposing the Boundary Condition for Functionals}
\label{sec: OlocEL}
Lemma \ref{lem: locfcnlresolution} helps us to understand the structure of $\sO_{loc}(\sE)$; 
however, we are more interested in $\Oloc$, 
i.e. we are not interested in distinguishing two functionals if they agree when restricted to $\condfieldscs$. 
To this end, note that there is a surjective map $\sO_{loc}(\sE)\to \Oloc$. 
Our aim now is to characterize the kernel of this map.
Proposition \ref{prop: dualsofcondfields} and its Corollary A tell us which functionals on $\sE$ vanish when restricted to $\condfields$, 
namely those which depend in one of their inputs only on the boundary information of the field's value in $\Eb/L$. 
Let us use this information to get a better grasp on $\Oloc$. 

Note the following: if a local functional depends only on the boundary information of one input, then it lies in the space $F^0 \Oloc$, i.e. is described by an integral over $\bdyM$.
This is because local functionals of order $k$ in the fields have their support on the small diagonal of $M^k$.
Within $\dmdrrel\otimes \sO_{red,loc}(\sJ(\sE))$ (which, by Lemma \ref{lem: locfcnlresolution}, resolves $\sO_{loc,red}(\sE)$), the space of functionals which are supported on $\bdyM$ is 
\begin{equation}
\iota_*( \Omega^\bullet_{\bdyM, tw})\otimes_{\cinfty_{M}}\sO_{red}(\sJ(\sE)).
\end{equation}
We have
\begin{equation}
\iota_*( \Omega^\bullet_{\bdyM, tw})\otimes_{\cinfty_{M}}\sO_{red}(\sJ(\sE))\cong \iota_*\left(\Omega^\bullet_{\bdyM, tw}\otimes_{\cinfty_{\bdyM}}\sO_{red}(\dpull (\sJ(\sE)))\right),
\end{equation}
where $\sO_{red}(\dpull(\sJ(\sE)))$ is defined similarly to $\sO_{red}(\sJ(\sE))$, but by taking continuous duals and symmetric powers of $\dpull \sJ(\sE)$ over $\cinfty_\bdyM$.

Let $L':=\Eb/L$. There is a canonical map of shifted $L_\infty$-algebras in the category of $D_{\bdyM}$-modules
\begin{equation}
\Upsilon: \dpull (\sJ(\sE)) \to \sJ(\sL')
\end{equation}
constructed as the composite of the following two maps.
First, we construct a map
\begin{equation}
\Upsilon_1: \dpull (\sJ(\sE))=\cinfty_\bdyM \otimes_{\iota^{-1}(\cinfty_M)}\iota^{-1}(\sJ(\sE)) \to \sJ(\sEb)
\end{equation}
given by
\begin{equation}
\Upsilon_1(f\otimes \sigma) = f\rho(\sigma);
\end{equation}
it is straightforward to check that $\Upsilon_1$ is well-defined (using the fact that $\rho$ arises from restriction to the boundary followed by a bundle operation) and moreover respects the tensor product defining $\dpull (\sJ(\sE))$.
Next, we consider the map
\begin{equation}
\Upsilon_2: \sJ(\sEb)\to \sJ(\sL')
\end{equation}
induced from the bundle map $\Eb\to L'$, and we set $\Upsilon =\Upsilon_2\circ \Upsilon_1$.
It follows that there is an inclusion
\begin{equation}
\Upsilon^\vee:\sJ(\sL')^\vee \to (\dpull \sJ(\sE))^\vee
\end{equation}
of left $D_\bdyM$ modules. Because $\Upsilon$ is a map of shifted $L_\infty$-algebras, it follows that 
\begin{equation}
\sJ(\sL')^\vee\otimes_{\cinfty_\bdyM} \sO(\dpull(\sJ(\sE)))\subset \sO_{red}(\dpull (\sJ(\sE)))
\end{equation}
is closed under the Chevalley-Eilenberg differential on $\sO_{red}(\dpull(\sJ(\sE)))$. 
The inclusion also respects the $D_\bdyM$ module structure, and it is straightforward to show that this sub $D_\bdyM$-module is flat. Hence, we can make the following definition:

\begin{definition}
The complex of \textbf{de Rham functionals vanishing on $L$} is the complex
\begin{equation}
\sV(L): = \iota_*\left( \dmdrbdy \otimes_{D_{\bdyM}} \left(\sJ(\sL')^\vee\otimes_{\cinfty_\bdyM} \sO_{red}(\dpull (\sJ(\sE)))\right)\right).
\end{equation}
The discussion preceding this definition guarantees that $\sV(L)$ is a subcomplex of $\dmdrrel\otimes_{D_M}\sO_{red}(\sJ(\sE))$. 
\end{definition}

\begin{notation}
We define the following complex of sheaves
\begin{equation}
\fullOloc:=\left(\dmdrrel\otimes_{D_M}\sO_{red}(\sJ(\sE))\right)/\sV(L)
\end{equation}
\end{notation}

\begin{lemma}
There is a quasi-isomorphism of complexes of sheaves
\begin{equation}
\fullOloc\to \Olocred.
\end{equation}
\end{lemma}

\begin{proof}
Since $\dmdrbdy\to \Omega^{n-1}_{\bdyM,tw}$ is a quasi-isomorphism of right $D_\bdyM$-modules and $\sJ(\sL')^\vee\otimes_{\cinfty_\bdyM} \sO(\dpull(\sJ(\sE)))$ is a flat left $D_\bdyM$-module, we have a quasi-isomorphism
\begin{equation}
\sV(L)\to \iota_*\left(\Omega^{n-1}_{\bdyM,tw}\otimes_{D_\bdyM} \sJ(\sL')^\vee\otimes_{\cinfty_\bdyM} \sO(\dpull(\sJ(\sE)))\right).
\end{equation}
Using Corollary A of the appendix, we identify this latter space of functionals as precisely the space of local functionals on $\sE_c$ which vanish when restricted to $\condfieldscs$.
Hence, we have a map of short exact sequences:

\begin{equation}
\xymatrix{
0 \ar[r] & \sV(L) \ar[r]\ar[d]& A_1\ar[r]\ar[d] & \fullOloc\ar[r]\ar[d]&0\\
0 \ar[r] &A_2  \ar[r]& \sO_{loc,red}(\sE)\ar[r]& \Olocred\ar[r]&0
},
\end{equation}
where 
\begin{equation}
A_1 = \dmdrrel\otimes_{D_M}\sO_{red}(\sJ(\sE))
\end{equation}
and 
\begin{equation}
A_2 = \iota_*\left(\Omega^{n-1}_{\bdyM,tw}\otimes_{D_\bdyM} \sJ(\sL')^\vee\otimes_{\cinfty_\bdyM}\sO(\dpull(\sJ(\sE)))\right).
\end{equation}
Since the first two vertical maps are quasi-isomorphisms, it follows by the snake lemma and the five lemma that the last vertical arrow is a quasi-isomorphism as well.
\end{proof}

\subsection{BF Theory, an Extended Example}
\label{subsec: bfthyloc}
The preceding discussion is a bit dense and difficult to follow. 
To help the reader to get a sense of these construction, we now make an extended study of BF theory on the upper half-space $\HH^n$.
We will study BF theory with the $A$ and $B$ boundary conditions (the terminology follows Example \ref{ex: bfbdycond}).

As in the preceding discussion, we will start by understanding $\sO_{loc,red}(\sE)$. Recall (Lemma \ref{lem: locfcnlresolution}), that there is a quasi-isomorphism

\begin{equation}
\dmdrrel\otimes_{D_M}\sO_{red}(\sJ(\sE))\to \sO_{loc,red}(\sE).
\end{equation}

One can use an argument from the Proposition 7.6 of \autocite{rabaxial} (although this argument is itself derivative from other sources) to construct a quasi-isomorphism of left $D_M$-modules
\begin{equation}
\sO_{red}(\sJ(\sE)) \overset{\sim}{\to} \cinfty_M \otimes C^\bullet_{red}(\fg, \Sym(\fg[n-2])).
\end{equation}
Hence, using Lemma \ref{lmm: dmdrrelflat}, one has a quasi-isomorphism
\begin{equation}
\dmdrrel\otimes_{D_M}\cinfty_M\otimes \left(C^\bullet_{red}(\fg, \Sym(\fg[n-2]))\right) \to \sO_{loc,red}(\sE).
\end{equation}
It is straightforward to show that
\begin{equation}
\dmdrrel\otimes_{D_M}\cinfty_M \cong \Omega^\bullet_{M}[n]\overset{\iota^*}{\to}\Omega^\bullet_{\bdyM}[n-1]=:\drrel,
\end{equation}
hence, we have constructed a quasi-isomorphism
\begin{equation}
\drrel \otimes C^\bullet_{red}(\fg, \Sym (\fg[n-2]))\to \sO_{loc,red}(\sE).
\end{equation}

Now, let us understand $\sV(L)$ better. Let us take first the $A$ boundary condition.
Here, $L'=\Lambda^\bullet T^* \bdyM\otimes \fg^\vee[n-2]$.
As before, we can show 
\begin{equation}
\sJ(\sL')\simeq \cinfty_\bdyM\otimes \fg^\vee[n-2]
\end{equation}
and 
\begin{equation}
\dpull(\sJ(\sE))\simeq \cinfty_\bdyM \otimes (\fg[1]\oplus \fg^\vee[n-2]).
\end{equation}
Hence, we have a quasi-isomorphism
\begin{equation}
\iota_*\left(\Omega^\bullet_{\bdyM}\right)\otimes \left( \fg[2-n]\otimes C^\bullet(\fg, \Sym(\fg[n-2])\right)\to \sV(L).
\end{equation}
Let 
\begin{equation}
\sO(\fg;A)
\end{equation}
denote the total complex
\begin{equation}
\Omega^\bullet_{M}[n]\otimes C^\bullet_{red}(\fg, \Sym (\fg[n-2]))\to \iota_*(\Omega^\bullet_{\bdyM})[n-1]\otimes C^\bullet_{red}(\fg),
\end{equation}
where the horizontal map is the tensor product of the maps
\begin{equation}
\iota^* :\Omega^\bullet_M \to \iota_*(\Omega^\bullet_{\bdyM})
\end{equation}
and the natural projection
\begin{equation}
C^\bullet_{red}(\fg, \Sym(\fg[n-2]))\to C^\bullet_{red}(\fg).
\end{equation}
We have thus constructed a quasi-isomorphism
\begin{equation}
\sO(\fg; A) \to \Olocred
\end{equation}
for BF theory with the $A$ boundary condition.
Similarly, define
\begin{equation}
\sO(\fg; B):=\Omega^\bullet_{M}[n]\otimes C^\bullet_{red}(\fg, \Sym (\fg[n-2]))\to \iota_*(\Omega^\bullet_{\bdyM})[n-1]\otimes \Sym^{\geq 1}(\fg[n-2]);
\end{equation}
we can construct a quasi-isomorphism
\begin{equation}
\sO(\fg; B) \to \Olocred
\end{equation}
for BF theory with the $B$ boundary condition.

Let us now specialize to the case $M=\HH^2$, $\bdyM = \RR$ and take global sections. In this case, we have equivalences
\begin{align}
\sO(\fg;B)(\HH^2) &\simeq \left( C^\bullet_{red}(\fg, \Sym(\fg))[2] \to \Sym^{\geq 1}(\fg)[1]\right)\\
&\simeq \left(H^{\geq 1}(\fg, \Sym(\fg))\right)[2] \oplus \left(\Sym^{\geq 1}(\fg)/\Sym^{\geq 1}(\fg)^{\fg}\right)[1].
\end{align}
When $\bdyM=\emptyset$, the zeroth and first cohomology of the complex of local functionals represent the space of deformations and obstructions, respectively, for quantizations of classical BF theory.
Adopting the same terminology here, we find that the space of obstructions for quantizations of 2D BF theory with $B$ boundary condition is $H^3(\fg, \Sym(\fg))$, and the space of deformations for 2D BF theory is $H^2(\fg, \Sym(\fg))$.
The cohomology acquires a grading from the symmetric degree in $\Sym(\fg)$; we will call this the \emph{B-weight}.
If we restrict to those functionals which have a B-weight of $+1$, and further assume that $\fg$ is semi-simple, we find that 
\begin{equation}
H^\bullet\left( \sO(\fg; B)(\HH^2)\right)_{1} \cong \fg[1],
\end{equation}
the subscript $1$ reminding us that we are computing the cohomology in $B$-weight 1.
To perform this computation, we use Whitehead's theorem and the fact that semi-simple Lie algebras have trivial centers.
Given $x\in \fg$, one may exhibit the isomorphism just described via the $B$-field functional
\begin{equation}
J_x(\beta) = \left(\int_\RR \iota^* \beta\right)(x).
\end{equation}
It is straightforward to check that this functional is closed in $\Olocred$.

%% file: Chapters/Chapter3.tex
\chapter{Free Quantum Bulk-Boundary Systems}
\label{chap: freequantum}

\section{Overview of the Chapter}
In the previous chapter, we have constructed a factorization algebra for a general classical bulk-boundary system. 
In this chapter, we carry out the quantization procedure for \emph{free} bulk-boundary systems, i.e. those for which $\ell_k$ is non-zero only for $k=1$.
As in the case $\bdyM=\emptyset$, such theories admit a factorization algebra of quantum observables which is much easier to describe than the corresponding interacting factorization algebras.

A bulk-boundary system on a manifold of the form $\bdyM \times \RR_{\geq 0}$ includes in particular a choice of a boundary condition $\sL$ for the bulk theory. 
As discussed in \autocite{butsonyoo}, the boundary condition $\sL$ admits a structure of (--1)-shifted \emph{Poisson} formal moduli space. 
(Recall that the usual Batalin-Vilkovisky formalism describes (--1)-shifted symplectic geometry.)
Conversely, a Poisson BV theory on a manifold $\bdyM$ gives rise to an ``honest'' BV theory on $\bdyM \times \RR_{\geq 0}$.
For more details on Poisson BV theories, we refer the reader to \autocite{butsonyoo} (wherein Poisson BV theories are referred to as ``degenerate field theories'').

As far as the author is aware, there is no general mechanism for the quantization of Poisson BV theories. 
However, for \emph{free} Poisson BV theories, one can describe an ansatz for such a quantization.
This ansatz gives rise to a factorization algebra $\Obq_\sL$ on $\bdyM$.

On the other hand, a free bulk-boundary system gives a free BV theory on $\mathring M$ and so a factorization algebra $\Obq_\sE$ on this space.

Finally, we introduce a factorization algebra $\Obq_{\sE,\sL}$ on $\bdyM \times \R_{\geq 0}$ of bulk-boundary quantum observables.

The main results of this chapter (cf. Theorems \ref{thm: maingenlcl} and \ref{thm: maingenlq}) relate the three different factorization algebras $\Obq_\sL$, $\Obq_\sE$, and $\Obq_{\sE,\sL}$. 
Note that they are all factorization algebras defined on different spaces.
Theorems \ref{thm: maingenlcl} and \ref{thm: maingenlq} will relate the three factorization algebras through the pushforward and restriction operations described in Section \ref{sec: review}. 
Recall that these operations allow one to compare factorization algebras on different spaces.

\subsection{A model case of the general result}
In this portion we describe a specific instance of the general Theorems, in order to orient the physically-minded reader.

Our focus is on the following geometric situation. 
Let $\Sigma$ be an oriented smooth 2-dimensional manifold,
which we equip with a complex structure.
Let $M$ denote the closed half-space $\RRge \times \Sigma $,
let $\mathring M$ denote the open half-space $\RRgt \times \Sigma $,
and let $\pi: M \to \Sigma$ denote the projection map.
We view $\RRge$ as providing a kind of ``time direction'' and use $t$ to denote its coordinate.
See Figure~\ref{fig:3mfldwbdry}.

\begin{figure}
\centering
\begin{tikzpicture}
\draw[semithick](-6,0) ellipse (0.8 and 2);
\draw[semithick](-6,0.3) arc (-30:30:1);
\draw[semithick](-6,1.2) arc (150:210:0.82);
\draw[semithick](-6,-1.3) arc (-30:30:1);
\draw[semithick](-6,-0.4) arc (150:210:0.82);
\draw[semithick](-2,0) ellipse (0.8 and 2);
\draw[semithick](-2,0.3) arc (-30:30:1);
\draw[semithick](-2,1.2) arc (150:210:0.82);
\draw[semithick](-2,-1.3) arc (-30:30:1);
\draw[semithick](-2,-0.4) arc (150:210:0.82);
\draw[dotted](2,0) ellipse (0.8 and 2);
\draw[dotted](2,0.3) arc (-30:30:1);
\draw[dotted](2,1.2) arc (150:210:0.82);
\draw[dotted](2,-1.3) arc (-30:30:1);
\draw[dotted](2,-0.4) arc (150:210:0.82);
\draw[semithick](-2,2) -- (3,2);
\draw[dotted](3,2) -- (4,2);
\draw[semithick](-2,-2) -- (3,-2);
\draw[dotted](3,-2) -- (4,-2);
\node at (-7.2,1){$\Sigma$};
\node at (0,-1){$\mathring{M}$};
\node at (-3.2,1){$M$};
\draw[->,semithick] (-3,0) -- (-5,0);
\node at (-4,0.25){$\pi$};
\draw[->, semithick](-1,-2.65) -- (2,-2.65);
\node at (0.5, -2.4){$t$};

\end{tikzpicture}
\caption{Projection of $M$ onto the boundary $\Sigma$}
\label{fig:3mfldwbdry}
\end{figure}
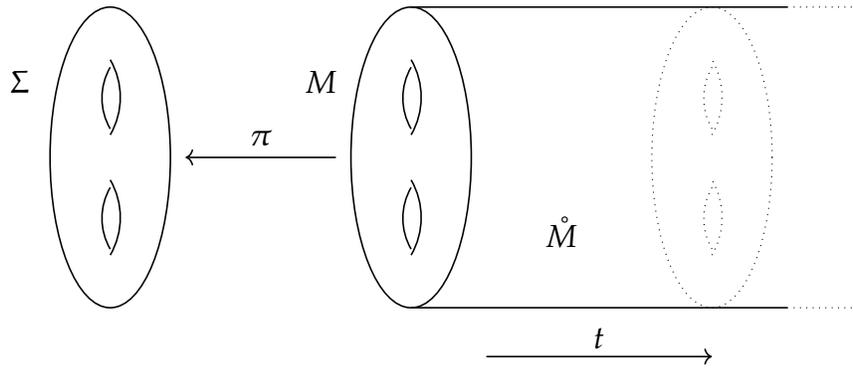

In the interior $\mathring M$, which is a manifold without boundary, 
we put (perturbative) Chern-Simons theory with gauge group~$U(1)$ with level $\kappa$.
It has a factorization algebra $\Obq_{CS}$ of quantum observables.
(See \S4.5 of~\autocite{CG1}.)

On the boundary $\Sigma = \partial M$, 
there is a factorization algebra $\Cur^q_{WZW}$ encoding the chiral $U(1)$ currents, with the Schwinger term determined by $\kappa$.
$\Cur^q_{WZW}$ can be described as the factorization algebra of quantum observables for the chiral WZW boundary condition.
(See \S5.4 of~\autocite{CG1} for its construction and verification that it recovers the standard vertex algebra and OPE.)

In this paper we will construct a factorization algebra $\Obq_{CS/WZW}$ for abelian Chern-Simons theory on $M$ with a particular boundary condition called the chiral WZW boundary condition. 
(As far as we are aware, 
this is the first construction of a factorization algebra of observables of a field theory arising on a manifold with boundary.)
It interpolates between the Chern-Simons observables and the chiral currents in the following precise sense.

\begin{theorem}
\label{thm: main}
The factorization algebra $\Obs^\q_{CS/WZW}$ is stratified in the sense that
\begin{itemize}
\item on the interior $\mathring M$, there is a natural \emph{isomorphism} 
\begin{equation}
\Obs^q_{CS} \simeq \left(\Obs^\q_{CS/WZW}\right)\Big|_{\mathring M}
\end{equation} 
of factorization algebras, and
\item on the boundary $\partial M = \Sigma$, there is a quasi-isomorphism 
\begin{equation}
\Cur^q_{WZW} \simeq \pi_*\left(\Obs^\q_{CS/WZW}\right)
\end{equation} 
of factorization algebras. 
\end{itemize}
\end{theorem}

The factorization algebra $\Obq_{\sE,\sL}$ thus exhibits the desired phenomenon, 
as it is precisely the abelian Chern-Simons system in the ``bulk'' $\mathring M$ but becomes the chiral currents on the boundary~$\partial M$. 
The full factorization algebra $\Obq_{CS/WZW}$ contains more information still: it encodes an \emph{action} of the bulk observables $\Obq_{CS}$ on the boundary observables~$\Cur^q_{WZW}$.

A compelling phenomenon happens at the quantum level: 
the canonical BV quantization of abelian Chern-Simons theory in the bulk forces the appearance of the Kac-Moody cocycle $\int \alpha \wedge\partial \beta$ (i.e., Schwinger term) on the boundary.
We emphasize that these constructions are wholly rigorous, not requiring any leaps of physical intuition.
They also yield naturally a stratified factorization algebra,
and hence the theorem suggests that other bulk-boundary correspondences in the physics literature may also admit formulations in these terms.
We will describe a few such correspondences, notably a generalization of abelian CS/WZW to higher dimensions with a $4n+3$-dimensional bulk and a $4n+2$-dimensional boundary equipped with a complex structure.

One drawback of our work is that we only deal with perturbative and Lie algebraic statements here, 
not with nonperturbative and group-level versions, where many fascinating issues arise.
(As merely a jumping-off point and not a complete list of citations for this enormous subject,
we point to \autocite{FMS, FFS, HopSing, KSonCS, BD, WittenJones,EMSS,FreedDQ, BecBenSchSza} as places where such issues are addressed.)
We expect that a rigorous extension of the BV formalism to global derived geometry would fold those nonperturbative issues together with our perturbative efforts.

\subsection{Consequences and applications}

One payoff here is a new view on Chern-Simons states in bundles of conformal blocks for chiral WZW models.
Factorization algebras, like sheaves, are local-to-global objects,
and so the homology of these stratified factorization algebras encode nontrivial global information.
Here, in particular, they automatically produce maps from the space of boundary observables into the global observables of the theory. 
As an example, we obtain the Chern-Simons states of the chiral WZW theory from studying the map from the boundary observables on a Riemann surface to the observables of a compact 3-manifold bounding that surface.
The higher dimensional analogs of Chern-Simons states are sections of interesting vector bundles over the intermediate Jacobians of any complex $2n+1$-fold that admits an oriented null-cobordism.

\subsection{Outline of the Chapter}

Section \ref{sec: BBFTs} defines the class of bulk-boundary theories that we study in this paper and introduces many examples of such.
The definition is modeled on the definition of a free BV theory in \autocite{cost},
but we hope it is transparent to anyone already familiar with the BV formalism in some guise.

Section \ref{sec: FAconstruction} recalls the factorization algebras that appear purely in the bulk or on the boundary,
which were constructed in \autocite{CG1}, in various guises.
We then construct the natural factorization algebra for the bulk-boundary system,
modeled on those constructions.
Functional analytic subtleties are addressed in the appendix.

Section \ref{sec: main theorem} states and proves the main theorem,
both for classical and for quantum observables.
Section \ref{sec: apps} addresses specific examples of the theorems.

\section{Free bulk-boundary field theories}
\label{sec: BBFTs}
Definition \ref{def: clbv} introduced the notion of a classical TNBFT; Definition \ref{def: bdycond} introduced the notion of a suitable class of boundary conditions for classical TNBFTs.
We have also defined bulk-boundary systems, which are a choice of TNBFT together with a boundary condition.
In this chapter, we consider the quantization of \emph{free} bulk-boundary systems:

\begin{definition}
A \textbf{free} bulk-boundary system is a bulk-boundary system whose underlying TNBFT has $\ell_k=0$ for $k>1$.
\end{definition}

In the previous chapter, we took great pains to emphasize the homotopy-theoretic nature of our constructions; for example, we described the process of imposing boundary conditions via a homotopy pullback.
The notion of a free bulk-boundary system is not particularly meaningful at the homotopical level.
Nevertheless, it is a useful condition to impose from the standpoint of quantization.
We will see in this chapter that free bulk-boundary systems possess very elegantly-described factorization algebras of quantum observables.
The non-free (i.e. interacting) case is left for Chapter \ref{chap: interactingquantum}.

\subsection{Examples of free bulk-boundary systems}
\label{sec: examples}

In Chapter \ref{chap: classical}, we introduced a number of bulk-boundary systems which have free analogs. We revisit those bulk-boundary systems and a few others here.

\begin{example}
\label{ex: toplmech2}
Suppose $V$ is a symplectic vector space with symplectic form $\omega$. Let $M=[0,\epsilon)$ and $\sE=\Omega^\bullet_{[0,\epsilon)}\otimes V$, together with the pairing $\ip$ induced from the Poincar\'e duality pairing and $\omega$. Here, $\Eb=\sEb=V$, and $\ip_{loc, \del}=\omega$. This theory is \textbf{topological mechanics}. A Lagrangian subspace $L$ of $V$ gives a boundary condition for topological mechanics.
(Cf. Examples \ref{ex: toplmech} and \ref{ex: toplmechbc}.)
\end{example}

\begin{example}
\label{ex: psm2}
Let $\Sigma$ be any surface with boundary, and let $V$ be a vector space with a constant Poisson structure, i.e., $V$ is a vector space equipped with a skew-symmetric map $\Pi: V^\vee \to V$. 
Let 
\begin{equation}(\sE,\diff)=( \Omega^\bullet_{\Sigma}\otimes V\oplus \Omega^\bullet_{\Sigma}\otimes V^\vee[1],\d_{dR}\otimes 1+1\otimes \Pi).\end{equation}
The pairing $\ip_{loc}$ is defined using the wedge product and the natural pairing between $V^\vee$ and $V$. It is evident that one can write
\begin{equation}
(\sEb,\Qb)=(\Omega^\bullet_{\del \Sigma}\otimes V\oplus \Omega^\bullet_{\del \Sigma}\otimes V^\vee[1],\d_{dR}\otimes 1+1\otimes \Pi),
\end{equation}
and $\ip_{loc, \del}$ is again defined using the wedge product of forms and the canonical pairing between $V^\vee$ and $V$. This theory is a special case of the \textbf{Poisson sigma model} \autocite{MR1854134}. The subcomplex $\Omega^\bullet_{\del \Sigma} \otimes V\subset \sEb$ gives a boundary condition for this theory. 
\end{example}

\begin{example}
Suppose $\fA$ is a complex vector space together with a non-degenerate symmetric bilinear pairing $\kappa$. Let $M$ be an oriented 3-manifold with boundary. For $(\sE,\diff)$ we take $(\Omega^\bullet_{M}\otimes \fA[1],d_{dR})$. For the pairing $\ip_{loc}$ we take 
\begin{equation}
\ip[\mu,\nu]_{loc} = \kappa(\mu, \nu),
\end{equation}
where we are implicitly taking a wedge product of forms and only keeping the top-form component of the resulting wedge product. From these characterizations, it is evident that $(\sEb,\Qb) =(\Omega^\bullet_{\del M}\otimes \fA[1],d_{dR})$, and 
\begin{equation}
\ip[\mu,\nu]_{loc, \del}=\kappa(\mu, \nu).
\end{equation}
This theory is an \textbf{abelian Chern-Simons theory}. 
In the bulk $3$-manifold, $M \setminus \partial M$, the elliptic complex $(\sE, \diff)$ is simply abelian Chern-Simons theory where we view $\fA$ as an abelian Lie algebra.
The solutions to the bulk equations of motion are the $\fA$-valued closed one-forms. 

If $M=\Sigma\times \RRge$ (where $\Sigma$ is a Riemann surface), the space of fields is endowed with the decomposition
\begin{equation}
\sE=\Omega^{0,\bullet}_{\Sigma}\,\hotimes\, \Omega^\bullet_{\RRge}\otimes \fA[1]\oplus \Omega^{1,\bullet}_{\Sigma}\,\hotimes\, \Omega^\bullet_{\RRge}\otimes\fA,
\end{equation}
with differential $\diff=\del+\delbar+d_{dR}$; if we replace $\diff$ with $(\diff)_\chi :=\chi \del +\delbar+d_{dR}$, we obtain \textbf{Chern-Simons at level~$\chi$.}

The boundary condition we consider depends on the choice of a complex structure on the boundary $\del M$.
Henceforth, when we want to stress the dependence on the complex structure, we denote the boundary Riemann surface by~$\Sigma$.

Given a holomorphic vector bundle $V$ on $\Sigma$, there is a resolution for its sheaf of holomorphic sections $\sV^{hol}$ given by the Dolbeault complex $\left(\Omega^{0,\bullet}(\Sigma, V), \dbar\right)$. 
The differential is the Dolbeault operator $\dbar : \Omega^{0}(\Sigma, V) \to \Omega^{0,1}(\Sigma, V) = \Gamma(T^{*0,1} \otimes V)$ defining the complex structure on $V$. 
In the case that $V = T^{*1,0}$, we denote this Dolbeault complex by $\Omega^{1,\bullet}(\Sigma)$ with the $\dbar$-operator understood. 

The subcomplex $\Omega^{1,\bullet}_\Sigma\otimes \fA\subset \Omega^\bullet_\Sigma\otimes \fA[1]$ defines a boundary condition for abelian Chern-Simons theory (at any level if $M=\CC\times \RRge$). 
To see this, consider $\sL$ as the sections of a vector bundle $L$ on $\Sigma$.
It is clear that the rank of $L$ is half that of $\Eb$.
Also, $\ip_{loc, \del}$ is identically zero on $L \otimes L$ since only forms of type $(1,\bullet)$ appear in $\sL$. 
Finally, the cochain complex $ \Omega^{1,\bullet}_\Sigma\otimes \fA$ is a subcomplex of the full de Rham complex since $\dbar \alpha = \d \alpha$ for forms $\alpha$ of type $(1,\bullet)$. 
We call this boundary condition the \textbf{chiral WZW boundary condition}. Notice that although Chern-Simons theory is topological, we may choose a non-topological boundary condition for the theory. In this situation, the boundary condition has a chiral, or holomorphic, nature.
\end{example}

\begin{example}
\label{ex: highercs}
Let $M$ be an oriented manifold of dimension $4n+3$, and suppose $\del M$ has the structure of a complex $(2n+1)$-fold that we denote $X$.
Let the fields be the (shifted) de Rham forms
\begin{equation}
\sE= \Omega^\bullet_M[2n+1], 
\end{equation} 
with $\diff=d_{dR}$.
This complex can be understood geometrically as encoding deformations of the trivial flat $U(1)$ $n$-gerbe.
(Taking $n=0$, we note that a flat $U(1)$ 0-gerbe is a flat $U(1)$-bundle.)
We have $\sEb=\Omega^\bullet_{\del M}\otimes \fA[2n+1]$, and $\Qb=d_{dR}$. 
The pairings are defined exactly as in the previous example, by wedging and integration. 
This theory is \textbf{higher-dimensional abelian Chern-Simons theory}. As a boundary condition, we take 
\begin{equation}
\sL = \Omega^{> n,\bullet}[2n+1],
\end{equation} 
which we call the \textbf{intermediate Jacobian} boundary condition, due to it being a component of the Hodge filtration.
(As in the previous example, we could work with some $(\fA,\kappa)$,  a finite-dimensional vector space together with a non-degenerate symmetric bilinear pairing. 
This would amount to tensoring the above complexes with $\fA$.
This extension would correspond to working with higher gerbes for a higher-dimensional abelian Lie group.)
\end{example}

\begin{example}
\label{ex: Riemhighercs}
There is an alternative boundary condition for higher dimensional Chern--Simons that depends on a Riemannian metric  rather than a complex structure. 
As above, let $M$ be an oriented manifold of dimension $4n+3$, and suppose the boundary $N = \partial M$ is equipped with a Riemannian structure. 
In turn, we decompose the middle de Rham forms on $N$ into the $\pm \sqrt{-1}$-eigenspaces
\begin{equation}\label{eqn:decomp}
\Omega^{2n+1}(N) = \Omega^{2n+1}_+ (N) \oplus \Omega^{2n+1}_-(N) 
\end{equation}
of the Hodge star operator. 

Consider the subcomplex of~$\Omega^\bu_{\partial M} \otimes \fA [2n+1]$:
\begin{equation}
\sL = \bigg(\Omega^{2n+1}_+(N) \otimes \fA \xrightarrow{\d} \Omega^{2n+2} (N) \otimes \fA [-1] \xrightarrow{\d} \cdots \xrightarrow{\d} \Omega^{4n+2}(N) \otimes \fA [-2k-1]\bigg) .
\end{equation}
It defines a boundary condition for $(4n+3)$-dimensional abelian Chern-Simons theory. 

For a complement to $\sL$ in $\sEb$, we may take
\begin{equation}\label{eqn:dminus}
\sL^\perp = \bigg( \Omega^{0} (N) \otimes \fA [2n+1].\xto{\d} \Omega^1(N) \otimes \fA[2n] \to \cdots \to \Omega^{2n} (N) \otimes \fA [1] \xto{\d_-} \Omega^{2n+1}_- (N) \otimes \fA \bigg)
\end{equation}
where $\d_- : \Omega^{2n} (N) \to \Omega^{2n+1}_- (N)$ denotes the de Rham differential followed by the projection using the decomposition~(\ref{eqn:decomp}).
\end{example}

\section{The factorization algebras at play}
\label{sec: FAconstruction}

In this section we describe the three factorization algebras that appear in a bulk-boundary system:
\begin{itemize}
\item the observables $\Obs_\sE$ living purely in the bulk $\mathring M$, which depend only on the BV theory in the bulk,
\item the observables $\Obs_{\sL}$ of the boundary condition, which live only on the boundary $\partial M$, and
\item the observables $\Obs_{\sE,\sL}$ of the bulk-boundary system, which lives on the whole manifold $M$ with boundary.
\end{itemize}
There are classical and quantum versions of both factorization algebras.
Now aware of the these three algebras, 
the reader can skip to Section~\ref{sec: main theorem} and understand the statement of our main theorems.

The bulk observables $\Obs_\sE$ arising here were defined in \autocite{CG1}, and they are a straightforward interpretation of the observables in a free BV theory.
The observables of the boundary condition $\Obs_{\sL}$ are defined in a similar way.
At the classical level, they are simply functions on the space $\sL$, but  the quantization uses a Poisson structure arising from the map $\sL\to\sEb$ which identifies $\sL$ as a Lagrangian in $\sEb$.
In this sense, the boundary condition behaves like a Poisson field theory,
in contrast to the symplectic-type bulk theory.

The observables $\Obs_{\sE,\sL}$ are constructed in an analogous way to the other algebras.
The classical observables realize, in a homotopical sense, the algebra of functions on the space of solutions to the equations of motion that satisfy the boundary condition.
The quantization is in the spirit of the BV formalism; 
it amounts to changing the differential by adding an operator determined by the natural pairing on the fields, with boundary condition imposed.
Our main theorems show that $\Obs_{\sE,\sL}$ interpolates between $\Obs_\sE$ and~$\Obs_{\sE,\sL}$,
and in this way we see that there is a natural quantization of the bulk-boundary system that realizes a correspondence between the bulk and boundary systems themselves.

\subsection{Bulk observables}
\label{sec: bulk obs}

Chapter 4 of \autocite{CG1} is devoted to constructing and analyzing the observables, both classical and quantum, of a free BV theory on a smooth manifold.
Here we simply recall the definitions.

\begin{definition}
Let $\sE$ be a free TNBFT. 
The factorization algebra of \textbf{classical observables for $\sE$} assigns to an open subset $U\subset \mathring M$ the (differentiable) cochain complex
\begin{equation}
(\Sym(\sE_c[1](U)),\diff)=:\Obcl_{\sE}(U),
\end{equation}
where the symmetric powers are taken with respect to the completed bornological tensor product of convenient vector spaces (see, e.g. Definition B.4.9 and Section B.5.2 of~\autocite{CG1}).
\end{definition}

Note that for a smooth vector bundle $V\to M$, these completed tensor products can be understood concretely as
\begin{equation}
(\sV_{c}(U))^{\hotimes_\beta k}\cong C^\infty_{c}(U^{\times k}; V^{\boxtimes k}). 
\end{equation}
In other words, they are the compactly supported sections on the $k$-fold product $U^k$ with values in the natural vector bundle $V^{\boxtimes k} \to U^k$.

Something a bit subtle is happening in this definition. 
{\it A priori} the classical observables ought to consist of functions on the fields $\sE$;
in other words, they ought to be a symmetric algebra on the linear dual vector space or, better yet, the continuous linear dual.
Here, however, we took a symmetric algebra on $\sE_c[1]$,
which looks different.
Two facts combine to explain our choice.
First, the local pairing lets us identify the continuous linear dual of $\sE$ with the distributional and compactly supported sections of $E[1] \to M$:
every such section determines a linear functional on $\sE$ by plugging it into the pairing.
Second, the Atiyah-Bott lemma (see Appendix E of \autocite{CG1}) shows that the elliptic complex of distributional, compactly supported sections of $E[1] \to M$ is continuously quasi-isomorphic to the subcomplex of smooth, compactly supported sections of $E[1] \to M$.
Together, these facts show that our definition captures correctly --- up to quasi-isomorphism --- the most natural choice of classical observables.
Concretely, we are working with {\em smeared} observables.

With our definition, BV quantization is straightforward,
because the pairing determines a natural BV Laplacian $\Delta: \Sym(\sE_c[1](U) \to \Sym(\sE_c[1](U)$ as follows.
We set $\Delta = 0$ on the constant and linear terms (i.e., the subspace $\Sym^{\leq 1}(\sE_c[1](U)$), 
and we require
\begin{equation}
\Delta(a b) = \Delta(a) b + (-1)^{|a|} a \Delta(b) + \{a,b\}
\end{equation}
for arbitrary $a$ and $b$. 
Here, $\{\cdot,\cdot\}$ is the unique biderivation (with respect to the product in the symmetric algebra) on 
\begin{equation}
\Sym(\sE_c[1](U))\times \Sym(\sE_c[1](U))
\end{equation}
which coincides with $\ip[\cdot,\cdot]$ on 
\begin{equation}
\sE_c[1](U)\times \sE_c[1](U).
\end{equation}
This equation defines $\Delta$ inductively on the higher symmetric powers. 

For instance, if $a$ and $b$ are linear, then $ab \in \Sym^2(\sE_c[1](U)$, 
and we see that
\begin{equation}
\Delta(a b) = \ip[a,b]
\end{equation}
because we have set $\Delta(a) = 0 = \Delta(b)$. 

By construction, $\Delta$ is a second-order differential operator on the graded commutative algebra $\Sym(\sE_c[1](U)$.
It is straightforward to verify that $\Delta^2 = 0$ and that $\Delta$ commutes with $\diff$ (because $\diff$ is compatible with the pairing $\ip$).
Hence we posit the next definition, following the BV prescription for deformation quantization.

\begin{definition}
Let $\sE$ be a free TNBFT. 
The factorization algebra of \textbf{quantum observables for $\sE$} assigns to an open subset $U\subset \mathring M$, the (differentiable) cochain complex
\begin{equation}
(\Sym(\sE_c[1](U))[\hbar],\diff+\hbar \Delta)=:\Obq_{\sE}(U),
\end{equation}
where the symmetric powers are taken with respect to the completed bornological tensor product of convenient vector spaces.
\end{definition}

\subsection{Observables of the boundary condition}

A boundary condition $\sL$ leads to factorization algebras on the boundary in a parallel fashion.

At the classical level, the idea is that we want to use a commutative algebra of functions on $\sL$,
which we take to be a symmetric algebra on the continuous linear dual $\sL^*$.
It is convenient to work with a smeared (and hence smooth) version of $\sL^*$.
One approach is to note that $\sL$ is a subspace of $\sEb$, 
and so we could work with the quotient of $\Obcl_{\sEb}$ by the ideal of functions that vanish on the subspace $\sL$.
This approach is canonically determined by the map $\sL \to \sEb$, and hence manifestly meaningful.
On the other hand, it is convenient to have an explicit graded vector bundle to use, 
particularly when we quantize and need to transport the BV Laplacian for the bulk theory to an operator on the boundary observables.
Hence we now introduce a different approach that we will see, later, is equivalent.

\begin{constr}
Let $\sL$ be a boundary condition for a free TNBFT associated to the graded subbundle $L$ of $\Eb$. 
Let $L^\perp$ be a complementary subbundle so that $\Eb = L \oplus L^\perp$. 
Let $\sL^\perp$ denote the sheaf of smooth sections of $L^\perp$,
and let $\sL^\perp_c$ the cosheaf of compactly supported smooth sections of~$L^\perp$.
With respect to this splitting, the differential $\Qb$ decomposes as $Q_{L}+Q_{L^\perp}+Q_{rel}$, 
where $Q_{L}$ preserves $\sL$, $Q_{L^\perp}$ preserves $\sL^\perp$, and $Q_{rel}$ maps $\sL^\perp$ to $\sL$. 
(There is no operator from $\sL$ to $\sL^\perp$ because we have assumed that $\Qb$ preserves~$\sL$.)  
\end{constr}

Notice that every element of $\sL^\perp_c$ determines a continuous linear functional on $\sL$ via the local pairing $\ip_{loc, \del}$ on $\sE_\partial$.
In fact, these smeared observables encompass essentially all the linear functionals:
by the Atiyah-Bott lemma, the complex $(\sL^\perp_c, Q_{L^\perp})$ is continuously quasi-isomorphic to the complex of compactly supported distributional sections of $\Eb/L$ with the differential induced by $\Qb$.
Hence a symmetric algebra on $\sL^\perp_c$ deserves to be understood as an algebra of observables for $\sL$.

\begin{definition}
\label{dfn: classical currents}
Let $\sL$ be a boundary condition for a free TNBFT.
The factorization algebra of \textbf{classical boundary observables for $\sL$} assigns to an open subset $U\subset \partial M$, 
the (differentiable) cochain complex
\index[notation]{Obsclbdy@$\Obcl_{\sL}$}
\begin{equation}
(\Sym(\sL^\perp_c(U)),Q_{L^\perp})=:\Obcl_{\sL}(U),
\end{equation}
where the symmetric powers are taken with respect to the completed bornological tensor product of convenient vector spaces.
\end{definition}

At the quantum level, one obtains a Heisenberg-type deformation of $\Obcl_\sL$ as a factorization algebra.
The relevant deformation arises from a canonical bilinear form on $\sL^\perp_c$ determined by our construction.
Let $\mu$ be the following local degree $-1$ cocycle on $\sL^\perp_c$:
for any pair of compactly-supported sections $e_1$ and $e_2$ on an open $U\subset \del M$, define
\begin{equation}
\label{eq: twistcocycle}
\mu(e_1,e_2)=\int_{\del M}\ip[e_1,Q_{rel} e_2]_{loc,\del},
\end{equation}
We use this pairing to define a second-order differential operator $\hbar \Delta_\mu$ on $\Sym(\sL^\perp_c(U))[\hbar]$ of cohomological degree 1,
just as we constructed the BV Laplacian $\Delta$ on the bulk observables.

\begin{definition}
\label{dfn: quantum currents}
Let $\sL$ be a boundary condition for a free TNBFT.
The factorization algebra of \textbf{quantum boundary observables for $\sL$} assigns to an open subset $U\subset \partial M$, 
the (differentiable) cochain complex
\index[notation]{Obsqbdy@$\Obq_\sL$}
\begin{equation}
(\Sym(\sL^\perp_c(U))[\hbar],Q_{L^\perp} + \hbar \Delta_\mu)=:\Obq_{\sL}(U),
\end{equation}
where the symmetric powers are taken with respect to the completed bornological tensor product of convenient vector spaces.
\end{definition}

\begin{remark}
The quotient map $q_L: \Eb \to \Eb/L$ makes $L^\perp$ canonically isomorphic to the quotient bundle $\Eb/L$,
and hence we can identify $L^\perp$ with the image of a splitting of that quotient map.
Any two choices of splitting $L^\perp_0$ and $L^\perp_1$ are related by a bundle automorphism of $\Eb$.
We emphasize this isomorphism is at the point set level; 
it is an automorphism of graded vector bundles.
Using this automorphism one gets a natural equivalence between the associated pairings $\mu_0$ and~$\mu_1$.
Hence, any two versions of the construction above are isomorphic.
\end{remark}

\begin{remark}
The pairing $\mu$ determines a central extension $\widehat{\sL^\perp_c}(U)$ of $\sL^\perp_c(U)$ (as an abelian dg Lie algebra) by the vector space $\CC\hbar$ placed in degree~1.
That is, $\widehat{\sL^\perp_c}(U)$ is a kind of Heisenberg Lie algebra.
Since $\mu$ is defined for any open subset $U$ of the whole manifold $\del M$, 
we get a precosheaf of central extensions on $\bdyM$.
The quantum observables are then the Chevalley-Eilenberg chains of this dg Lie algebra $\widehat{\sL^\perp_c}$.
Thus our definition above is a case of taking a twisted enveloping factorization algebra. 
See Definition 3.6.4 of \autocite{CG1} for an extensive discussion,
and Chapter 4 for an explanation of why this construction encodes canonical quantization.
\end{remark}

\subsection{Observables of the bulk-boundary system}

There is a natural way to extend our methods above to obtain observables on $\sE_\sL$, 
which describes solutions to the equations of motion for fields in $\sE$ that must live in $\sL$ on the boundary.
We will begin by describing the corresponding functor
\begin{equation}
\mathrm{Opens}(M)\to \mathrm{Ch}
\end{equation}
and then turn to verifying it is a factorization algebra.

\begin{definition}
\label{def: freeclobs}
Let $(\sE,\sL)$ be a free bulk-boundary field theory. 
The prefactorization algebra of \textbf{bulk-boundary classical observables for $(\sE,\sL)$} assigns to each open subset $U\subset M$, 
the (differentiable) cochain complex
\index[notation]{Obscl@$\Obcl_{\sE,\sL}$}
\begin{equation}
(\Sym(\condfieldscs[1](U)),\diff)=:\Obcl_{\sE,\sL}(U),
\end{equation}
where $\condfieldscs$ denotes the cosheaf of compactly-supported fields for the bulk-boundary system
(i.e., elements of $\sE_{\sL}(U)$ whose support is compact).
The symmetric powers are taken with respect to the completed bornological tensor product of convenient vector spaces.
\end{definition}

To see that $\Obcl_{\sE,\sL}$ is a prefactorization algebra, 
one can borrow verbatim Section 3.6 of~\autocite{CG1}.

\begin{remark}
Definition \ref{def: freeclobs} differs from Definition \ref{def: classFA}.
However, we note here that in the appendices, 
we provide two useful results,
\begin{itemize}
\item a geometric interpretation of the tensor powers $\sE_{\sL}(U)^{\hotimes_\beta k}$ and
\item a version of the Atiyah-Bott lemma for the bulk-boundary fields (cf. Appendix D,~\autocite{CG1}),
\end{itemize}
that underpin our choice of smeared observables for the bulk-boundary system.
Analogs of these results played a key role in the case of free BV theories on manifolds without boundary.
The first allows us to recognize why the completed bornological tensor product is natural here,
and it also plays a role in the proof that we get a factorization algebra.
The second justifies that working with the continuous linear dual $\sE_{\sL}(U)^\vee$ adds no further information than $\sE_{\sL,c}(U)[1]$, 
up to continuous quasi-isomorphism.
\end{remark}

In fact, we can, without much difficulty, show that the classical observables form a factorization algebra, that is, they satisfy the local-to-global condition of Definition 6.1.4 in~\autocite{CG1}. 

\begin{theorem}
\label{thm: obcl}
For a free bulk-boundary theory $\sE$ with local Lagrangian boundary condition $\sL$, 
the classical observables $\Obcl_{\sE,\sL}$ form a factorization algebra.
\end{theorem}

\begin{proof}
The context here is nearly identical to that of Theorem 6.5.3(ii) of \autocite{CG1}. By the same arguments as in the proof of that theorem, we need only to show that, given any Weiss cover $\fU=\{U_i\}_{i\in I}$ of an open subset $U\subset M$, the map 
\begin{equation}
\label{eq: cechmap2}
\bigoplus_{n=0}^\infty \bigoplus_{i_1,\cdots, i_n}\Sym^m\left( \condfieldscs[1](U_{i_1}\cap\cdots\cap U_{i_n})\right) [n-1]\to \Sym^m (\condfieldscs[1](U))
\end{equation}
is a quasi-isomorphism, where the left-hand side is endowed with the \v{C}ech differential. According to the appendix, particularly the Corollary of \ref{thm: tensorofdirichlet}, 
\begin{equation}
\Sym^m\left( \condfieldscs[1](U_{i_1}\cap\cdots\cap U_{i_n})\right)
\end{equation}
is the subspace of 
\begin{equation}
\Sym^m\left( \sE_{c}[1](U_{i_1}\cap\cdots\cap U_{i_n})\right)\subset C^\infty_{c}((U_{i_1}\cap \cdots \cap U_{i_n})^m,(E[1])^{\boxtimes m})
\end{equation}
consisting of those sections that lie in $(L\oplus \Eb \d t)_{x_1}\otimes E_{x_2}\otimes\cdots \otimes E_{x_m}$ whenever the first of the points $x_1,\cdots, x_m\in (U_{i_1}\cap \cdots \cap U_{i_n})^m$ lies on $\del M$, and similarly for $x_2, \cdots, x_m$. The proof of Lemma A.5.7 of \autocite{CG1} constructs a contracting homotopy of the mapping cone of Equation \ref{eq: cechmap2} without any conditions imposed at the boundary of $M$. Because the contracting homotopy involves only multiplication by smooth functions and addition of sections, it preserves the lie-in condition for $(\condfieldscs)^{\hotimes_\beta m}$. Hence, the contracting homotopy from the proof of Lemma A.5.7 of \autocite{CG1} gives a contracting homotopy for the mapping cone of Equation~\ref{eq: cechmap2}, so that the map of that equation is a quasi-isomorphism.
\end{proof}

We also define a factorization algebra of quantum observables.

\begin{definition}
Let $(\sE,\sL)$ be a free bulk-boundary field theory. 
The prefactorization algebra of \textbf{bulk-boundary quantum observables for $(\sE,\sL)$} assigns to each open subset $U\subset M$, 
the (differentiable) cochain complex
\begin{equation}
(\Sym(\condfieldscs[1](U)[\hbar],\diff+\hbar \Delta)=:\Obq_{\sE,\sL}(U).
\end{equation}
Here $\Delta$ is the restriction of the BV Laplacian for $\Obq_\sE$ to this graded subspace. 
\end{definition}

\begin{remark}
The fact that $\diff+\hbar\Delta$ is a differential on $\Obq_{\sE,\sL}(U)$ requires some proof. 
In the case where $\del M$ is empty, it follows from the invariance of $\ip$ under $\diff$ (see equation \ref{eq: invpairing}). 
In the present case, equation \ref{eq: invpairing} is satisfied for $\condfields$, 
so that $\diff+\hbar \Delta$ squares to zero on $\Obq_{\sE,\sL}(U)$. 
This property motivates the use of local Lagrangian boundary conditions for TNBFTs.
\end{remark}

\begin{theorem}
\label{thm: freequantumFA}
The functor $\Obq_{\sE,\sL}$ is a factorization algebra. 
\end{theorem}

\begin{proof}
That $\Obq_{\sE,\sL}$ is a prefactorization algebra is an immediate consequence of the fact that $\Obcl_{\sE,\sL}$ is, since the BV Laplacian is local. 
To see that the local-to-global condition is also satisfied, note that $\Obq_{\sE,\sL}(U)$ has a filtration given by 
\begin{equation}
F^n\Obq_{\sE,\sL}(U) = \bigoplus_{j+k \leq n} \hbar^j \Sym^k(\sE_{\sL,c}(U)[1])
\end{equation}
for every open subset $U$.
The differential on $\Obq_{\sE,\sL}(U)$ preserves this filtration. Moreover, for any Weiss cover $\fU$ of $U$, the \v{C}ech complex $\check{C}(\fU, \Obq_{\sE,\sL})$ for this cover also has a filtration and the map 
\begin{equation}
\label{eq: cech2}
\check{C}(\fU, \Obq_{\sE,\sL})\to \Obq_{\sE,\sL}(U)
\end{equation}
respects this filtration, hence induces a map of spectral sequences. The induced map on the associated graded spaces (the $E^1$ page) is the map
\begin{equation}
\check{C}(\fU,\Obcl_{\sE,\sL}[\hbar])\to \Obcl_{\sE,\sL}(U)[\hbar],
\end{equation}
which was shown to be a quasi-isomorphism in the proof of Theorem \ref{thm: obcl}. Hence the map in Equation \ref{eq: cech2} is a quasi-isomorphism.
\end{proof}

\section{The main theorems}
\label{sec: main theorem}

In this section, we state and prove a generalization of Theorem~\ref{thm: main} that applies to a general free bulk-boundary field theory $\sE$ with boundary condition $\sL$. Without loss of generality , we will assume that the underlying manifold is of the form $M=\Mbdy\times \RR_{\geq 0}$, so that $\partial M = M_\partial$.  Let $\pi: M\to \Mbdy$ denote projection onto the boundary. We will also assume that the space of fields is globally of the form $\sE\,\hotimes\, \Omega^\bullet_{\RRge}$, with the pairing $\ip$ of the form specified in Definition~\ref{def: tnbft}. 

\begin{remark}
The assumption that $M=\Mbdy\times \RR_{\geq 0}$ is purely for convenience. Our methods construct factorization algebras on an arbitrary manifold with boundary, so long as one can find a tubular neighborhood of the boundary on which the fields decompose to be ``topological normal to the boundary.''
\end{remark}

Here is our generalization of Theorem~\ref{thm: main} at the classical level.

\begin{theorem}
\label{thm: maingenlcl}
For a free bulk-boundary field theory $(\sE,\sL)$, we have the following identifications:
\begin{enumerate}
\item Let $\Obcl_\sE$ denote the factorization algebra on $\mathring M$ of classical observables for $\sE$, constructed using the techniques of Chapter 4 of \autocite{CG1}. Then, there is an isomorphism
\begin{equation}
\Obcl_{\sE,\sL}\big|_{\mathring M} \cong \Obcl_\sE.
\end{equation}
\item There is a quasi-isomorphism
\begin{align}
\label{eq: qicl}
\II^{\cl}: \Obcl_\sL&\to \pi_*\Obcl_{\sE,\sL}.
\end{align}
\end{enumerate}
\end{theorem}

We will state now the quantum analogue of this theorem before turning to the proofs.

\begin{theorem}
\label{thm: maingenlq}
For a free bulk-boundary field theory $(\sE,\sL)$, we have the following identifications:
\begin{enumerate}
\item Let $\Obq_\sE$ denote the factorization algebra on $\mathring M$ of quantum observables for $\sE$, constructed using the techniques of Chapter 4 of \autocite{CG1}. Then, there is an isomorphism
\begin{equation}
\Obq_{\sE,\sL}\big|_{\mathring M} \cong \Obq_\sE.
\end{equation}
\item There is a quasi-isomorphism
\begin{align}
\label{eq: qiq}
\II^{q}: \Obq_\sL&\to \pi_* \Obq_{\sE,\sL}.
\end{align}
\end{enumerate}
\end{theorem}

\begin{remark}
One consequence of this theorem is that the quantum boundary observables, for any choice of splitting $L^\perp$, are explicitly identified, up to quasi-isomorphism, with $\pi_* \Obq_{\sE,\sL}$.
Hence we see again that the choice of splitting is irrelevant.
\end{remark}

\begin{remark}
Theorems \ref{thm: maingenlcl} and \ref{thm: maingenlq} are characterizations of the ``boundary value'' and the ``bulk value'' of the factorization algebras $\Obcl_{\sE,\sL}$, $\Obq_{\sE,\sL}$. However, the bulk-boundary factorization algebras contain more information than their bulk and boundary values alone---they also encode an action of the bulk observables on the boundary observables. This is a rich structure. For example, in the Poisson sigma model we believe the structure to be related to the formality quasi-isomorphism of Kontsevich \autocite{KontPSM}. We study this action for topological mechanics and the Chern-Simons/chiral WZW system in Proposition \ref{prp: toplmech} and Lemma \ref{lem: cswzwonhalfline}, respectively.
\end{remark}

We now turn to proving these theorems. 

\begin{proof}[Proof of classical theorem]
The first statement of the theorem follows immediately from the fact that $\condfields(U)=\sE(U)$ when $U\cap \del M=\emptyset$. 

It remains, therefore, to prove the second statement. 
Throughout the proof, let $U$ be an open subset of $\Mbdy$. 
Let us first construct the cochain map
\begin{equation}
\II^{\cl}(U): \Obcl_\sL(U)\to \Obcl_{\sE,\sL}(U\times \RRge)
\end{equation}
for each open subset $U\subset M_\del$. 
To this end, let $\phi$ be a compactly-supported function on $\RRge$ whose integral over $\RRge$ is 1, and let $\Phi(t):=\int_0^t \phi(s)ds$. 
Both the boundary and bulk observables arise as symmetric algebras built on cochain complexes, 
so the map $\II^{\cl}$ will be induced from a cochain map on the linear observables. 

As a first step, we decompose the fields $\condfieldscs$ further. 
By hypothesis, we have the isomorphism
\begin{equation}
\sE_{c}(U\times \RRge)\cong (\sEb)_{c}(U) \,\hotimes_\beta\, \Omega^\bullet_{\RRge,c}(\RRge).
\end{equation}
Recall that in the construction of the boundary observables, we have a decomposition
\begin{equation}
\Qb=Q_L+Q_{L^\perp}+Q_{rel},
\end{equation} 
where $Q_L$ preserves $\sL$, $Q_{L^\perp}$ preserves $\sL^\perp$, and $Q_{rel}$ maps $\sL^\perp$ to $\sL$. We can therefore write
\begin{equation}
\condfieldscs(U\times \RRge)\cong \left( \sL_{c}(U)\,\hotimes\, \Omega^\bullet_{\RRge,c}(\RRge)\rtimes \sL^\perp_{c}(U)\,\hotimes\, \dirforms(\RRge)\right),
\end{equation}
where $\dirforms(\RRge)$ is the cochain complex (concentrated in degrees 0 and 1)
\begin{equation}
\begin{tikzcd}
\left\{f\in \Omega^0_{\RRge,c}(\RRge)\mid f(t=0)=0\right\}\ar[r,"d_{dR}"]&\Omega^1_{\RRge,c}(\RRge),
\end{tikzcd}
\end{equation}
and the symbol $\rtimes$ reminds us that $\sL^\perp$ is not a subcomplex of $\sEb$. 
Note that our boundary condition requires that only the $\sL^\perp$-valued fields vanish at the boundary. 

Define the map $\II^{\cl}(U): \sL^\perp_c(U)[-1]\to \condfieldscs(U\times \RR_{\geq 0})$ by 
\begin{equation}
\II^{\cl}(U)(\alpha)=\alpha \wedge \phi \, \d t-(-1)^{|\alpha|}(\Phi-1)Q_{rel}\alpha,
\end{equation}
where $|\alpha|$ denotes the cohomological degree of $\alpha$ in $\sL^\perp$ (not $\sL^\perp[-1]$). 
The map $\II^{\cl}(U)$ is of cohomological degree zero because of the terms $\wedge \phi \,\d t$ and $Q_{rel}$. 
Moreover, $\II^{\cl}(U)(\alpha)$ does indeed have compact support if $\alpha$ does, since $(\Phi-1)(t)=0$ for $t>>0$. 
By construction $\II^{\cl}$ is a map of precosheaves. 
We also see that $\II^{\cl}(U)(\alpha)$ satisfies the boundary condition because 
\begin{equation}
\rho\left( \II^{\cl}(U)(\alpha)\right)= (-1)^{|\alpha|}Q_{rel}\alpha
\end{equation}
and $Q_{rel}\alpha$ lives in $\sL_c$.
Finally, we check that $\II^{\cl}(U)$ is a cochain map:
on the one hand,
\begin{equation}
\II^{\cl}(U)(Q_{L^\perp}\alpha) = (Q_{L^\perp}\alpha)\wedge \phi \, \d t-(-1)^{|\alpha|+1}(\Phi-1)Q_{rel}Q_{L^\perp}\alpha,
\end{equation}
and on the other,
\begin{align}
\diff\II^{\cl}(U)(\alpha)& = Q_{L^\perp}\alpha \wedge \phi \, \d t+ Q_{rel}\alpha \wedge \phi \, \d t\\
&-Q_{rel}\alpha \wedge \phi \, \d t -(-1)^{|\alpha|}(\Phi-1)Q_{L}Q_{rel}\alpha.
\end{align}
Once one uses the relation $Q_L Q_{rel}=-Q_{rel}Q_{L^\perp}$, 
one sees that the two expressions are equal. 
Since $\II^{\cl}(U)$ respects the differentials on the complexes as well as the extension maps, 
it extends to a map of factorization algebras $\Obcl_\sL\to \pi_*\Obcl_{\sE,\sL}$.

It remains to show that $\II^{\cl}(U)$ is a quasi-isomorphism. 
We will exhibit, in fact, something much stronger: a deformation retraction.
Namely, we will produce a cochain map $\PP^{\cl}(U)$ such that $\PP^{\cl}(U)\II^{\cl}(U)=\id$ and 
a cochain homotopy $\KK^{\cl}(U)$ between $\II^{\cl}(U)\PP^{\cl}(U)$ and the identity~$\id$. 

To this end, consider the map
\begin{equation}
\PP^{\cl}(U): \condfieldscs(U\times \RRge) \to \sL^\perp_c(U)[-1]
\end{equation}
where
\begin{equation}
\PP^{\cl}(U)(e) =p_{L^\perp}\left( \int_{\RRge} e \right)
\end{equation}
and where $p_{L^\perp}$ is the canonical map $\sEb\to \sL^\perp$ induced by the quotient bundle map $\Eb \to L^\perp$.
Notice that 
\begin{align}
\PP^{\cl}(U)\left(Q_{L^\perp}e+Q_{L}e+Q_{rel}e+(-1)^{|e|}\frac{de}{dt}\wedge \d t \right)
&=Q_{L^\perp}\int_{\RRge}e+(-1)^{|e|}p_{\sL^\perp}\int_{\RRge}\frac{de}{dt}\wedge \dt
\\
&=Q_{L^\perp}\int_{\RRge}e\\
&=Q_{L^\perp}\PP^{\cl}(U)(e),
\end{align}
where the second equality holds because $e$ is compactly supported and $p_{L^\perp} e(0)=0$. 
Hence $\PP^{cl}(U)$ is a cochain map.
Direct computation verifies that $\PP^{\cl}(U)\II^{\cl}(U)=\id$.

Consider now the degree --1 map
\begin{equation}
\KK^{\cl}(U): \condfieldscs(U\times \RRge)\to \condfieldscs(U\times \RRge)
\end{equation}
where
\begin{equation}
\big(\KK^{\cl}(U)(e)\big)(t)= (-1)^{|e|-1}(\Phi(t)-1)(\PP^{\cl}(U)(e))-(-1)^{|e|}\int_t^\infty e(s).
\end{equation}
The field $\KK^{\cl}(U)(e)$ satisfies the required boundary condition because 
\begin{equation}
\KK^{\cl}(U)(e)(0)=(-1)^{|e|}p_{L^\perp}\int_{\RRge}e-(-1)^{|e|}\int_{\RRge}e
\end{equation}
and hence $\KK^{\cl}(U)(e)(0)$ is an element of $\sL$.
Direct computation shows that $\KK^{\cl}(U)$ is a cochain homotopy between $\II^{\cl}(U)\PP^{\cl}(U)$ and the identity. 

Just as $\II^{\cl}(U)$ extends to a map of symmetric algebras, extend $\KK^{\cl}(U)$ and $\PP^{\cl}(U)$ to maps
\begin{align}
\KK^{\cl}(U)&:\Sym(\condfieldscs[1](U\times \RRge))=\left(\pi_*\Obq_{\sE,\sL}\right)(U)\to \left(\pi_*\Obq_{\sE,\sL}\right)(U)\\
\PP^{\cl}(U)&: \Obq_{\sE,\sL}(U)\to \Sym(\sL^\perp_c(U)[-1]) = \Obcl_\sL(U)
\end{align}
by the usual procedure extending a deformation retraction at the linear level to symmetric powers. 
(One treatment with the necessary formulas is Section 2.5 of~\autocite{othesis}.)
\end{proof}

Proving the quantum theorem is a modest modification of the classical argument.

\begin{proof}[Proof of quantum theorem]
The first statement of the theorem again follows immediately from the fact that $\condfields(U)=\sE(U)$ when $U\cap \del M=\emptyset$. 

It remains, therefore, to prove the second statement,
using the constructions from the proof of the classical theorem.
Throughout the proof, let $U$ be an open subset of $\Mbdy$. 
Recall that the cocycle $\mu$ determines $\Obq_\sL$ and the cocycle $\ip$ determines $\Obq_{\sE,\sL}$. 
We will show that $\II^{\cl}(U)$ respects the cocycles and hence determines the desired map $\II^q$ between the quantized factorization algebra.
In particular, we must show that
\begin{equation}
\mu(\alpha_1,\alpha_2)=\ip[\II^{\cl}\alpha_1,\II^{\cl}\alpha_2].
\end{equation}
To see this, compute 
\begin{align}
\ip[\II^{\cl}\alpha_1,\II^{\cl}\alpha_2]&=\int_{\Mbdy\times \RRge} \ip[\alpha_1,Q_{rel}
\alpha_2]_{loc, \del}\phi (1-\Phi) \, \d t +\int_{\Mbdy\times \RRge} \ip[Q_{rel}\alpha_1,
\alpha_2]_{loc, \del}\phi (\Phi-1) \, \d t\\
&=2\mu(\alpha_1,\alpha_2)\int_{\RRge}\phi(1-\Phi)\, \d t.
\end{align}
Since $\frac{d}{dt}(1-\Phi)=-\phi$, we find
\begin{equation}
\int_{\RRge}\phi(\Phi-1)=-\int_{-1}^0 u \, \d u=\frac{1}{2},
\end{equation}
which verifies that $\II^{\cl}$ preserves the cocycles defining the quantum observables, as needed.
\end{proof}

\section{Applications}
\label{sec: apps}

In this section, we apply our main theorems to several bulk-boundary systems,
namely the examples already mentioned in Section~\ref{sec: examples}.
In the low-dimensional examples, we can relate the factorization algebras to more familiar objects, 
such as associative algebras and vertex algebras.
For instance, in the example of topological mechanics, we find that our procedure is equivalent to the canonical quantization of the algebra $\sO(V)$ (on the ``bulk'' line) and its Fock space (on the boundary point).

In higher dimensions, our factorization algebras recover familiar phenomena when using simple product spaces and performing compactifications ({\em aka} pushforwards). For example, on a slab of the form $N\times [0,1]$ with $N$ an oriented 2-manifold, the CS/WZW system is shown to be equivalent to the free massless scalar boson on $N$ (see Lemmas \ref{lem: fullwzwscalarcl} and~\ref{lem: fullwzwscalarq}).

For the sake of space, we omit some examples which appear in the pre-print \autocite{GRW}; we direct the interested reader there.

\subsection{Topological mechanics}
\label{subsec: toplmech}

In this subsection, we study the factorization algebras for topological mechanics with values in $V$ and with boundary condition $L$. We will see that the factorization algebra of classical bulk-boundary observables encodes the commutative algebra $\Sym(V)$ together with the module $\Sym(V/L)$. For the quantum observables, we obtain the Weyl algebra $W(V)$ and the Fock module $F(L)$ built on $L$. (We define these objects in the sequel.)

Recall that a  symplectic vector space $(V,\omega)$ together with a Lagrangian subspace $L\subset V$ define a free bulk-boundary system on $[0,\epsilon)$, which we call topological mechanics (cf. Example~\ref{ex: toplmech}). 
(We can take $V$ to be $\ZZ$-graded, if we like, but of bounded total dimension.)
The main theorem \ref{thm: maingenlcl} identifies $\left.\Obcl_{\sE,\sL}\right|_{(0,\epsilon)}$ with the factorization envelope on $(0,\epsilon)$ of the abelian Lie algebra $V$. 
Proposition 3.4.1 of \autocite{CG1} shows that this factorization algebra is equivalent to the locally constant factorization algebra on $(0,\epsilon)$ corresponding to the associative algebra $\sO(V):=\Sym(V)$. 
Similarly, $\left.\Obq_{\sE,\sL}\right|_{(0,\epsilon)}$ is equivalent to the factorization algebra on $(0,\epsilon)$ corresponding to the Weyl algebra $W(V)$.
(Recall that the Weyl algebra is the algebra generated by $V$ and $\hbar$ and subject to the relation $v_1v_2-v_2v_1=\omega(v_1,v_2)\hbar$.)
We freely use the isomorphism $V\cong V^\vee$ induced by $\omega$, so in particular, we may identify $\sO(V)\cong \Sym(V^\vee)$.

For any $0<\delta\leq 0$, the main theorems also identify the bulk-boundary observables $\Obcl_{\sE,\sL}([0,\delta))$ and $\Obq_{\sE,\sL}([0,\delta))$ with the boundary observables
\begin{equation}
\Obcl_{\sL}\cong \Sym(V/L)
\end{equation}
and
\begin{equation}
\Obq_{\sL}\cong \Sym(V/L)[\hbar],
\end{equation}
respectively, for any $\delta\leq\epsilon$. The second isomorphism arises from the fact that $Q_{rel}=0$. 

These identifications are purely identifications of factorization algebras on $\{0\}$; they do not take into account the actions of $\Obcl_\sE$, $\Obq_\sE$ on the boundary observables. In this subsection, we show how the bulk and boundary observables interact through the bulk-boundary factorization algebras $\Obcl_{\sE,\sL}$ and $\Obq_{\sE,\sL}$. 
Namely, we will examine the structure maps involving one or more intervals including the boundary point.
These structure maps will give the boundary observables the structure of a right module over the corresponding algebras in the bulk.

More precisely, given an algebra $A$ and a pointed right module $M$ of $A$, there is a stratified locally constant factorization algebra $\cF_{A,M}$ on $[0,\epsilon)$ which assigns $A$ to any open interval, and $M$ to any half-closed interval (cf. \S 3.3.1 of \autocite{CG1}). We will show that the cohomology factorization algebras $H^\bullet\Obcl_{\sE,\sL}$ and $H^\bullet \Obq_{\sE,\sL}$ will be of this form for particular choices of $A$ and $M$. We have already discussed that the corresponding algebras are $\sO(V)$ and $W(V)$ for the classical and quantum observables, respectively. It remains only to identify the relevant modules.

The Lagrangian $L\subset V$ determines a (right) module for the commutative algebra $\sO(V)$, namely $\sO(V/L):=\Sym(V/L)\cong \Sym(L^\vee)$ with the module structure induced from the quotient map $V\to V/L$. 
Similarly, $L$ determines a right module $F(L)$ for the Weyl algebra, namely the quotient of $W(V)$ by the right-submodule generated by $L$. The underlying vector space for $F(L)$ is $\Sym(V/L\oplus \hbar)$. Having established all the relevant notation, we can now state the main proposition.

\begin{proposition}
\label{prp: toplmech}
For $(\sE,\sL)$ corresponding to topological mechanics of Example~\ref{ex: toplmech}, 
there is a quasi-isomorphism of factorization algebras on~$\RR_{\geq 0}$
\begin{equation}
\Obcl_{\sE,\sL} \xto{\simeq} \cF_{\sO(V),\sO(V/L)}
\end{equation}
from classical observables to the stratified locally constant factorization algebra associated to the algebra $\sO(V)$ and the module $\sO(V/L)$.
Likewise, there is a quasi-isomorphism of factorization algebras
\begin{equation}
\Obq_{\sE,\sL} \xto{\simeq} \cF_{W(V),F(L)}
\end{equation}
from the quantum observables to the stratified locally constant factroziation algebra associated to the Weyl algebra $W(V)$ and the Fock module~$F(L)$.
\end{proposition}

\begin{proof}
It is straightforward to verify that the factorization algebras of both the classical and quantum bulk-boundary observables are stratified locally constant with respect to the stratification $\{0\}\subset [0,\epsilon)$. 
Hence, each factorization algebra corresponds to some pair $(A,M)$. 
We need only to determine the modules living on the boundary. 
To this end, let $I_1=(0,\epsilon)$ and $I_2=[0,\epsilon)$. 
Consider the structure maps for the inclusion $I_1\subset I_2$. Let $A$ stand momentarily for either of $\sO(V)$, $W(V)$, and similarly let $M$ stand for either of the two modules on the boundary. The structure map $m_{I_1}^{I_2}$ induces a map $A\to M$. The associativity axiom of a prefactorization algebra guarantees that this is a map of $A$ modules.

Recall from the proof of Theorem~\ref{thm: maingenlcl} that the map $\II^{\cl}$ is induced from a choice $\phi$ of compactly-supported function on $I_2$ whose total integral is 1. Let us suppose that $\phi$ is supported on $I_1$. Then, we have a quasi-isomorphism
\begin{equation}
\II^{\cl}_{int}: V\to \condfieldscs(I_1)[1]
\end{equation}
where
\begin{equation}
\II^{\cl}_{int}(v) = \phi \, \d t\otimes v.
\end{equation}
The symmetrization of this map, which we also denote by $\II^{\cl}_{int}$, induces a quasi-isomorphism 
\begin{equation}
\II^{\cl}_{int}:\Sym(V) \to \Obcl_{\sE,\sL}(I_1).
\end{equation}
Consider the composite map
\begin{equation}
\xymatrix{
\Sym(V) \ar[r]|-{\II^{\cl}_{int}}& \Obcl_{\sE,\sL}(I_1)\ar[r]|-{m_{I_1}^{I_2}}&\Obcl_{\sE,\sL}(I_2)\ar[r]|-{\PP^{\cl}(I_2)}&\Sym(V/L),
}
\end{equation}
where $\PP^{cl}(I_2)$ is introduced in the proof of Theorem \ref{thm: maingenlcl}. It follows directly from the definitions that the composite is the map $\Sym(V)\to \Sym(V/L)$ induced from the projection $V\to V/L$. The statement of the proposition for the classical observables follows.

We now ``perturb'' the classical information. 
We would like to understand the structure map 
\begin{equation}
m_{I_1}^{I_2}:\Obq_{\sE,\sL}(I_1)\to \Obq_{\sE,\sL}(I_2)
\end{equation} 
at the level of cohomology. 
We know that the cohomology of $\Obq_{\sE,\sL}(I_1)$ is the underlying vector space of $W(V)$, and the cohomology of $\Obq_{\sE,\sL}(I_2)$ is $\Sym(V/L)[\hbar]$, which is  the underlying vector space of a module $M$ for $W(V)$. 
The structure map $m_{I_1}^{I_2}$ induces a map $T: W(V)\to M$ which intertwines the right $W(V)$ actions. 
Because $\Obq_{\sE,\sL}$ is filtered by powers of $\hbar$, and because the associated graded factorization algebra is $\Obcl_{\sE,\sL}\otimes_{\CC}\CC[\hbar]$, $T$ is surjective. 
Hence, to understand $M$, we simply need to identify the kernel of $T$. In the proof of Theorem \ref{thm: maingenlcl}, we constructed maps $\II^{cl}(I_2),\PP^{cl}(I_2),\KK^{cl}(I_2)$ which fit into a deformation retraction. 
Hence, the homological perturbation lemma (see, e.g., \autocite{crainic}) gives a formula for a quasi-isomorphism 
\begin{equation}
\PP^{\q}: \Obq_{\sE,\sL}(I_2)\to \Sym(V/L)[\hbar].
\end{equation} 
On the sub-complex $\condfields(I_2)[1]\subset \Obq_{\sE,\sL}(I_2)$, $\PP^q$ agrees with $\PP^{\cl}$. 
Moreover, as demonstrated in \autocite{CG1}, the map 
\begin{equation}
\begin{tikzcd}
V\arrow[r,"\II^{\cl}_{int}"] &\condfields(I_1)[1]\arrow[r,hookrightarrow] &\Obq_{\sE,\sL}(I_1)
\end{tikzcd}
\end{equation}
induces the canonical map $V\to W(V)$ on cohomology. 
Finally, tracing through the definitions, the composite 
\begin{equation}
\begin{tikzcd}
V\ar[r,"{\II^{\cl}_{int}}"] &\condfields(I_1)[1]\ar[r,hookrightarrow] &\Obq_{\sE,\sL}(I_1)\ar[r,"{m_{I_1}^{I_2}}"]&\Obq_{\sE,\sL}(I_2)\ar[r,"{\PP^{\q}}"]&M
\end{tikzcd}
\end{equation} 
is seen to be the quotient map $V\to V/L$ followed by the inclusion $V/L\to \Sym(V/L)[\hbar]$. Thus, $L$ is in the kernel of $T$, and hence so too is the whole submodule of $W(V)$ generated by $L$. For dimension reasons, this implies that $M\cong F(L)$.  
\end{proof}

\subsection{Free Poisson sigma model}

In this section, we examine the Poisson sigma model with linear target (Example \ref{ex: psm2}).
Our discussion is brief; see \autocite{GRW} for a much more extensive discussion of this theory. 

Recall from Example~\ref{ex: psm2} that a vector space $V$ and a skew-symmetric linear map $\Pi: V^\vee\to V$ determine a field theory on any oriented surface $\Sigma$ with boundary. The (underlying graded vector) space of fields of this theory is 
\begin{equation}
\left(\Omega_\Sigma^\bullet \otimes V^\vee [1]\right)\oplus \left(\Omega_\Sigma^\bullet\otimes V\right).
\end{equation} 
The differential is of the form $(\d_{dR}\otimes \id)\oplus (\id \otimes \Pi)$ (the first term preserves each summand in the decomposition above, and the second term maps the first summand to the second). In particular, one obtains a field theory on the upper half-plane $\Sigma = \HH$. For the choice of Lagrangian $\sL_0= \Omega^\bullet_{\RR}\otimes V$, we have $\sL_0^\perp= \Omega^\bullet_\RR\otimes V^\vee[1]$. 

Recall our notational convention: given an associative algebra $A$, let $\cF_A$ denote the locally constant factorization algebra on $\RR$ constructed from $A$ (cf. the first example in Section 3.1.1 of~\autocite{CG1}).
The following lemma is a straightforward application of Proposition 3.4.1 of~\autocite{CG1} and Theorems \ref{thm: maingenlcl} and \ref{thm: maingenlq}.

\begin{lemma}
\label{lmm:bdryL0}
For $\sL_0$ as defined just above, 
there is a quasi-isomorphism of factorization algebras on~$\RR$
\begin{equation}
\Obcl_{\sL_0} \xto{\simeq} \cF_{\Sym(V^\vee)}
\end{equation}
for classical observables and a quasi-isomorphism of factorization algebras
\begin{equation}
\Obq_{\sL_0} \xto{\simeq} \cF_{(\Sym(V^\vee)[\hbar],\ast)},
\end{equation}
where $\ast$ refers to the Kontsevich star product on $\Sym(V^\vee)[\hbar]$. 
The product is characterized by the relation
\begin{equation}
\nu_1\ast\nu_2-\nu_2\ast\nu_1 = \hbar \nu_1(\Pi\nu_2),
\end{equation}
where $\nu_1,\nu_2\in V^\vee$. 
\end{lemma}

In our situation the dual space $V^\vee$ can be decomposed into a direct sum of a vector space $V_t^\vee$ with a trivial pairing and a vector space $V_s^\vee$ with a nondegenerate (i.e., symplectic) pairing.
Thus the quantum observables corresponds to a tensor product of a commutative algebra $\Sym(V_t^\vee)$ and a Weyl algebra~$W(V_s)$.
It is thus natural to analyze just these two cases since the general answer can be assembled from them.

\begin{remark}
One can define the boundary observables $\Obq_{\sL_0}$ without making any reference to the theory on $\HH$ for which $\sL_0$ is a boundary condition. 
Theorem~\ref{thm: maingenlq} tells us that these observables are equivalent to the pushforward of the coupled bulk-boundary observables. 
When $V$ ceases to be a linear Poisson manifold, there is no direct definition of $\Obq_{\sL_0}$, and in fact the quantization of $\Sym(V^\vee)$ produced in \autocite{KontPSM} requires in an essential way the study of the theory on $\HH$. 
We expect that once one constructs the factorization algebra of quantum observables for the (interacting) Poisson sigma model on $\HH$, its pushforward to $\RR$ recovers the algebra $\Sym(V^\vee)[\hbar]$ with the Kontsevich star product. 
In fact, we expect that one can apply a similar procedure to any theory with a shifted Poisson structure, using the so-called ``universal bulk theory''~\autocite{butsonyoo}.
\end{remark}

As a special case, when $\Pi$ is zero, the quantum observables correspond to the commutative algebra generated by $V^\vee$ and $\hbar$; 
moreover, a new boundary condition becomes available. 
Namely, we take $\sL_1= \Omega^\bullet_\RR \otimes V^\vee[1]$, so that $\sL^\perp_1 = \Omega^\bullet_\RR\otimes V$. We have a similar lemma.

\begin{lemma}
There are quasi-isomorphisms 
\begin{align}
\Obcl_{\sL_1}&\xto{\simeq} \cF_{\Lambda^\bullet V}\\
\Obq_{\sL_1}& \xto{\simeq} \cF_{(\Lambda^\bullet V)[\hbar]};
\end{align}
where we use $\Lambda^\bu V$ to denote the free graded algebra generated by $V$ in degree~$1$.
\end{lemma}

\subsection{Abelian CS/WZW}
\label{sec: cswzw}

As we have seen, a finite-dimensional complex vector space $\fA$ endowed with a non-degenerate symmetric pairing $\kappa$ defines abelian Chern-Simons theory on an oriented 3-manifold $M$, with space of fields $\sE = \Omega^\bullet_{M}\otimes \fA[1]$.
We focus here on the boundary condition $\sL = \Omega^{1,\bullet}_{\Sigma}\otimes \fA$,
which encodes chiral currents, and we examine this system in two different cases of interest. 

First, we study the system on a manifold of the form $N\times [0,1]$, where $N$ is an oriented 2-manifold endowed with a complex structure at $t=0$ and the conjugate complex structure at $t=1$. 
Let $p: N\times [0,1]\to N$ denote the projection onto the first factor. 
We study the pushforwards $p_* \Obcl_{\sE,\sL}$ and $p_* \Obq_{\sE,\sL}$,
which is a kind of ``slab compactification.''
In this case, we find that the pushforwards are equivalent to the factorization algebras of observables of the massless free scalar on $N$ (Lemmas \ref{lem: fullwzwscalarcl} and~\ref{lem: fullwzwscalarq}).
At the level of factorization algebras, we are recovering a ``full'' CFT by intertwining a chiral and antichiral CFT.

Next, we study the system on a 3-manifold of the form $\Sigma \times \RR_{\geq 0}$, where $\Sigma$ is a Riemann surface.
Here, we push forward via the projection $p'$ onto $\RR_{\geq 0}$, and find that the systems are equivalent to topological mechanics on $\RR_{\geq 0}$ with values in $H^\bullet(\Sigma)[1]$ and with boundary condition $H^{1,\bullet}(\Sigma)$ (see Lemma~\ref{lem: cswzwonhalfline}).

\subsubsection*{Fixing the complement $L^\perp$}

The elliptic complex on $\Sigma$ complementary to the boundary condition $\sL$ is 
\begin{equation}
\sL^\perp = \Omega^{0,\bu} (\Sigma) \otimes \fA [1]
\end{equation}
equipped with the $\dbar$ differential (this is not a subcomplex of $\sEb$).
Using the obvious splitting of $\Omega^\bu(\partial M)$ into the components $\Omega^{0,\bu}(\Sigma)$ and $\Omega^{1,\bu}(\Sigma)$, we see that the differential $\Qb = \d_{dR}$ decomposes as
\begin{equation}
\Qb = Q_L + Q_{L^\perp} + Q_{rel} = \dbar_{\Omega^{1,\bu}} + \dbar_{\Omega^{0,\bu}} + \partial
\end{equation}
where we view $Q_{rel} = \partial$ as the map of elliptic complexes $\partial : \Omega^{0,\bu}(\Sigma) \otimes \fA [1] \to \Omega^{1,\bu}(\Sigma) \otimes \fA$. 

\subsubsection*{The classical observables}

The sheaf $\sE_{\sL}$ of $\sL$-conditioned fields has the following explicit description.
For $U \subset M$:
\begin{align}
\sE_{\sL} (U) & = \{\alpha \in \sE(U) \; | \; \pi(\alpha) \in \iota_* \sL(U) \} \\ & = \{\alpha \in \Omega^{\bu}(U) \otimes \fA[1] \; | \; \iota^* \alpha \in \Omega^{1, \bu} (\partial U) \otimes \fA \} .
\end{align}
That is, the $\sL$-conditioned fields supported on $U \subset M$ consist of differential forms on $U$ whose pullback to the boundary are forms of type $(1,\bu)$. 
Likewise, we have the cosheaf $U \mapsto \sE_{\sL,c} (U)$ on $M$ which consists of compactly supported differential forms on $U$ whose pullback to the boundary are compactly supported forms of type $(1,\bu)$. 

The factorization algebra of classical boundary observables $\Obs^{\cl}_\sL$ on  $\Sigma$ assigns the cochain complex
\begin{equation}
\Sym(\sL_c^\perp(U)) = \left(\Sym\left(\Omega_c^{0,\bu} (U) \otimes \fA [1]\right), \dbar \right) .
\end{equation}
to an open set $U \subset \Sigma$. 
Note that this is the (untwisted) enveloping factorization algebra of the cosheaf of abelian dg Lie algebras $\Omega_c^{0,\bu} \otimes \fA$ on $\Sigma$. See \S 3.6.2 of \autocite{CG1}.

\subsubsection*{The quantum observables}
The factorization algebra of bulk-boundary quantum observables $\Obs^{\q}_{\sE,\sL}$ assigns to the open set $U \subset M$ the cochain complex $(\Sym(\condfieldscs[1](U))[\hbar],\diff+\hbar \Delta)$.

From the general prescription in Section \ref{sec: FAconstruction} the factorization algebra of quantum boundary observables $\Obs_{\sL}^\q$ is the enveloping factorization algebra of $\sL_c^\perp[-1] = \Omega_c^{0,\bu} (U) \otimes \fA$ twisted by a local cocycle $\mu$ whose formula appears in Equation \eqref{eq: twistcocycle}.  

Since $Q_{rel}=\del$ we have the explicit formula for $\mu$: 
\begin{equation}
\mu(\alpha_1,\alpha_2) = \int_{\CC}\kappa(\alpha_1, \del \alpha_2).
\end{equation}
Explicitly, this local cocycle defines the factorization algebra on $\Sigma$ which assigns the cochain complex
\begin{equation}
\left(\Sym\left(\Omega_c^{0,\bu} (U) \otimes \fA [1]\right) [\hbar] , \dbar + \hbar \mu \right) .
\end{equation}
to an open set $U \subset \Sigma$.
In other words, this is the twisted factoriazation enveloping algebra 
of the cosheaf of abelian dg Lie algebras $\Omega_c^{0,\bu}  \otimes \fA$ on $\Sigma$ associated to the local cocycle $\mu$.  
(See \S 3.6.3 of~\autocite{CG1}.)

In Chapter 5 of \autocite{CG1} it is shown that locally on $\Sigma = \CC$, this factorization algebra is a model for the abelian Kac-Moody vertex algebra associated to the level~$\kappa$. 

\subsubsection{Slab compactification}

Let $N$ be an oriented $2$-manifold.
We consider a three-manifold of the form $M= N\times [0,1]$.
Moreover, we equip $N \times \{0\}$ with a complex structure and denote the resulting Riemann surface by $\Sigma$; 
we equip $N \times \{1\}$ complex conjugate complex structure and denote the resulting Riemann surface by $\overline{\Sigma}$. 
Let $\iota_0$ and $\iota_1$ denote the inclusions of $N$ at $t=0$ and $t=1$, respectively. 
Let $\pi: M\to N$ be the projection onto the ``space'' slice of $M$. 
See Figure~\ref{fig: slab}.

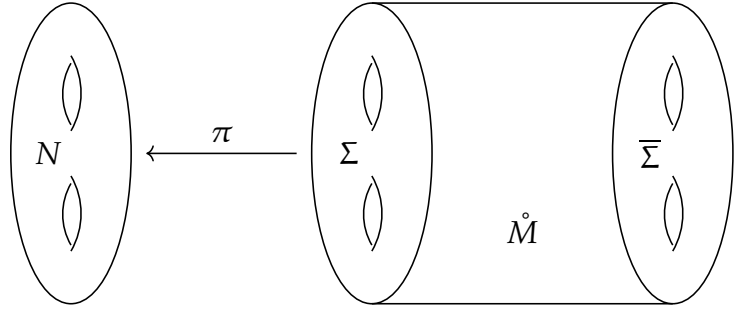
\begin{figure}
\centering
\begin{tikzpicture}
\draw[semithick](-6,0) ellipse (0.8 and 2);
\draw[semithick](-6,0.3) arc (-30:30:1);
\draw[semithick](-6,1.2) arc (150:210:0.82);
\draw[semithick](-6,-1.3) arc (-30:30:1);
\draw[semithick](-6,-0.4) arc (150:210:0.82);
\draw[semithick](-2,0) ellipse (0.8 and 2);
\draw[semithick](-2,0.3) arc (-30:30:1);
\draw[semithick](-2,1.2) arc (150:210:0.82);
\draw[semithick](-2,-1.3) arc (-30:30:1);
\draw[semithick](-2,-0.4) arc (150:210:0.82);
\node at (-2.3,0){$\Sigma$};
\draw[semithick](2,0) ellipse (0.8 and 2);
\draw[semithick](2,0.3) arc (-30:30:1);
\draw[semithick](2,1.2) arc (150:210:0.82);
\draw[semithick](2,-1.3) arc (-30:30:1);
\draw[semithick](2,-0.4) arc (150:210:0.82);
\node at (1.7,0){$\overline{\Sigma}$};
\draw[semithick](-2,2) -- (2,2);
\draw[semithick](-2,-2) -- (2,-2);
\node at (-6.3,0){$N$};
\node at (0,-1){$\mathring{M}$};
\draw[->,semithick] (-3,0) -- (-5,0);
\node at (-4,0.25){$\pi$};
\end{tikzpicture}
\caption{Projection of $N \times [0,1]$ onto $N$}
\label{fig: slab}
\end{figure}

For these choices 
\begin{equation}
\condfields =\left\{ \mu \in \Omega^\bullet_{N\times[0,1]}\otimes \fA[1]\,|\, \iota_0^*\mu \in \Omega^{1,\bullet}_\Sigma\otimes \fA, \iota_1^* \mu \in \Omega^{\bullet, 1}_\Sigma\otimes \fA\right\}. 
\end{equation}
is the space of $\sL$-conditioned fields for the Chern-Simons/chiral WZW bulk-boundary system.

We study now the ``slab compactification" of the factorization algebra of bulk-boundary observables.
This is the factorization algebra on $N$ obtained by pushing forward $\Obs_{\sE, \sL}$ along $\pi$. 
To decongest the notation, we assume that $\fA=\CC$, since all proofs proceed with little change for general~$\fA$.

Let $\sE_{\rm scalar}$ denote the cochain complex underlying the BV theory of the scalar field on $\Sigma$. Namely, it is the two-term chain complex
\begin{equation}
\begin{tikzcd}
\Omega^0_\Sigma \ar[r,"{\del\delbar}"]& \Omega^2_\Sigma
\end{tikzcd}
\end{equation}
concentrated in cohomological degrees 0 and 1, together with the natural degree --1 pairing between top forms and functions. 
Let $\Obcl_{\sE_{\rm scalar}}$ and $\Obq_{\sE_{\rm scalar}}$ denote the factorization algebras of classical and quantum obervables, respectively, for the massless free scalar.

We now show that there is a quasi-isomorphism of factorization algebras between the observables of the free scalar and the slab compactification of the bulk-boundary observables of Chern-Simons theory.
In words, this quasi-isomorphism says that the free massless 2-dimensional scalar field emerges as the theory describing a ``thin'' slab with chiral currents on one side coupled to antichiral currents on the other via a Chern-Simons theory between them.

\begin{lemma}
\label{lem: fullwzwscalarcl}
There is a quasi-isomorphism 
\begin{equation}
\Obcl_{\sE_{\rm scalar}}\to \pi_*\Obcl_{\sE,\sL}
\end{equation}
of factorization algebras on $N$.
\end{lemma}

\begin{proof}
Define a map $I: \sE_{\rm scalar}\to \pi_*\sE_\sL$ by the formulas:
\begin{align}
I(f) &= f\otimes \d t +\delbar f\otimes t -\del f\otimes (1-t)\\
I(\omega) &= \omega\otimes 1.
\end{align}
The sheaf $\pi_*\sE_\sL$ is a subsheaf of $\Omega^\bullet_\Sigma\hotimes_\beta \Omega^\bullet_{[0,1]}([0,1])$, so we ``factorize'' forms into their tangential and normal components and write elements of $\pi_*\sE_\sL$ as tensor products.
Strictly speaking, not all forms can be written as tensor products, or even as finite sums of such. 
However, all formulas we write down have a canonical extension to the full (completed bornological) tensor product $\Omega^\bullet_\Sigma\hotimes \Omega^\bullet_{[0,1]}([0,1])$.
The map $I$ is manifestly a sheaf map, and it is relatively straightforward to check that it is a cochain map,
Furthermore, $I$ induces the desired quasi-isomorphism, as we proceed to show.

We construct an inverse quasi-isomorphism $P$ to $I$.
Let $f\in\Omega^0_\Sigma$, $\alpha\in \Omega^{0,1}_\Sigma$, $\beta\in \Omega^{1,0}_\Sigma$, $\omega\in \Omega^2_\Sigma$, and $\nu_1,\nu_2,\nu_3,\nu_4\in \Omega^\bullet_{[0,1]}$. 
Define the map $P: \pi_*\condfields \to \sE_{\rm scalar}$ by the formulas
\begin{equation}
P(f\otimes \nu_1+\alpha \otimes \nu_2+\beta\otimes \nu_3+\omega\otimes \nu_4) =(\iota_0^* \nu_4)\omega  +f\int_{[0,1]} \nu_1-\delbar \beta \int_{[0,1]}\nu_3.
\end{equation}
Let us check that $P$ is a cochain map. Let $f\otimes \nu_1$ be a zero form on $M$ which lies in $\condfields$, i.e. it vanishes at $t=0$ and $t=1$. Then,
\begin{equation}
P( (\del+\delbar)f\otimes \nu_1 +f\otimes d\nu_1) = f\otimes \nu_1(1)-f\otimes \nu_1(0)=0=Q_{\rm scalar}P(f\otimes \nu_1).
\end{equation}
Let $\alpha\otimes \nu_1+\beta\otimes \nu_2+f\otimes \nu_3$ be a one-form on $M$ which satisfies the boundary conditions to lie in $\condfields$ (here, $\nu_1,\nu_2\in \Omega^0_{[0,1]}$, and $\nu_3\in \Omega^1_{[0,1]}$; the boundary conditions are $\iota_0^*\nu_1=0$ and $\iota_1^*\nu_2=0$). Then,
\begin{align}
P(\del\alpha \otimes \nu_1 &-\alpha\otimes d\nu_1+\delbar \beta\otimes \nu_2 -\beta\otimes d\nu_2+ (\del+\delbar)f \otimes \nu_3)\\
&= \nu_1(0)\del\alpha  + \nu_2(0)\delbar \beta + \delbar \beta (\nu_2(1)-\nu_2(0))-\delbar\del f\int_{[0,1]}\nu_3\\
&= \del\delbar f\int_{[0,1]}\nu_3 = \diff\left( P(\alpha\otimes \nu_1+\beta\otimes \nu_2+f\otimes \nu_3)\right).
\end{align}
This exhausts all the non-trivial checks that $P$ intertwines differentials.

It is immediate that $PI=\id$. 
We now construct a homotopy between $IP$ and $\id$. 
Let $\eta_0$ denote the degree --1 endomorphism of the de Rham forms $\Omega^\bullet_{[0,1]}$ which takes a one-form $\nu$ to the unique anti-derivative of $\nu$ which vanishes at $t=0$. 
Similarly, define $\eta_1$ to be the anti-derivative which vanishes at $t=1$. Now, define
\begin{align}
K: \pi_*\condfields &\to \pi_*\condfields [-1]\\
K(f\otimes \nu)& = f\otimes \eta_0(\nu) -\left(\int_{[0,1]}\nu\right) f\otimes t\\
K(\alpha\otimes \nu) & = -\alpha \otimes \eta_0(\nu)\\
K(\beta\otimes \nu) & = -\beta\otimes \eta_1(\nu)\\
K(\omega\otimes \nu) & = \omega \otimes \eta_0(\nu),
\end{align}
One can verify by straightforward computation that $\diff K+K\diff=IP-\id$,
which proves that $P$ and $I$ are inverse quasi-isomorphisms.

All maps involved are manifestly sheaf-theoretic over $\Sigma$, and moreover they preserve compact support. 
Hence, we also have  quasi-isomorphisms
\begin{equation}
\sE_{\rm scalar,c}(U)[1]\to \condfieldscs(U\times[0,1])[1]
\end{equation}
for each $U\subset N$, and the quasi-isomorphisms respect the extension by zero maps. The lemma follows, using the usual extension of a deformation retraction between cochain complexes to a deformation retraction between the corresponding symmetric algebras.
\end{proof}

A similar lemma holds for the quantum observables.

\begin{lemma}
\label{lem: fullwzwscalarq}
There is a quasi-isomorphism
\begin{equation}
\Obq_{\sE_{\rm scalar}}\to \pi_* \Obq_{\sE,\sL}
\end{equation}
of factorization algebras on~$N$.
\end{lemma}

\begin{proof}
By direct inspection, the map $I$ defined in the proof of Lemma \ref{lem: fullwzwscalarcl} respects the (--1)-shifted pairings on $\sE_{\rm scalar}$ and $\condfields$. Hence it induces also a quasi-isomorphism on the quantum observables.
\end{proof}

\begin{ucorollary}
\label{crl: fullcswzw}
Let $\Obcl_{\chi}$ denote the boundary observables for the chiral WZW boundary condition on $\Sigma$, and similarly let $\Obcl_{\bar\chi}$ denote the boundary observables for the anti-chiral WZW boundary condition on $\overline{\Sigma}$. There is a map of factorization algebras on $N$:
\begin{equation}
\Obcl_{\chi}\hotimes_\beta \Obcl_{\bar \chi}\to \Obcl_{\sE_{\rm scalar}}.
\end{equation} 
There is an analogous map for the quantum factorization algebras.
\end{ucorollary}

This map encodes the chiral and antichiral ``sectors'' of the full CFT.
When evaluated on a disk, it determines a map from a vertex algebra tensored with its conjugate into the OPE-algebra of the massless scalar field.
On a closed Riemann surface, 
the global sections of $\Obcl_\chi$ and $\Obcl_{\bar\chi}$  are (dual to) the conformal blocks of the holomorphic and anti-holomorphic Kac-Moody vertex algebras, respectively. 
This pairing of the factorization algebras gives a local-to-global description of the ``holomorphic factorization" of the conformal blocks of the full WZW theory \autocite{WittenFactorization} in the case of an abelian group.

\begin{proof}
By Theorem \ref{thm: maingenlcl}, 
\begin{equation}
\Obcl_\chi \simeq \pi_*\left.\Obcl_{\sE,\sL}\right |_{N\times [0,1/2)},
\end{equation}
and
\begin{equation}
\Obcl_{\bar\chi} \simeq \pi_*\left.\Obcl_{\sE,\sL}\right |_{N\times (1/2,1]}.
\end{equation}
By Lemma \ref{lem: fullwzwscalarcl}, 
\begin{equation}
\Obcl_{\sE_{\rm scalar}}\simeq\pi_*\Obcl_{\sE,\sL} 
\end{equation}
The map of the present corollary is then induced from the structure maps of $\Obcl_{\sE,\sL}$ for inclusions of the form $U\times[0,1/2)\sqcup U\times (1/2,1] \subset U\times [0,1]$.
\end{proof}

\subsubsection{The other projection}

Let $M=\Sigma \times \RR_{\geq 0}$, and consider the projection $p:M\to \RR_{\geq 0}$. 
In this section, we study the pushforward factorization algebras $p_* \Obcl_{\sE,\sL}$ and~$p_* \Obq_{\sE, \sL}$,
which can be seen as studying canonical quantization of Chern-Simons theory (cf. \S4.4 of~\autocite{CG1}).

Let $V=H^\bullet(\Sigma)[1]$
and endow it with the symplectic structure induced from the Poincar\'{e} duality pairing.
This graded vector space models the tangent complex at the basepoint of the $U(1)$-character stack for $\Sigma$;
its symplectic structure is also known as the Atiyah-Bott form. 
Let $L$ denote the cohomology of the holomorphic 1-forms on $\Sigma$;
it is the Lagrangian in $V$ given by the $(1,\bullet)$-part of the Dolbeault cohomology of $\Sigma$.
(A choice of K\"{a}hler metric on $\Sigma$ gives such an embedding $L\to V$.)
Conceptually, picking this Lagrangian corresponds to choosing a polarization of the character stack.

By pushing forward from $M$ to $\RR_{\geq 0}$, 
we reduce the study of abelian Chern-Simons theory to the study of topological mechanics for the pair $(V,L)$,
which we treated in Section~\ref{subsec: toplmech}.

\begin{lemma}
\label{lem: cswzwonhalfline}
As factorization algebras on $\RR_{\geq 0}$, we find
\begin{itemize}
\item the classical observables $p_* \Obcl_{\sE,\sL}$ are quasi-isomorphic to $\cF_{\sO(V),\sO(L)}$, 
and
\item  the cohomology of the quantum observables $H^\bullet(p_* \Obq_{\sE,\sL})$ is isomorphic to $\cF_{W(V),F(L)}$,
which encodes the Weyl algebra associated to that tangent complex as well as the Fock space determined by the Lagrangian. 
\end{itemize}
\end{lemma}

In other words, at the classical level, our factorization algebra encodes the symplectic geometry of the $U(1$)-character stack near its base point, including the natural polarization associated with choosing a complex structure on the surface.
Our quantization recovers the canonical quantization of that data.
In short, our process recovers a shadow of the geometric quantization of abelian Chern-Simons theory.

\begin{proof}
Choose a K\"{a}hler metric on $\Sigma$. 
Let $(\sE,\sL)$ denote the Chern-Simons/chiral WZW bulk-boundary system on $\Sigma\times \RR_{\geq 0}$. 
Let $(\sF,\sK)$ denote topological mechanics on $\RR_{\geq 0}$ with values in $H^\bullet(\Sigma)[1]$ and with boundary condition $H^{1,\bullet}(\Sigma)$. 
Hodge theory using the K\"{a}hler metric allows one to construct a quasi-isomorphism
\begin{equation}
\sF_{\sK,c}(U)[1]\to \condfieldscs(\Sigma\times U)[1]
\end{equation}
for any open subset $U\subset \RR_{\geq 0}$.
This quasi-isomorphism manifestly preserves the cocycles used to define the quantum observables and the extension-by-zero maps for inclusions $U\subset V$. 
It follows that  $p_* \Obcl_{\sE,\sL}$ and $p_*\Obq_{\sE,\sL}$ are equivalent to the corresponding factorization algebras for topological mechanics. 
The lemma follows via Proposition~\ref{prp: toplmech}.
\end{proof}

%% file: Chapters/Chapter4.tex
\chapter{Interacting Quantum Bulk-Boundary Systems}
\label{chap: interactingquantum}
\section{Introduction}

In this chapter, we articulate what we mean by a quantum bulk-boundary system, and we show that quantum bulk-boundary systems give rise to factorization algebras of observables. 
Our approach follows the work of Costello and Gwilliam (\autocite{cost, CG1, CG2}) very closely.

Recall from the Introduction that Costello defines a quantum field theory on a manifold $M$ (without boundary) to be a collection $\{I[L]\}$ of interaction functionals, one for each $L>0$, satisfying
\begin{enumerate}
    \item The (homotopical) renormalization group flow equation, which specifies $I[L]$ in terms of $I[L']$ for any $L$ and $L'$,
    \item The (renormalized) quantum master equation (QME),
    \item An asymptotic locality condition as $L\to 0$.
\end{enumerate}
 
Given such a quantum field theory, Costello and Gwilliam \autocite{CG2} show that one may assign to it a factorization algebra on $M$.

We will establish a similar definition for quantum bulk-boundary systems (cf. Definition \ref{def: QTNBFT}).
Having made this definition, we will then discuss the obstruction theory for the quantization of bulk-boundary systems.
Finally, we will construct factorization algebras of observables associated to quantum bulk-boundary systems.
Many of the constructions of Costello and Gwilliam carry over \emph{mutatis mutandis}.

\emph{Inter mutandes}, however, is where the main technical details of this dissertation lie.
Namely, in making sense of the RG equation and the QME on manifolds without boundary, one needs to introduce a renormalized propagator and BV Laplacian.
It takes some care to define the appropriate objects in the case of bulk-boundary systems.
To this end, we use the so-called ``method of image charges.''
In mathematical terms, we replace a manifold with boundary $M$ with its double $\Mdbl$, i.e. the union of two copies of $M$ along $\bdyM$.
(Recall that $\bdyM$ has a number of preferred tubular neighborhoods in $M$, namely those for which the axioms of a classical bulk-boundary system are satisfied. The gluing happens using one of these choices. 
The double $\Mdbl$ possesses an involution $\involution$, and we relate the study of the theory on $M$ to that of a $\involution$-invariant theory on $\Mdbl$.
We will call this doubling procedure ``the doubling trick.'' For our purposes, ``the doubling trick'' and ``the method of image charges'' represent the same techniques.

We emphatically do \emph{not} make the claim that a bulk-boundary system on $M$ is in general the same as a $\involution$-invariant theory on $\Mdbl$.
Nevertheless, $\involution$-invariant objects on $\Mdbl$ will furnish the desired heat kernels and propagators for the bulk-boundary system on $M$.

To make use of the doulbing trick, we will need to impose additional restrictions on our bulk-boundary systems (in addition to the restrictions we have already made in Chapter \ref{chap: classical}).
The universal bulk theories, as well as most examples we consider, will survive this additional culling.

To the reader familiar with the formalism of Costello and Gwilliam, we note that the formalism for bulk-boundary systems exhibits essentially all of the same features. In particular, we show a bijection between the space of local functionals and the so-called quantum ``pre-theories;'' show that the obstruction theory of quantum bulk-boundary systems is governed by the complex of local functionals introduced and studied in Section \ref{sec: localfcnls}; and construct a factorization algebra of observables associated to a quantum bulk-boundary system.

\section{Strong Boundary Conditions and the Doubling Trick}
\label{sec: firstdbling}
In this section, we describe a procedure which introduces additional boundary conditions on the space of fields. The role of the additional boundary conditions is to replace the spacetime $M$ with the manifold (without boundary) $M_{DBL}$ obtained by gluing two copies of $M$ along $\bdyM$, at least for the study of the underlying free theory to $\sE$. However, to make the doubling procedure work, we need to restrict the class of field theories we consider still further.
\begin{definition}
Let $\sL$ be a boundary condition for a TNBFT. A \textbf{complement to $\sL$} is a vector bundle complement $L'$ to $L$ in $\Eb$. We write \index[notation]{Lprime@$\sL'$} $\sL'$ for the sheaf of sections of $L'$. 
\end{definition}

\begin{definition}
\label{def: amenable}
Let $(\sE,\sL)$ be a bulk-boundary system. ($\sE,\sL$) is \textbf{amenable to the doubling trick} if one can find a complement $\sL'$ to $\sL$ such that there is a decomposition 
\begin{equation}
E=A\oplus B
\end{equation}
where $A$ and $B$ are graded vector bundles on $M$. We require the following:
\begin{enumerate}
\item The pairing $\ip_{loc}$ is zero when restricted to $A\otimes A$,
\item In the same tubular neighborhood $\tubnhd$ over which $E\mid_{\tubnhd}\,\cong \Eb\boxtimes \Lambda^\bullet(T^*[0,\epsilon))$, we have $A\mid_{\tubnhd}\,\cong \Ab\boxtimes \Lambda^\bullet(T^*[0,\epsilon))$ and similarly for $B$. Here $\Ab, \Bb$ are subbundles of $\Eb$, and $\Eb=\Ab\oplus \Bb$. 
\item Let $\Gamma$ be the bundle map $\Eb\to \Eb$ which acts by 1 on $L$ and $-1$ on $L'$. We require $\Gamma$ to preserve $\Ab$ and $\Bb$.
\item Let \index[notation]{A@$\sA$}$\sA$, \index[notation]{B@$\sB$} $\sB$ denote the sheaves of sections of $A,B$, and similarly for $\sAb $.
We require that one can write
\index[notation]{Q@$\diffonnocrossdiff$}\index[notation]{ell1@$\pertdiff$}
$\diff=\diffonnocrossdiff+\pertdiff$, where $\pertdiff: \sA\to \sB$ and $\diffonnocrossdiff$ is of the form $\diffonnocrossdiffbdy\otimes 1+1\otimes d_{dR}$ when acting on sections in $\cinfty(\tubnhd; E|_T)$. Here, $\diffonnocrossdiffbdy$ is a differential on $\sEb$ making $\sEb$ into an elliptic complex and $\diffonnocrossdiffbdy$ preserves $\sL$, $\sL'$, $\sAb$, and $\sBb$.
\item $\diffonnocrossdiff$ preserves $\sA$ and $\sB$. (Over $\tubnhd$, this statement follows from the previous condition.)
\item $\diffonnocrossdiff$ and $\pertdiff$ are separately (almost) skew-symmetric for the pairing $\ip$, i.e.
\begin{equation}
\ip[\diffonnocrossdiff e_1,e_2]+(-1)^{|e_1|}\ip[e_1,\diffonnocrossdiff e_2]=0
\end{equation}
(and similarly for $\pertdiff$) for~$e_1,e_2\in \cinfty_{c.s.}(\mathring M; E)$.
\end{enumerate}
\end{definition}

\begin{remark}
We admit that some of these conditions may be redundant. 
We do not investigate this issue here.
\end{remark}

\begin{remark}
The purpose of Definition \ref{def: amenable} is to introduce a general sub-class of bulk-boundary systems containing most examples of interest to us while keeping the formalism as general as we can make it.

For the reader familiar with the Alexandrov-Kontsevich-Schwarz-Zaboronsky formalism, we the reader may think of Definition \ref{def: amenable} to be a notion of AKSZ-type theories whose target is a (possibly twisted) shifted cotangent bundle.
(This condition is imprecise, but we hope it can help to orient some readers.)
\end{remark}

When we say that a field theory $(\sE,\sL)$ is amenable to doubling, we will usually have in mind that $(\sE,\sL)$ is already equipped with a particular choice of amenability data $L',A,B,\cdots$. 
In \autocite{cost}, Costello gives a careful accounting of the dependence of his constructions on choices of this nature.
We believe that a similar accounting can be done in the context at hand; however, for the sake of space, we do not undertake this here.

\begin{example}
\label{ex: nocrossdifffromstart}
Suppose $(\sE,\sL)$ is a TNBFT with boundary condition and one can find an $\sL'$ so that $\Qb$ preserves $\sL$ and $\sL'$. Then we can take $A=0$, $B=E$, $\diffonnocrossdiff=\diff$. Hence, any field theory with $\Qb$ of this form is amenable to the doubling trick (e.g. BF theory). 
\end{example}

\begin{example}
In \autocite{butsonyoo}, Butson and Yoo describe a classical bulk-boundary system on a manifold of the form $\bdyM\times \RR_{\geq 0}$ associated to any Poisson BV theory.
In the interest of space, we have not defined these objects here. 
Nevertheless, we briefly describe some of the features of this ``universal bulk theory,'' so that we can show that it is amenable to doubling.
The sheaf of fields $\sF$ for a degenerate theory is the sheaf of sections on $\bdyM$ of a vector bundle.
The space of fields for the universal bulk theory for $\sF$ is 
\[
\sF\hotimes_\beta \Omega^{\bullet}_{\RR_{\geq 0}}\oplus \sF^!\hotimes \Omega^{\bullet}_{\RR_{\geq 0}};
\]
the differential $\diff$ on this complex is $Q_{\sF\oplus \sF^!}\otimes 1+ 1_{\sF\oplus \sF^!}\otimes d_{dR} + \Pi\otimes 1$,
where 
\[
\Pi: \sF \to \sF^!
\]
is the degree +1 operator induced from the constant term in the Poisson structure on $\sF$ and $Q_{\sF\oplus \sF^!}$ preserves both $\sF$ and $\sF^!$.
It follows from this description that the boundary space of fields is $\sF\oplus \sF^!$, together with the differential $Q_{\sF\oplus \sF^1}+\Pi$.
To show that such a theory is amenable to doubling, take 
\begin{equation}
L=F \quad L'=F^!,
\end{equation}
\begin{equation}
A=L\boxtimes (\Lambda^\bullet(T^*\R_{\geq 0})) \quad B=L'\boxtimes (\Lambda^\bullet(T^*\R_{\geq 0})),
\end{equation}
\begin{equation}
\diffonnocrossdiffbdy = Q_{\sF}\oplus Q_{\sF}^!,
\end{equation}
and $\pertdiff$ induced from the constant term in the $P_0$ structure on $\sF$.
\end{example}

\begin{example}
In this example, we describe a bulk-boundary system which is an example and a non-example, depending on the manifold on which it is formulated.
Chern-Simons theory with chiral WZW boundary condition may be formulated on a general 3-manifold with a transversely holomorphic foliation, where the foliation is transverse also to the boundary.
In such a general situation, it is not possible to show that this bulk-boundary condition is amenable to doubling.
However, for a 3-manifold of the form $\Sigma\times \RR_{\geq 0}$, this bulk-boundary system is the universal bulk theory associated to a Poisson BV theory on $\Sigma$; hence, on such manifolds, the CS/WZW system is amenable to doubling.
\end{example}

\begin{remark}
The purpose of Definition \ref{def: amenable} is to split the differential $\diff$ on $\sE$ into the sum of a term $\diffonnocrossdiff$ which has ``nice'' behavior near the boundary and a term $\pertdiff$ whose corresponding contribution to the action 
\begin{equation}
S_\pertdiff(e) = \frac{1}{2}\ip[e,\pertdiff e]
\end{equation}
can be treated as an interaction. Using the skew-symmetry of $\pertdiff$, one finds that $S_\pertdiff(e)$ is a functional of $\sA$ fields only.
This will be a crucial property in guaranteeing that one can in fact treat $S_\pertdiff(e)$ as an interaction term (cf. Definition-Lemma \ref{deflem: rgflow}).
\end{remark}

We now note a number of nice properties of bulk-boundary systems which are amenable to doubling.
Once we do this, we will be in a better position to see what makes such bulk-boundary systems ``amenable to doubling.''

We will let 
\index[notation]{Eo@$\nocrossdiff$}
$\nocrossdiff$ denote the complex of sheaves
\begin{equation}
(\sE,\diffonnocrossdiff);
\end{equation}
we will often write $\sE$ for the complex $(\sE, \diff)$.
Hence, $\nocrossdiff$ and $\sE$ have the same underlying graded spaces;
they differ only in their differentials.
Similarly, we will denote by 
\index[notation]{ELo@$\condnocrossdiff$}
$\condnocrossdiff$ the complex of sheaves consisting of fields in $\nocrossdiff$ satisfying the boundary condition $\sL$. This is a subcomplex of $\nocrossdiff$ by Condition 4 in Definition \ref{def: amenable}.

\begin{definition}
\label{def: ultra}
Let $(\sE,\sL)$ be amenable to doubling. The \textbf{sheaf of \ultrafieldsterm fields} is the following graded sheaf; for an open $U$ with $U\cap \bdyM= \emptyset$, $\ultrafields(U)=\nocrossdiff$. If $U\cap \bdyM\neq \emptyset,$ for $e\in \nocrossdiff(U)$, write $e= \lambda_\sL+\mu_\sL \d t+\lambda_{\sL'}+\mu_{\sL'}\d t$ near the boundary, and set 
\index[notation]{ELultra@$\ultrafields$}
\begin{equation}
\ultrafields(U):= \left\{ e\in \nocrossdiff(U)\left| \frac{d^n\lambda_\sL}{dt^n}(0)\right.=\frac{d^m\mu_\sL}{dt^m}(0)=\frac{d^m\lambda_{\sL'}}{dt^m}(0)=\frac{d^n\mu_{\sL'}}{dt^n}(0)=0, n \text{ odd}, m \text{ even}\right\}.
\end{equation}
$\ultrafields$ is a subsheaf of $\condnocrossdiff$.
\end{definition}

\begin{lemma}
$\ultrafields$ is a subcomplex of $\condnocrossdiff$. 
\end{lemma}
\begin{proof}
A straightforward computation, using the fact that $\diffonnocrossdiffbdy$ preserves \emph{both} $\sL$ and $\sL'$.
\end{proof}

The following proposition is a slight generalization of Section A.1 of \autocite{CMRquantumgaugetheories}. In the language of that reference, $\ultrafields$ is the space of fields whose $\sE_1$ component is ultra-Neumann and whose $\sE_2$ component is ultra-Dirichlet.

\begin{proposition}
\label{prop: ultraharm}
The inclusion $I:\ultrafields\to \condnocrossdiff$ is a quasi-isomorphism of complexes of sheaves. 
\end{proposition}

\begin{remark}
Recall that we have $\diff = \diffonnocrossdiff+\pertdiff$, and we are intending to treat the term $\pertdiff$ as an interaction term.
Hence $\condnocrossdiff$ should be thought of as the space of fields of the underlying free theory of our bulk-boundary system.
On the other hand, the complex $\ultrafields$ represents fields satisfying a much stronger boundary condition than the one included in $\condnocrossdiff$.
Proposition \ref{prop: ultraharm} shows that, at least for the underlying free theory, imposing the extra boundary conditions on our fields does not change the space of classical solutions.
\end{remark}

\begin{proof}
Let $U$ be an open subset of $M$; we may assume $U\cap \bdyM\neq \emptyset$, for otherwise $\ultrafields(U)=\condnocrossdiff(U)$. We assume further that $U$ is of the form $U'\times [0,\epsilon)$, where $U'$ is open in $\bdyM$. Consider the degree --1 map $h(U): \condnocrossdiff(U)\to \condnocrossdiff(U)$ given by
\begin{equation}
h(U)(e)(t) = (-1)^{|\mu_\sL|}\int_0^t \mu_\sL(s)\d s+(-1)^{|\mu_{\sL'}|}\int_0^t \mu_{\sL'}(s)\d s,
\end{equation}
where we are using the notation of Definition \ref{def: ultra}. Note that $h(U)$ preserves $\ultrafields(U)$. Let us write $\diffonnocrossdiffbdy=\QL+\QLprime$ (where $\QL$ preserves $\sL$ and $\QLprime$ preserves $\sL'$), and note that
\begin{align}
(\diffonnocrossdiff h(U)+h(U)\diffonnocrossdiff)(e)(t) &= \mu_\sL(t)+\mu_{\sL'}(t)\\
&+ (-1)^{|\mu_\sL|}\int_0^t \QL \mu_\sL(s)ds+(-1)^{|\mu_{\sL'}|}\int_0^t \QLprime\mu_{\sL'}(s)ds\\
&+(-1)^{|\mu_\sL|+1}\int_0^t \QL\mu_\sL(s)ds+(-1)^{|\mu_{\sL'}|+1}\int_0^t \QLprime\mu_{\sL'}(s)ds
\\&+\lambda_{\sL}(t)-\lambda_{\sL}(0)+\lambda_{\sL'}(t)-\lambda_{\sL'}(0)\\
&= e(t)- \rho(e);
\end{align}
therefore, if one defines $P(U): \condnocrossdiff\to\ultrafields\,$ to be $P(U)(e)(t)=\rho(e)$, one finds that $P(U)$ and $I(U)$ are mutually inverse homotopy equivalences, with $h(U)$ as the witnessing homotopy between both $I(U)P(U)$ and $P(U)I(U)$ and their respective identity maps. We have therefore shown that, for each point $x\in M$ and each open $V\subset M$ containing $x$, we can find an open subset $U\subset V\subset M$ containing $x$ on which $I(U)$ is a quasi-isomorphism. Hence, $I$ is a quasi-isomorphism of sheaves.
\end{proof}

We reiterate that one should only understand Proposition \ref{prop: ultraharm} as a statement about the underlying free theories encoded in $\ultrafields$ and $\condnocrossdiff$, and not as a statement about the full interacting theory whose free part is $\condnocrossdiff$. 
The following example shows that the brackets on $\condnocrossdiff$ do not restrict along the inclusion $\ultrafields\subset \condnocrossdiff$.

\begin{example}
Consider one-dimensional BF theory on $M=[0,1)$, with space of fields $\Omega^\bullet_M\otimes \fg[1]\oplus \Omega^\bullet_M\otimes \fg^\vee[-1]$ and boundary space of fields $\sEb= \fg[1]\oplus \fg^\vee[-1]$. 
Choose $\sL=\fg^\vee[-1]$ and $\sL'=\fg[1]$.
This theory is amenable to doubling, and $\condnocrossdiff=\condfields$. 
Given $f_1,f_2\in \cinfty_M\otimes \fg$ satisfying the conditions to lie in $\ultrafields$ (i.e. all even derivatives of $f_1$ and $f_2$ are zero at $t=0$), then $[f_1,f_2]$ does not necessarily satisfy these conditions, since the second derivative of $[f_1,f_2]$ at $t=0$ has a term which is the bracket of the first derivatives of $f_1$ and $f_2$ at $t=0$. 
\end{example}

In light of the previous example, we note that, in general, $\ultrafields$ cannot be readily made to contain information about the interacting bulk-boundary system $(\sE,\sL)$.
However, since the propagators and heat kernels of Costello are derived from the underlying free theory of an interacting theory,
we will be able to use $\ultrafields\,$ to extract the necessary ingredients in the definition of a quantum bulk-boundary system.

Now, we are in a position to explain the doubling procedure a bit better.
Let $\tubnhd$ denote a tubular neighborhood of the boundary $\bdyM$, and fix a diffeomorphism $\tubnhd\cong \bdyM\times [0,\epsilon)$.
We let, as above, 
\index[notation]{Mdbl@$\Mdbl$}
$\Mdbl$ denote the manifold without boundary obtained by gluing two copies of $M$ along $\bdyM$. 
More precisely, $\Mdbl$ is covered by three open sets: two copies of $\mathring M$ and one copy of $\bdyM\times (-\epsilon, \epsilon)$.
The three open sets are glued together via the following pair of maps:

\begin{align}
f_1,f_2:\tubnhd\backslash\bdyM\cong \bdyM\times(0,\epsilon)&\to \bdyM \times(-\epsilon,\epsilon)\\
f_1(x,t)&= (x,t)\\
f_2(x,t)&=(x,-t).
\end{align}
We will call the copy of $\mathring M$ glued to $\bdyM\times (-\epsilon,\epsilon)$ via $f_1$ the \textbf{positive half-manifold}; similarly, we call the copy glued via $f_2$ the \textbf{negative half-manifold}. Similarly, we can define the \textbf{non-negative and non-positive half-manifolds.}
There is a natural involution of $\Mdbl$ which exchanges the positive and negative half-manifolds and fixes $\Mbdy\subset \Mdbl$.

Now, we proceed to double $E$ to form a bundle $\Edbl\to \Mdbl$.
Over both the positive and negative half-manifolds, we set 
\index[notation]{Edbl@$\Edbl$}
$\Edbl= E\mid_{M\backslash\bdyM}$. 
Over $\bdyM\times(-\epsilon,\epsilon)$, we set 
\begin{equation}
\Edbl = \Eb\boxtimes \Lambda^\bullet(T^*(-\epsilon,\epsilon)).
\end{equation}
Let $\Gamma$ be the involution on $\Eb$ which is the identity on $L$ and minus the identity on $L'$. 
The transition function covering $f_1$ is the isomorphism $\phi:E
\mid_\tubnhd\,\cong \Eb\boxtimes (\Lambda^\bullet T^*(0,\epsilon))$. The transition function for $\Edbl$ covering $f_2$ is the map $(\Gamma\otimes \tau^*)\circ\phi$, where $\tau$ is the map $(-\epsilon,0)\to (0,\epsilon)$ given by $\tau(t)=-t$. 

Further, there is a natural bundle isomorphism
\index[notation]{sigma@$\involution_E$}
\begin{equation}
    \involution_E : \Edbl \to \involution^*\Edbl
\end{equation}
covering the natural involution $\involution$ on $\Mdbl$.
Over each copy of $M\backslash \bdyM$, the involution on $\Mdbl$ simply exchanges the two copies and the bundle map covering this map is the identity on $E\mid_{M\backslash\bdyM}$. 
Over $\bdyM\times (-\epsilon,\epsilon)$, the involution on $\Mdbl$ takes $t\mapsto -t$, and $\involution_E\mid_{\bdyM \times(-\epsilon,\epsilon)}$ is simply $\Gamma\otimes \tau^*$, where we extend $\tau$ to an automorphism of $(-\epsilon,\epsilon)$.

\begin{example}
\label{ex: BFdbl}
Consider BF theory on an oriented manifold $M$. Let $\sL= \Omega^\bullet_{\bdyM}\otimes \fg[1]$ and $\sL' = \Omega^\bullet_{\bdyM}\otimes\fg^\vee[n-2]$.
One readily sees that $\Mdbl$ coincides with the oriented smooth manifold obtained by the usual doubling procedure, and $\involution$ corresponds to the natural orientation-reversing diffeomorphism.
One may also identify 
\begin{equation}
\sEdbl \cong \Omega^\bullet_{\Mdbl}\otimes \fg[1]\oplus \Omega^\bullet_{\Mdbl}\otimes \fg^\vee[n-2];
\end{equation}
however $\involution$ extends to $\sEdbl$ as $\involution^*$ on the first summand and $-\involution^*$ on the second summand.
\end{example}

\begin{lemma}
\label{lem: qandipdouble}
The differential $\diffonnocrossdiff$ on $\nocrossdiff$ extends naturally to a $\involution_E$-equivariant one on $\sEdbl$. Likewise, the pairing $\ip_{loc}$ extends to a $\involution_E$-equivariant pairing $\ip_{loc,DBL}$ on $\sEdbl$, and $\diffonnocrossdiff$ is skew-self-adjoint for the pairing $\ip_{DBL}$ induced from $\ip_{loc,DBL}$.
\end{lemma}

\begin{proof}
It is enough to check these facts locally.
For $\diffonnocrossdiff$ this fact follows from the fact that along the normal direction, $d_{dR}$ is diffeomorphism invariant and the fact that $\diffonnocrossdiffbdy$ preserves $\sL$ and $\sL'$.
We extend $\ip_{loc}$ to $\Edbl$ as follows: over the positive half-manifold, $\Edbl = E$ and $\Dens_{\Mdbl}=\Dens_M$.
So, we extend $\ip_{loc}$ in the obvious way, i.e. set $\ip_{loc,DBL}=\ip_{loc}$.
Over the negative half-manifold, we do the opposite: we set $\ip_{loc, DBL}=-\ip_{loc}$.
This negative sign accounts for the fact that we identify the copy of $\tubnhd$ in the negative half-manifold with $\bdyM\times(-\epsilon,0]\subset \bdyM\times(-\epsilon,\epsilon)$.
If the reader remains befuddled by this minus sign, note that in the context of Example \ref{ex: BFdbl}, the negative half-manifold is given the opposite orientation to $M$, so we must make the opposite identification of top forms and densities over the negative half-manifold.
Finally, over $\bdyM\times(-\epsilon,\epsilon)$, we set $\ip_{loc,DBL}$ to be of the form described in the very last axiom listed in Definition \ref{def: clbv}.
The remaining checks, for example that $\ip_{loc,DBL}$ indeed gives a bundle map, are routine.
\end{proof}

The next Lemma relates the ultra-boundary conditions discussed in Definition \ref{def: ultra} to $\involution_E$-invariance in $\Edbl$.

\begin{lemma}
\label{lem: ultraharmdbl}
For any open subset $U\subset M$, let $\Udbl$ denote the open subset in $\Mdbl$ obtained by doubling $U$.
Then, the cochain complex $\ultrafields(U)$ is isomorphic (as a complex of differentiable vector spaces) to the cochain complex of invariants $(\sEdbl(\Udbl),\diffonnocrossdiff)^{\involution_E}$.
\end{lemma}
\begin{proof}
First, we show that there is an isomorphism of $\ZZ$-graded topological vector spaces.
To this end, let us state an elementary fact that differs from the case at hand only cosmetically.
Define $\widehat\cinfty(\RR_{\geq 0})$ to be the space of functions on $\RR_{\geq 0}$ whose odd derivatives vanish at $t=0$.
This space acquires a topology as a closed subset of $\cinfty(\RR_{\geq 0})$.
Let $\cinfty_{ev}(\RR)$ denote the closed subspace of $\cinfty(\RR)$ consisting of functions which are even under the interchange $x\mapsto -x$.
Then, there is a natural isomorphism of topological vector spaces
\begin{equation}
\widehat \cinfty(\RR_{\geq 0}) \cong \cinfty_{ev}(\RR).
\end{equation}
Given an even function on $\RR$, one simply restricts to $\RR_{\geq 0}$ and finds that the resulting function has the required vanishing of derivatives.
On the other hand, one may check that the natural even extension of a function in $\widehat \cinfty(\RR_{\geq 0})$ to $\RR$ is smooth at $t=0$.
There is a similar statement involving even derivatives at $t=0$ and odd functions on $\RR$.
As we have stated, this example differs from the desired statement only cosmetically.
Hence, we have constructed an isomorphism of $\ZZ$-graded topological vector spaces as in the statement of the lemma.

We need further to show that this isomorphism respects the differentials on both complexes, but this follows directly from the description of the isomorphism.
\end{proof}

\begin{remark}
In fact, it easy to show that $U\mapsto (\sEdbl(\Udbl), \diffonnocrossdiff)^{\involution_E}$ is a complex of sheaves on $M$, and the above Lemma establishes an isomorphism of complexes of sheaves on $M$.
\end{remark}

Lemma \ref{lem: ultraharmdbl}, together with Lemma \ref{prop: ultraharm}, is the main justification for studying the \ultrafieldsterm fields: they allow us to replace computations on $M$--a manifold with boundary--with $\involution$-invariant computations on $\Mdbl$. We will use this to produce the parametrices  we will use for renormalization in Section \ref{sec: renorm}.

Before we move on, we note that the doubling construction applies equally well to $A$ and $B$ (cf. \ref{def: amenable}), producing bundles $\Adbl$, $\Bdbl$ satisfying $\Edbl\cong \Adbl\oplus \Bdbl$.

\section{Regularized Heat Kernels and Propagators}
\label{sec: renorm}
In this section, we use the doubling trick to obtain regularized BV Laplacians and propagators suitable for studying bulk-boundary systems.
As we have remarked, the doubling trick allows one to replace $M$ with $\Mdbl$, where we will be able to leverage the full theory of elliptic complexes and elliptic operators.
We will then symmetrize these constructions with respect to $\involution$ and restrict to the non-negative half-manifold to obtain heat kernels and propagators relevant to our boundary value problem.

In this section, let us reiterate that we have fixed a tubular neighborhood $\tubnhd$ of $\bdyM$ in $M$ and a diffeomorphism $\tubnhd\cong \bdyM\times[0,\epsilon)$, and we require $\sE\cong \sEb\hotimes_\beta \Omega^\bullet_{[0,\epsilon)}$ over $\tubnhd$.

\subsection{Gauge fixings and the doubling trick}
\begin{definition}
Let $\sE$ be a TNBFT. A \textbf{gauge-fixing} for $\sE$ is a degree --1 operator 
\index[notation]{QGF@$Q^{\GF}$}
$Q^{\GF}$ such that 
\begin{enumerate}
\item $[\diffonnocrossdiff,Q^{\GF}]$ is a generalized Laplacian in the sense of Definition 2.2 \autocite{bgv}.
\item $(Q^{\GF})^2=0$.
\item $Q^{\GF}$ is self-adjoint for the pairing $\ip$ when acting on fields which are zero over $\bdyM$. 
\end{enumerate}
\end{definition}

\begin{definition}
\label{def: amenablegf}
A gauge-fixing is \textbf{amenable to the doubling trick} if, over $\tubnhd$, $Q^{\GF}$ is of the form $\Qbgf+\iota_{\partial/\partial t}\frac{\partial}{\partial t}$, where $[\Qbgf,\diffonnocrossdiffbdy]$ is a generalized Laplacian on $\bdyM$. We require that $\Qbgf$ preserve $\sL$ and $\sL'$, and that $Q^{\GF}$ preserve $\sA,\sB$. 
\end{definition}

Note that it follows from the definition that the metric on $M$ determined by the generalized Laplacian $[Q,Q^{GF}]$ is cylindrical near the boundary, i.e. is of the form $g= g_\del +\d t^2$, when $Q^{GF}$ is amenable to doubling. 

The following lemma is proved by the same straightforward computations as Lemma \ref{lem: qandipdouble}.
\begin{lemma}
\label{lem: qgfextends}
If $Q^{\GF}$ is amenable to the doubling trick, then $Q^{\GF}$ extends to an operator on $\sEdbl$ and $Q^{\GF}$ commutes with $\involution_E$. 
\end{lemma}

Let us recall a major motivation for the introduction of the operator $Q^{GF}$ in the case that $M$ is closed.
In this case, if $(\sE,Q)$ is an elliptic complex, one may choose metrics on $E$ and $M$, and form the formal adjoint $Q^*$ to $Q$.
Then, the commutator $[Q,Q^*]$ is an elliptic operator on $\sE$, and one has an identification $H^\bullet(\sE,Q)\cong \ker [Q,Q^*]$.
This method for studying the cohomology of elliptic complexes on closed manifolds is known as formal Hodge theory \autocite{wellshodge}.
In other words, we have reduced the problem of computing the cohomology $H^\bullet(\sE,Q)$ from that of the computation of a subquotient to that of identifying a particular subspace of $\sE$.
In any case, the operator $[Q,Q^*]$ contains all the information of the cohomology of $\sE$, and further, one may apply the well-developed machinery of elliptic operators to it.

Now, let us return to the situation in which $\bdyM\neq \emptyset$, and let $(\sE,\sL,Q^{\GF})$ be amenable to the doubling trick.
We are interested in the cohomology of the complex $\condnocrossdiff$.
We have seen that, to this end, we may replace $\condnocrossdiff$ with $\ultrafields\,$.
In turn, $(\ultrafields\, ,\diffonnocrossdiff)$ may be identified with the complex $(\sEdbl, \diffonnocrossdiff)^{\involution_E}$.
When $M$ is compact, and under certain additional assumptions, we saw in the previous paragraph that the kernel of the operator $[Q,Q^{GF}]$ (acting on $\sEdbl$) computes the cohomology of $(\sEdbl, \diffonnocrossdiff)$;
the next lemma shows that an analogous statement is true for $\involution$ invariants as well.

\begin{lemma}
\label{lem: justifyingheatkernel}
Let $(\sE,\sL,Q^{\GF})$ be a theory and a gauge-fixing which are amenable to the doubling trick, and suppose that $M$ is compact. Assume further that $Q^{\GF}$ is the formal adjoint of $\diffonnocrossdiff$ with respect to some metric on $\sEdbl$ (for which $\involution$ is an isometry) and the metric on $\Mdbl$ induced from $[\diffonnocrossdiff,Q^{\GF}]$. Then, 
\begin{equation}
H^\bullet \condnocrossdiff(M)\cong (H^\bullet \sEdbl(\Mdbl))^{\involution_E}\cong (\ker [\diffonnocrossdiff,Q^{\GF}])^{\involution_E},
\end{equation}
so that the $\involution_E$-invariant elements of the kernel of $[\diffonnocrossdiff,Q^{\GF}]$ compute the cohomology of $\condnocrossdiff(M)$.
\end{lemma}

\begin{proof}
The first isomorphism arises from Proposition \ref{prop: ultraharm}, Lemma \ref{lem: ultraharmdbl}, and the fact that taking $\involution$-invariants is an exact functor. For the second isomorphism, note that Hodge theory gives an isomorphism
\begin{equation}
H^\bullet \sEdbl(\Mdbl)\cong \ker [\diffonnocrossdiff,Q^{\GF}].
\end{equation}
Let $[e] \in H^\bullet (\sEdbl(\Mdbl)^{\involution_E})\cong H^\bullet(\sEdbl(\Mdbl)^{\involution_E})$, i.e. $\diffonnocrossdiff e=0$ and $\involution_E e = e$. Let $e_{harm}$ be a harmonic representative of $[e]$, i.e. $e=e_{harm}+\diffonnocrossdiff e'$. 
Here, $e_{harm}$ and $\diffonnocrossdiff e'$ are unique.
Now,
\begin{align}
\involution_E e &= \involution_E e_{harm}+\involution_E \diffonnocrossdiff e'\\
&= \involution_E e_{harm}+ \diffonnocrossdiff\involution_E e'\\
& = e_{harm}+\diffonnocrossdiff e',
\end{align}
where in the second equality we have used the fact that $\involution$ and $\diffonnocrossdiff$ commute and in the third equality we have used $\involution e = e$.
The section $\involution e_{harm}$ is still harmonic because $\involution$ commutes with $\diffonnocrossdiff$ and $Q^{\GF}$. 
It follows that $e_{harm}=\involution e_{harm}$; 
hence, the map $[e]\mapsto e_{harm}$ takes $\involution$-invariants to $\involution$-invariants. 
The inverse map $\ker [\diffonnocrossdiff,Q^{\GF}]\to H^\bullet (\sEdbl(\Mdbl))$ manifestly sends invariants to invariants, so that we have established the second isomorphism of the Lemma.
\end{proof}

\subsection{Topological Mechanics: an Extended Example}
\label{subsec: topmechdbl}
Consider topological mechanics (Example \ref{ex: toplmech}) on $M=[0,1]$. For this theory, $\sE=\Omega^\bullet_{[0,1]}\otimes V$, where $V$ is a symplectic vector space. 
Choose two Lagrangians $L_1,L_2\subset V$ and vector space complements $L'_1,L'_2$. 
Since $\sEb=V\oplus (-V)$, $L:=L_1\oplus L_2$ is a boundary condition, and $L'_1\oplus L'_2$ is a complement for $L_1\oplus L_2$.

Because $\Qb = 0$ manifestly preserves $L$ and $L'$, we may verify that $(\sE,\sL)$ is amenable to doubling by taking $A=0$, $B=E$ (cf. Example \ref{ex: nocrossdifffromstart}).
It follows that $(\sE,\diffonnocrossdiff)=(\condfields,\diff)$.
Moreover, it is straightforward to compute that $H^0 \condfields(M)\cong L_1\cap L_2$ and $H^1 \condfields(M) \cong V/(L_1+L_2)$.
(Note that $H^\bullet\condfields(M)$ has a canonical $(-1)$-shifted symplectic pairing.)

Now, use the standard Euclidean metric to define the degree --1 operator $\delta = i_{\partial /\partial t} \frac{d}{dt}$. 
This operator is manifestly amenable to doubling.
The operator $H:=d_{dR}\delta+\delta d_{dR}$ is identically the operator $d^2/dt^2$ on~$\sE$.

Let us turn to the study of $\Mdbl$ and $\Edbl$.
$\Mdbl$ is easily seen to be a circle of circumference 2. We write $\Mdbl = \RR/2\ZZ$.
Then $\involution$ can be identified with the map induced from the orientation-reversing automorphism $t\mapsto -t$ of $\RR$.
Let $\mathrm{M\ddot o b}\to S^1$ denote the M\"obius bundle on $S^1$.
One may check that
\begin{align}
    \label{eq: edbltopmech}
\Edbl&\cong (\Lambda^\bullet(T^*S^1))\otimes \mathrm{M\ddot{o}b}\otimes (L'_1\cap L_2\oplus L_1\cap L'_2) \\&\oplus (\Lambda^\bullet(T^*S^1))\otimes (L_1\cap L_2\oplus L'_1\cap L'_2). \nonumber
\end{align}
The differential $\diffonnocrossdiff$ on the first summand of $\sEdbl$ is the natural one induced on $\Omega^\bullet_{S^1}(\Mob)$ from the flat connection on the M\"obius bundle; on the second summand, $\diffonnocrossdiff$ is simply the de Rham differential.
The formula $i_{\del/\del t}\frac{d}{dt}$ is well-defined on $\Omega^\bullet_{S^1}(\Mob)$ (as well as on $\Omega^\bullet_{S^1}$), and coincides with the extension of $Q^{GF}$ from $E$ to $\Edbl$.
The operator $[\diffonnocrossdiff,Q^{GF}]$ can be written simply $\frac{d^2}{dt^2}$,
i.e. there is no ambiguity in writing this operator even on sections of the M\"obius bundle.

On sections in $\sEdbl$, $\involution_E$ acts as follows.
Let us identify the space of global sections of $\Mob$ with the space of smooth functions $f$ on $\RR$ satisfying
\begin{equation}
f(t)=-f(t+2)
\end{equation}
for all $t$.
Consider the involution $\sigma_0$ of $\cinfty(S^1, \Mob)$ given by $f(t)\mapsto f(2-t)$.
The involution $\sigma_0$ preserves the flat connection on $\Mob$; hence it extends to an involution of $\Omega^\bullet(S^1;\Mob)$.
Let $\involution^*$ denote the involution of $\Omega^\bullet(S^1)$ induced by pullback along $\involution:S^1 \to S^1$.
Then, using the ordering of the summands in $\Edbl$ as in Equation \eqref{eq: edbltopmech}, we find that $\involution_E$ acts on $\sEdbl$ as
\begin{equation}
-\sigma_0 \oplus \sigma_0 \oplus \involution^* \oplus -\involution^*
\end{equation}

It follows directly from our discussions that the kernel of $[\diffonnocrossdiff,Q^{GF}]$ is 
\begin{equation}
1\otimes \RR[dt] \otimes (L_1\cap L_2\oplus L'_1\cap L'_2).
\end{equation}
Finally, the $\involution_E$-invariants in this space are
\begin{equation}
1\otimes L_1\cap L_2 \oplus dt\otimes L'_1\cap L'_2.
\end{equation}
Using the identification $L'_1\cap L'_2\cong V/(L_1+L_2)$, we find that our answer for 
\begin{equation}
(\ker [\diffonnocrossdiff,Q^{GF}])^{\involution_E}
\end{equation}
agrees with the computation of $H^\bullet(\condfields(M),\diff)$ we performed earlier.
\subsection{Finding a parametrix for the Laplacian}
\label{sec: parametrices}
Having established the utility of the class of gauge fixings which are amenable to the doubling trick, let us turn to a discussion of parametrices, which will be an essential tool for the construction of quantum bulk-boundary systems.
Throughout this subsection, we will assume $(\sE,\sL)$ is amenable to doubling, and $Q^{\GF}$ is a gauge fixing which is amenable to doubling.
We will also liberally use the identification $\Edbl\cong \Edbl^![-1]$ afforded to us by the existence of $\ip_{loc,DBL}$.
Finally, we will use simply the symbol $\involution$ in place of $\involution_E$, since there will be no need to use the involution of $\Mdbl$ in this section.

\begin{definition}
\label{def: param}
A \textbf{$\involution$-invariant parametrix} for the operator $H:=[\diffonnocrossdiff,Q^{\GF}]$ is a distributional section 
\index[notation]{Phi@$\Phi$}
\begin{equation}
\Phi \in \overline{\sEdbl}\widehat{\otimes}_\pi \overline{\sEdbl}
\end{equation}
of the bundle $\Edbl\boxtimes\Edbl$ over $\Mdbl\times \Mdbl$ with the following properties:
\begin{enumerate}
\item $\Phi$ is symmetric under the natural $\Z/2$ action exchanging the two factors of the tensor product.
\item $\Phi$ is of cohomological degree $+1$.
\item $\Phi$ has proper support, i.e. the projection maps $\supp(\Phi)\subset M^2\to M$ are both proper.
\item Let $T_\Phi$ denote the induced operator $\sEdblcs\to \overline{\sEdbl}$. We require that the operator
\begin{equation}
HT_{\Phi} - (\involution+1)
\end{equation}
have a representation as convolution with a smooth section of $E\boxtimes E$.
\item The relation
\begin{equation}
(1\otimes \involution) \Phi = \Phi
\end{equation}
holds, where $\involution$ is the involution on distributional sections induced in the usual way from the action of $\involution$ on smooth sections. (Together with condition (1), this condition also implies $(\involution\otimes 1)\Phi=\Phi$.)
\item With respect to the decomposition $\Edbl\cong \Adbl\oplus \Bdbl$, we require that $\Phi$ have no component in $\Adbl\boxtimes \Adbl$. 
\end{enumerate}
A \textbf{non-invariant parametrix} (see Definition 8.2.4.2 of \autocite{CG2}) satisfies conditions (1)-(3) of the above definition, along with condition (4) with the identity operator in place of $(\involution+1)$.
\end{definition}

\begin{remark}
In the sequel, we will often use the terms ``parametrix'' and ``$\involution$-invariant parametrix'' interchangeably. If we mean specifically a non-invariant parametrix, we will always use that term.
\end{remark}

\begin{remark}
Condition (6) in Definition \ref{def: param} is a technical condition that will allow us to treat the term in the action induced from $\pertdiff$ as an interaction term. 
\end{remark}
\begin{lemma}
\label{lem: noninvttoinvtparam}
If $\Psi$ is a non-invariant parametrix which satisfies condition (6) in Definition \ref{def: param}, then 
\begin{equation}
\Psi^\involution:=\frac{1}{2}(1\otimes 1+1\otimes \involution + \involution\otimes 1+\involution\otimes \involution)\Psi
\end{equation}
is a $\involution$-invariant parametrix.
\end{lemma}

\begin{proof}
Conditions (1)-(3) and (5) are immediate. For condition (4), we note that $H$ commutes with $\involution$, and $\ultrafields(M)\cong \sEdbl(\Mdbl)^\involution$. Let us write $T_\Psi$ for the operator induced from $\Psi$. Then, given $e\in \sEdbl(\Mdbl)$,
\begin{align}
HT_{\Psi^\involution} e-(e+\involution e)&=\frac{1}{2}\left(HT_\Psi e+HT_\Psi\involution e+\involution HT_\Psi e+\involution H T_\Psi \involution e\right) - e+\involution e\\
&\sim \frac{1}{2}\left( e+\involution e+\involution e+e\right)-(e+\involution e)=0.
\end{align}
where the symbol ``$\sim$'' means ``differs only by smoothing operators from''. Finally, we note that since $\involution$ preserves the decomposition $\Edbl\cong \Adbl\oplus \Bdbl$, $\Psi^\involution$ satisfies condition (6) if $\Psi$ does. 
\end{proof}

\begin{ucorollary}
$\involution$-invariant parametrices exist.
\end{ucorollary}

\begin{proof}
It is a standard fact of elliptic theory (cf. Theorem IV.4.4 of \autocite{wellshodge}) that non-invariant parametrices exist. 
We need to show that a non-invariant parametrix satisfying condition (6) in Definition \ref{def: param} exists. 
To see this, note that  $H=[\diffonnocrossdiff,Q^{\GF}]$ respects the decomposition $\sEdbl\cong \sAdbl\oplus \sBdbl$. 
We can therefore find non-invariant parametrices $\Phi_A$, $\Phi_B$ for $H$ on $\sAdbl$ and $\sBdbl$ separately. 
The standard theory presents $\Phi_A$ as a distributional section of $A\boxtimes A^!$, and similarly for $\Phi_B$.  
We would like, however, to use the pairing $\ip_{loc}$ to identify $E$ with $E^!$, and therefore $\Phi_A$ and $\Phi_B$ with distributional sections of $E\boxtimes E$. 
Under this identification, $A^!$ is identified with a subbundle of $B$ (since $A$ pairs trivially with itself). 
Hence, none of $\Phi_A,\Phi_B$ has a non-zero component in $A\boxtimes A$.
\end{proof}

\begin{example}
\label{ex: toplmechprop}
Let us consider topological mechanics again, but with $M=[0,\infty)$. We let $V$ be a symplectic vector space, $L$ a Lagrangian in $V$, and $L'$ a Lagrangian complement to $L$. In this example, $A=0$, and $B=\Lambda^\bullet(T^*M)\otimes V$. Then, $\Mdbl=\R$, $\sEdbl=\Omega^\bullet_\R\otimes V$, $\involution= \tau^*\otimes \Gamma$. If we set $Q^{\GF}=i_{\partial/\partial t}\frac{\partial}{\partial t}$, then $H=\frac{d^2}{dt^2}$. A non-invariant parametrix for $H$ is given by
\begin{equation}
\Phi_0:=\frac{1}{2}F(t,s)|t-s|(\d t-\d s) \otimes \omega^{-1},
\end{equation}
where $F$ is a symmetric function of $s$ and $t$ which has proper support and is 1 in a neighborhood of the diagonal in $\Mdbl\times\Mdbl\cong \R^2$. We also require that $F(t,s)=F(-t,-s)$. 
For example, we may choose a compactly-supported, even function $f_0$ on $\RR$ which is identically 1 in a neighborhood of 0, and let $F(t,s)=f_0(t-s)$

We compute
\begin{align}
(\involution\otimes 1)\Phi_0 &= \frac{-1}{2}F(-t,s)|t+s|\left( (\d s+\d t)\otimes ((\Gamma\otimes 1)\omega^{-1})\right)\\
(1\otimes \involution)\Phi_0 &= \frac{1}{2}F(t,-s)|t+s|\left( (\d t+ \d s )\otimes (( 1\otimes \Gamma)\omega^{-1})\right)\\
(\involution\otimes \involution)\Phi_0 &= -\frac{1}{2}F(t,s)|t-s|\left((\d t - \d s)\otimes ((\Gamma\otimes \Gamma)\omega^{-1})\right).\\
\end{align}
Note that 
\begin{equation}
\label{eq: omegadecomp}
\omega^{-1}\in L\otimes L' \oplus L'\otimes L,
\end{equation} so that $(\Gamma\otimes \Gamma)\omega^{-1}=-\omega^{-1}$, and $(\involution\otimes \involution)\Phi_0=\Phi_0$. Similarly, one finds that $(\involution\otimes 1)\Phi_0=(1\otimes \involution)\Phi_0$. If we write $\omega^{-1}=\omega^{-1}_{LL'}+\omega^{-1}_{L'L}$ to denote the decomposition of $\omega^{-1}$ given by Equation \ref{eq: omegadecomp}, then we find that 
\begin{align}
\label{eq: full1dprop}
\Phi_0^\involution=& \frac{1}{2}F(t,s)|t-s|\left(	(\d t-\d s)\otimes \omega^{-1} \right)\nonumber\\
&-\frac{1}{2}F(-t,s)|t+s|\left ((\d t+\d s)\otimes (\omega^{-1}_{LL'}-\omega^{-1}_{L'L})\right)
\end{align}
is a $\involution$-invariant parametrix.
\end{example}

The following is  a slight modification of Lemma 8.2.4.3 of \autocite{CG2}:
\begin{lemma}
\label{lem: propsofparametrix}
$\involution$-invariant parametrices possess the following properties:
\begin{enumerate}
\item The difference of two $\involution$-invariant parametrices $\Phi,\Psi$ is a smooth section of $\Edbl\boxtimes\Edbl$ over $\Mdbl\times \Mdbl$. 
\item A $\involution$-invariant parametrix $\Phi$ is smooth away from the diagonal and the anti-diagonal in $\Mdbl\times \Mdbl$, i.e. $\Phi$ is smooth away from the locus of points of the form $(x,x)$ and $(x,\involution x)$. 
\item Given a $\involution$-invariant parametrix $\Phi$, $(\diffonnocrossdiff\otimes 1+1\otimes \diffonnocrossdiff)\Phi$ is a smooth section of $\Edbl\boxtimes\Edbl$ over $\Mdbl\times \Mdbl$.
\end{enumerate}
\end{lemma}
\begin{proof}
\begin{enumerate}
\item $(H\otimes 1+1\otimes H)(\Phi-\Psi)$ is smooth, and $H\otimes 1+1\otimes H$ is a generalized Laplacian on $\sE\otimes \sE$, so is elliptic. $\Phi-\Psi$ is therefore smooth by elliptic regularity.
\item $(H\otimes 1+1\otimes H)\Phi$ is smooth away from this locus, so the result follows by elliptic regularity.
\item The proof of this statement is nearly identical to the corresponding proof in \autocite{CG2}, except we use also that $\diffonnocrossdiff$ commutes with $\involution$. 
\end{enumerate}
\end{proof}
Now, we define 
\index[notation]{PPhi@$P(\Phi)$}
\begin{equation}
\label{eq: propagatordef}
P(\Phi):=\frac{1}{2}(Q^{\GF}\otimes 1+1\otimes Q^{\GF})\Phi \in \overline{\sEdbl}(\Mdbl)\hotimes_\pi\overline{\sEdbl}(\Mdbl),
\end{equation}

and
\index[notation]{KPhi@$K_\Phi$}
\begin{equation}
\label{eq: heatkerneldef}
K_\Phi:=K_{(\involution+1)}-(\diffonnocrossdiff\otimes 1+1\otimes \diffonnocrossdiff)P(\Phi)\in \overline{\sEdbl}(\Mdbl)\hotimes_\pi \overline{\sEdbl}(\Mdbl).
\end{equation}

\begin{definition}
The distribution $P(\Phi)$ defined in Equation \eqref{eq: propagatordef} is the \textbf{propagator for $\Phi$}.
The distribution $K_\Phi$ defined in Equation \eqref{eq: heatkerneldef} is the \textbf{BV heat kernel for $\Phi$}.
\end{definition}

Now, we show that the BV heat kernel for $\Phi$ is actually a smooth section of $\Edbl\boxtimes\Edbl$ over $\Mdbl\times \Mdbl$.
Note that  
\begin{align}
(\diffonnocrossdiff\otimes 1+1\otimes \diffonnocrossdiff)P(\Phi)&=\frac{1}{2}(\diffonnocrossdiff\otimes 1+1\otimes\diffonnocrossdiff)(Q^{\GF}\otimes 1+1\otimes Q^{\GF})\Phi\\
&=\frac{1}{2}(H\otimes 1+1\otimes H)\Phi\\
&-\frac{1}{2}\left( Q^{\GF}\diffonnocrossdiff\otimes 1+1\otimes Q^{\GF}\diffonnocrossdiff-\diffonnocrossdiff\otimes Q^{\GF}+Q^{\GF}\otimes \diffonnocrossdiff\right)\Phi\\
&=\frac{1}{2}(H\otimes 1+1\otimes H)\Phi \\
&-\frac{1}{2}\left( (Q^{\GF}\otimes1)(\diffonnocrossdiff\otimes 1+1\otimes \diffonnocrossdiff)+(1\otimes Q^{\GF})(\diffonnocrossdiff\otimes 1+1\otimes \diffonnocrossdiff)\right)\Phi\nonumber\\
&= K_{(\involution+1)}+\text{smooth kernels},
\end{align}
where the last equality comes from Item 3 in Lemma \ref{lem: propsofparametrix}.
As a consequence, we have the first statement of the following lemma:

\begin{lemma}
\label{lem: propsofkernels}
The following facts about the BV heat kernel and the propagator hold:
\begin{enumerate}
    \item The kernel $K_\Phi$ is a smooth section of $\Edbl\boxtimes \Edbl$ over $\Mdbl\times \Mdbl$.
    \item One has the equation
    \begin{equation}K_\Phi-K_\Psi=(\diffonnocrossdiff\otimes 1+1\otimes \diffonnocrossdiff)(P(\Psi)-P(\Phi)),\end{equation}
    relating the BV heat kernel to the propagator.
    \item One has the identity
    \begin{equation}
        \label{eq: bvkernelclosed}
        (\diffonnocrossdiff\otimes 1+1\otimes\diffonnocrossdiff )K_\Phi =0
    \end{equation}
    concerning the $\diffonnocrossdiff$-closedness of the BV heat kernel.
    \item As smooth sections of $\Edbl\times \Edbl$ over $\Mdbl\times \Mdbl$, both $K_\Phi$ and $P(\Phi)-P(\Psi)$ are separately $\involution$-invariant in both arguments.
    In particular, upon restriction to the non-negative half-manifold, the kernels $(P(\Phi)-P(\Psi)), K_\Phi$ can be seen as elements of the tensor square of the (underlying graded) space of fields $\condfields$.
\end{enumerate}
\end{lemma}
\begin{proof}
The first statement has already been proved. The second statement is immediate from the definition of $K_\Phi$, and the third statement follows from the fact that $(\diffonnocrossdiff\otimes 1+1\otimes \diffonnocrossdiff)K_{(\involution+1)/2}=0$, which we have already used. For the final statement, note that
\begin{align}
(\involution\otimes 1)(P(\Phi)-P(\Psi))=(Q^{\GF}\otimes 1+1\otimes Q^{\GF})(\involution\otimes 1)(\Phi-\Psi)&=P(\Phi)-P(\Psi)\\
(\involution\otimes 1)K_\Phi = K_{(\involution^2+\involution)}-(\involution\otimes 1)(\diffonnocrossdiff\otimes 1+1\otimes \diffonnocrossdiff)P(\Phi)&=K_\Phi,
\end{align}
where we have made ample use of the fact that $\involution$ commutes with $\diffonnocrossdiff$ and $Q^{\GF}$; the corresponding statements for $(1\otimes \involution)$ are similar.
\end{proof}

Lemma \ref{lem: propsofkernels} specifies a number of very important properties of the propagators $P(\Phi)$ and heat kernels $K_\Phi$, properties which we will use repeatedly in the sequel.

\subsection{Parametrices in Topological Mechanics}
In this subsection, we demonstrate the above-described properties of parametrices, propagators, and heat kernels for the case of topological mechanics.

Let us start with the parametrix $\Phi^\involution_0$ described in Equation \eqref{eq: full1dprop}.
We find that
\begin{align}
\label{eq: propfortopmech3}
P(\Phi^\involution_0) &= \frac{1}{2}\left(\left(\frac{\partial F}{\partial t}(t,s)-\frac{\partial F}{\partial s}(t-s)\right)|t-s| + F(t,s)\sgn(t-s)\right) \otimes \omega^{-1}\nonumber\\
-&\frac{1}{2} \left( \left( -\frac{\partial F}{\partial t}(-t,s)+\frac{\partial F}{\partial s}(-t,s)\right)|t+s|+F(-t,s)\sgn(t+s)\right)\otimes(\omega^{-1}_{LL'}-\omega^{-1}_{L'L}).
\end{align}
Let
\begin{equation}
G(t,s) = |t-s|\left(\frac{\partial F}{\partial t}-\frac{\partial F}{\partial s}\right) (t,s).
\end{equation}
Note that, because $F$ is identically one in a neighborhood of the diagonal in $\RR^2$, all its derivatives are zero near the diagonal. Hence, $G$ is zero in a neighborhood of the diagonal, and therefore smooth on $\RR^2$.
Equation \eqref{eq: propfortopmech3} then can be written 
\begin{align}
P(\Phi^\involution_0)   &= \frac{1}{2}\left( G(t,s)+F(t,s)\sgn(t-s)\right)\otimes \omega^{-1}\nonumber\\
-&\frac{1}{2}\left(-G(-t,s)+F(-t,s) \sgn(t+s) \right) \otimes(\omega^{-1}_{LL'}-\omega^{-1}_{L'L}).
\end{align}
Let 
\begin{equation}
\widetilde G(t,s) = -G(-t,s), \quad \widetilde F(t,s) = F(-t,s)
\end{equation}
It follows that 
\begin{align}
 (\diffonnocrossdiff\otimes 1+ 1\otimes \diffonnocrossdiff) P(\Phi^\involution_0) &=\frac{1}{2}\left(\d G+\d F \sgn(t-s)+ 2FK_{id}\right)\otimes \omega^{-1}\nonumber\\
 &-\frac{1}{2}(\d\widetilde G + \d\widetilde F \sgn(t+s)- 2 \widetilde F K_{\tau^*})\otimes (\omega^{-1}_{LL'}-\omega^{-1}_{L'L}).
\end{align}
Here, $\tau$ is, as above, the involution of $\RR$ taking $t$ to $-t$.
As before, we find that $\d F\sgn(t-s)$ and $\d \widetilde F \sgn(t+s)$ are smooth.
Moreover, because $F$ is one on a neighborhood of the diagonal and $K_{id}$ has support on the diagonal $FK_{id}=K_{id}$, and similarly, $\widetilde F K_{\tau^*}=K_{\tau^*}$.
Hence, we find that
\begin{align}
K_{\Phi^\involution_\sigma}&=\frac{1}{2}(\d G+\d F\sgn(t-s))\otimes \omega^{-1}\nonumber\\
&-\frac{1}{2}(\d\widetilde G + \d\widetilde F \sgn(t+s))\otimes(\omega^{-1}_{LL'}-\omega^{-1}_{L'L})
\end{align}
is indeed smooth.
It is straightforward to verify, using the fact that $\d F$ is zero near the diagonal and $\d \widetilde F$ is zero near the anti-diagonal, that $K_{\Phi^\involution_0}$ is closed for $\diffonnocrossdiff$.

Finally, let us verify that $K_{\Phi^\involution_0}$ is $\involution$-invariant in both arguments.
Note that 
\begin{equation}
\widetilde G=-(\tau^*\otimes 1) G=(1\otimes \tau^*)G
\end{equation}
and
\begin{equation}
\widetilde F = (\tau^*\otimes 1) F =(1\otimes \tau^*)F.
\end{equation}
We will also use the fact that $d$--the de Rham differential on $\RR^2$--commutes individually with $\tau^*$ acting in the $t$ and $s$ variables.
Hence,
\begin{align}
(\involution\otimes 1) K_{\Phi^\involution_0} &= \frac{1}{2}( - \d \widetilde G - \d \widetilde F \sgn(t+s))\otimes (\omega^{-1}_{LL'}-\omega^{-1}_{L'L})\\
&+\frac{1}{2}( \d G + \d F \sgn(t-s))\otimes\omega^{-1}=K_{\Phi^\involution_0}. 
\end{align}
The check for $1\otimes \involution$ is similar.

\section{Definition of Quantum Bulk-Boundary Systems}
\label{sec: qbbsdef}
In this section, we define quantum bulk-boundary systems. As mentioned in the Introduction of this Chapter and in the Introduction of the dissertation, we need to address three aspects: the RG equation (Section \ref{subsec: RG}, the QME (Section \ref{subsec: QME}), and asymptotic locality (Section \ref{subsec: def}).
First, however, we discuss more precisely the sorts of interaction functionals we will use in our definition.

\subsection{Spaces of Functionals}
In this section, we introduce the various classes of functionals we will consider in our definition of a bulk-boundary system.
Since we have primarily used the language of local $L_\infty$ algebras to describe classical field theories, let us comment briefly on the switch from $L_\infty$ algebras to local functionals.
Recall that a classical TNBFT comes equipped with the information of a collection $\{\diff, \ell_2,\cdots\}$ of brackets, and with a pairing 
\begin{equation}
\ip: (\condfieldscs(M))^{\hotimes_\beta 2} \to \RR.
\end{equation}
For a classical bulk-boundary system, we may therefore form the action functional
\begin{align}
    S(\varphi) &= \frac{1}{2}\ip[\varphi, \diff \varphi] +\sum_{k=2}^\infty \frac{1}{(k+1)!}\ip[ \varphi, \ell_k \varphi]\nonumber\\
    &=\frac{1}{2}\ip[\varphi, \diffonnocrossdiff \varphi]+\frac{1}{2}\ip[\varphi,\pertdiff\varphi] +\sum_{k=2}^\infty \frac{1}{(k+1)!}\ip[ \varphi, \ell_k \varphi],
\end{align}
which in mathematical terms is a continuous (non-linear) function 
\begin{equation}
\condfieldscs \to \RR.
\end{equation}

We make the definitions
\begin{align}
    S_{\mathrm{free}}(\varphi) &:= \frac{1}{2}\ip[\varphi,\diffonnocrossdiff\varphi]\\
    \label{eq: interactiondef}
    I(\varphi) &:= \frac{1}{2}\ip[\varphi,\pertdiff\varphi] +\sum_{k=2}^\infty \frac{1}{(k+1)!}\ip[ \varphi, \ell_k \varphi].
\end{align}
We will call $S_{\mathrm{free}}$ the \textbf{free part of the action} and $I$ \textbf{the interaction}.
We will define a quantum bulk-boundary system (with underlying free part $S_{\mathrm{free}}$) to be a collection of interaction functionals $I[\Phi]$, one for each parametrix $\Phi$, satisfying some properties we have already foreshadowed.
The $I[\Phi]$ are to be understood as the (ill-defined) result of the integration of $e^{I/\hbar}$ over all fields of ``energy'' higher than $\Phi$.
In this section, we will describe a space of functionals which we would expect to contain the $I[\Phi]$ if the procedure of integrating out the high-energy modes were well-defined.

In the sequel, it will be useful to us to use the language of bornological and convenient vector spaces.
We refer the reader to Appendix B of \autocite{CG1} for detailed definitions and properties of the categories $\mathrm{BVS}$ and $\mathrm{CVS}$.
We only note here that both categories are closed symmetric monoidal, with symmetric monoidal products $\otimes_\beta$ and $\hotimes_\beta$, respectively.
\begin{definition}
\label{def: functionaldefs}
We make the following definitions regarding functionals on the space $\condfields$.
\begin{itemize}
    \item By a \textbf{functional} on the \condfieldsterm fields, we mean an element $I$ belonging to the convenient vector space
\begin{equation}
    \label{eq: functionaldfn}
\sO(\condfieldscs):=\prod_{k=0}^\infty \underline{\Hom}_{CVS}((\condfieldscs(M))^{\hotimes_\beta k},\RR)_{S_k}/\RR.
\end{equation}

(Note that we quotient out by the space corresponding to $k=0$ in the product.)

\item Given a functional $I\in \sO(\condfieldscs)$, we let $I_k$ denote its component in the $k$-th factor of the product in Equation \ref{eq: functionaldfn}. 
We call $I_k$ the \textbf{order $k$ component of $I$}.
Similarly, if a functional $I$ has a non-zero component only in the $k$-th factor of the product, we say that $I$ has \textbf{order $k$}.
\item Recall that the $k$-fold completed bornological tensor product $(\condfieldscs(M))^{\hotimes_\beta k}$ can be identified with the space of compactly-supported sections of $E^{\boxtimes k}$ over $M^{k}$ (a manifold with corners) satisfying a certain boundary condition, cf. the Corollary to Theorem \ref{thm: tensorofdirichlet}.
By the \textbf{support} of an order $k$ functional $I_k$, we mean the complement of the union of all open sets $U\subset M^k$ such that $I(e)=0$ whenever $e$ has support on a compact subset of $U$. 
\item  We say that $I$ has \textbf{proper support} if, for each $k$, the maps $\supp(I_k)\subset M^k\to M$ induced from the projections $M^k\to M$ are all proper. We denote by 
\index[notation]{Oprop@$\ps$}
$\ps$ the space of all functionals with proper support.
\end{itemize}

\end{definition}

\begin{remark}
The methods of the appendix---Theorem \ref{prop: dualsofcondfields}, in particular---provide a splitting $\sO(\condfields)\hookrightarrow \sO(\sE)$ of the natural restriction map $\sO(\sE)\to \sO(\condfields)$. 
One can show that the support of a functional $I\in \sO(\condfields)$ coincides with the support of its image under this splitting.
\end{remark}

Given any functional $I$, we can view (via the closed symmetric monoidal structure of $\mathrm{CVS}$) each Taylor component $I_k$ as a multilinear map
\begin{equation}
\label{eq: firstderiv}
\condfieldscs^{\hotimes_\beta (k-1)}\to \underline{\Hom}_{CVS}(\condfieldscs, \RR).
\end{equation}
Functionals with proper support are precisely those functionals which arise as the composition of the natural map
\begin{equation}
\condfieldscs^{\hotimes_\beta (k-1)} \to \condfields^{\hotimes_\beta (k-1)}
\end{equation}
with a map
\begin{equation}
\condfields^{\hotimes_\beta (k-1)}\to \underline{\Hom}_{CVS}(\condfieldscs,\RR).
\end{equation}

In \autocite{cost}, Costello considers additionally functionals with \emph{smooth first derivative}, which is to say functionals whose corresponding map as in Equation \ref{eq: firstderiv} has image in $\sE^!$ (i.e. in the space of smooth distributions.) 
He then defines a field theory to be a collection of functionals \emph{with smooth first derivative} $\{I[\Phi]\}$, one for each parametrix $\Phi$.
The functionals are required to satisfy, among other things, the requirement (P1) that their support can be made to ``approach'' the diagonal by choosing $\Phi$ ``small enough.''
The functionals with smooth first derivative and support on the diagonal are precisely the local action functionals; hence, the requirement (P1), together with the requirement that the $I[\Phi]$ have smooth first derivative, constitutes an asymptotic locality condition on the~$I[\Phi]$.

The case at hand differs from the situation in \autocite{cost} in several respects. 
First, local functionals include also integrals over $\bdyM$ of Lagrangian densities.
Such functionals do not have smooth first derivative.

To understand how to address this difficulty, let us define the sheaves
\begin{align}
\delsmoothdistrE&:= \delsmoothdistr\otimes_{D_M} \sJ(\sE)^\vee\\
\conddelsmoothdistr &:= \left(\delsmoothdistr \otimes_{D_M}\sJ(\sE)^\vee\right)/\left( \iota_* \Omega^{n-1}_{\bdyM, tw}\otimes_{D_\bdyM} \sJ(\sL')^\vee\right) \label{eq: delsmoothdistr}
\end{align}
(all notation here is coincident with that of Section \ref{sec: localfcnls}).
We may identify the space of global sections of $\conddelsmoothdistr$ with a subspace of $\sO(\condfields)$.
Following the discussion of Section \ref{sec: localfcnls}, the sheaf $\conddelsmoothdistr$ can be identified with the space of local functionals on $\condfields$ which are linear in their input.
\begin{definition}
A global section $\digamma\in \conddelsmoothdistr(M)$ is called a \textbf{$\partial$-smooth distribution}. 
A functional $I$ of order $k$ has \textbf{$\partial$-smooth first derivative} if the corresponding map 
\begin{equation}
\label{eq: delsmoothrequirement}
(\condfieldscs)^{\hotimes_\beta (k-1)}\to \innerhom{\condfieldscs}{\RR}
\end{equation}
has image in $\conddelsmoothdistr$, the space of $\partial$-smooth distributions. 
A functional has \textbf{smooth first derivative} if the map has image in $\condfields^!\subset \innerhom{\condfieldscs}{\RR}$. 
We will denote by $\sO_{sm}(\condfieldscs)$
\index[notation]{Osm@$\sO_{sm}(\condfieldscs)$}
the space of functionals with $\partial$-smooth first derivative. 
We will denote by
\index[notation]{Opsm@$\sO_{P,sm}(\condfieldscs)$}
$\sO_{P,sm}(\condfieldscs)$ the space of functionals with $\partial$-smooth first derivative and proper support.
\end{definition}

Let us now show that the functionals with $\del$-smooth first derivative and support on the small diagonal $M\subset M^k$ are precisely the local action functionals.

\begin{lemma}
\label{lem: localfcnlssupportdiag}
The set of order $k$ functionals with $\del$-smooth first derivative and support on the small diagonal $M\subset M^k$ is precisely the set of order $k$ local action functionals.
\end{lemma}

\begin{proof}
It is straightforward to show that the local functionals have support on the small diagonal and have $\del$-smooth first derivative, using integration by parts (which is encoded in the $D_M$-module tensor product relations in Equation \eqref{eq: delsmoothdistr}).

Let us show the converse, namely that any order $k$ functional $J$ which has support on the small diagonal and $\del$-smooth first derivative is local.

Recall that the space of local functionals is actually a \emph{sheaf} on $M$;
similarly, the space of $\del$-smooth functionals is also a sheaf on $M$.
The space of $\del$-smooth functionals with support on the small diagonal is a subsheaf of this space.
There is an inclusion of sheaves
\begin{equation}
\label{eq: locintosmooth}
    \Oloc\to \sO_{sm}(\condfieldscs)
\end{equation}
The sheaf of local functionals also possesses a natural $\cinfty_M$-module structure induced from the action of $\cinfty_M$ on the left on $\delsmoothdistr$ and on $\iota_* \Omega^{n-1}_{\bdyM,tw}$.
The sheaf $\sO_{sm}(\condfieldscs)$ also possesses a $\cinfty_M$ structure, induced from the action of $\cinfty_M$ on the target in Equation \eqref{eq: delsmoothrequirement},
and the sheaf of $\del$-smooth functionals with support on the small diagonal is closed under this $\cinfty_M$ action.
Hence, all three sheaves under consideration are fine.
The map in Equation \eqref{eq: locintosmooth} is manifestly a map of $\cinfty_M$-modules.
It follows from this discussion that it suffices to check the statement of the Lemma locally on $M$.

We may therefore assume that $E$ is the trivial bundle and $M=\HH^n$.

Suppose that $I$ is a functional with support on the diagonal and $\del$-smooth first derivative.
Let us first note that $\condfieldscs^{\hotimes_\beta k}$ is a closed subspace of $\sE^{\hotimes_\beta k}$.
Hence, by the Hahn-Banach theorem, we may extend $I$ to a functional on the $\sE$ fields.
We will use the letter $I$ to denote this extension in the sequel.
Note that, given $x\in \HH^n$, $\dot x \in \RR^{n-1}$, and $m\in \NN$ there are continuous evalution maps
\begin{align}
    \mathrm{ev}_{x} &: \delsmoothdistr \to \RR\\
    \mathrm{ev}_{\dot x,m} &: \delsmoothdistr \to \RR;
\end{align}
they arise from the isomorphism $\delsmoothdistr\cong \Dens_{\HH^n}\oplus \Dens_{\RR^{n-1}}[\del_t]$ (cf. Equation \eqref{eq: noncanonicaliso1}).
We therefore obtain distributions
\begin{equation}
    I_{x},I_{\dot x,m}: \cinfty_c(\HH^{n(k-1)}) \to \RR.
\end{equation}
Because $I$ is supported on the diagonal, the distribution $I_x$ (respectively $I_{y,m}$) is supported at the point $(x,\ldots, x) \in \HH^{n(k-1)}$ (respectively $(y,\ldots, y)\in \HH^{n(k-1)}$).
Hence (cf. Theorem 24.6 of \autocite{treves}; to apply this theorem, we used the fact that $I$ was extended to the $\sE$ fields), $I_x$ (respectively, $I_{y,m}$) is a finite linear combination of derivatives at $(x,\ldots, x)$ (respectively, $(y,\ldots, y)$).
For fixed $\phi \in \cinfty_c(\HH^{n(k-1)})$, $I_x(\phi)$ must be smooth in $x$.
Hence, we conclude that, for $\psi\in \cinfty_c(\HH^n)$,
\begin{align}
    I(\psi\otimes\phi) &= \sum_{I\in \N^n}\int_{\HH^n}\psi(x) c^I(x) \del_I \phi(x,\ldots, x) \d^n x\nonumber\\
    &+\sum_{m\in \N, J\in \N^n}\int_{\RR^{n-1}}  d^{J,m}(y) \del_J \phi( y, \ldots,  y)\left.\frac{\del ^m }{(\del  x_n)^m}\right|_{x_n=0}(\psi(y,x_n)) \d^{n-1}y,
\end{align}
where each coefficient $c^I$ (respectively $d^{J,m}$) is smooth on $\HH^n$ (respectively, $\RR^{n-1}$).
Each summand in the above sums is a local functional, and we know that at each $x\in \HH^n$ (respectively, $y\in \RR^{n-1}$), only finitely many of the $c^I$ (respectively, $d^{J,m}$) are non-zero.
We know, also that, for every point $y\in \RR^{n-1}$ there exists a neighborhood of $y$ throughout which the $d^{J,m}$ are zero for all but finitely many $m$.
This is because $\delsmoothdistr$ contains the term $\Omega^{n-1}_{M,tw}[\del_{x_n}]$ which allows only for polynomials in $\del/\del x_n$.
The only thing that remains to be shown is that the same statement is true throughout a neighborhood of each point $x$ (respectively, $y$).
The morphism of sheaves:
\begin{align}
    \Upsilon&: \cinfty_{\HH^n} \to \cinfty_{\HH^n}\\
    \Upsilon(f)(x) &= I_x(f^{\otimes (k-1)})=\sum_{I}c^I \del_I f^{\otimes(k-1)}
\end{align}
is manifestly regular in the sense of Definition 2.2 of \autocite{slovak2}, i.e. it sends a smooth family of functions on $\HH^n$ parametrized by an auxiliary space $T$ to a smooth family of functions on $\HH^n$ parametrized by $T$.
It follows by Lemma 2.6 therein that, for every point $x\in \HH^n$, there exists a neighborhood of $x$ throughout which all but finitely many of the $c^I$ are zero.
By extending the $d^{J,m}$ into $\HH^n$, a similar argument can be used to show the same fact for the $d^{J,m}$ at a fixed $m\in \NN$.
Since, as we have already discussed, there exists a neighborhood of every point $y\in \RR^{n-1}$ for which $d^{J,m}=0$ for all but finitely many $m$, the Lemma follows.
\end{proof}

We define yet another type of functional, whose purpose is to codify which sorts of quadratic interactions we may allow in our theories.

\begin{definition}
We will call a functional $I \in \sO(\condfieldscs)$ a \textbf{quadratic perturbation functional} if it depends only on $A$ fields and is quadratic in the $A$ fields.
\end{definition}

The quadratic interaction
\begin{equation}
S_1(\alpha_1,\alpha_2) = \ip[\alpha_1,\pertdiff\alpha_2]
\end{equation}
is a quadratic perturbation functional. Indeed, $\alpha_2$ must lie in $\sA$ for the above expression to be non-zero. Moreover, we have assumed that $\ip$ is invariant with respect to $\diffonnocrossdiff$; furthermore, the same is true for the full differential $\diff$. As a consequence, it follows that the same is true also for $\pertdiff$. Hence
\begin{equation}
S_1(\alpha_1,\alpha_2)=\pm \ip[\pertdiff\alpha_1,\alpha_2],
\end{equation}
which is zero unless $\alpha_1\in \sA$.

We will denote by
\index[notation]{OPplus@$\pertfunctionals$}
$\pertfunctionals$ the space of all functionals in $\ps[\![\hbar]\!]$ which are at least cubic modulo $\hbar$ and a quadratic perturbation functional. In other words, an element of $\pertfunctionals$ is the sum of a quadratic perturbation functional, an element of $\ps$ which is at least cubic, and an unconstrained element of $\hbar \ps[\![\hbar]\!]$. 
We establish a similar notation $\pertfunctionalspsmcs$ to denote perturbation functionals with proper support and $\del$-smooth first derivative.

\subsection{Renormalization Group Flow}
\label{subsec: RG}
In this section, we define---for the case at hand---one of the central concepts of the theory of perturbative renormalization as developed in \autocite{cost}, namely the notion of renormalization group flow. Having defined $\involution$-invariant parametrices in Section \ref{sec: parametrices}, and shown that they satisfy (almost) all of the same properties as non-invariant parametrices, we can essentially proceed exactly as in \autocite{cost, CG2} to define renormalization group flow for TNBFTs which are amenable to the doubling trick and which have a gauge-fixing amenable to the doubling trick.

\begin{deflem}
\label{deflem: rgflow}
The \textbf{renormalization group flow} from ($\involution$-invariant) parametrix $\Phi$ to ($\involution$-invariant) parametrix $\Psi$ is the operator
\begin{equation}
\rgflow{\Phi}{\Psi}{\cdot} : \pertfunctionals \to \pertfunctionals,
\end{equation}
defined by the equation 
\index[notation]{WPhiPsi@$\rgflow{\Phi}{\Psi}{I}$}
\begin{equation}
\rgflow{\Phi}{\Psi}{I} = \hbar \log \left( e^{\hbar \partial_{P(\Phi,\Psi)}}e^{I/\hbar}\right),
\end{equation}
where $P(\Phi,\Psi): = P(\Phi)-P(\Psi)$ ($P(\Phi)$ is defined by Equation \ref{eq: propagatordef}).
We consider $P(\Phi,\Psi)$ as a section of $E\boxtimes E\to M\times M$ which satisfies the boundary condition at both boundaries $\bdyM\times M$ and $M\times \bdyM$.
By Theorem \ref{thm: tensorofdirichlet}, this implies that $\prop\in \condfields^{\hotimes_\beta 2}$.
Furthermore, $\partial_{\prop}$ is the unique order-two differential operator on $\sO_P(\condfieldscs)$ which is zero on functionals of order less than 2 and which is contraction with $\prop$ on functionals of order exactly 2.
The infinite series which defines this operator is well-defined.
\end{deflem}

\begin{proof}
\emph{A priori}, our functionals may only take as input compactly-supported sections of $E$.
The propagator $\prop$ is not compactly-supported.
However, our functionals are properly-supported, so as long as we ignore functionals of order 0---as we have been doing---this issue of support does not appear.
(See Lemma 14.5.1 of Chapter 2 of \autocite{cost} for a discussion of this issue.)

The algebraic formula defining the renormalization group flow operator $\rgflow{\Phi}{\Psi}{I}$ has a combinatorial interpretation as a sum over connected graphs.
The graphs may possess ``loose'' half-edges, which we call ``tails'' or ``external'' edges.
Each vertex (sometimes we will also call the vertices ``internal vertices'' to match terminology from physics)of the diagram carries also an integer label called the \emph{genus} of the vertex. Let the \emph{order} of a diagram be the number of tails, and the \emph{genus} be the sum of the first Betti number of the diagram and the genera of its internal vertices. A vertex of valence $k$ and genus $g$ corresponds to a term in $I$ which is of order $k$ and is accompanied by a power $\hbar^g$---a diagram of a fixed genus and order can therefore contain, in principle, vertices of arbitrarily high valence. 

Figure \ref{fig: halfedges} depicts the building blocks of the diagrams we are considering, and Figure \ref{fig: typicaldiag} shows a typical diagram one may build out of the building blocks.

We will divide the vertices into two classes. The $\AA$ class of vertices corresponds to quadratic perturbation functionals and the $\BB$ class of vertices corresponds to vertices that are either at least trivalent or have genus at least 1.  A standard argument shows that both the number of $\BB$ vertices and the number of half-edges incident on $\BB$ vertices in a diagram of a given order and genus is bounded above by a number depending only on the order and genus of the diagram.

\begin{figure}
     \centering
     \begin{subfigure}[b]{0.2\textwidth}
         \centering
         \begin{tikzpicture}
         \draw[Ahalfedge] (-1,0)--(0,0);
         \end{tikzpicture}
         \caption{An $A$ half-edge}
         \label{fig: ahalfedge}
     \end{subfigure}
     \hfill
     \begin{subfigure}[b]{0.2\textwidth}
         \centering
         \begin{tikzpicture}
         \draw[Bhalfedge] (-1,0)--(0,0);
         \end{tikzpicture}
         \caption{A $B$ half-edge}
         \label{fig: bhalfedge}
     \end{subfigure}
     \hfill
    \begin{subfigure}[b]{0.2\textwidth}
         \centering
         \begin{tikzpicture}
         \draw[ABedge] (-2,0)--(0,0);
         \end{tikzpicture}
         \caption{An $AB$ edge}
         \label{fig: abedge}
     \end{subfigure}
     \hfill
    \begin{subfigure}[b]{0.2\textwidth}
         \centering
         \begin{tikzpicture}
         \draw[BBedge] (-2,0)--(0,0);
         \end{tikzpicture}
         \caption{A $BB$ edge}
         \label{fig: bbedge}
     \end{subfigure}
     \hfill
    \begin{subfigure}[b]{0.4\textwidth}
         \centering
         \begin{tikzpicture}
         \node at (0,-1.55){};
         \draw[Ahalfedge] (-2,-1)--(-1,-1);
         \draw[Ahalfedge] (0,-1)--(-1,-1);
         \end{tikzpicture}
         \caption{The only admissible bivalent vertex at $\hbar=0$.}
         \label{fig: quadratic vertex}
     \end{subfigure}
     \hfill
         \begin{subfigure}[b]{0.4\textwidth}
         \centering
         \begin{tikzpicture}
         \node at (0,1){};
         \draw[Ahalfedge] (-.75,.75)--(0,0);
         \draw[Bhalfedge] (0,0)--(.75,.75);
         \draw[Bhalfedge] (0,0)--(-.75,-.75);
         \draw[Bhalfedge] (0,0)--(.75,-.75);
         \end{tikzpicture}
         \caption{A vertex of higher valency.}
         \label{fig: higher vertex}
     \end{subfigure}
        \caption{The various building blocks of our diagrams.}
        \label{fig: halfedges}
\end{figure}
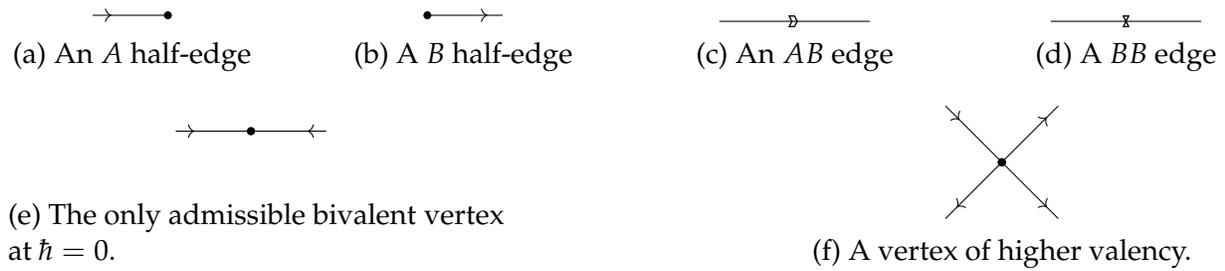

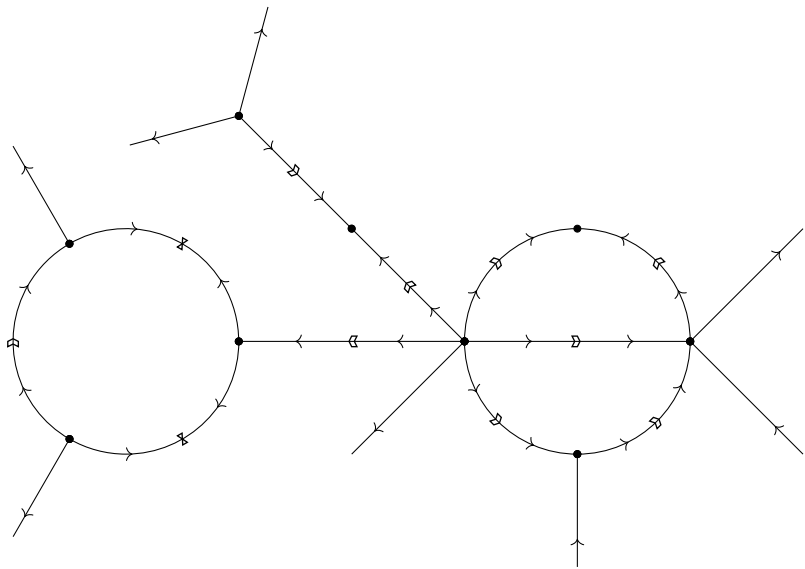
\begin{figure}
    \centering
\begin{tikzpicture}[scale=1.5]
         \node at (0,1){};
         \draw[ABfulledge] (1cm, 0) arc (0:90:1cm) --+(0,0) node(a){};
         \path (a) arc (90:180:1cm) --+(0,0) node (b) {};
         \draw[ABfulledge] (b.center) arc (180:90:1cm);
         \draw[ABfulledge] (b.center) arc (180:270:1cm) --+(0,0) node (c) {};
         \draw[ABfulledge] (c.center) arc (270:360:1cm) --+(0,0) node (d){};
         \draw[ABfulledge] (b.center)+(0,0)-- ++(-2,0) node (f){};
         \draw[ABfulledge] (b.center) -- +(-1,1) node (h){};
         \draw[ABfulledge] (h.center)+(-1,1) node (j){} -- (h.center);
         \draw[Bhalfedge] (j.center) -- +(75:1);
         \draw[Bhalfedge] (j.center) -- +(195:1);
         \draw[Bhalfedge] (b.center) -- +(-1,-1);
         \draw[ABfulledge] (b.center) -- (d.center);
         \draw[Bhalfedge] (d.center) -- +(1,1);
         \draw[Ahalfedge] (d.center)+(1,-1) -- (d.center);
         \path (c.center) -- +(0,-1) node (cdist){};
         \draw[Ahalfedge] (cdist.center) -- (c.center);
         \draw[BBfulledge] (f.center) arc (0:-120:1cm) -- +(0,0) node (g){};
         \draw[ABfulledge] (g.center) arc (-120:-240:1cm) -- +(0,0) node (i){};
         \draw[Bhalfedge] (g.center) -- +(-120:1);
         \draw[BBfulledge] (i.center) arc (120:0:1cm);
         \draw[Bhalfedge] (i.center) -- +(120:1);
         \end{tikzpicture}
         \caption{A typical diagram. All vertices have genus 0.}
         \label{fig: typicaldiag}
\end{figure}

An $\AA$ vertex can only accept $\sA$ inputs, so no two $\AA$ vertices can be connected by an internal edge (since the parametrices have no components in $A\boxtimes A$). Thus, for a given $\AA$ vertex, the half-edges incident on the vertex either lead to an external vertex (aka a tail) or to a $\BB$ vertex. In a connected diagram $G$, it is not possible (because of the assumption of connectivity) for both half-edges of an $\AA$ vertex to end on a tail unless $G$ consists only of an $\AA$ vertex with its tails. Thus, except in this case, each $\AA$ vertex is connected by an internal edge to at least one $\BB$ vertex. There can therefore be at most as many $\AA$ vertices as half-edges incident on $\BB$ vertices. Thus, the number of vertices which appear in $G$ is constrained by the genus and order of $G$. The valence of the $\BB$ vertices appearing in $G$ is bounded above because, as we have mentioned at the end of the previous paragraph, the number of half-edges incident on $\BB$ vertices has been bounded. Hence, there are only finitely many $G$ of a given order and genus. 
\end{proof}

\begin{lemma}
Renormalization group flow restricts to an operator
\begin{equation}
\rgflow{\Phi}{\Psi}{\cdot} : \pertfunctionalspsmcs\to \pertfunctionalspsmcs
\end{equation}
on the space of functionals with proper support and smooth first derivative which are cubic modulo $\hbar$ and a quadratic perturbation functional. 
\end{lemma}

\begin{definition}
A collection of functionals $I[\Phi]\in \pertfunctionalspsmcs$, one for each $\involution$-invariant parametrix $\Phi$, satisfies the \textbf{homotopical renormalization group (HRG) equation} if 
\begin{equation}
\rgflow{\Phi}{\Psi}{ I[\Phi]}=I[\Psi]
\end{equation}
for all parametrices $\Phi$, $\Psi$. In other words, the functional for parametrix $\Psi$ is obtained from the functional for parametrix $\Phi$ by the application of the RG flow operator $\rgdefault$.
\end{definition}

A quantum bulk-boundary system will be a collection of action functionals $\{I[\Phi]\}$, one for each $\involution$-invariant parametrix $\Phi$, satisfying the HRG equation, along with two other major criteria, which we specify in the next subsubsections.

\subsection{The Quantum Master Equation}
\label{subsec: QME}

In the case of classical TNBFTs, the classical master equation was an expression of the gauge-invariance properties of the TNBFT. We would like to introduce a notion of the analogous equation in the case of the quantum theories. Our approach is identical to that taken in \autocite{cost}, so we will summarize the story very briefly.

\begin{definition}
Let $\Phi$ be a parametrix. The \textbf{scale $\Phi$ BV Laplacian} is the cohomological degree $+1$ operator 
\index[notation]{DeltaPhi@$\Delta_\Phi$}
\begin{equation}
\Delta_\Phi:=\partial_{K_\Phi};
\end{equation}
here, $\partial_{K_\Phi}$ is defined analogously to $\partial_{P_\Phi^\Psi}$, which was defined in Definition-Lemma \ref{deflem: rgflow}.
\end{definition}

The following are standard facts concerning the operator $\Delta_\Phi$.
\begin{lemma}
\begin{enumerate}
\item $\Delta_\Phi^2=0$;
\item Let $\diffonnocrossdiff$ denote the operator induced on $\sO(\condfieldscs)$ from the operator with the same name on $\condfields$. $\diffonnocrossdiff\Delta_\Phi+\Delta_\Phi \diffonnocrossdiff=0$.
\end{enumerate}
\end{lemma}

\begin{definition}
The \textbf{scale $\Phi$ BV bracket} is the map
\index[notation]{Zbracket@$\{\cdot,\cdot\}_\Phi$}
\begin{equation}
\sO_P(\condfieldscs)\times \sO_P(\condfieldscs)\to \sO_P(\condfieldscs)
\end{equation} 
denoted $\{\cdot, \cdot\}_\Phi$, and defined by 
\begin{equation}
\{ I, J\}_\Phi = \Delta_\Phi (I\cdot J)-\Delta_\Phi(I)\cdot J-(-1)^{|I|}I\cdot \Delta_\Phi J.
\end{equation}
We will denote by the same symbol the binary operation on $\sO(\condfieldscs)[\![\hbar]\!]$ defined by $\hbar$-linear extension.
\end{definition}

\begin{remark}
The scale $\Phi$ bracket descends to a binary operation on $\pertfunctionalspsmcs$.
\end{remark}

The following is a standard fact concerning the BV Laplacian $\Delta_\Phi$.

\begin{lemma}
The BV Laplacian for scale $\Phi$ is a derivation for the BV bracket for scale $\Phi$. 
\end{lemma}

\begin{definition}
The \textbf{quantum master equation for scale $\Phi$} (scale-$\Phi$ QME) is the equation
\begin{equation}
\label{eq: qme}
\diffonnocrossdiff I+\frac{1}{2}\{I,I\}_\Phi +\hbar \Delta_\Phi I=0,
\end{equation}
where $I\in \pertfunctionals$.
\end{definition}

The following Lemma is due, in this form, to Costello \autocite{cost}, and the proof given there carries over unchanged to the present context.

\begin{lemma}
\label{lem: rgqme}
The homotopical RG equation from scale $\Phi$ to scale $\Psi$ takes a solution of the scale-$\Phi$ QME to a solution of the scale-$\Psi$ QME.
\end{lemma}

Lemma \ref{lem: rgqme} tells us that, given a collection of functionals $\{I[\Phi]\}_{\Phi}$ related by the homotopical RG equation, $I[\Phi]$ satisfies the scale-$\Phi$ QME if and only if $I[\Psi]$ solves the scale-$\Psi$ QME. We will demand that a quantum TNBFT be given by a collection of functionals $\{I[\Phi]\}_\Phi$ related by the homotopical RG equation such that $I[\Phi]$ satisfies the scale-$\Phi$ QME for some (and therefore all) $\Phi$. To capture the notion that the quantum TNBFT be specified by a local action functional, we require one more condition, which we detail in the next subsubsection.

\subsection{A First Definition for Quantum Bulk-Boundary Systems}
\label{subsec: def}

Let us note that the parametrices possess a partial order defined as follows: we say that $\Phi\leq \Psi$ if $\supp(\Phi)\subset \supp(\Psi)$. With this notion in mind, we can finally define what we mean by a quantum bulk-boundary system:

\begin{definition}
\label{def: QTNBFT}
A \textbf{perturbative quantum bulk-boundary system} is a collection of functionals $\{I[\Phi]\}$ ($I[\Phi]\in \pertfunctionalspsmcs$) of cohomological degree 0, one for each parametrix $\Phi$, satisfying:
\begin{enumerate}
\item The $I[\Phi]$ are related by the homotopical renormalization group flow (Definition \ref{deflem: rgflow}), i.e. 
\begin{equation}
I[\Phi]=\rgflow{\Psi}{\Phi}{I[\Psi]}
\end{equation}
for any two parametrices $\Phi$, $\Psi$.
\item The functional $I[\Phi]$ solves the scale $\Phi$ quantum master equation \eqref{eq: qme} for some (hence all) $\Phi$.
\item Let $I_{i,k}[\Phi]$ denote the part of $I[\Phi]$ which is of degree $i$ in $\hbar$ and degree $k$ in the fields $\condfieldscs$. Given any neighborhood $U$ of the diagonal $M\subset M^k$, we can find a parametrix $\Phi_U$ such that whenever $\Phi\leq \Phi_U$, $I[\Phi]$ has support contained in $U$. More precisely, let $\supp(\Phi)^n$ denote the set of points $(x,y)\in M^2$ such that there exists a collection of points $x_0 = x, x_1,\cdots, x_n = y$ with $(x_i,x_{i+1})\in \supp(\Phi) $. Let $\eta(i,k)$ denote the maximum number of edges which can appear in a Feynman diagram of genus $i$ and order $k$; similarly, we let $\upsilon(i,k)$ denote the maximum number of vertices which appear. We require that
\begin{equation}
I_{i,k}(e_1,\cdots, e_k) = 0
\end{equation}
unless 
\begin{equation}
\supp(e_l)\times \supp(e_m) \subset \supp(\Phi)^{\eta(i,k)(1+\upsilon(i,k))}
\end{equation}
for all $1\leq l,m\leq k$.
\item The mod $\hbar$ part of $I[\Phi]$, which we denote $I_0[\Phi]$, is obtained by HRG flow from scale $0$ of a classical bulk-boundary system (see Lemma \ref{lem: classRG}). 
\end{enumerate}
\end{definition} 
  
  \begin{remark}
  The particular exponent $\eta(i,k)(1+\nu(i,k))$ appearing in Item 3 above may be somewhat mysterious.
  The value of the exponent will become clearer once we introduce the machinery of heat-kernel renormalization.
  This discussion culminates in \ref{prop: equivtheories}, wherein the necessity of the particular value of the exponent is derived.
  \end{remark}
 \subsection{An alternative definition of quantum TNBFTs}
 
 In the previous subsection, we outlined a definition of perturbative quantum bulk-boundary systems, based on the the method of effective actions of \autocite{cost}. That definition was based on the notion of a parametrix, and lends itself readily to the construction of a factorization algebra of quantum observables, which we undertake in Section \ref{sec: quantumFA}. However, for the purposes of renormalization and the construction of counter-terms, we need an alternative definition of a quantum bulk-boundary system, which we study in this section. The definition of this section will be based on the notion of a \emph{fake heat kernel:}
 
In this section, we assume that $(\sE,\diffonnocrossdiff, Q^{\GF})$ are amenable to doubling. The following is based on Definition 14.3.1 of Chapter 2 of \autocite{cost}
\begin{definition}
\label{def: fakeheat}
A \textbf{fake heat kernel} for the theory $\sE$ is a smooth, cohomological degree 1 section  
\begin{equation}
\tilde K \in \sEdbl \hotimes_\beta \sEdbl \hotimes_\beta \cinfty(\R_{>0}).
\end{equation}
We denote by $\tilde K_L$ the element of $\sEdbl\hotimes_\beta \sEdbl$ obtained by evaluating at $L$. We require $\tilde K$ to satisfy
\begin{enumerate}
\item $\tilde K$ extends, at $L=0$, to a distribution 
\begin{equation}
\tilde K\in \overline{\sEdbl}\hotimes_\beta \overline{\sEdbl}\hotimes_\beta \cinfty(\R_{\geq 0})
\end{equation}
such that $\tilde K_0$ is the kernel (using the pairing $\ip_{loc}$ to identify $\sEdbl\to \sEdbl^![-1]$) for the identity map $\sEdbl\to \sEdbl$.
\item The support of $\tilde K$, as a subset of $M\times M\times \R_{>0}$, is proper, i.e. the projection maps $\supp \tilde K\to M\times \R_{>0}$ are proper.
\item The quantity 
\begin{equation}
\frac{\d}{\d L}\tilde K + ([\diffonnocrossdiff,Q^{\GF}]\otimes 1) \tilde K
\end{equation}
lies in the space
\begin{equation}
\cinfty(\Mdbl\times \Mdbl\times \RR_{\geq 0}, \Edbl\boxtimes \Edbl)
\end{equation}
and vanishes at $L=0$, with all derivatives in $L$ and on $\Mdbl$, faster than any power of $L$.
\item The kernel $\tilde K$ admits a small $L$ asymptotic expansion which can be written, in normal coordinates $x,y$ near the diagonal in $\Mdbl^2$, in the form 
\begin{equation}
\tilde K_L\simeq L^{-\frac{\dim M}{2}}e^{-||x-y||^2/(4L)}\sum_{i\geq 0}L^i \Omega_i(x,y),
\end{equation}
where the $\Omega_i$ are smooth sections of $\Edbl\boxtimes \Edbl$.
\item The kernel $\tilde K$ is symmetric under the $\Z/2$ action which interchanges the order of the tensor factors $\sEdbl\hotimes_\beta \sEdbl$ (with appropriate Koszul signs).
\item The kernel $\tilde K$ has no component in $\overline{\sAdbl}\hotimes_\beta \overline{\sAdbl}\hotimes_\beta \cinfty(\R_{\geq 0})$.
\end{enumerate}
\end{definition}

\begin{lemma}
A fake heat kernel exists.
\end{lemma}

\begin{proof}
An object satisfying the first four properties of our definition are shown to hold in Lemma 14.3.3 of Chapter 2 of \autocite{cost}. The fact that fake heat kernels satisfying (5) exist follows from the symmetry properties of $[\diffonnocrossdiff,Q^{\GF}]$ with respect to the pairing $\ip$. The fact that fake heat kernels satisfying property (6) follows by the same argument that showed that parametrices satisfying the analogous condition exist (cf. Lemma \ref{lem: noninvttoinvtparam} and its corollary).
\end{proof}

Notice that if $\tilde K$ is a fake heat kernel, then 
\begin{equation}
\tilde \Phi_L : = \int_0^L \tilde K_T\d T
\end{equation}
is a non-invariant parametrix in the sense of Definition \ref{def: param}. We will denote by $\Phi_L$ the $\involution$-invariant parametrix constructed in Lemma \ref{lem: noninvttoinvtparam} from $\tilde \Phi_L$. In other words,
\index[notation]{PhiL@$\Phi_L$}
\begin{equation}
\Phi_L = \frac{1}{2}\left( 1\otimes 1+\involution \otimes 1 +1\otimes \involution + \involution\otimes \involution\right) \tilde \Phi_L
\end{equation}
 By the propagator from scale $\epsilon$ to scale $L$ we mean the propagator between the parametrices $\Phi_\epsilon$ and $\Phi_L$, namely
 \begin{equation} 
 \frac{1}{2}\left( Q^{\GF}\otimes 1+1\otimes Q^{\GF}\right) \left(\Phi_L-\Phi_\epsilon\right).
 \end{equation}
Similarly, by HRG flow from scale $\epsilon$ to scale $L$ (which we denote $\rgflow{\epsilon}{L}{\cdot}$), we mean the HRG flow from parametrix $\Phi_\epsilon$ to parametrix $\Phi_L$. We can also define the scale $L$ QME in a similar way.

\begin{definition}
A \textbf{theory in the heat kernel sense}  is a collection of functionals $\{I[L]\mid L\in \R_{>0}\}$ of elements of $\sO^+_{P,sm}(\condfieldscs)$ of cohomological degree 0 such that
\begin{enumerate}
\item The $I[L]$ satisfy the HRG equation 
\begin{equation}
I[L] = \rgflow{t}{L}{I[t]}.
\end{equation}
\item The functional $I[L]$ solves the scale $L$ quantum master equation for some (hence all) $L$.
\item There is a small $L$ asymptotic expansion 
\begin{equation}
I[L] \simeq \sum f_i(L)\Psi_i
\end{equation}
in terms of local action functionals $\Psi_i\in \sO_{loc}^+(\condfieldscs)$. Further, the $f_i(L)$ have at most a finite-order pole at $L=0$. 
\end{enumerate}
A collection $\{I[L]\}$ satisfying conditions (1) and (3) is called a \textbf{pre-theory.}
\end{definition}

Given a heat kernel theory $\{I[L]\}$, one can define a collection of functionals (one for each parametrix $\Phi$) via
\begin{equation}
I[\Phi]=\rgflow{\Phi_L}{\Phi}{I[\Phi_L]},
\end{equation}
and it is straightforward to verify that the collection satisfies the HRG equation and the QME. We will show in Section \ref{subsec: equivtheories} that $\{I[\Phi]\}$ also satisfies the locality axiom, and that any parametrix theory arises in this way. 

The heat kernel parametrices give us much more control over the small scale behavior of our parametrices; in fact, we can prove a lemma which will be useful to us in the sequel. 
Let us, for the rest of this section, use the symbols $\sE, \sEdbl,\sE_c,\ldots$ to denote the spaces of global sections of the (co)sheaves which are normally represented by the same letters.
Recall that, when $\bdyM=\emptyset$, the parametrices $\Phi$ give well-defined operators $\sE_c\to \sE_c$. In the present situation, z non-invariant parametrix $\tilde \Phi$ gives rise to an operator $(\sEdbl)_c\to (\sEdbl)_c$; 
elements of $\condfieldscs$ can be understood as compactly-supported sections of $\Edbl$ with step-function singularities (by extending the sections into the negative half-manifold to be $\involution$-invariant), and so as elements of $\innerhom{\sEdbl}{\RR}$, for which we will use the symbol $(\sEdbl)^\vee$, since this inner-hom space is the strong topological dual to $\sEdbl$. 
The non-invariant parametrices $\tilde \Phi_L$ act on $(\sEdbl)^\vee$ by transpose, so each $\tilde \Phi_L$ gives a map $\tilde \Phi_L:\condfieldscs\to (\sEdbl)^\vee$. 
Each invariant parametrix gives rise to a similar map, which we will denote also by the symbol~$\Phi_L$.
Let us note the following: given an element $e'\in (\ultrafields)_c$, the $\involution$-invariant extension is smooth on the double, so that $\Phi_L e'\in (\sEdbl)_c^\involution$, and $\Phi_L$ may be viewed as an operator on~$(\ultrafields)_c$.
However, we would like to extend this operator to the full space of $\sL$-conditioned fields.
The following lemma shows that $\Phi_L$ factors through the inclusion $\condfieldscs\to (\sEdbl)^\vee$.

\begin{lemma}
\label{lem: propdefined}
The parametrix $\Phi_L$ lifts to a well-defined operator
\begin{equation}
T_L : \condfieldscs \to \condfieldscs.
\end{equation}

A similar statement holds to construct an operator $P_L$ whose kernel on $\Edbl$ is~$P(\Phi_L)$.
\end{lemma}

\begin{proof}

One may show the Lemma by direct computation or by using a number of well-known facts about the properties of pseudo-differential operators, as we do now.

First, given a section $e\in \condfields$, we may extend it to a $\involution$-invariant section $e'$ of $\Edbl$ in the same way that one may extend a function on $\RR_{\geq 0}$ to an even or odd function on $\RR$.
In general, $e'$ is not smooth, since it and/or its derivatives possess a step-function singularity at $\bdyM\subset \Mdbl$.
It is nevertheless smooth away from $\bdyM$.
The operators $\tilde \Phi_L$, $\tilde P(\Phi_L)$, when considered as operators on the space of distributional sections of $\Edbl$, are pseudodifferential operators, and hence they preserve the singular support of distributions.
Therefore, $\tilde \Phi_L e'$ and $\tilde P(\Phi_L)e'$ are smooth away from $\bdyM$. 
We define $T_L e$ (respectively, $P_L e$) to be one-half of the unique smooth extension of $\tilde \Phi_L e'$ (respectively, $\tilde P(\Phi_L)e'$) from the interior of $M$ to all of $M$, if it exists.
Hence, we need to show that such an extension exists.
First we note that $\tilde P(\Phi_L)e'$, as a section of $\Edbl$, is simply the gauge-fixing operator $Q^{GF}$ applied to $\tilde \Phi_L e'$.
Hence, if $\tilde \Phi_L e'$ is smooth up to $\bdyM$, so too will be $\tilde P(\Phi_L)e'$.
It therefore suffices to show the lemma for the operator $T_L$.
To this end, we may write $e'=e_1+e_2$, where $e_1$ is zero on the negative half-manifold in $\Mdbl$, and $e_2$ is smooth on $\Mdbl$.
It therefore suffices to show that, given a section $e_1$ of $\Edbl$ which is smooth away from $\bdyM$ and which is zero on the negative half-manifold, $\tilde \Phi_L e_1$ is smooth on $M$ up to the boundary.

A sufficient condition for this to be the case is that $\tilde \Phi_L$ satisfy the transmission condition (see Definition 2.3 and Theorem 2.13 of \autocite{Schrohe}).
The operator $\tilde \Phi_L$ is the (non-invariant) parametrix for an elliptic differential operator $[Q,Q^{GF}]$.
Elliptic differential operators, having polynomial symbols, are readily seen to satisfy the transmission condition, and by Proposition 2.7(c) of \autocite{Schrohe}, any (non-invariant) parametrix for an elliptic differential operator satisfies the transmission condition.
(Strictly speaking, Schrohe studies the transmission condition in the upper half-space, but since parametrices have their singularities on the diagonal, we may readily assume without loss of generality that $M$ is the upper half-space.)
The lemma follows.
\end{proof}

\begin{lemma}
\label{lem: qprop}
The equality
\begin{equation}
[Q,P_L]e = e - K_{\Phi_L}\star e 
\end{equation}
holds, where
\begin{equation}
K_{\Phi_L}\star e = (-1)^{|e|}\id\otimes\ip (K_{\Phi_L}\mid_{M\times M} \otimes e)
\end{equation}
is convolution (over $M$) of $K_{\Phi_L}$ with $e$ using the pairing $\ip$.
\end{lemma}

\begin{proof}
It suffices to check this equation by pairing with $e'\in \ultrafields\,$, i.e. it suffices to show that
\begin{equation}
    \ip[{e',[\diffonnocrossdiff,P_L]e}] = \ip[e',e-K_{\Phi_L}\star e]
\end{equation}
for all $e'\in (\ultrafields)_c$.
This is because these fields are dense in $\condfieldscs$ in the $L^2$ sense.
Now, note that, by construction, $P_L$ and $K_L\star$ are graded symmetric operators on $\condfieldscs$ with respect to the pairing $\ip$.
Because $e$ and $e'$ satisfy the boundary condition, $\diffonnocrossdiff$ is also graded skew-symmetric for the pairing.
Hence, we find that
\begin{equation}
    \ip[{e', [\diffonnocrossdiff,P_L]e}]=\ip[{[\diffonnocrossdiff,P_L]e',e}]
\end{equation}
Let
\begin{equation}
    \cD: \ultrafields \to (\sEdbl)^\involution
\end{equation}
denote the natural doubling map. 
Further, let $\mathrm{res}_M$ denote the restriction map
\begin{equation}
    \mathrm{r}_M: (\sEdbl)^\involution\to \ultrafields.
\end{equation}
The maps $\mathrm{r}_M$ and $\cD$ are mutually inverse isomorphisms.
For $e'\in \ultrafields$, we have, by construction,
\begin{equation}
    P_L e' = \frac{1}{2}\mathrm{r}_M P(\Phi_L)\cD(e').
\end{equation}
We also have the identities
\begin{align}
    \mathrm{r}_M \diffonnocrossdiff \cD(e') &= \diffonnocrossdiff e',\qquad \mathrm{r}_M Q^{GF}\cD(e')=Q^{GF}e'\\
    \cD \diffonnocrossdiff \mathrm{r}_M f &= \diffonnocrossdiff f, \qquad \cD Q^{GF}\mathrm{r}_M f = Q^{GF} f
\end{align}
for $e'\in \ultrafields$ and $f\in (\sEdbl)^\involution$.
In other words, the above identities show that the maps $\cD$ and $\mathrm{r}_M$ intertwine both $\diffonnocrossdiff$ and $Q^{GF}$
Hence,
\begin{equation}
    [\diffonnocrossdiff,P_L ]e' = [\diffonnocrossdiff,\mathrm{r}_{M}P(\Phi_L)\cD e']=\mathrm{r}_M [\diffonnocrossdiff,P(\Phi_L)]\cD e';
\end{equation}
now, the kernel for $P(\Phi_L)$ on $\sEdbl$ is precisely $(Q^{GF}\otimes 1+1\otimes Q^{GF})\Phi_L$, and the kernel for $[\diffonnocrossdiff, P_L]$ is
\begin{align}
    K_{\involution +1}-K_{\Phi_L};
\end{align}
$K_{\involution +1}$ acts on $\cD(e')$ by multiplication by 2.
Hence, we find that 
\begin{equation}
    [Q,P_L]e'= \frac{1}{2}\mathrm{r}_M P_{DBL,L}\cD(e') = \frac{1}{2}\mathrm{r}_M (2e'- (-1)^{|e|}\id \otimes \ip_{\Mdbl}(K_{\Phi_L}\otimes e')) = e' - K_{\Phi_L}\star e',
\end{equation}
where we have used the fact that because $e'$ and $K_{\Phi_L}$ are $\involution$-invariant,
\[
id \otimes \ip_{\Mdbl}(K_{\Phi_L}\otimes e')=2id \otimes \ip(K_{\Phi_L}|_{M\times M}\otimes e').
\]
This establishes the Lemma.
\end{proof}

The preceding results combine to show the following lemma:

\begin{lemma}
\label{lem: classRG}
The HRG flow from scale zero to scale $\Phi$ is well-defined modulo $\hbar$ for functionals with smooth first derivative. More precisely, in the \emph{tree} diagrams appearing in the definition of $\rgflow{\Psi}{\Phi}{\cdot}$, one may replace $P(\Phi)-P(\Psi)$ with $P(\Phi)$ and obtain a well-defined answer. Let 
\index[notation]{W0@$\clrgflow{0}{\Phi}{\cdot}$}
$\clrgflow{0}{\Phi}{\cdot}$ denote the corresponding map at the level of functionals (with smooth first derivative).
$W_0(P(0,\Phi),\cdot)$ takes interaction functionals corresponding to classical bulk-boundary systems to solutions of the scale$-\Phi$ quantum master equation mod $\hbar$.
The same statement holds true for functionals of the form $I+\delta \cO$, where $I$ has smooth first derivative and proper support, $\delta^2=0$, and $\cO$ is an arbitrary functional. Moreover, the equation
\begin{equation}
\frac{d}{d\delta}\clrgflow{0}{\Phi}{I+\delta (\diffonnocrossdiff\cO+\{I,\cO\}_\Phi})=d_{CE}\frac{d}{d\delta}\clrgflow{0}{\Phi}{I+\delta\cO}
\end{equation} 
holds.
\end{lemma}

\begin{proof}
In light of Lemmas \ref{lem: qprop} and \ref{lem: propdefined}, the proof of the first statement is by an identical argument to the one used when $\bdyM=\emptyset$.
Here, it is important that we consider only functionals with smooth first derivative (as opposed to $\del$-smooth first derivative) because the trees arising in this way can be seen as compositions of poly-differential operators and pseudo-differential operators.
The second statement of the lemma follows from the mod $\hbar$ part of the argument used to prove Lemma \ref{lem: rgqme}; the argument now applies with zero for one of the parametrices. The fact that this argument also applies to functionals of the form $I+\delta \cO$ uses that the diagrams involved only have one vertex corresponding to $\cO$.
\end{proof}

\section{Counterterms and Renormalization}
So far, we have managed to define a quantum bulk-boundary system.
However, we have said nothing about how to construct a quantum bulk-boundary system.
In this section, we show that one may---given any local functional $I$---construct a quantum pre-theory corresponding to $I$ by the mechanisms of renormalization.
In other words, we will obtain from $I$ a collection $I[L]$ of functionals, one for every $L>0$, which satisfy all of the axioms of a bulk-boundary system in the heat kernel sense except that they may not satisfy the QME.
In Section \ref{subsec: obstruction}, we address the possible failure of a pre-theory to satisfy the QME.
The results in this section are essentially a summary of \autocite{albertrenorm} and an explanation of the consequences of that paper in the present context.
\subsection{Counterterms}

In this and the next subsection, we restrict our attention to the study of \emph{pre-theories} in the heat kernel formalism, that is, collections of interaction functionals $\{I[L]\}$ satisfying the HRG equation and asymptotic locality, but not necessarily the QME. Our goal is to show that such theories are in bijection with $\hbar$ power series of local action functionals $\sum_{i,k}\hbar^i I_{i,k}$, where $I_{i,k}$ is homogeneous of degree $k$ in the fields. To this end, suppose that $\sum_{i,k}\hbar^i I_{i,k}$ is an action functional, where $\sum_{k}I_{0,k}$ is the action functional of a classical bulk-boundary system (we fix also the boundary condition $\sL$ and a gauge-fixing and assume that both the theory and the gauge-fixing are amenable to the doubling trick). We would like to set
\begin{equation}
I[L]=\rgflow{0}{L}{ \sum_{i,k}\hbar^i I_{i,k}};
\end{equation}
however, the HRG flow is not defined from scale 0, since the propagator $P(0,L)$ does not have a smooth kernel. Instead, one must study the singularities for small $\epsilon$ of the quantity
\begin{equation}
\rgflow{\epsilon}{L}{ \sum_{i,k}\hbar^i I_{i,k}}.
\end{equation}
One then shows that these singularities can be canceled by the introduction of local counter-terms---a local, $\epsilon$-dependent functional $I^{CT}(\epsilon)$---so that the limit
\begin{equation}
\lim_{\epsilon\to 0}\rgflow{\epsilon}{L}{\sum_{i,k}\hbar^i I_{i,k}-I^{CT}(\epsilon)}
\end{equation}
exists.

Let $\gamma$ be a Feynman diagram appearing in the sum defining $\rgflow{\epsilon}{L}{\sum \hbar^i I_{i,k}}$. Let $k$ be the number of external vertices, or tails, of $\gamma$. Let $E(\gamma)$ denote the set of edges of $\gamma$, and let $\alpha$ be a compactly-supported section in $\condfieldscs^{\hotimes_\beta k}$. For each such diagram, the associated weight can be written 
\begin{equation}
w_{\gamma,I}(\epsilon,L,\alpha)=\int_{(\epsilon,L)^{|E(\gamma)|}} f_{\gamma,I}(\mathbf{ t},\alpha)\prod_{e\in E(\gamma)}dt_e,
\end{equation}
where $\bf t$ represents the collection of variables $t_e$ as $e$ ranges over $E(\gamma)$ and $f_{\gamma, I}(\mathbf t,\alpha)$ is determined by the Feynman rules. The following result concerning the small $\epsilon$ behavior of the $f_{\gamma,I}(\mathbf t,\alpha)$ is proved for scalar field theory in \autocite{albertrenorm}, but its proof only uses the asymptotic expansion of the heat kernel of the form given in Definition \ref{def: fakeheat} and so applies to the present case.

\begin{theorem}
\label{thm: smallt}
There exists a cover $C_1,\cdots C_\ell$ of $(0,L)^{|E(\gamma)|}$ by a finite number of closed sets such that there exist:
\begin{enumerate}
\item a finite set of inequalities of the form $t_{e_1}^{R_1}-t_{e_2}^{R_2}\geq 0$ defining each $C_i$,
\item $g_{ij}(\mathbf t)\in \cinfty((0,L)^{|E(\gamma)|})$, $i\in \{1,\cdots, \ell\}$, $j\in \NN$,
\item local action functionals $\Psi_{ij}(\alpha)$,
\item $d_{ir},m_{ir}\in \NN$, with $d_{ir}\to \infty$ as $r\to \infty$, $i\in \{1,\cdots, \ell\}$, so that
\begin{equation}
\left | f_{\gamma,I}(\mathbf{t}, \alpha) -\sum_{j=1}^r g_{ij}(\bf t)\Psi_{ij}(\alpha) \right |\leq ||\alpha||_{m_{ir}}^{|T(\gamma)|} \left(\max_{e\in E(\gamma)}t_{e}\right)^{d_{ir}}
\end{equation}
for all $\mathbf t \in C_i$ with $\max t_{e}$ sufficiently small. Here, $||\alpha||_{m_{ir}}$ is the $C^{m_{ir}}$ norm of $\alpha$.
\end{enumerate}
\end{theorem}

Theorem \ref{thm: smallt} shows that we can approximate $w_{\gamma,I}(\epsilon,L,\alpha)$ by  a finite number of terms in the sum
\begin{equation}
\sum_{i,j} \left(\int_{C_i\cap (\epsilon,L)^{|E(\gamma)|}}g_{ij}(\bf t)d\bf t\right)\Psi_{ij}(\alpha),
\end{equation}
and the error in doing so (if sufficiently many terms are kept in the above sum) is bounded by a function of the coordinates $\bf t$ whose integral over $(0,L)^{|E(\gamma)|}$ exists. We can also use Theorem \ref{thm: smallt} to derive a small-$\epsilon$ asymptotic expansion for the weights $w_{\gamma,I}(\epsilon,1,\alpha)$:

\begin{proposition}
\label{prop: smallepsilon}
The Feynman weights $w_{\gamma,I}(\epsilon,1,\alpha)$ admit small-$\epsilon$ asymptotic expansions of the form:
\begin{equation}
\sum_{i=0}^\infty g_j(\epsilon)\Psi_j(\alpha),
\end{equation}
where the $\Psi_j(\alpha)$ are local action functionals. More precisely, for every $N>0$, there exist a $j_0>0$ and an $m\in \ZZ_{\geq 0}$ so that 
\begin{equation}
\left | w_{\gamma,I}(\epsilon,1,\alpha)- \sum_{j=0}^{j_0} g_j(\epsilon)\Upsilon_j(\alpha) \right| < \epsilon^{N+1}||\alpha||_{m}^{|T(\gamma)|}
\end{equation}
for sufficiently small $\epsilon$. 
\end{proposition}

\begin{proof}
The proof follows from Theorem \ref{thm: smallt} in the same way that Lemma 15.7.7 in \autocite{costrenormBV} follows from Lemma 15.7.3 therein.
\end{proof}

\begin{theorem}
\label{thm: smallepsilon}
The weights $w_{\gamma,I}(\epsilon, L, \alpha)$ have the following properties:
\begin{enumerate}
\item They admit a small-$\epsilon$ asymptotic expansion
\begin{equation}
w_{\gamma,I}(\epsilon,L,\alpha) \simeq \sum_{i=0}^\infty h_i (\epsilon)\Upsilon_i(L,\alpha),
\end{equation}
where the $\Upsilon_i$ are continuously $L$-dependent functionals of the field $\alpha$ (not necessarily local) and the $h_i$ are continuous functions of $\epsilon\in (0,1)$.
\item The $\Upsilon_i(L,\alpha)$ admit a small-$L$ asymptotic expansion 
\begin{equation}
\Upsilon_i(L,\alpha) \simeq \sum_{i=0}^\infty f_{i,j}(L)\Phi_{i,j},
\end{equation}
where the $\Phi_{i,j}$ are local.
\end{enumerate}
\end{theorem}

\begin{proof}
Item (1) follows from using the small-$\epsilon$ asymptotic expansion 
\begin{equation}
w_{\gamma,I}(\epsilon,1,\alpha) \simeq \sum_{j=0}^\infty g_{i}(\epsilon )\Psi_{i}(\alpha)
\end{equation}
and then subjecting these weights to RG flow from $1$ to $L$. 
Item (2) follows from the fact that the $\Upsilon_i$ can be described as 
\begin{equation}
w_{\gamma,J}(1,L,\alpha), 
\end{equation}
where $J$ is a local functional obtained from the $\Psi_{i}$, so that we can apply Proposition \ref{thm: smallepsilon} to study the small-$L$ behavior of the $\Upsilon_i$.
\end{proof}

\begin{remark}
We note that the $g_i$ of Proposition \ref{prop: smallepsilon}, as well as the $h_i$  and the $f_{i,j}$ of Theorem \ref{thm: smallepsilon}, can be shown to have finite-order poles in their variables. We do not prove this here, but it follows from a direct examination of the structure of the Feynman weights and the functions $g_{ij}$ appearing in Theorem \ref{thm: smallt}. 
\end{remark}

With Theorem \ref{thm: smallepsilon} in hand, we can construct a quantum pre-theory by constructing counter-terms along the lines of \autocite{cost}.
To this end, let $\sA_{\geq 0}\subset  C((0,1))$ denote the subspace of continuous functions of $\epsilon\in (0,1)$ which admit an $\epsilon\to 0^+$ limit. 

\begin{definition}
A \textbf{renormalization scheme} is a choice of complement $\sA_{<0}$ to $\sA_{\geq 0}$ in $C((0,1))$. 
\end{definition}

For $f\in C((0,1))$, we write $\Sing f$ and $\Reg f$ for the projections of $f$ onto $\sA_{<0}$ and $\sA_{\geq 0}$, respectively.

\begin{theorem}
\label{thm: constructionofcounterterms}
Given a local action functional $I=\sum_{i,k}\hbar^i I_{i,k}\in \pertfunctionalspsmcs$ and a renormalization scheme $\sA_{<0}$, there exists a unique series of counter-terms $I^{CT}_{i,k}(\epsilon)$, where each $I^{CT}_{i,k}(\epsilon)$ belongs to the space
\begin{equation}
\Oloc \otimes_{alg} \sA_{<0}
\end{equation}
and is homogeneous of degree $k$. The $I^{CT}_{i,k}(\epsilon)$ are such that the limit
\begin{equation}
\lim_{\epsilon\to 0^+} \rgflow{\epsilon}{L}{I-\sum_i \hbar^i I^{CT}_{i,k}(\epsilon)}
\end{equation}
exists.
\end{theorem}

\begin{proof}
Theorem \ref{thm: smallepsilon} implies that, given a connected diagram $\gamma$, the Feynman weight $w_{\gamma, I}(\epsilon,L)$ admits a small-$\epsilon$ asymptotic expansion of the form 
\begin{equation}
w_{\gamma,I}(\epsilon,L)\simeq \sum_{i=0}^\infty g_i(\epsilon) \Upsilon_i,
\end{equation}
where the $g_i(\epsilon)$ are continuous functions of $\epsilon\in (0,1)$ and the $\Upsilon_i$ are smooth families of functionals parametrized by $L$. The existence of an asymptotic expansion implies there exists an $N>0$ such that $g_i(\epsilon)$ admits an $\epsilon\to 0^+$ limit for $i>N$. We define 
\begin{equation}
\Sing w_{\gamma,I}(\epsilon,L) = \sum_{i=0}^\infty\Sing_\epsilon g_i(\epsilon) \Upsilon_i=\sum_{i=0}^N \Sing_\epsilon g_i(\epsilon).
\end{equation}
Let $\Gamma_{i,k}$ denote the set of connected Feynman diagrams of genus $i$ and order $k$. Denote
\begin{equation}
W_{i,k}(P(\epsilon,L),I) = \sum_{\gamma\in \Gamma_{i,k}}w_{\gamma,I}(\epsilon,L);
\end{equation}
we have 
\begin{equation}
\rgflow{\epsilon}{L}{I} = \sum \hbar^i W_{i,k}(P(\epsilon,L),I).
\end{equation}
In \autocite{cost}, Costello constructs the counter-terms by induction on $(i,k)\in \Z_{\geq 0}\times \Z_{\geq 0 }$. We perform a similar induction; however, since in the present work we have allowed Feynman diagrams which have bivalent vertices at order $\hbar^0$, we need to perform an additional induction. To this end, recall that we have postulated a decomposition $\sE\cong \sA\oplus \sB$ (Definition \ref{def: amenable}) and let $W_{i,k,l}$ denote the component of $W_{i,k}$ which accepts $l$ inputs from $\sA$ (and therefore $k-l$ inputs from $\sB$). Similarly, we write $I_{i,k,l}$ for the corresponding component of $I_{i,k}$. We will call $l$ the \emph{$\sA$-grading} of $I_{i,k,l}$, $W_{i,k,l}$. We will prove the theorem by an induction on the $(i,k,l)$, endowed with the lexicographic ordering. As we have seen, the $W_{0,k,l}(P(\epsilon,L),I)$ are non-singular as $\epsilon\to 0$, since they are obtained by classical RG flow. Now, we define 
\begin{equation}
I^{CT}_{1,0,0}(\epsilon) : = \Sing_{\epsilon} W_{1,0,0}(P(\epsilon,L), I)=0,
\end{equation}
since the only diagrams which could contribute to $W_{1,0,0}$ are the contraction of $I_{0,2}$ with $P(\epsilon, L)$ and $I_{1,0}$. The former is zero, because $P$ has an $\sA$ and $\sB$ factor, while $I_{0,2}$ accepts only $\sA$ inputs. The contribution of $I_{1,0}$ is independent of $\epsilon$ and so non-singular. This establishes the base case for the induction. Let us now denote by $W_{<(i,k,l)}(P(\epsilon,L),J)$ the sum over all diagrams contributing to $W(P(\epsilon,L),J)$ with order $k'$, genus $i'$, and $\sA$-grading $l'$, where $(i',k',l')< (i,k,l)$. We will also write $W_{\leq (i,k,l)}$ for a similar sum. Suppose that, for a given $(i,k,l)$, we have constructed counter-terms $I^{CT}_{i',k',l'}$ for all $(i',k',l')< (i,k,l)$ such that 
\begin{equation}
W_{<(i,k,l)}\left(P(\epsilon,L),I-\sum_{(i',k',l')< (i,k,l)}I^{CT}_{i',k',l'}(\epsilon)\right)
\end{equation}
admits an $\epsilon\to 0$ limit. Define
\begin{equation}
I^{CT}_{i,k,l}(\epsilon) = \Sing_\epsilon W_{i,k,l}\left( P(\epsilon,L), I-\sum_{(i',k',l')<(i,k,l)}I^{CT}_{i',k',l'}(\epsilon)\right). 
\end{equation} 
(In principle, $I^{CT}(\epsilon)$ depends on $L$ as well as $\epsilon$; we will see below that this is not the case). We have the following equation whenever $(i',k',l')<(i,k,l)$:
\begin{equation}
\label{eq: lessthan}
W_{i',k',l'}\left( P(\epsilon,L), I-\sum_{(i'',k'',l'')\leq (i,k,l)}I^{CT}_{i'',k'',l''}(\epsilon)\right)=W_{i',k',l'}\left( P(\epsilon,L),  I-\sum_{(i'',k'',l'')< (i,k,l)}I^{CT}_{i'',k'',l''}(\epsilon)\right).
\end{equation}
Similarly, we have
\begin{align}
\label{eq: equal}
W_{i,k,l}&\left(  P(\epsilon,L), I-\sum_{(i',k',l')\leq (i,k,l)}I^{CT}_{i'',k'',l''}(\epsilon)\right)\\
\nonumber
&=W_{i,k,l}\left( P(\epsilon,L),  I-\sum_{(i',k',l')< (i,k,l)}I^{CT}_{i',k',l'}(\epsilon)\right)-I^{CT}_{i,k,l}(\epsilon).
\end{align}

To prove  Equations \eqref{eq: lessthan} and \eqref{eq: equal}, let $\gamma\in \Gamma_{i',k',l'}$ be a connected graph appearing in the left-hand-side of Equation \eqref{eq: lessthan}. Let $V(\gamma)$ denote the set of vertices of $\gamma$, $E(\gamma)$ the set of (internal) edges of $\gamma$, and for each $v\in V(\gamma)$, let $val(v)$ denote the valence of $v$ and $g(v)$ the genus of $v$. For $\gamma$ to have genus $i'$, we must have the equality
\begin{equation}
|E(\gamma)|-|V(\gamma)|+\sum_{v\in V} g(v) +1 =i'.
\end{equation}
Moreover, $|E(\gamma)|-|V(\gamma)|+1\geq 0$. Hence, $\gamma$ cannot have any vertices $v$ of genus exceeding $i'$. If $\gamma$ has a single vertex $v_0$ of genus $i'$, then all remaining vertices have genus zero and the diagram is a tree. 
Then, because all of the genus 0 vertices are at least bivalent, $\gamma$ has at least as many tails as the valence of $v_0$, so the valence of $v_0$ cannot exceed $k'$. 
The same can be said about the valence of any vertex in $\gamma$.
Finally, if in addition the valence of $v_0$ is $k'$, $\gamma$ must consist of some number of genus 0 bivalent vertices connected to $v_0$. 
Note that, by our assumptions on the structure of $P(\epsilon,L)$ and of the bivalent vertices (given in Definitions \ref{def: amenable} and \ref{def: param}), attaching a bivalent vertex to another vertex increases the $\sA$-order of the functional by 1. Hence, for any functional $J=\sum_{i,k}\hbar^i J_{i,k}$, the contributions to 
\begin{equation}
W_{i,k,l}(P(\epsilon,L),J)
\end{equation}
come from terms $J_{i',k',l'}$ with $(i',k',l')\leq (i,k,l)$, and the contribution of $J_{i,k,l}$ to $W_{i,k,l}(P(\epsilon,L),J)$ is simply $J_{i,k,l}$. This establishes Equations \eqref{eq: lessthan} and \eqref{eq: equal}. The remainder of the proof proceeds as in \autocite{cost}: one shows that $I^{CT}_{i,k,l}(\epsilon)$ is independent of $L$, local, and that 
\begin{equation}
W_{\leq(i,k,l)}\left(P(\epsilon,L), I-\sum_{(i',k',l')\leq (i,k,l)}I^{CT}_{i',k',l'}(\epsilon)\right)
\end{equation}
admits an $\epsilon\to 0$ limit. The main facts used in the proofs of those statements in \autocite{cost} are stated in Theorem 9.3.1 of Chapter 2 there. The analogous statement in the present work is Theorem \ref{thm: smallepsilon}.
\end{proof}

Theorem \ref{thm: constructionofcounterterms} implies the following:

\begin{theorem}
Given an action functional $I\in \pertfunctionals$ and a renormalization scheme, let $I^{CT}(\epsilon)$ denote the sum $\sum_{i,k,l}\hbar^i I^{CT}_{i,k,l}(\epsilon)$ of all the counter-terms constructed in Theorem \ref{thm: constructionofcounterterms}. Then, the equation 
\begin{equation}
I[L]:= \lim_{\epsilon\to 0^+}W\left( P(\epsilon,L),I-I^{CT}(\epsilon)\right)
\end{equation}
defines a quantum pre-theory in the heat kernel sense.
\end{theorem}

The following Theorem is proved in our context in an identical fashion to the proof in Section 11 of Chapter 2 of \autocite{cost}:

\begin{theorem}
\label{thm: mainbijection}
Let $I_0$ be the interaction functional of a classical bulk-boundary system. Suppose the classical bulk-boundary system is amenable to doubling and has a gauge-fixing amenable to doubling. A choice of renormalization scheme leads to a bijection between the set of local action functionals $\hbar \Oloc[\![\hbar]\!]$ and the set of quantum bulk-boundary pre-theories $\{I[L]\}_{L>0}$ (in the heat kernel sense) whose $\hbar^0$ component is obtained from $I_0$ by classical RG flow from scale 0. 
\end{theorem}

\subsection{Equivalence of the two notions of quantum field theory}
\label{subsec: equivtheories}
In this section, we show the following Proposition:

\begin{proposition}
\label{prop: equivtheories}
Given $\{I[L]\}$, a quantum pre-theory in the heat kernel sense, the collection 
\begin{equation}
I[\Phi]:= \rgflow{\Phi_L}{\Phi}{I[L]}
\end{equation}
gives a quantum pre-theory in the parametrix sense. Conversely, any quantum pre-theory $\{I[\Phi]\}$ in the parametrix sense gives a quantum pre-theory in the heat kernel sense:
\begin{equation}
I[L]:= I[\Phi_L].
\end{equation}
These two correspondences are mutually inverse.
\end{proposition}

\begin{proof}
For the first statement, fix a renormalization scheme and a local action functional $I$. Let $\{I[L]\}$ be the quantum pre-theory in the heat kernel sense associated to these choices.
Let $\{I[\Phi]\}$ the corresponding collection of functionals defining the putative pre-theory in the parametrix sense. $\{I[\Phi]\}$ manifestly satisfies the HRG flow equation, so we only need to show that the collection $\{I[\Phi]\}$ has $\partial$-smooth first derivative and that $\{I_{i,k}[\Phi]\}$ can be chosen to be ``close to local'' if $\Phi$ has support close to the diagonal $M\subset M^k$. The precise sense in which me mean the latter statement is specified as Item (3) in Definition \ref{def: QTNBFT}. 
The analogous statement in \autocite{cost} is Lemma 12.0.2 of Chapter 2. 
Our proof is almost identical to the one there. 
Let $\Psi$ be a smooth function on $M^2$ with proper support and suppose that $\Psi$ is identically 1 in a neighborhood of the diagonal. 
Then, as discussed in \autocite{cost}, $\Psi \Phi_L$ is a parametrix and 
\begin{equation}
\label{eq: parambehavior}
I[\Psi \Phi_L] = \lim_{\epsilon\to 0} W(P(\Psi(\Phi_L-\Phi_\epsilon)), I-I^{CT}(\epsilon)).
\end{equation}
The fact that $I[\Psi \Phi_L]$ has $\partial$-smooth first derivative follows from the fact that the functionals $I-I^{CT}(\epsilon)$ have $\partial$-smooth first derivative (because they are local) and the fact that HRG flow preserves the space of functionals with smooth first derivative. The fact that the $I[\Phi]$ satisfy HRG flow is immediate. To show that the $I[\Phi]$ give a pre-theory, it remains to check that the $I[\Phi]$ have the desired properties as $\Phi\to 0$. We have
\begin{equation}
\label{eq: parambehavior2}
I[\Phi] = W(P(\Phi-\Psi\Phi_L) , I[\Psi \Phi_L]).
\end{equation}
Let us first note that by choosing $\Psi$ supported arbitrarily close to the diagonal in $M^2$, we can make $I_{i',k'}[\Psi \Phi_L]$ ``close'' to local for $(i',k')$ smaller than or equal to (in the lexicographic ordering) some $(i,k)$. More precisely, there exists some bilinear function $\eta(i,k)$ so that 
\begin{equation}
I_{i',k'}[\Psi\Phi_L](e_1,\cdots, e_{k'}) = 0
\end{equation}
unless 
\begin{equation}
\supp e_m \times \supp e_n \subset (\supp( \Psi\Phi_L))^{\eta(i,k)}.
\end{equation}
This follows from Equation \eqref{eq: parambehavior}, where $\eta$ is, as in Definition \ref{def: QTNBFT}, the maximum number of edges that can appear in a Feynman diagram of genus $i$ and order $g$.

Now, we may choose $\Psi$ supported close enough to the diagonal so that $\Psi\Phi_L \leq \Phi$. Fix a Feynman diagram $\gamma$. Let $V$ be the number of internal vertices of $\gamma$, $E$ the number of internal edges, and $w_\gamma$ the contribution of $\gamma$ to the sum $\rgflow{\Psi\Phi_L}{\Phi}{I[\Psi\Phi_L]}$. Then, 
\begin{equation}
w_\gamma(e_1,\cdots, e_k) = 0
\end{equation}
unless 
\begin{equation}
\supp e_m \times \supp e_n \subset (\supp( \Psi\Phi_L))^{E+V\eta(i,k)} \subset (\supp( \Psi\Phi_L))^{(\eta(i,k)+1)\upsilon(i,k)}.
\end{equation}
This shows that the collection $I[\Phi]$ satisfies the locality condition in Definition \ref{def: QTNBFT}. This completes the proof that the $I[\Phi]$ give a pre-theory in the parametrix sense.

To show that all parametrix theories arise in this way, suppose that two parametrix theories $\{I'[\Phi]\}$ and $\{I[\Phi]\}$ satisfy
\begin{equation}
I'_{i',k',l'}[\Phi]=I_{i',k',l'}[\Phi]
\end{equation}
for all $(i',k',l')< (i,k,l)$. Then, let 
\begin{equation}
J_{i,k,l}[\Phi] = I'_{i,k,l}[\Phi]-I_{i,k,l}[\Phi];
\end{equation}
the combinatorics of the HRG flow equation imply that $J_{i,k,l}[\Phi]$ is independent of $\Phi$; hence, because $I$ and $I'$ satisfy asymptotic locality, $J$ must be supported on the small diagonal $M\subset M^k$. Finally, $J$ has $\partial$-smooth first derivative because $I$ and $I'$ do. 
The functionals with support on the small diagonal and $\partial$-smooth first derivative are precisely the local action functionals (cf. Proposition \ref{lem: localfcnlssupportdiag} ; hence, $J$ is a local functional.
In other words, suppose that $I[\Phi]$ is a parametrix pre-theory defined for all $(i',k',l')< (i,k,l)$; we have shown that the set of all lifts of such a theory to one defined for all $(i',k',l')\leq (i,k,l)$ is in bijection with the space of local functionals. But this is precisely the corresponding set of lifts for a heat kernel pre-theory.
\end{proof}

\section{Obstruction Theory }
\label{subsec: obstruction}
Suppose that $\{I[\Phi]\}$ is a quantum bulk-boundary system defined modulo $\hbar^{n+1}$; in other words, we have $I[\Phi]\in \sO(\sE)[\![\hbar]\!]/(\hbar^{n+1})$,
\begin{equation}
\rgflow{\Psi}{\Phi}{I[\Psi]}=I[\Phi]\quad \mod \hbar^{n+1},
\end{equation} 
\begin{equation}
QI[\Phi]+\frac{1}{2}\{I[\Phi],I[\Phi]\}_{\Phi} +\hbar \Delta_\Phi I[\Phi]=0\quad \mod \hbar^{n+1},
\end{equation}
and the $I[\Phi]$ satisfy an asymptotic locality condition as $\Phi\to 0$. 
We have seen (cf. the end of the proof of Proposition \ref{prop: equivtheories}) that, by choosing a local functional, we may extend $I[\Phi]$ to a collection $\{\widetilde I[\Phi]\}$ of functionals which are defined and satisfy the HRG equation modulo $\hbar^{n+2}$. Define 
\begin{equation}
O_{n+1}[\Phi]=\hbar^{-n-1}\left( \diff\widetilde I[\Phi]+\frac{1}{2}\{\widetilde I[\Phi],\widetilde I[\Phi]\}_\Phi+\hbar \Delta_\Phi \widetilde I[\Phi]\right)\quad \mod \hbar ;
\end{equation} 
The functional $O_{n+1}[\Phi]$ belongs to the space $\sO_{P}(\condfieldscs)$ of functionals with proper support, and has cohomological degree +1. 
This is because $\widetilde I[\Phi]$ has proper support and the operations $\diff, \{\cdot, \cdot\}_\Phi$, and $\Delta_\Phi$ preserve functionals with this property.

The following is Lemma 11.1.1 of Chapter 5 of \autocite{cost}; the proof given there works without changes in the present context.
\begin{lemma}
 Let $\epsilon$ be a parameter of cohomological degree $-1$ which squares to zero. Let $I_0[\Phi]$ be the classical part of $I[\Phi]$. Then, 
 \begin{equation}
I_0[\Phi]+\epsilon O_{n+1}[\Phi]
 \end{equation}
 satisfies the classical HRG equation and the (scale $\Phi$) classical master equation. The space of lifts of $\{I[\Phi]\}$ to a theory defined modulo $\hbar^{n+2}$ is the space of collections of functionals $\{J[\Phi]\}$ of cohomological degree 0 such that $I_0[\Phi]+\delta J[\Phi]=O_{n+1}[\Phi]$, such that $J[\Phi]$ satisfies the equation
 \begin{equation}
 QJ[\Phi]+\{I_0[\Phi],J[\Phi]\}_\Phi = O_{n+1}[\Phi],
 \end{equation}
and such that $J[\Phi]$ can be made to have support arbitrarily close to the diagonal for small $\Phi$; here, $\delta$ is a parameter of cohomological degree 0 which squares to 0.
\end{lemma}

\begin{remark}
To be more precise about the asymptotic locality condition specified on $J[\Phi]$ in the lemma, let $K_t$ be a fake heat kernel. We require that $J[\Phi_t]$ have a small-$t$ asymptotic expansion in terms of local functionals of the same sort that the functionals $I[\Phi_t]$ are required to satisfy in the heat kernel definition of quantum TNBFT.
\end{remark}

Combining the preceding lemma with Lemma \ref{lem: classRG}, we have

\begin{proposition}
Given a theory $\{I[\Phi]\}$ defined modulo $\hbar^{n+1}$, the obstruction to lifting $I[\Phi]$ to a theory defined modulo $\hbar^{n+2}$ is a local functional $O_{n+1}\in \Oloc$ of cohomological degree 1 which is closed for the Chevalley-Eilenberg differential. The set of lifts of $I[\Phi]$ to a theory defined modulo $\hbar^{n+2}$ is the set of elements $J\in \Oloc$ of degree zero which witness the exactness of $O_{n+1}$ as an element of $\Oloc$.  
\end{proposition}

\begin{proof}
By Lemma \ref{lem: classRG}, the limit $\lim_{\Phi\to 0} O_{n+1}[\Phi]$ exists and is closed for the Chevalley-Eilenberg differential on the complex of local functionals. 
The asymptotic locality of $\widetilde I[\Phi]$ implies the asymptotic locality of $O_{n+1}[\Phi]$, by an argument similar to the one in the proof of Lemma 9.4.0.2 of \autocite{CG2}. Namely, $\diff$ does not change the support of a functional, while $\supp \Delta_\Phi \widetilde I[\Phi]$ is a subset of the image of $\supp I[\Phi]$ under the projection $M^k\to M^{k-2}$, and similarly, the operation $\{\cdot, \cdot\}_\Phi$ ``widens'' the support of functionals by an amount which can be made to be small if $\Phi$ is small.

Let $\{J[\Phi]\}$ be a collection of functionals specifying a lift of $I[\Phi]$ to a theory mod $\hbar^{n+2}$. By Lemma \ref{lem: classRG}, $J:=\lim_{t\to 0}J[\Phi_t]$ exists. Moreover, since $J[\Phi_t]$ is assumed to have a small-$t$ asymptotic expansion in terms of local action functionals, it follows that $J$ is local. Finally, since 
\begin{equation}
QJ[\Phi]+\{I[\Phi],J[\Phi]\}_\Phi = O_{n+1}[\Phi],
\end{equation}
and since the classical HRG equation for $I_0+\delta J$ intertwines the Chevalley-Eilenberg differential acting on $J$ with 
\begin{equation}
QJ+\{I[\Phi],J\}_{\Phi},
\end{equation}
it follows that $J$ witnesses the exactness of $O_{n+1}$ in $\Oloc$.
\end{proof}

\section{Factorization Algebra of Quantum Observables}
\label{sec: quantumFA}

In this subsection, we construct a factorization algebra on $M$ which can be associated to any quantum bulk-boundary system $I[\Phi]$ on $M$, thus completing the main task of this dissertation.

We proceed as in \autocite{CG2}, introducing first a complex of global (on $M$) observables for each parametrix $\Phi$.
The complexes corresponding to different parametrices will be equivalent to each other; in this way, we obtain a single complex of global observables.
The observables associated to an open subset $U\subset M$ will be described as a subcomplex of the global observables.
Finally, we will describe the factorization product for observables.
The discussion follows \autocite{CG2} closely.
In fact, the proofs given there for the case $\bdyM=\emptyset$ carry over with little change to our context.
All constructions will be implicitly carried out in the category of differentiable vector spaces.
For the sake of space, however, we only describe the underlying vector spaces.
We refer the reader to the corresponding locations in \autocite{CG2} for the description of the differentiable structure on the vector spaces we define here.

\begin{definition}
The \textbf{cochain complex of global scale-$\Phi$ observables}
\index[notation]{ObsqPhi@$\Obq_{\sE,\sL}(M)[\Phi]$}
$\Obq_{\sE,\sL}(M)[\Phi]$ is the cochain complex
\begin{equation}
(\sO(\condfields)[\![\hbar]\!], Q+\{I[\Phi],\cdot\}_\Phi+\hbar \Delta_\Phi);
\end{equation}
the operator $Q+\{I[\Phi],\cdot\}_\Phi+\hbar \Delta_\Phi$ squares to zero as a consequence of the scale $\Phi$ QME.
\end{definition}

\begin{definition}
The \textbf{homotopical renormalization group flow for observables} is the operation
\index[notation]{Wobs@$\rgobs{\Phi}{\Psi}$}
\begin{equation}
\rgobs{\Phi}{\Psi}: \sO(\condfields)[\![\hbar]\!]\to \sO(\condfields)[\![\hbar]\!]
\end{equation}
defined by 
\begin{equation}
\rgobs{\Phi}{\Psi}(\cO) = \hbar \frac{d}{d\delta}\log\left(\exp(\hbar \partial_{P_\Psi-P_\Phi}) e^{I[\Phi]+\delta \cO/\hbar}\right).
\end{equation}
Write $\cO=\sum_{i,k}\hbar^i \cO_{i,k}$, where $\cO_{i,k}$ has order $k$ in the fields. The HRG flow for observables has a graphical representation as a sum over graphs with one special vertex labeled by $\cO_{i,k}$, where $i$ is the genus of the vertex and $k$ is its valence.
\end{definition}

The following lemma is proved by the same means as in \autocite{CG2}:
\begin{lemma}
\label{lem: obsrgflow}
The following equality
\begin{equation}
(Q+\{I[\Psi],\cdot\}_{\Psi}+\hbar \Delta_\Psi)(\rgobs{\Phi}{\Psi}(\cO))=\rgobs{\Phi}{\Psi}\left( (Q+\{I[\Phi],\cdot\}_\Phi+\hbar \Delta_\Phi)\cO\right)
\end{equation}
holds. In other words, the observable HRG flow is a map of complexes
\begin{equation}
\Obq_{\sE,\sL}(M)[\Phi]\to \Obq_{\sE,\sL}(M)[\Psi]
\end{equation}
from the scale-$\Phi$ observables to the scale-$\Psi$ observables.
\end{lemma}

\begin{definition}
Let $\cP$ be the poset of parametrices, considered as a category. By Lemma \ref{lem: obsrgflow}, the assignment $\Phi \to \Obq_{\sE,\sL}(M)[\Phi]$ is a functor $F:\cP \to \text{Ch}$. The \textbf{global observables} $\Obq_{\sE,\sL}(M)$ are the limit cochain complex $\lim F$. 
\end{definition}

In more pedestrian terms, an element of $\Obq_{\sE,\sL}(M)$ is an observable $\cO[\Phi]$, one for every $\Phi \in \cP$, such that the collection $\cO[\Phi]$ and $\cO[\Psi]$ are related by HRG flow from scale $\Phi$ to scale $\Psi$ for all $\Psi$ and $\Phi$.

$\Obq_{\sE,\sL}(M)$ will be the space of global sections of the factorization algebra we are constructing. We must also define local sections of this factorization algebra.

\begin{definition}
The space of \textbf{observables supported on an open $U\subset M$} is the subcomplex of $\Obq_{\sE,\sL}(M)$ consisting of functionals $\cO[\Phi]=\sum_{i,k}\hbar^i \cO_{i,k}[\Phi]$ such that for each $i,k$, there exists a compact subset $K\subset U^k$ and a parametrix $\Phi$ such that the support (Definition \ref{def: functionaldefs}) of $\cO_{i,k}[\Psi]\subset K$ for all $\Psi\leq \Phi$. We denote the space of observables supported on $U$ by 
\index[notation]{ObsqEL@$\Obq_{\sE,\sL}$}
$\Obq_{\sE,\sL}(U)$.
\end{definition}

The proof that $\Obq_{\sE,\sL}(U)$ is a subcomplex of $\Obq_{\sE,\sL}(M)$ is identical to the one given in \autocite{CG2}. 

We have given a cochain complex $\Obq_{\sE,\sL}(U)$ for every open subset of $M$; it is immediate from the definitions that whenever $U\subset V$, there is an induced map $\Obq_{\sE,\sL}(U)\to \Obq(V)$. To define the remaining factorization products, it suffices by associativity to define a map 
\begin{equation}
\Obq (U)\hotimes_\beta \Obq(V)\to \Obq(U\sqcup V)
\end{equation}
whenever $U$ and $V$ are disjoint. Let us denote by $\ast$ the commutative product in $\sO(\condfields)[\![\hbar]\!]$. Let $\cO\in \Obq_{\sE,\sL}(U)$, $\cO'\in \Obq(V)$. We define the two-point factorization product:

\begin{equation}
m(\cO,\cO')[\Phi]=\lim_{\Psi \to 0} \rgobs{\Psi}{\Phi}(\cO[\Psi]\ast \cO'[\Psi]).
\end{equation}

The proof that this limit exists in Section 9.5.1 of \autocite{CG2} relies only on the combinatorics of the relevant Feynman diagrams and the support properties of $I[\Phi]$, $\cO$, and $\cO'$. 
In both of these respects, the situation at hand is identical to the one in \autocite{CG2}; hence, this limit exists also in our case. 
The arguments in Section 9.5.1 of \autocite{CG2} then give:

\begin{theorem}
\label{thm: quantumFA}
The assignment $U\mapsto \Obq_{\sE,\sL}(U)$ gives a prefactorization algebra.
\end{theorem}

To show that this assignment is also a factorization algebra, we will prove the following proposition (cf. Proposition 9.6.1.1 of \autocite{CG2}):

\begin{proposition}
\label{prop: assocgradedFA}
Filter $\Obq$ by defining $F^k\Obq_{\sE,\sL}(U)$ to be the subspace of observables which vanish to order $k$ in $\hbar$. There is an isomorphism of prefactorization algebras
\begin{equation}
\text{Gr} \Obq \cong \Obcl \otimes_\CC \CC[\![\hbar]\!];
\end{equation}
hence, $\text{Gr}\Obq$ is a factorization algebra.
\end{proposition}

\begin{ucorollary}
The quantum observables $\Obq$ form a factorization algebra.
\end{ucorollary}

The proof of Proposition \ref{prop: assocgradedFA} depends in turn on the following proposition, whose proof in the case at hand is the only point of our construction of factorization algebras which differs from the corresponding proof in \autocite{CG2}, and depends on Lemma \ref{lem: classRG}.

\begin{proposition}
Let $\Obq_{(0)}$ denote the truncation of $\Obq$ to order $\hbar^0$, i.e. $\Obq_{(0)}(U)$ is the space of observables $\{\cO[\Phi]\}$ related to each other by the $\hbar^0$ part of the HRG flow equation; the differential on $\Obq_{(0)}(U)$ is induced from the differentials $Q+\{I_0[\Phi],\cdot\}_\Phi$, where $I_0[\Phi]$ is the $\hbar^0$ component of the interaction; etc. 
Then there is an isomorphism
\begin{equation}
 \Obcl \to \Obq_{(0)}
\end{equation}
of prefactorization algebras on $M$. 
\end{proposition}

\begin{proof}
For each $\Phi$, we construct a map 
\begin{equation}
cl\rgobs{0}{\Phi}: \Obcl(U) \to \Obq_{(0)}(U),
\end{equation}
which we call ``classical observable HRG flow from scale 0 to scale $\Phi$.''  Given $\cO \in \Obcl(U)$, $cl\rgobs{0}{\Phi}(\cO)$ has a graphical representation as a sum over connected trees with one special vertex (corresponding to $\cO$). The edges are decorated by $P(\Phi)$. The classical observable HRG flow from scale 0 to scale $\Phi$ bears the same relation to observable HRG flow that the map of Lemma \ref{lem: classRG} bears to HRG flow of action functionals, and a similar argument shows that it is well-defined and a cochain map. Given $\cO\in \Obcl(U)$, the collection $\{cl\rgobs{0}{\Phi}\cO\}_\Phi$ satisfies the observable HRG flow modulo $\hbar$. Hence, we have a cochain isomorphism 
\begin{equation}
\Obcl(U)\to \Obq_{(0)}(U)
\end{equation}
for all $U$. The same argument as in \autocite{CG2} shows that the map respects factorization products. 
\end{proof}

%% file: Chapters/Chapter5.tex
\chapter{BF Theory on $\RR_{\geq 0}$: An Example}
\label{chap: examples}

In this chapter, we consider the quantization of one-dimensional BF theory on the manifold $M=\RR_{\geq 0}$.
We will impose the $B$ boundary condition at $t=0$.
(See Examples \ref{ex: bf} and \ref{ex: bfbdycond} for the relevant definitions.)
One-dimensional BF theory is topological mechanics valued in the symplectic formal moduli problem $T^* B\fg$.
However, because BF theory involves the bracket on $\fg$, we find that we will need the full interacting formalism of Chapter \ref{chap: interactingquantum} to construct factorization algebras of observables.
Hence, this example will be a perfect one on which to cut our teeth, and on which to test the predictions of the general theory.

The main theorems of this chapter 1) show that a quantization of this bulk-boundary system exists and 2) describe the observables of the bulk-boundary system in terms of familiar representation-theoretic objects associated to $\fg$.

We note that this bulk-boundary system has been studied by Alekseev and Mn\"ev in \autocite{alekseevmnev}, so we make no claims of being the first to study this system.
What \emph{is} novel, however, are the heat kernel and factorization-algebraic methods used to study it.

\begin{theorem}
\label{thm: 1dbfquantization}
Let $\fg$ be a uni-modular Lie algebra (i.e. the trace of $x\in \fg$ in the adjoint representation vanishes for all $x$).
Let $I$ be the classical interaction of one-dimensional $\fg$-BF theory with $B$ boundary condition at $t=0$. 
Choose the standard Riemannian metric on $\RR_{\geq 0}$ and choose $Q^{GF}=d^*$ for this choice of metric.
Then, the collection of action functionals
\begin{equation}
I[L]=\lim_{\epsilon\to 0}\rgflow{\epsilon}{L}{I}
\end{equation}
defines a quantum bulk-boundary system.
\end{theorem}

It is a standard argument to show that $I[L]$ is only polynomial in $\hbar$.
(In fact, it contains only $\hbar^0$ and $\hbar^1$ terms.)
Moreover, observable RG flow and the differential on the quantum observables have the following properties:
\begin{enumerate}
    \item They either decrease or keep the same the degree of polynomial dependence on the $B$-fields of observables.
    \item If they increase the degree of dependence of an observable on $\hbar$, they decrease the degree of dependence on the $B$-field observables.
\end{enumerate}
Hence, we may form a factorization algebra $\Obq_{\sE,\sL,poly}$ of quantum observables which have a polynomial dependence on $\hbar$ and on the $B$-fields.

\begin{theorem}
\label{thm: 1dbfobs}
Let $\fg$ be a unimodular Lie algebra (i.e. the trace of $x\in \fg$ over the adjoint representation is 0 for all $x$).
Then, for the one-dimensional BF bulk-boundary system, we have quasi-isomorphisms (of pointed chain complexes)
\begin{equation}
C_\bullet(\fg) \to \Obq_{\sE,\sL,poly}(\RR_{\geq 0})\otimes_{\RR[\hbar]}\RR_{\hbar=1}
\end{equation}
and 
\begin{equation}
C^\bullet(\fg,C_\bullet(\fg)) \to \Obq_{\sE,\sL,poly}(\RR_{>0})\otimes \RR_{\hbar=1}.
\end{equation}
\end{theorem}

\begin{remark}
Let us comment on how we expect Theorems \ref{thm: 1dbfquantization} and \ref{thm: 1dbfobs} to change if the unimodularity hypothesis on $\fg$ is lifted.
First of all, we believe the results of Theorem \ref{thm: 1dbfquantization} to stand unchanged.
As for the statement concerning the observables, we string together the following observations.
First, note that that similar comptuations to those performed in Section \ref{subsec: bfthyloc} show that the complex of local functionals for this bulk-boundary system is equivalent to $C^\bullet(\fg)[1]$.
Note that the zeroth cohomology of the complex of functionals gives the space of deformations of a given quantum theory, and the first cohomology of the complex gives the space of possible obstructions to quantization.
(In particular, for $\fg$ semi-simple, there exists a unique quantization of the bulk-boundary system.)
For non-unimodular $\fg$, the trace in the adjoint representation gives a one-parameter family of deformations of any quantization of our bulk-boundary system at hand.
These quantizations correspond to a modification of the local interaction functional at order $\hbar$ by an interaction term
\[
I_\omega(\alpha) = \int_{\RR_{\geq 0}} \omega(\alpha).
\]
On the purely algebraic side of the question, given $\omega\in H^1(\fg)$, we can define $C_{\bullet, \omega}(\fg)$ by ``twisting'' the Chevalley-Eilenberg complex by $\omega$.
Then, we expect that the quantization of the bulk-boundary system with interaction $I+\hbar I_\omega$ to give rise to the complexes $C_{\bullet, \omega+\tr_\fg}(\fg)$ and $C^\bullet(\fg, C_{\bullet, \omega+\tr_\fg}(\fg))$ on the boundary and in the bulk, respectively.
\end{remark}

\begin{remark}
\label{rmk: algstructure}
Before setting $\hbar=1$, both $\Obq_{\sE,\sL,poly}(\RR_{\geq 0})$ and $\Obq_{\sE,\sL,poly}(\RR_{>0})$ possess $\mathrm{BD}_0$ structures.
We believe, though do not show, that Theorem \ref{thm: 1dbfobs} may be refined to a statement about $\mathrm{BD}_0$ algebras.

From a different perspective, the factorization algebra
\begin{equation}
\Obq_{\sE,\sL,poly}|_{\RR_{>0}}
\end{equation}
will be manifestly locally constant, hence will define an $E_1$ algebra, or equivalently a $dg$ associative algebra.
We believe that the resulting algebra structure on $C^\bullet(\fg,C_\bullet(\fg))$ is the following one, which appears in \autocite{KSBRS}.
Let $\{t^a\}_{a=1}^\ell$ form a basis for $\fg$ and $\{t_a\}$ the dual basis.
Consider the graded algebra $\textrm{Cl}(\fg)$ whose generators are 
\begin{equation}
\{e^a,e_b\}_{a,b=1}^\ell.
\end{equation}
The $e^a$ reside in cohomological degree --1 and the $e_a$ reside in cohomological degree +1.
The relations in $\textrm{Cl}(\fg)$ are simply the Clifford relations
\begin{equation}
[e^a,e_b] = 2 \delta^a_b.
\end{equation}
The set $\{e^a,e_b\}$ forms an ordered basis for $\fg\oplus \fg^\vee$; hence, we have a PBW isomorphism
\begin{equation}
\textrm{Cl}(\fg)\cong \Sym(\fg[1]\oplus \fg^\vee[-1]);
\end{equation}
the latter object is the underlying graded vector space for $C^\bullet(\fg, C_\bullet(\fg))$.
Define the structure constants $f^{ab}_c$ by
\begin{equation}
[t^a,t^b]= f^{ab}_c t^c
\end{equation}
(we use the Einstein summation convention here).
One may check by direct computation that---under the assumption of unimodularity of $\fg$---the differential 
\begin{equation}
\left [f^{ab}_c\,e_a\,e_b\,e^c,\cdot \right]
\end{equation}
on $\mathrm{Cl}(\fg)$
corresponds to the Chevalley-Eilenberg differential on $C^\bullet(\fg, C_\bullet(\fg))$ under the PBW isomorphism.

In this way, the complex $C^\bullet(\fg, C_\bullet(\fg))$ becomes a dg associative algebra.
Let $A$ be this dg associative algebra, and let $\cF_A$ be the locally constant factorization algebra on $\RR_{>0}$ which assigns $A$ to any connected open subset of $\RR_{>0}$.

\begin{conjecture}
 There is an equivalence of factorization algebras (or a zigzag of equivalences)
\begin{equation}
\cF_A \to \left.\Obq_{\sE,\sL,poly}\right|_{\RR_{>0}}
\end{equation}
which refines the result of Theorem \ref{thm: 1dbfobs}.
\end{conjecture}

To prove such a conjecture would require a detailed study of the factorization product in the bulk, which we do not do here.
\end{remark}

The following higher-dimensional generalizations of Theorems \ref{thm: 1dbfquantization} and \ref{thm: 1dbfobs} seem natural from this perspective, and we present them as a conjecture.

BF theory with $B$ boundary condition is a well-defined bulk-boundary system on $\HH^n$ for any $n$.
Moreover, let $\pi: \HH^n\to \RR^{n-1}$ denote the projection to the boundary.

\begin{conjecture}
BF theory with the $B$ boundary condition possesses a quantization on $\HH^n$ for any $n$.
Further, there are quasi-isomorphisms:
\begin{align*}
    \cF_{\fg,n-1}\to \pi_*\Obq_{\sE,\sL,poly}\otimes_{\RR[\hbar]}\RR_{\hbar=1}
\end{align*}

\[
 HH^\bullet(\cF_{\fg, n-1})\to \pi_*\left(\left.\Obq_{\sE,\sL,poly}\right|_{\RR^{n-1}\times \RR_{> 0}}\otimes_{\RR[\hbar]}\RR_{\hbar=1}\right)
\]
of factorization algebras on $\RR^{n-1}$, where $\cF_{\fg,n-1}$ is the $(n-1)$-factorization envelope of $\fg$, introduced in Example \ref{ex: factenvelope} and $HH^\bullet(\cF_{\fg, n-1})$ is a factorization-algebraic model for the $E_{n-1}$ center of this locally constant factorization algebra.
\end{conjecture} 

We hope that similar techniques to the ones used here can be used to prove these claims.

We devote the remainder of the chapter to the proof of Theorems \ref{thm: 1dbfquantization} and \ref{thm: 1dbfobs}.
We start in Section \ref{sec: 1dbfkernelprop} by writing down explicit formulae for the heat kernels and propagators for the bulk-boundary system.
In Section \ref{sec: 1dbfquant}, we prove Theorem \ref{thm: 1dbfquantization}.
In Section \ref{sec: 1dbfobs}, we prove Theorem \ref{thm: 1dbfobs}.

\section{Extracting the Heat Kernel and Propagators}
\label{sec: 1dbfkernelprop}
In this section, we write down explicit formulae for the heat kernel and propagators for 1D BF theory on $\RR_{\geq 0}$ with $B$ boundary condition.

The underlying free theory of one-dimensional BF theory is topological mechanics with values in the symplectic vector space $\fg[1]\oplus \fg^\vee[-1]$.
Hence, we have that~$\Mdbl\cong \RR$ and 
\begin{equation}
\Edbl = \Omega^\bullet_\RR\otimes(\fg[1]\oplus \fg^\vee[-1]).
\end{equation}
The involution $\involution$ simply sends $t\mapsto -t$ in $\Mdbl$.
At the level of $\Edbl,$ $\involution$ acts by $-\involution^*$ on the $A$ fields and by $\involution^*$ on the $B$ fields.

Choosing the standard Euclidean metric on $\RR_{\geq 0}$ (and hence, on $\RR$), we define $Q^{GF}=\delta$, where $\delta$ is the formal adjoint to the de Rham differential.

Hence, $[Q,Q^{GF}]$ is the standard Hodge Laplacian for the Euclidean metric on $\RR_{\geq 0}$ and $\RR$.

Let us choose a basis $\{t^a\}$ for $\fg$ and choose $\{t_a\}$ its dual basis. We write 
\begin{equation}
\tdownup:= \sum_a t_a\otimes t^a
\end{equation}
and 
\begin{equation}
\tupdown:= \sum_a t^a\otimes t_a.
\end{equation}
Under the identifications $\fg^\vee \otimes \fg\cong \Hom(\fg, \fg)$ and $\fg\otimes \fg^\vee \cong \Hom(\fg^\vee,\fg^\vee)$, $\tdownup$ and $\tupdown$ correspond to the identity operators on the respective spaces.
However, given that one may also identify $\Hom(\fg,\fg)\cong \Hom(\fg^\vee,\fg^\vee)$ (via the transpose map), we have introduced the symbols $\tupdown$ and $\tdownup$ to make clear to which spaces the symbols most naturally belong (namely $\fg\otimes \fg^\vee$ and $\fg^\vee\otimes \fg$, respectively).

Having made these identifications, and letting $f_0$ be a compactly-supported even function on $\RR$ which is identically 1 in a neighborhood of 0, we can construct the following smooth section of $\Edbl\boxtimes\Edbl$ on $\RR^2$:
\begin{equation}
\widetilde K_L(t_1,t_2):= \frac{f_0(t_1-t_2)}{\sqrt{4\pi L}}e^{-\frac{(t_1-t_2)^2}{4L}}(dt_1-dt_2)(\tupdown+ \tdownup).
\end{equation}

It is straightforward to verify that $K_L$ is a fake heat kernel for the operator $[ \diffonnocrossdiff ,Q^{GF}]$ on $\RR$. 
We may therefore form the parametrix
\begin{equation}
\widetilde \Phi_L(t_1,t_2) =\int_0^L \widetilde K_{L'}\d L'.
\end{equation}
and the propagator
\begin{align}
\widetilde P(\epsilon, L) &= \int_{\epsilon}^L (Q^{GF}\otimes 1) K_{L'}\d L'\\
&= \frac{f_0(t_1-t_2)(t_1-t_2)}{\sqrt{4\pi}}(\tupdown+\tdownup) \int_\epsilon^L \frac{e^{-\frac{(t_1-t_2)^2}{4L'}}}{(L')^{3/2}}\d L'.
\end{align}

Upon taking $\epsilon\to0$ and $L\to\infty$, we find
\begin{equation}
\widetilde P(0,\infty) = f_0(t_1-t_2)\sgn(t_1-t_2)(\tupdown+\tdownup),
\end{equation}
which coincides with what we found in Example \ref{ex: toplmechprop}.

The following formula explicitly describes the $\involution$-invariant propagator corresponding to $\widetilde P(\epsilon,L)$:
\begin{align}
\label{eq: 1dbfprop}
P(\Phi_L-\Phi_\epsilon)&= \frac{f_0(t_1-t_2)(t_1-t_2)}{4\sqrt{\pi}}(\tupdown+\tdownup) \int_\epsilon^L \frac{e^{-\frac{(t_1-t_2)^2}{4L'}}}{(L')^{3/2}}\d L' \nonumber\\
&+\frac{f_0(t_1+t_2)(t_1+t_2)}{4\sqrt{\pi}}(\tupdown-\tdownup) \int_\epsilon^L \frac{e^{-\frac{(t_1+t_2)^2}{4L'}}}{(L')^{3/2}}\d L'.
\end{align}

Finally, we write the BV heat kernel associated to these propagators:
\begin{align}
\label{eq: 1dbfkernel}
    K_{\Phi_L}(t_1,t_2)&= \frac{f_0(t_1-t_2)}{\sqrt{4\pi L}}e^{-\frac{(t_1-t_2)^2}{4L}}(dt_1-dt_2)(\tupdown+ \tdownup)\nonumber\\
    &+ \frac{f_0(t_1+t_2)}{\sqrt{4\pi L}}e^{-\frac{(t_1+t_2)^2}{4L}}(dt_1+dt_2)(\tupdown- \tdownup)+ R(t_1,t_2,L),
\end{align}
where $R$ is a smooth, properly-supported function on $(\RR_{\geq 0})^3$ which vanishes to all orders in $L$ as $L\to 0$.
More precisely, for all $d\in \ZZ$,
\begin{equation}
    \label{eq: remaindervanishing}
    \lim_{L\to 0}\frac{1}{L^d}R(\cdot,\cdot,L) = 0
\end{equation}
in the locally convex topological vector space of functions on $(\RR_{\geq 0})^2$.

Equations \eqref{eq: 1dbfprop} and \eqref{eq: 1dbfkernel} will be used extensively to prove Theorems \ref{thm: 1dbfquantization} and \ref{thm: 1dbfobs}.

\section{Quantization of the Bulk-Boundary System: Proof of Theorem \ref{thm: 1dbfquantization}}
\label{sec: 1dbfquant}
The proof of Theorem \ref{thm: 1dbfquantization} can be broken into two propositions:

\begin{proposition}
\label{prop: 1dbffinite}
The limit
\begin{equation}
I[L]=\lim_{\epsilon\to 0}\rgflow{\epsilon}{L}{I}
\end{equation}
exists in the topology underlying the convenient vector space of functionals.
\end{proposition}

It follows directly from their definition that the interaction functionals $I[L]$ satisfy the RG flow equation and the asymptotic locality condition of Definition \ref{def: QTNBFT}.
Hence, Theorem \ref{thm: 1dbfquantization} follows directly from Proposition \ref{prop: 1dbffinite} and the following Proposition:
\begin{proposition}
\label{prop: 1dBFQME}
The family of interactions $\{I[L]\}$ satisfies the QME.
\end{proposition}

\begin{proof}[Proof of Proposition \ref{prop: 1dbffinite}]
As is conventional for BF theory, we view the diagrams appearing in 
\begin{equation}
\rgflow{\epsilon}{L}{I}
\end{equation}
as being directed; all vertices have two ``in'' half-edges and one ``out'' half-edge.
It follows from this description of the diagrams appearing in $\rgflow{\epsilon}{L}{I}$ that only tree and one-loop diagrams appear therein.
Lemma \ref{lem: classRG} shows that the the tree diagrams appearing in $\rgflow{\epsilon}{L}{I}$ admit an $\epsilon\to 0$ limit.
It remains to show that the one-loop diagrams appearing are finite.
Because we are in one dimension, we obtain a further simplification: 
since the propagator is a zero-form on $M\times M$, and since each trivalent vertex requires at least one one-form input, the only one-loop diagrams appearing in $\rgflow{\epsilon}{L}{I}$ are wheels, i.e. one-loop diagrams which 
remain connected when one edge is severed.
Such diagrams are specified by their number of internal vertices, for which we use the symbol $k$.
Figure \ref{fig: bfwheel} shows one such diagram, with $k=3$.

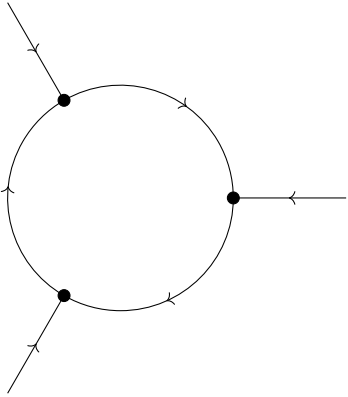
\begin{figure}[h]
    \centering
\begin{tikzpicture}[scale=1.5]
         \node at (0,1){};
         \draw[fill=black] (0,0) circle (1.5pt);
         \draw[fermion] (0,0) arc (0:-120:1cm) -- +(0,0) node (g){};
         \draw[fill=black] (g.center) circle (1.5pt);
         \draw[fermion] (g.center) arc (-120:-240:1cm) -- +(0,0) node (i){};
         \draw[fermionbar] (g.center) -- +(-120:1);
         \draw[fermion] (i.center) arc (120:0:1cm);
         \draw[fermionbar] (i.center) -- +(120:1);
         \draw[fermionbar] (0,0) -- +(0:1);
         \draw[fill=black] (i.center) circle (1.5pt);
         \end{tikzpicture}
    \caption{A typical wheel appearing in one-dimensional BF theory. Here, $k=3$.}
    \label{fig: bfwheel}
\end{figure}

Such diagrams require the placement of a compactly-supported, $\fg^{\otimes k}$-valued $k$-form $(f \d t_1\cdots \d t_k) \otimes \kappa$ as an external input.
(Here $\kappa\in \fg^{\otimes k}$.)
Further, each propagator is a sum of two-terms, one a function of $t_1-t_2$, and the other a function of $t_1+t_2$.

The weight $w_{k}(\epsilon, L, \{(f \d t_1 \cdots \d t_k)\otimes \kappa \}_{i=1}^k)$ can be written as a product
\begin{equation}
w_{\gamma,k}^{an}(\epsilon, L, f) w_{\gamma,k}^{alg}(\kappa),
\end{equation}
where $w_{\gamma,k}^{alg}(\kappa)$ depends only on Lie algebraic information in $\fg$.

Let 
\begin{equation}
\beta \in \mathrm{Set}( \{1,\ldots, k\} , \{-1,1\})
\end{equation}
be any function. 
For each $i\in \{1,\ldots, k\}$, let $u_i = t_i - \beta(i) t_{i+1}$, with the convention that $t_{k+1}=t_1$.
The analytic weight associated to a diagram with $k$ vertices is therefore a sum--as $\beta$ varies over all possible functions 
\begin{equation}
\{1,\ldots, k\}\to \{-1,1\}
\end{equation}
associating a sign to each internal edge of the wheel diagram--of weights of the form:
\begin{align}
&w^{an}_{k,\beta}(\epsilon,L, f):=\nonumber\\
&\int_{[\epsilon,L]^k}\int_{(\RR_{\geq 0})^k} f(t_1,\ldots, t_k) \left(\prod_{i=1}^k \frac{u_i f_0(u_i)}{\sqrt{4\pi}}\times \frac{\exp{\left( -\frac{\left(u_i\right)^2}{4L_i}\right)}}{(L_i)^{3/2}} \right)\,  \d t_1 \, \ldots \, \d t_k \, \d L_1\, \ldots \, \d L_k.
\end{align}

Let us perform the integral over the $L_i$ variables:
\begin{equation}
\label{eq: 1dbfpropexplicitly}
\int_\epsilon^L \frac{\exp{\left( -\frac{\left(u_i\right)^2}{4L_i}\right)}}{(L_i)^{3/2}} \d L_i= \frac{2}{u_i}\int_{u_i/(2\sqrt{L})}^{u_i/(2\sqrt{\epsilon})} e^{-v^2} \d v.
\end{equation}

We therefore find
\begin{equation}
\label{eq: weightbound}
 w^{an}_{k,\beta}(\epsilon,L, f) = \frac{1}{\pi^{k/2}}\int_{(\RR_{\geq 0})^k}f(t_1,\ldots, t_k)  \left(\prod_{i=1}^k f_0(u_i) \int_{u_i/(2\sqrt{L})}^{u_i/(2\sqrt{\epsilon})} e^{-v^2} dv \right)\, \d t_1 \cdots \d t_k.
\end{equation}
We claim that 
\begin{equation}
\left(\lim_{\epsilon \to 0}w^{an}_{k,\beta}(\epsilon,L)\right)(f)=\frac{1}{\pi^{k/2}}\int_{(\RR_{\geq 0})^k}\left(f(t_1,\ldots, t_k)  \prod_{i=1}^k f_0(u_i) \int_{u_i/(2\sqrt{L})}^{\infty} e^{-v^2} dv \right)\, \d t_1 \cdots \d t_k.
\end{equation}
Let us denote by $w^{an}_{k,\beta}(0,L)$ the distribution appearing on the right-hand side of the above equation.
Of course, it is intuitively clear that $w^{an}_{k,\beta}(0,L)$ is the $\epsilon \to 0$ limit of $w^{an}_{k,\beta}(\epsilon, L)$, and it is certainly true that 
\begin{equation}
\left(\lim_{\epsilon \to 0}w^{an}_{k,\beta}(\epsilon,L,f)\right)=w^{an}_{k,\beta}(0,L)(f)
\end{equation}
for each $f$.
However, we need to show that this convergence is uniform as $f$ varies over a bounded subset of $\condfieldscs^{\hotimes_\beta k}$.
On a bounded subset $B$ of $\condfieldscs^{\hotimes_\beta k}$, there exist  $C>0$ and a compact $K\subset M^k$ so that
\begin{equation}
\supp(g) \subset K, \quad |g|< C
\end{equation}
for all $g\in B$.
Let us choose $\delta>0$ and divide the integral in Equation \eqref{eq: weightbound} into one over $k+1$ regions
\begin{enumerate}
    \item The region $\cR$ consisting of points for which $|u_i|>\delta$ for all $i$,
    \item The $k$ regions $\cS_1,\ldots, \cS_k$, where $\cS_i$ consists of those points for which $|u_i|\leq \delta$.
\end{enumerate}
Note that $\cR$ is disjoint from each of the $\cS_i$, but the $\cS_i$ are not pairwise disjoint.
We may form the bound:
\begin{align}
\pi^{k/2}\left| w^{an}_{k,\beta}(\epsilon,L)(f)- w^{an}_{k,\beta}(0,L)(f)\right|&\leq\nonumber\\
\int_{\cR}|f(t_1,\ldots, t_k)| \prod_{i=1}^k |f_0(u_1)| \int_{u_i/(2\sqrt \epsilon)}^\infty e^{-v^2}&dv+\sum_i\int_{\cS_i}|f(t_1,\ldots, t_k)| \prod_{i=1}^k |f_0(u_1)| \int_{u_i/(2\sqrt \epsilon)}^\infty e^{-v^2}dv \nonumber\\
 \leq C\mathrm{vol}(K)&\prod_{i=1}^k\left(\int_{\delta/(2\sqrt{\epsilon})}^\infty e^{-v^2}dv\right)\,+\,\pi^{k/2}\sum_i \int_{\cS_i\cap K} C .
\end{align}
We may assume $K$ is a hypercube of side-length $z$.
Then, we have the bound
\begin{equation}
\mathrm{vol}(\cS_i)\cap K \leq 2\sqrt{2}\delta z^{k-1}.
\end{equation}
Thus, by choosing $\delta$ sufficiently small, we can make the integral over the $\cS_i$ as small as we like.
Having fixed such a $\delta$, we may also make a choice of sufficiently small $\epsilon$ to make the integral over $\cR$ sufficiently small.
These choices required only knowledge of $K$ and $C$, which are uniform constants for all $g\in B$, i.e. these are uniform choices for the given bounded subset of $\condfieldscs^{\hotimes_\beta k}$ which we are considering.
This completes the proof of the Proposition.
\end{proof}

\begin{proof}[Proof of Proposition \ref{prop: 1dBFQME}]
At each finite $L$, we may form the scale-$L$ obstruction
\begin{equation}
\cO[L]:= \frac{1}{\hbar}\left( QI[L] + \frac{1}{2}\{I[L],I[L]\}_{L}+\hbar \Delta_L I[L]\right).
\end{equation}
As we have seen, the obstruction $\cO[L]$ satisfies a version of classical RG flow.
Hence, it admits a limit 
\begin{equation}
\cO:= \lim_{L\to 0}\cO[L]
\end{equation}
which is local.
The argument which gives the equation immediately preceding Equation (C.3) of \autocite{LiLi} shows that 
\begin{align}
\cO &= \lim_{L\to 0}\lim_{\epsilon\to 0} \frac{d}{d\delta} \rgflow{\epsilon}{L}{I+\delta\left( QI +\frac{1}{2}\{I,I\}_\epsilon+\hbar \Delta_\epsilon I\right)}\nonumber\\
&=\lim_{L\to 0}\lim_{\epsilon\to 0} \frac{d}{d\delta} \rgflow{\epsilon}{L}{I+\delta\left( \frac{1}{2}\{I,I\}_\epsilon+\hbar \Delta_\epsilon I\right)}\nonumber\\
&= \lim_{L\to 0}\lim_{\epsilon\to 0} \frac{d}{d\delta} \rgflow{\epsilon}{L}{I+\delta\frac{1}{2} \{I,I\}_\epsilon} \label{eq: uvobstr}
\end{align}
where for the second equality we have used that $QI=0$ for BF theory (and in fact, for any classical theory which can be described by a dgla as opposed to an $L_\infty$-algebra).
In the third equality, we have used the uni-modularity of $\fg$ to show that $\Delta_\epsilon I=0$.
From this description of the obstruction, it follows that we need to consider--as in the proof of Proposition \ref{prop: 1dbffinite}--weights associated to wheel diagrams with at least 2 vertices, and one edge marked with $K_\epsilon$ rather than $P(\epsilon,L)$.
In the case at hand, it is now possible to introduce ``one-particle reducible'' diagrams, i.e. a wheel with a tree attached to it.
In other words, the diagrams appearing in the obstruction are wheels with a tree attached to at most one internal vertex.
However, it suffices to show that the obstruction vanishes by considering wheels, since we know trees to give finite contributions.
As in the proof of Proposition \ref{prop: 1dbffinite}, we are left to consider weights of the form
\begin{align}
O^{an}_{k,\beta}(\epsilon,L, f)&:=\int_{[\epsilon,L]^{k-1}}\int_{(\RR_{\geq 0})^k} f_1(t_1,
\ldots, t_k) \left(\frac{f_0(u_k)\exp{\left(-\frac{(u_k)^2}{4\epsilon}\right)}}{\sqrt{4\pi \epsilon}}- R(t_k,t_1,\epsilon)\right)\nonumber
\\\times\prod_{i=1}^{k-1}& \left(u_i\frac{f_0(u_i)}{\sqrt{4\pi}}\times \frac{\exp{\left( -\frac{\left(u_i\right)^2}{4L_i}\right)}}{(L_i)^{3/2}} \right) \d t_1 \, \ldots \, \d t_k\,\d L_1\, \ldots \, \d L_{k-1}
\end{align}
Because $K_\epsilon$ is a one-form, the external half-edge incident on either the $k$-the vertex or the first vertex corresponds to a zero-form $A$-field input.
Let us assume without loss of generality that this is the case for the $k$-th input.
It follows that $f(t_1,\ldots, t_k) = t_k g(t_1,\ldots, t_k)$ for some continuous function $g$.
Let us first assume that $k>2$, and let us perform the integrals over the $L_i$.
We find that
\begin{align}
O^{an}_{k,\beta}(\epsilon,L, f)=&\nonumber\\
\int_{(\RR_{\geq 0})^k}\, &f(t_1,\ldots, t_k)\left(\prod_{i=1}^{k-1} \frac{f_0(u_i)}{\sqrt{\pi}}\int_{u_i/\sqrt{4L}}^{u_i/\sqrt{4\epsilon}}e^{-v^2}dv\right)\nonumber\\ 
&\Bigg(\frac{f_0(u_k)\exp{\left(-\frac{(u_k)^2}{4\epsilon}\right)}}{\sqrt{4\pi \epsilon}}- R(t_k,t_1,\epsilon)\Bigg)\d t_1 \, \ldots \, \d t_k, 
\end{align}
where $u_i = t_{i}-\beta(i)t_{i+1}$.
The remainder $R(t_k,t_1,\epsilon)$ is smooth in its variables and has proper support. 
It and all its derivatives vanish to all orders in $\epsilon$ uniformly in $t_1,t_k$ as $\epsilon\to 0$, cf. Equation \eqref{eq: remaindervanishing}.
Since we have the bound
\begin{equation}
\label{eq: crudebound}
\frac{f_0(u_i)}{\sqrt{\pi}}\int_{u_{i}/2\sqrt{L}}^{u_{i}/2\sqrt{\epsilon}}e^{-v^2}dv \leq 1,
\end{equation}
the part of $O^{an}_{k,\beta}(\epsilon, L, \{f_i\})$ involving $R(t_k,t_1,\epsilon)$ vanishes to all orders in $\epsilon$ as $\epsilon\to 0$.

Next, let us consider
\begin{align}
&\left|\int_0^\infty  f(t_1,\ldots, t_k) f_0(u_{k-1})f_0(u_k)\frac{\exp{\left(-\frac{(u_k)^2}{4\epsilon}\right)}}{\sqrt{4\pi \epsilon}}\d t_k\,\right|\nonumber\\
&\leq \frac{1}{\sqrt{4\pi\epsilon}}\int_0^\infty |f(t_1,\ldots t_k) f_0(u_{k-1})f_0(u_k)|\, e^{-\frac{u_k^2}{4\epsilon}}\d t_k \,\leq \frac{C}{\sqrt{4\pi \epsilon}} \int_{-\infty}^\infty  e^{-\frac{u_k^2}{4\epsilon}}\d u_k\,\leq \frac{C}{2},
\label{eq: boundoftkint}
\end{align}
where $C$ is the maximum value of $|f f_0 f_0|$ on $(\RR_{\geq 0})^k$.
Let $K$ be the support of $f$; we use the crude bound of Equation \eqref{eq: crudebound} for $i=k-1$ and the bound of Equation \eqref{eq: boundoftkint} to obtain
\begin{align}
|O^{an}_{k,\beta}(\epsilon,L,f)|\leq \frac{C'}{2}\int_{K} \left( \prod_{i=1}^{k-2}\frac{1}{\sqrt{\pi}} \int_{u_i/\sqrt{4L}}^{u_i/\sqrt{4\epsilon}}e^{-v^2}dv\right)\d t_1\cdots \d t_{k-1} \,.  
\end{align}
Proceeding very similarly to the way we did in the proof of Proposition \ref{prop: 1dbffinite}, we find that the $\epsilon\to 0$, $L\to 0$ limit of the above is 0.

It remains to consider the integral
\begin{equation}
\label{eq: 2legbfint}
\int_{[\epsilon,L]}  \int_{(\RR_{\geq 0})^2} f(t_1, t_2)\frac{f_0(u_1)f_0(u_2) u_1}{4\pi \sqrt{\epsilon (L_1)^3}}\exp{\left(-\frac{(u_1)^2}{4L_1}-\frac{(u_2)^2}{4\epsilon}\right)}\d t_1\, \d t_2 \, \d L_1 \,.
\end{equation}
There are four cases to consider here, corresponding to the four signs in $u_1=t_1\pm t_2$, $u_2 =t_2\pm t_1$.
We divide these cases into three classes:
\begin{align*}
    \beta(1) &= \beta(2) \tag{C1}\\
    \beta(1) &= 1 = -\beta(2) \tag{C2}\\
    \beta(1)&= -1 = -\beta(2) \tag{C3}
\end{align*}
Let us suppose that we are in case (C1).
Then
\begin{equation}
\label{eq: sitrick}
\frac{u_1}{L_1}\exp\left( -\frac{(u_1)^2}{4L_1}-\frac{(u_2)^2}{4\epsilon}\right)=-\frac{\epsilon}{\epsilon+L_1}\frac{\del}{\del t_1} \exp\left( -\frac{(u_1)^2}{4L_1}-\frac{(u_2)^2}{4\epsilon}\right).
\end{equation}
The same identity holds (up to, possibly, a minus sign, depending on $\beta(1)$) if one replaces the $t_1$-derivative with a $t_2$ derivative.
It is at this point that we remember that $f(t_1,t_2)$ vanishes either at $t_1=0$ or $t_2=0$.
Let us assume without loss of generality that it vanishes at $t_1=0$; the opposite case is handled by using the $t_2$-case of Equation \eqref{eq: sitrick}.
Let 
\begin{equation}
    \tau : = \frac{\epsilon}{\epsilon+L_1}\frac{\del}{\del t_1}.
\end{equation}
The quantity appearing in Expression \eqref{eq: 2legbfint} may therefore be written
\begin{equation}
\label{eq: 2legbfint2}
\int_{[\epsilon,L]}  \int_{(\RR_{\geq 0})^2} \tau \left(f(t_1, t_2)\frac{f_0(u_1)f_0(u_2)}{4\pi }\right)\frac{1}{\sqrt{\epsilon L_1}}\exp{\left(-\frac{(u_1)^2}{4L_1}-\frac{(u_1)^2}{4\epsilon}\right)}\d t_1\, \d t_2 \,\d L_1 \,,
\end{equation}
using an integration by parts; the boundary condition on $f$ prevents the existence of a boundary term.
The operator $\tau$ is  bounded, and therefore continuous on the LF-space $\condfieldscs^{\hotimes_\beta 2}$.
Hence, the term $\tau(\cdots)$ appearing in Equation \eqref{eq: 2legbfint2} may be bounded by a constant $C$ which is a uniform choice as $f$ varies over a bounded subset of $\condfieldscs^{\hotimes_\beta 2}$.
Let $K\subset \RR_{\geq 0}$ denote the volume of the closure of the set
\begin{equation}
\{c\in \RR\mid c= t_1+\beta(1) t_2, f(t_1,t_2) \neq 0\};
\end{equation}
the compact support of $f$ guarantees that $K<\infty$.
We therefore have the bound
\begin{equation}
|O^{an}_{k,\beta}(\epsilon, L, f)|\leq C K\int_{[\epsilon,L]}\int_{\RR} \frac{1}{\sqrt{\epsilon L_1}} \exp\left( -\left( \frac{1}{4L_1}+\frac{1}{4\epsilon}\right) (u_1)^2\right)\d u_1 \, \d L_1= 2KC \left( \sqrt{L+\epsilon} - \sqrt{2\epsilon}\right).
\end{equation}
It follows that, for $\beta$ belonging to class (C1), 
\begin{equation}
    \lim_{L\to 0}\lim_{\epsilon \to 0}|O^{an}_{k,\beta}(f)| =0.
\end{equation}
Let us now proceed to consider the cases (C2) and (C3).
Let us define, for both cases, $u= t_1+t_2$ and $v=t_1-t_2$.
(Hence, for (C2), we have $u_1=v$, $u_2=u$, and vice versa for (C3).)
Then, we are left to consider the integrals
\begin{align}
\int_{[\epsilon,L]}  \int_{(\RR_{\geq 0})^2} & f(t_1, t_2)\frac{f_0(v)f_0(u) v}{4\pi \sqrt{\epsilon (L_1)^3}}\exp{\left(-\frac{v^2}{4L_1}-\frac{u^2}{4\epsilon}\right)}\d t_1 \,\d t_2\,\d L_1 \, \tag{C2}\\
\int_{[\epsilon,L]}  \int_{(\RR_{\geq 0})^2} & f(t_1, t_2)\frac{f_0(u)f_0(v) u}{4\pi \sqrt{\epsilon (L_1)^3}}\exp{\left(-\frac{u^2}{4L_1}-\frac{v^2}{4\epsilon}\right)}\d t_1 \,\d t_2\,\d L_1 \,
\tag{C3}.\\\label{eq: 2legbfint3}\
\end{align}
We may bound the norm of the factors $f(t_1,t_2)f_0(u_1)f_0(u_2)/(4\pi)$ in the above integrals, and change to the $(u,v)$ coordinate system on $(\RR_{\geq 0}$, to find that the integrals in Equation \eqref{eq: 2legbfint3} are bounded by 
\begin{align}
&C\int_{[\epsilon, L]}\int_0^\infty \int_{-u}^u  \frac{|v|}{\sqrt{\epsilon(L_1)^3}}\exp{\left(-\frac{v^2}{4L_1}-\frac{u^2}{4\epsilon}\right)}\d v\, \d u\, \d L_1 \tag{C2}\\
&C\int_{[\epsilon, L]}\int_0^\infty \int_{-u}^u  \frac{u}{\sqrt{\epsilon(L_1)^3}}\exp{\left(-\frac{u^2}{4L_1}-\frac{v^2}{4\epsilon}\right)} \d v\, \d u \, \d L_1
\tag{C3}\\\label{eq: 2legbfint4}.
\end{align}
We find the following explicit value for the integral (C2) in Equation \eqref{eq: 2legbfint4}:
\begin{equation}
C' \int_{[\epsilon,L} \d L_1 \left( \frac{1}{\sqrt{L_1}}- \frac{1}{\sqrt{\epsilon+L_1}}\right),
\end{equation}
where $C'$ is some constant differing from $C$ by factors independent of $\epsilon, L,$ and $f$.
This integral manifestly vanishes once we take $\epsilon \to 0$ and then $L\to 0$.
For the case (C3), we may bound the integral by integrating over $v\in (-\infty, \infty)$ instead of $v\in (-u,u)$, and we find
\begin{equation}
C\int_{[\epsilon, L]}\int_0^\infty  \int_{-\infty}^\infty  \frac{u}{\sqrt{\epsilon(L_1)^3}}\exp{\left(-\frac{u^2}{4L_1}-\frac{v^2}{4\epsilon}\right)} = C''\int_{[\epsilon,L]}\frac{1}{\sqrt{L_1}}\d v \d u \d L_1= 2C'' \left( \sqrt{L}-\sqrt{\epsilon}\right),
\end{equation}
which manifestly vanishes under the $\epsilon \to 0, L\to 0$ limit.
This completes the proof.
\end{proof}

\section{The Observables of One-Dimensional BF Theory}
\label{sec: 1dbfobs}
In this section, we prove Theorem \ref{thm: 1dbfobs}. We break the Theorem up into two Propositions.

\begin{proposition}
\label{prop: 1dbfbdyobs}
For the one-dimensional BF theory with $B$ boundary condition, we have a quasi-isomorphism
\begin{equation}
C_\bullet(\fg) \to \Obq_{\sE,\sL,poly}(\RR_{\geq 0})\otimes_{\RR[\hbar]}\RR_{\hbar=1}
\end{equation}
of pointed cochain complexes.
\end{proposition}

\begin{proposition}
\label{prop: 1dbfblkobs}
For the one-dimensional BF theory with $B$ boundary condition, we have a quasi-isomorphism
\begin{equation}
C^\bullet(\fg,C_\bullet\fg) \to \Obq_{\sE,\sL,poly}(\RR_{>0})\otimes \RR_{\hbar=1}.
\end{equation}
of pointed cochain complexes.
\end{proposition}

\begin{proof}[Proof of Proposition \ref{prop: 1dbfbdyobs}]
Let us begin by defining a linear map
\begin{equation}
    \cO_{\ast}: \Sym(\fg[1]) \to \Obq_{\sE,\sL,poly}(\RR_{\geq 0})\otimes_{\RR[\hbar]}\RR_{\hbar=1}.
\end{equation}
Let, for $\omega\in \Sym^k(\fg[1])$, $\cO_{\omega}[0]$ denote the following function of order $k$ on the space of $B$ fields
\begin{equation}
\cO_{\omega}[0](\beta\otimes \left(\chi_1\wedge \cdots \wedge \chi_k\right)) = (-1)^{|\beta|k}(\iota_{(0,\ldots, 0)}^* \beta) (\chi_1\wedge \cdots \wedge \chi_k)(\omega),
\end{equation}
where $\beta$ is a form on $(\RR_{\geq 0})^k$, the $\chi_i\in \fg^\vee$, and $\iota_{(0,\ldots, 0)}$ is the inclusion of the origin into~$(\RR_{\geq 0})^k$.
We define
\begin{equation}
\cO_\omega[\Phi]= \lim_{\epsilon\to 0}\frac{\d}{\d \delta} \rgflow{\Phi_\epsilon}{\Phi}{I+\delta \cO_\omega[0]}.
\end{equation}
Lemma \ref{lem: 1dbffiniteobservables} shows that this limit exists.

Moreover, 
\begin{align}
 &\rgobs{\Phi}{\Psi}(\cO_\omega[\Phi]) \nonumber\\
&= \frac{\d}{\d\delta}\left( \rgflow{\Phi}{\Psi}{I[\Phi]+\delta \lim_{\epsilon\to 0}\frac{\d}{\d \delta'} \rgflow{\Phi_\epsilon}{\Phi}{I+\delta' \cO_\omega[0]}}\right)\nonumber\\
&= \frac{\d}{\d\delta}\left( \rgflow{\Phi}{\Psi}{\lim_{\epsilon\to 0}\rgflow{\Phi_\epsilon}{\Phi}{I}+\delta \lim_{\epsilon\to 0}\frac{\d}{\d \delta'} \rgflow{\Phi_\epsilon}{\Phi}{I+\delta' \cO_\omega[0]}}\right)\nonumber\\
& = \frac{\d}{\d\delta}\left( \rgflow{\Phi}{\Psi}{\lim_{\epsilon\to 0}\rgflow{\Phi_\epsilon}{\Phi}{I+\delta \cO_\omega[0]}}\right)\nonumber\\
& =\lim_{\epsilon \to 0}\frac{\d}{\d \delta}\left(\rgflow{\Phi_\epsilon}{\Psi}{I+\delta\cO_\omega[0]}\right)\nonumber\\
& = \cO_{\omega}[\Psi],
\end{align}
which shows that $\cO_\omega[\Phi]$ satisfies the RG flow equation for observables.

Now, let $\widehat Q_{\Phi}$ denote the total differential
\begin{equation}
\diffonnocrossdiff+\{I[\Phi],\cdot \}_{\Phi}+\Delta_\Phi
\end{equation}
on the scale $\Phi$ observables.

We would like to show that
\begin{equation}
\label{eq: CEdiffgoal}
\widehat{Q}_\Phi \cO_\omega [\Phi]= \cO_{d_{CE}\omega}[\Phi].    
\end{equation}
That is the content of Lemma \ref{lem: diff1dbfobs}, below.
This completes the construction of a cochain map
\begin{equation}
   \Upsilon: C_\bullet(\fg) \to \Obq_{\sE,\sL,poly}(\RR_{\geq 0})\otimes_{\RR[\hbar]}\RR_{\hbar=1}.
\end{equation}
It is manifest that $\Upsilon$ sends the point $1\in C_\bullet(\fg)$ to the point
\[
1\in\Obq_{\sE,\sL,poly}(\RR_{\geq 0})\otimes_{\RR[\hbar]}\RR_{\hbar=1}.
\]
To show that $\Upsilon$ is a quasi-isomorphism, we note that both sides are filtered---the left-hand side by degree in the symmetric algebra, and the right-hand side by degree of dependence on $B$ fields---and the map $\Upsilon$ preserves this filtration.
The associated graded object on the left-hand side is $\Lambda^{-\bullet}\fg$, while the associated graded on the right-hand side is the algebra of global observables for the bulk-boundary system.
The associated graded map $\Upsilon$ is the inverse quasi-isomorphism of 
Proposition \ref{prop: 1dbfclass}.
Hence $\Upsilon$ is also a quasi-isomorphism; this completes the proof.
\end{proof}

\begin{proof}[Proof of Proposition \ref{prop: 1dbfblkobs}]
The proof is similar to the proof of \ref{prop: 1dbfbdyobs}, and combines Lemmas \ref{lem: 1dbffiniteobservables} and \ref{lem: diffon1dbfobsblk}, and Proposition \ref{prop: 1dbfclass}

First, for $\alpha\in \Sym^k(\fg[1]\oplus \fg^\vee[-1])$ and a fixed $t>0$, we will define a function $\cO_{\alpha,t}[0]$ on the space of fields as follows.
Let $\phi\in \Omega^\bullet(\RR_{\geq 0}^k)\otimes (\fg[1]\oplus \fg^\vee[-1])^{\otimes k}$.
 
Let $\iota_t: \{\ast\} \to (\RR_{\geq 0})^k$ denote the inclusion of the point $(t,\ldots, t)$ into $(\RR_{\geq 0})^k$.
We define
\begin{equation}
    \cO_{\alpha,t}[0](\phi) = \alpha( \iota_t^* \phi),
\end{equation}
where we have used the natural pairing on $\Sym^k(\fg[1]\oplus \fg^\vee[-1])$ to pair $\alpha$ and $\iota^*_t\phi$.
Lemma \ref{lem: 1dbffiniteobservables} shows that $\cO_{\alpha,t}[0]$ gives rise to a \emph{global} observable for the one-dimensional BF theory.
This observable is manifestly supported on $\RR_{>0}\subset \RR_{\geq 0}$; in fact, it is supported on any open neighborhood of $\{t\}$.

We must now show that the differential on the quantum observables induces the desired one on $\Sym(\fg[1]\oplus \fg^\vee[-1])$.
To this end, we will study $C^\bullet(\fg, C_\bullet(\fg))$ not in its standard incarnation, but in the following form.
The underlying graded vector space of $C^\bullet(\fg,C_\bullet(\fg))$ is $\Sym(\fg[1]\oplus \fg^\vee[-1])$, and the symbols $t^a,t_a$ generate this symmetric algebra.
Let $\del_a=\frac{\del}{\del t^a}$ and $\del^a= \frac{\del}{\del t_a}$.
Define (using the Einstein summation convention throughout)
\begin{equation}
    \del_{tr}:= \del_a\del^a, \quad d^q := e^{\del_{tr}}d_{C^\bullet(\fg,C_\bullet(\fg))}e^{-\del_{tr}};
\end{equation}
One may check that $d^q$ has the following expression:
\begin{equation}
    \frac{1}{2}f^{ab}_ct_a t_b \del^c+ f^{ab}_c t_bt^c \del_a +\frac{1}{2}f^{ab}_c \del_{a}\del_b\del^c.
\end{equation}

The first two terms in the above expression constitute the differential on $C^\bullet(\fg, \Sym(\fg[1]))$; the third has a geometric interpretation described in Remark \ref{rmk: shiftedPoissonBF}.

Let us denote by $\sO^q_{P_0}(T^*B\fg)$ the graded vector space $\Sym(\fg[1]\oplus \fg^\vee[-1])$ together with the differential $d^q$.
Lemmas \ref{lem: 1dbffiniteobservables} and \ref{lem: diffon1dbfobsblk} establish a cochain map from $\sO^q_{P_0}(T^*B\fg)$ into the quantum observables.
The map preserves the natural filtrations on both sides---the filtration by symmetric powers of $\fg[1]$ in $\sO^q_{P_0}(T^*B\fg)$ and the filtration by symmetric powers of dependence on the $B$-fields in the quantum observables, and is manifestly the inverse to the map constructed in Proposition \ref{prop: 1dbfclass} at the level of associated graded spaces.
The Proposition follows.
\end{proof}

\begin{remark}
\label{rmk: shiftedPoissonBF}
Let us comment briefly on the notation $\sO^q_{P_0}(T^*B\fg)$ for the cochain complex introduced above.
The cochain complex $C^\bullet(\fg, \Sym(\fg[1]))$ can be understood as the algebra of functions on the 0-shifted symplectic formal moduli problem $T^*B\fg$.
An avatar of this fact is that the above algebra is manifestly a dg Poisson, or $P_1$, algebra.
By Poisson/Rozenblyum-Safronov additivity \autocite{safronovadditivity}, one may view a $P_1$ algebra as an algebra in the category of $P_0$-algebras.
In particular, one may forget a $P_1$ structure on a space down to a (homotopy) $P_0$ structure.
Proposition 2.16 of \autocite{safronovpoissonlie} can be adapted \autocite{safronovpersonalcomm} to show that the homotopy $P_0$ structure on $C^\bullet(\fg, \Sym(\fg[1]))$ is described by a vanishing bivector and the degree +1 trivector
\begin{equation}
    f^{ab}_c \del_a\wedge \del_b\wedge \del^c;
\end{equation}
the differential $d^q$ described above is precisely the deformation of the Chevalley-Eilenberg differential on $C^\bullet(\fg,\Sym(\fg[1]))$ by the BV Laplacian associated to this $P_0$ structure.
\end{remark}

In the remainder of this section, we prove the technical lemmas which were necessary in the proofs of Propositions \ref{prop: 1dbfbdyobs} and \ref{prop: 1dbfblkobs}.
First, we construct observables $\cO_\omega$ (respectively $\cO_{\alpha,t}$) with support in $\RR_{\geq 0}$ (respectively, $\RR_{>0})$ for each $\omega\in \Sym(\fg[1])$ (respectively,~$\alpha\in \Sym(\fg[1]\oplus \fg^\vee[-1])$):

\begin{lemma}
\label{lem: 1dbffiniteobservables}
Let $\cO_{\omega}[0]$, $\cO_{\alpha,t}[0]$ be as in the proofs of Propositions \ref{prop: 1dbfbdyobs} and \ref{prop: 1dbfblkobs}, respectively.
Then the limits
\begin{align}
\cO_\omega[\Phi]&= \lim_{\epsilon\to 0} \rgobs{\Phi_\epsilon}{\Phi}(\cO_\omega[0])\nonumber\\
\cO_{\alpha,t}[\Phi]&= \lim_{\epsilon\to 0} \rgobs{\Phi_\epsilon}{\Phi}(\cO_{\alpha,t}[0])
\end{align}
exist in the natural topology on $\sO(\condfields(\RR_{\geq 0}))$.
\end{lemma}

\begin{proof}
It suffices to check this for $\Phi=\Phi_L$, since the continuity of RG flow will guarantee that this limit exists for any other parametrix $\Phi$.
A diagram $\gamma$ representing a term in~$\rgobs{\Phi_\epsilon}{\Phi_L}(\cO_{\omega}[0])$ (respectively, $\rgobs{\Phi_\epsilon}{\Phi_L}(\cO_{\alpha,t}[0])$) can be decomposed as follows.
First note that---supposing $\omega\in \Sym^k(\fg[1])$ (respectively,m $\alpha\in \Sym^k(\fg[1]\oplus \fg^\vee[-1])$)---$\gamma$ consists of one special vertex with $k$ ``out'' half-edges and any number of vertices of valence 3 with two ``in'' half-edges and one ``out'' half-edge.
Let us sever any edges incident on the special vertex.
The result is a disconnected graph consisting of an isolated $k$-valent vertex; some number of ``loose'' edges that, in $\gamma$, connect pairs of half-edges of the special vertex; and a disjoint union of graphs appearing in $\rgflow{\epsilon}{L}{I}$.
Let us call the connected graphs of this latter form ``type I graphs.''
The edges in the original graph $\gamma$ which connect the special vertex to the type I graphs simply tell us where to input the smooth function $P(\Phi_L-\Phi_\epsilon)(t_1=t,t_2)$ as inputs into the type I graphs.
Let us observe a number of further simplifications which can be made on the graphs appearing in $\rgobs{\Phi_\epsilon}{\Phi_L}$.

First, note that the ``loose'' edges can only appear for the observable $\cO_{\alpha, t}[0]$, since $\cO_\omega[0]$ has only ``out'' half-edges, which cannot join to form edges.
The ``loose'' edges contribute factors of
\begin{equation}
    \frac{4f_0(2t)}{\sqrt{4\pi}}\int_{t/\sqrt{L}}^{t/\sqrt{\epsilon}} e^{-v^2}dv,
\end{equation}
which manifestly possesses an $\epsilon\to 0$ limit.
Hence, we may assume without loss of generality that no such ``loose'' edges appear in $\gamma$, as they contribute a convergent factor.

Now, we assess which type I graphs may occur in a diagram $\gamma$.
Note that type I wheel graphs cannot appear, since such a diagram would involve three propagators incident on a single vertex; because the propagator is a zero-form, and a vertex requires exactly one one-form input, such an incidence of propagators yields zero.
We may further assume that exactly one type I tree appears, since the analytic weight associated to multiple trees will simply be a product of weights associated to the individual components.
To describe the analytic part of the weights of diagrams with one type I tree, let us denote by $\gamma'$ this single tree.
Let $E(\gamma')$ denote the set of (internal) edges of $\gamma'$.
Let $T(\gamma')$ denote the set of external edges or tails of $\gamma'$.
Let $IT(\gamma')\subset T(\gamma')$ denote the set of tails which are connected to the special vertex in $\gamma$.
Let $V(\gamma')$ denote the set of (internal) vertices of $\gamma'$, and choose some bijection $\{1,\ldots, m\} \to V(\gamma')$.
For any $e\in E(\gamma')$, let $v_b(e)$ denote the vertex at the beginning of $e$ and $v_e(e)$ denote the vertex at the end of $e$ (our graphs are directed; hence, this is well-defined).
Similarly, for any $p\in IT(\gamma')$, let $v_i(p)$ denote the vertex in $\gamma'$ to which it is attached.
Let $\beta: E(\gamma')\to \{-1,1\}$ be a labeling of the edges of $\gamma'$ by elements of $\ZZ/2\ZZ$.
Finally, let $t_1,\ldots, t_m$ be the natural coordinates on $(\RR_{\geq 0})^{V(\gamma')}$ induced from the fixed bijection $\{1,\ldots, m\}\to V(\gamma')$, 
and
\begin{align}
    u_e &= t_{v_b(e)}-\beta(e)t_{v_e(e)}, \quad e\in E(\gamma')\\
    u_p &= t_{v_i(p)}, \quad p\in IT(\gamma')
\end{align}
Then, the analytic weight associated to $\gamma$ is a sum over possible $\beta$ of the following weight:
\begin{align}
    w^{an}_{\beta}(&\epsilon, L,f; \cO):=\nonumber\\
    C&\int_{(\RR_{\geq 0})^m} f(t_1,\ldots, t_m)\nonumber\\
    &\prod_{f\in IT(\gamma')}\left( f_0(u_f)\int_{u_f/\sqrt{4L}}^{u_f/\sqrt{4\epsilon}} e^{-v^2}dv\right)\prod_{e\in E(\gamma')} \left( f_0(u_e) \int_{u_e/\sqrt{4L}}^{u_e/\sqrt{4\epsilon}} e^{-v^2}dv\right)\d t_1 \cdots \d t_m.
\end{align}
It follows by the same sort of reasoning as in the proof of Proposition \ref{prop: 1dbffinite} that this integral admits a finite $\epsilon \to 0$ limit uniformly in $f$.
\end{proof}

The remaining lemmas in this section show that the differential on $\Obq_{\sE,\sL,poly}$ induces the desired differentials on $\Sym(\fg[1])$ and~$\Sym(\fg[1]\oplus \fg^\vee[-1])$.

\begin{lemma}
\label{lem: diff1dbfobs}
Let $\widehat Q$ denote the total differential on $\Obq_{\sE,\sL}(\RR_{\geq 0})$.
The equality
\begin{equation}
    \widehat Q \cO_{\omega}= \cO_{d_{CE}\omega}
\end{equation}
holds.
\end{lemma}
\begin{proof}
Let $\delta$ be a parameter of square-zero, and for $J \in \sO(\condfieldscs)$, $\cO\in \sO(\condfields)$, define
\begin{align}
    D_\Phi J &= QJ + \frac{1}{2}\{J,J\}_\Phi+\Delta_\Phi J\\
    D_{\Phi,J}\cO &= Q\cO + \{J,\cO\}_\Phi + \Delta_\Phi \cO\\
    \cR_{\Phi_1,\Phi_2, J}(\cO) & = \frac{d}{d \delta}\rgflow{\Phi_1}{\Phi_2}{J+\delta \cO}.
\end{align}
Let us note a few identities which follow from these definitions:
\begin{align}
    D_\Phi(J+\delta \cO) &= D_\Phi J + D_{\Phi,J}\cO\\
    \rgobs{\Phi}{\Psi}(\cO) &= \cR_{\Phi,\Psi, I[\Phi]}(\cO)\\
    (D_{\Phi}J)e^J&= (Q+\Delta_\Phi)e^J\\
    \widehat Q_\Phi \cO &= D_{\Phi, I[\Phi]}\cO
\end{align}

Now, we apply the equation at the end of the proof of Lemma 9.2.2 of Chapter 5 of \autocite{cost} (this equation encodes the compatibility of RG flow with the QME):

\begin{align}
    D_\Phi\Big( \rgflow{\Phi_\epsilon}{\Phi}{I+\delta \cO}\Big) &\exp\Big( \rgflow{\Phi_\epsilon}{\Phi}{I+\delta \cO}\Big)\nonumber\\
    &=\exp\left( \partial_{P(\Phi-\Phi_\epsilon)}\right) D_{\Phi_\epsilon}(I+\delta \cO) \exp(I+\delta \cO).
\end{align}
Let us analyze the parts of this equation which are linear in $\delta$. 
With the aim of compactness of notation, let us set
\begin{equation}
I_\epsilon^\Phi:= \rgflow{\Phi_\epsilon}{\Phi}{I}
\end{equation}
We find
\begin{align}
    \left(D_{\Phi, I_\epsilon^\Phi}(\cR_{\Phi_\epsilon,\Phi,I}(\cO))+\cR_{\Phi_\epsilon,\Phi,I}(\cO) D_{\Phi}I_{\epsilon}^\Phi\right)&\exp\left(I_\epsilon^\Phi\right)\nonumber\\
    &= \exp\left( \partial_{P(\Phi-\Phi_\epsilon)}\right) \left(\cO D_{\Phi_\epsilon}I+ D_{\Phi_\epsilon,I}\cO\right) \exp(I)\nonumber\\
    &=\cR_{\Phi_\epsilon,\Phi,I}\left( \cO D_{\Phi_\epsilon}I+D_{\Phi_\epsilon,I}\cO\right)\exp\left( I^\Phi_\epsilon\right)
    \label{eq: epsto0limitobsrg}
\end{align}
Let us suppose that $\cR_{\Phi_\epsilon,\Phi,I}(\cO)$ admits an $\epsilon\to 0$ limit $\cO[\Phi]$
(which happens, for example, if $\cO=\cO_{\omega}[0]$).
Then, upon taking the $\epsilon\to 0$ limit on the left-hand side of Equation \eqref{eq: epsto0limitobsrg}, the term $D_\Phi I^\Phi_\epsilon$ goes to zero by the arguments of Proposition \ref{prop: 1dBFQME}.
Meanwhile, $D_{\Phi,I^\Phi_\epsilon}\cO\to \widehat {Q}_\Phi \cO[\Phi]$ as $\epsilon\to 0$.
We therefore find
\begin{equation}
\label{eq: obsdiff}
    \widehat {Q}_{\Phi}\cO[\Phi] = \lim_{\epsilon \to 0}\cR_{\Phi_\epsilon,\Phi,I}\left(\cO D_{\Phi_\epsilon}I +D_{\Phi_\epsilon, I}\cO\right).
\end{equation}
Let us simplify the right-hand side of Equation \eqref{eq: obsdiff} in the case that $\cO= \cO_\omega [0]$.
We note that $\cO_\omega[0]$ is $\diffonnocrossdiff$- and $\Delta_\Phi$-closed for all $\Phi$; hence 
\begin{equation}
D_{\Phi_\epsilon,I}\cO_\omega[0]=\{I,\cO_\omega[0]\}_{\Phi_\epsilon}.
\end{equation}
In the proof of Proposition \ref{prop: 1dBFQME}, we saw that 
\begin{equation}
    D_{\Phi_\epsilon}I = \frac{1}{2}\{I,I\}_{\Phi_\epsilon}.
\end{equation}
We therefore find that $\widehat{Q}_\Phi \cO[\Phi]$ may be represented by the $\epsilon\to 0$ limit of a sum of directed graphs.
The graphs appearing in this sum have the following properties:
\begin{itemize}
    \item They have one special vertex of valence $k$, where~$\omega\in \Sym^k(\fg[1])$, with only outgoing half-edges.
    \item They have any number of trivalent vertices with two incoming half-edges and one outgoing half-edge.
    \item There is a special edge, labelled by $K_{\Phi_\epsilon}$.
    \item The remaining edges are labelled by $P(\Phi_\epsilon,\Phi)$.
\end{itemize}
To complete the proof, we will show
\begin{equation}
    \label{eq: findingCEdiff}
    \lim_{L\to 0}\widehat{Q}_{\Phi_L} \cO_\omega[\Phi_L]= \cO_{d_{CE}\omega}[0].
\end{equation}
A simple manipulation with the observable RG flow equation shows that Equation \eqref{eq: findingCEdiff} implies \eqref{eq: CEdiffgoal}.

Let us first note that we will ignore the term $R(t_1,t_2,\epsilon)$ in $K_{\Phi_\epsilon}$, since this term vanishes uniformly in $t_1,t_2$ to all orders in $\epsilon$ as $\epsilon\to 0$, so will cause the Feynman weights under consideration to vanish.

We make an additional observation which reduces the number of Feynman graphs $\gamma$ under consideration.
First, note that---as in the proof of Lemma \ref{lem: 1dbffiniteobservables}---upon severing the edges incident on the special vertex in $\gamma$, we obtain a disjoint union of connected graphs appearing in $\rgflow{\Phi_\epsilon}{\Phi_L}{I}$.
Precisely one of these connected graphs contains the special edge, or is attached to the special vertex via the special edge.
The remainder of the connected graphs are labelled on their edges by the propagator, and are attached to the special vertex by a propagator.
By arguments similar to those used in the proof of Proposition \ref{prop: 1dbffinite}, these will contribute 0 in the $\epsilon \to 0$, $L\to 0$ limit.
As long as the part of the diagram containing the special vertex admits an $\epsilon\to 0, L\to 0$ limit, we may ignore these contributions.
Hence, we may assume that $\gamma$ consists of a single connected graph appearing in the graphical expansion for $I[\Phi_\epsilon]$ attached to the special vertex.
Precisely one edge of $\gamma$ is labeled by a heat kernel.
We may assume that this edge is incident on the special vertex of $\gamma$, for otherwise, because of the boundedness of the propagators and because of Proposition \ref{prop: 1dBFQME}, the resulting diagrams give a trivial $\epsilon \to 0$, $L\to 0$ contribution.

To describe the analytic part of the weights of such diagrams, let us consider the graph $\gamma'$ obtained by severing the edges incident on the special vertex in $\gamma$.

Let $E(\gamma')$ denote the set of (internal) edges of $\gamma'$.

Let $T(\gamma')$ denote the set of external edges or tails of $\gamma'$.

Let $IT(\gamma')\subset T(\gamma')$ denote the set of tails which are connected to the special vertex in $\gamma$ via an edge labeled by a propagator.
Let $p_0$ denote the tail of $\gamma'$ which attaches to the special vertex via the special edge.

Let $V(\gamma')$ denote the set of (internal) vertices of $\gamma'$, and choose some bijection 
\[
\{1,\ldots, m\} \to V(\gamma').
\]
of sets.

For any $e\in E(\gamma')$, let $v_b(e)$ denote the vertex at the beginning of $e$ and $v_e(e)$ denote the vertex at the end of $e$ (our graphs are directed; hence, this is unambiguous).
Similarly, for any $p\in IT(\gamma')\sqcup \{p_0\}$, let $v_i(p)$ denote the vertex to which it is attached.
Let $\beta: E(\gamma')\to \{-1,1\}$ be a labeling of the edges of $\gamma'$ by elements of $\ZZ/2\ZZ$.

Finally, let $t_1,\ldots, t_m$ be the natural coordinates on $(\RR_{\geq 0})^{V(\gamma')}$ induced from the fixed bijection $\{1,\ldots, m\}\to V(\gamma')$, 
and
\begin{align}
    u_e &= t_{v_b(e)}-\beta(e)t_{v_e(e)}, \quad e\in E(\gamma')\\
    u_p &= t_{v_i(p)}, \quad p\in IT(\gamma')\sqcup\{p_0\}
\end{align}
Then, the analytic weight associated to $\gamma$ is a sum, as $\beta$ varies over all possible set maps $E(\gamma')\to \ZZ/2\ZZ$, of the weights:
\begin{align}
    O^{an}_{\gamma,\beta}(&\epsilon, L,f; \cO_\omega[0]):=\nonumber\\
    C&\int_{(\RR_{\geq 0})^m}f(t_1,\ldots, t_m)\prod_{p\in IT(\gamma')}\left( f_0(u_p)\int_{u_p/\sqrt{4L}}^{u_p/\sqrt{4\epsilon}} e^{-v^2}dv\right)\nonumber\\
    &\times\frac{f_0(u_{p_0})}{\sqrt{\epsilon}}\exp\left( -\frac{u_{p_0}^2}{4\epsilon}\right)\prod_{e\in E(\gamma')} \left( f_0(u_e) \int_{u_e/\sqrt{4L}}^{u_e/\sqrt{4\epsilon}}e^{-v^2}\d v\right)\d t_1 \cdots \d t_m.
\end{align}

As in the proof of Proposition \ref{prop: 1dBFQME}, let us first suppose that there is an internal edge $e_0$ in $\gamma'$.

Let us suppose with some loss of generality that $v_i(p_0)=v_1$ and $v_b(e_0)=v_2$, $v_b(e_0)=v_3$
(there remains the possibility that $e_0$ is incident on $v_i(p_0)$).
For all $e\neq e_0$ and $p\neq p_0$, we may bound the integral over $v$ by $\sqrt{\pi}$; hence, we obtain the bound
\begin{align}
    |O^{an}_{\gamma, \beta}(\epsilon,L,f;&\cO_\omega[0])|\leq C'\int_{(\RR_{\geq 0})} \frac{e^{-\frac{(u_{p_0})^2}{4\epsilon}}}{\sqrt{\epsilon}}\d t_{v_i(p_0)}\, \nonumber\\
    &\times \int_{K\subset (\RR_{\geq 0})^{m-1}} \left(\int_{u_{e_0}/\sqrt{4L}}^{{u_{e_0}}/\sqrt{4\epsilon}}e^{ -\frac{u_{e_0}^2}{4\epsilon}} e^{-v^2} \right)\d v \d t_1 \cdots \widehat {\d t_{v_i(p_0)}} \cdots \d t_m,
\end{align}
where $K$ is the (compact) intersection of the support of $f$ with the $(\RR_{\geq 0})^{m-1}$ spanned by all the $t_i$ excluding $t_{v_i(p_0)}$.
The integral over $t_{v_i(p_0)}$ is independent of $\epsilon$; the remaining integral goes to zero as $\epsilon \to 0$ and $L\to 0$.

Now, if $e_0$ is incident on $v_i(p_0)$, so we may assume $v_b(e_0)=v_1$, $v_e(e_0)=v_2$, then we may change change coordinates to $s_1 = t_1$, $s_2 = t_1 -\beta(e_0)t_2$, $s_3=t_3,\ldots, s_m=t_m$.
The above bound then suffices once again.

It therefore remains to consider diagrams for which there is no such edge $e_0$.
Figure \ref{fig: diagrams1dbf} shows the only two possible classes of such diagrams.
We may assume without loss of generality that $\omega\in \fg[1]$ for the first diagram, and for the second diagram that $\omega\in \Sym^2(\fg[1])$.

\begin{figure}
     \centering
    \begin{subfigure}[b]{0.4\textwidth}
         \centering
         \begin{tikzpicture}
         \draw[fermion, draw = red] (0,0) node (a) {$\star$}--+(45:2.6) node (c){};
         \draw[fermion] (a.center)--+(65:2.6);
         \draw[fermion] (a.center)--+(135:2.6);
         \draw[fermion] (a.center)--+(115:2.6);
         \path (a.center) --+(0,2) node (b){$\cdots$};
         \draw[Bhalfedge] (c.center) -- +(-15: 2);
         \draw[fermionbar] (c.center) -- +(105 : 2);
         \end{tikzpicture}
         \caption{The classical diagram}
         \label{fig: classicaldiagram}
     \end{subfigure}
     \hfill
            \begin{subfigure}[b]{0.4\textwidth}
         \centering
         \begin{tikzpicture}
         \path (0,0) -- +(55:2.6) node (c){};
         \draw[fermion, draw =red] (0,0) node (a) {$\star$} to[bend right] (c.center);
         \draw[fermion] (a.center) to[bend left] (c.center);
         \draw[fermion] (a.center)--+(135:2.6);
         \draw[fermion] (a.center)--+(115:2.6);
         \path (a.center) --+(0,2) node (b){$\cdots$};
         \draw[Bhalfedge] (c.center) -- +(55: 2);
         \end{tikzpicture}
         \caption{The quantum diagram. There are two such diagrams, one for each choice of red edge.}
         \label{fig: quantumdiagram}
     \end{subfigure}
        \caption{The diagrams contributing to the computation of $\widehat Q \cO_\omega$. The red edges are the special ones.}.
        \label{fig: diagrams1dbf}
\end{figure}
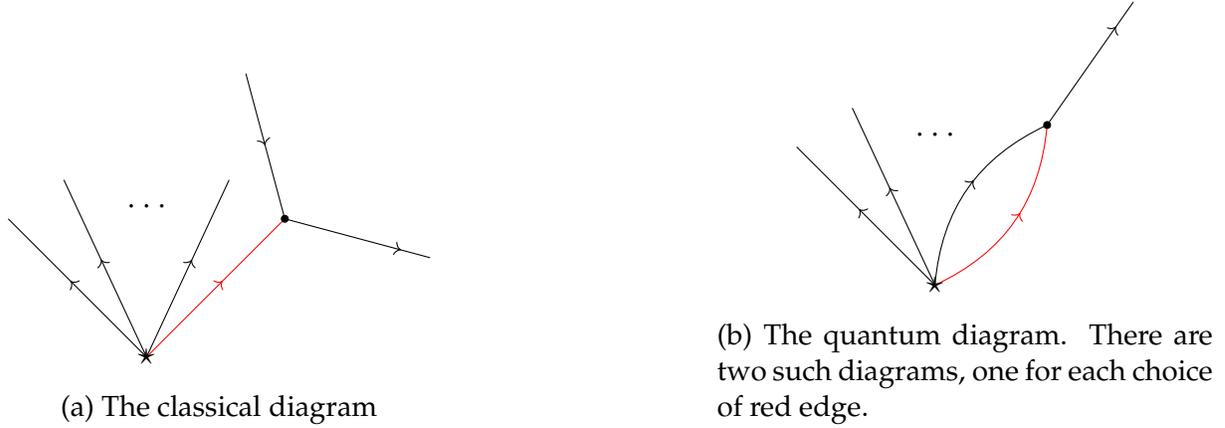

The analytic weight associated to the first diagram is the following functional.
Let $f\in \cinfty((\RR_{\geq 0})^2)$ be such that $f(0,t_2)=0$ for all $t_2$.
(This is an avatar of the enforcement of the boundary condition on the fields.)
The classical diagram gives the weight:
\begin{align}
    O(f)=\lim_{\epsilon \to 0}\int_0^\infty-2\frac{f_0(t)f(t,t)}{\sqrt{4\pi \epsilon}}e^{-\frac{t^2}{4\epsilon}}\d t.
\end{align}
Let us use that $f(t,t)=tg(t)$ for some continuous, compactly-supported $g$. Then, we find that
\begin{equation}
|O(f)|\leq C\lim_{\epsilon \to 0}\int_0^\infty \frac{t}{\sqrt{4\pi \epsilon}}e^{-\frac{t^2}{4\epsilon}}\d t=0.
\end{equation}
It therefore remains to consider the quantum diagram.
As we have already noted, it suffices to consider the case that $\omega \in \Lambda^2\fg$.
Moreover, the Lie-algebraic factor from the quantum diagram is precisely the term $d_{CE}\omega$.
It remains only to compute the analytic factor, for $f\in \cinfty(\RR_{\geq 0})$:
\begin{align}
O^q(f):= \lim_{L\to 0}\lim_{\epsilon \to 0}\int_{[\epsilon,L]}\int_0^\infty \frac{t\,f_0(t)^2\,f(t)}{2\pi\sqrt{\epsilon}(L')^{3/2}}\exp \left( -\frac{t^2}{4}\left(\frac{1}{\epsilon}+\frac{1}{L'}\right)\right) \d t\, \d L'.
\end{align}
Now, we Taylor expand $t\,f_0(t)^2\, f(t)= t f(0)+ t^2 g(t)$, where $g$ is some function with compact support on $\RR_{\geq 0}$.
We therefore have
\begin{align}
O^q(f)&= \lim_{L\to 0}\lim_{\epsilon \to 0}f(0)\int_{[\epsilon,L]}\int_0^\infty \frac{t}{2\pi\sqrt{\epsilon}(L')^{3/2}}\exp \left( -\frac{t^2}{4}\left(\frac{1}{\epsilon}+\frac{1}{L'}\right)\right) \d t\, \d L'\nonumber\\
&+\lim_{L\to 0}\lim_{\epsilon \to 0}\int_{[\epsilon,L]}\int_0^\infty \frac{t^2\, g(t)}{2\pi\sqrt{\epsilon}(L')^{3/2}}\exp \left( -\frac{t^2}{4}\left(\frac{1}{\epsilon}+\frac{1}{L'}\right)\right) \d t\, \d L'.
\end{align}
Using simple bounds derived from the AM-GM inequality, one may show that the second summand above is zero, while the first summand gives $f(0)/2$.
Once we remember that there are \emph{two} quantum diagrams (the second one obtained by switching the roles of $P$ and $K$ in the first diagram),
we find that we obtain Equation \eqref{eq: findingCEdiff}.
\end{proof}

\begin{lemma}
\label{lem: diffon1dbfobsblk}
Let $\widehat Q$ denote the differential on the bulk observables of 1d BF theory.
The equality 
\begin{equation}
    \widehat Q \cO_{\alpha,t} =\cO_{d^q\alpha,t}
\end{equation}
holds, where $d^q$ is the differential on $\Sym(\fg[1]\oplus \fg^\vee[-1])$ defined in the proof of Proposition \ref{prop: 1dbfblkobs}.
\end{lemma}

\begin{proof}
The proof resembles that of Lemma \ref{lem: diff1dbfobs}.
Just as in the proof there, we may reduce to diagrams in which there is one non-special vertex.
The diagrams of Figure \ref{fig: diagrams1dbf} appear (though the tails emanating from the special vertex may point either towards or away from it now).
We will need to recompute the weights associated to these diagrams, however.
We obtain in addition the two new quantum diagrams depicted in Figure \ref{fig: newdiagrams1dbf}.

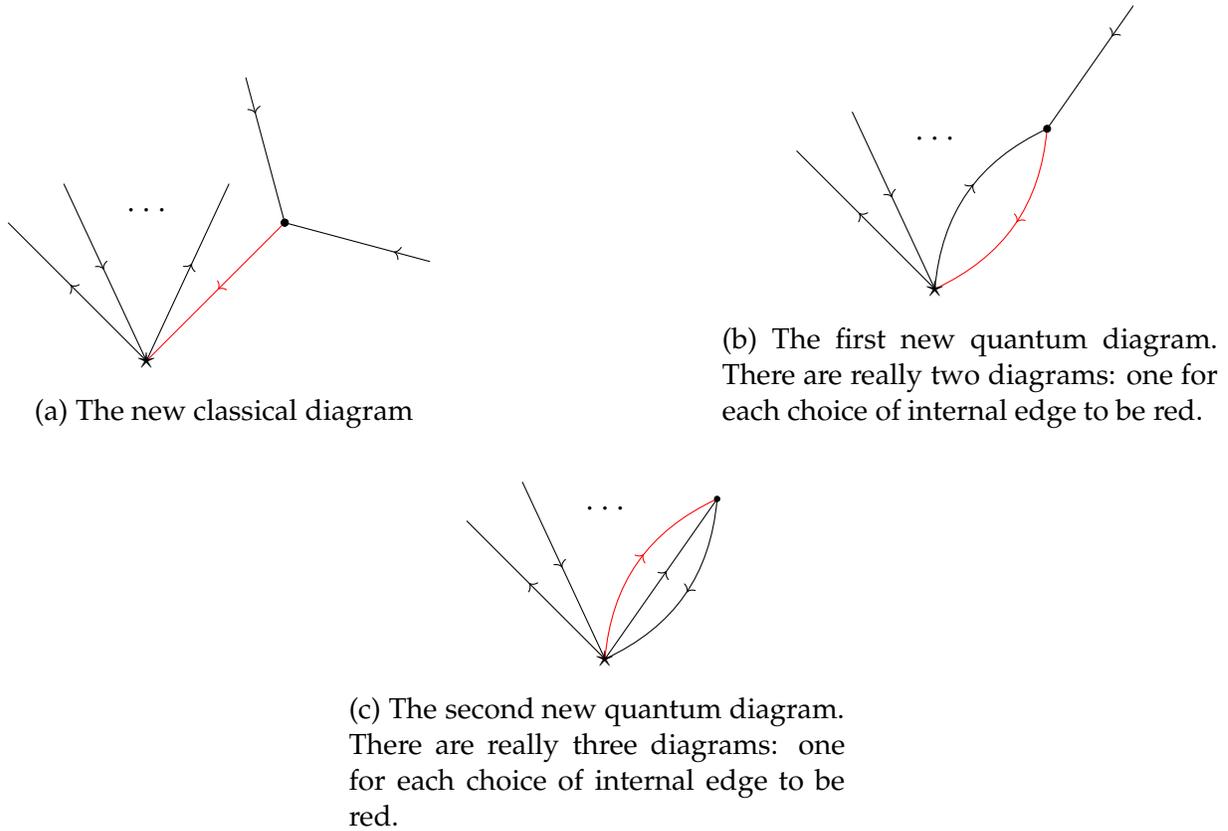
\begin{figure}
     \centering
    \begin{subfigure}[b]{0.40\textwidth}
         \centering
         \begin{tikzpicture}
         \draw[fermionbar, draw = red] (0,0) node (a) {$\star$}--+(45:2.6) node (c){};
         \draw[fermion] (a.center)--+(65:2.6);
         \draw[fermion] (a.center)--+(135:2.6);
         \draw[fermionbar] (a.center)--+(115:2.6);
         \path (a.center) --+(0,2) node (b){$\cdots$};
         \path (c.center) -- +(-15: 2) node (e){};
         \draw[Ahalfedge] (e.center)--(c.center);
         \path (c.center) -- +(105 : 2) node (d){};
         \draw[Ahalfedge] (d.center)--(c.center);
         \end{tikzpicture}
         \caption{The new classical diagram}
         \label{fig: 1newclassicaldiagram}
     \end{subfigure}
     \hfill
            \begin{subfigure}[b]{0.40\textwidth}
         \centering
         \begin{tikzpicture}
         \path (0,0) -- +(55:2.6) node (c){};
         \draw[fermionbar, draw =red] (0,0) node (a) {$\star$} to[bend right] (c.center);
         \draw[fermion] (a.center) to[bend left] (c.center);
         \draw[fermion] (a.center)--+(135:2.6);
         \draw[fermionbar] (a.center)--+(115:2.6);
         \path (a.center) --+(0,2) node (b){$\cdots$};
         \path (c.center) -- +(55: 2) node (d) {};
         \draw[Ahalfedge] (d.center)--(c.center);
         \end{tikzpicture}
         \caption{The first new quantum diagram. There are really two diagrams: one for each choice of internal edge to be red.}
         \label{fig: 1newquantumdiagram}
     \end{subfigure}

    \begin{subfigure}[b]{0.40\textwidth}
         \centering
         \begin{tikzpicture}
         \path (0,0) -- +(0,3);
         \path (0,0) -- +(55:2.6) node (c){};
         \draw[fermionbar] (0,0) node (a) {$\star$} to[bend right] (c.center);
         \draw[fermion,draw = red] (a.center) to[bend left] (c.center);
         \draw[fermion] (a.center) -- (c.center);
         \draw[fermion] (a.center)--+(135:2.6);
         \draw[fermionbar] (a.center)--+(115:2.6);
         \path (a.center) --+(0,2) node (b){$\cdots$};
         \filldraw[black] (c.center) circle (1pt);
         \end{tikzpicture}
         \caption{The second new quantum diagram. There are really three diagrams: one for each choice of internal edge to be red.}
         \label{fig: 2newquantumdiagram}
     \end{subfigure}
        \caption{The diagrams contributing to the computation of $\widehat Q \cO_{\alpha,t}$ which did not appear in Figure \ref{fig: diagrams1dbf}. The red edges are the special ones.}.
        \label{fig: newdiagrams1dbf}
\end{figure}

The original classical diagram will contribute precisely the term in the differential on  $C^\bullet(\fg,C_\bullet(\fg))$ arising from the action of $\fg$ on $C_\bullet(\fg)$.
Similarly, the new classical diagram will contribute the Chevalley-Eilenberg differential for $C^\bullet(\fg)$.
Both of these claims follow from the following identity, true for $t>0$:
\begin{equation}
    f(t) = \lim_{\epsilon \to 0} \frac{1}{\sqrt{4\pi \epsilon}}\int_{0}^\infty f(s) \left( f_0(s-t)\exp\left(-\frac{(s-t)^2}{4\epsilon}\right)\pm f_0(s+t) \exp\left(-\frac{(s+t)^2}{4\epsilon}\right)\right)ds;
\end{equation}
the limit converges uniformly as $f$ varies over bounded subsets of $\cinfty_c(\RR_{\geq 0})$.

We will only sketch the derivation of this identity, as it is similar to many other computations we have already done.
For $t>0$, the term involving $e^{-(s+t)^2/4\epsilon}$ vanishes, uniformly in $s$, faster than any power of $\epsilon$ as $\epsilon \to 0$.
It therefore does not contribute to the computation.
For the term involving $e^{-(s-t)^2/4\epsilon}$, we proceed just as in the case without boundary: we change coordinates from $s$ to $s-t$, Taylor expand $f(s)f_0(t-s)$ around $s=t$, and compute the Gaussian integral.

The first quantum diagram and the old quantum diagram give similar contributions. For $\beta:\{1,2\}\to \{-1,1\}$ a map, we need to consider
\begin{align}
    O^{q}_{\beta,blk}(\epsilon,L, f)&\nonumber\\
    = \int_\epsilon^L&\frac{1}{\sqrt{(4\pi)^2 \epsilon L_1^3}}\int_0^\infty f(t_2) f_0(t-\beta(1) t_2)(t-\beta(1)t_2)e^{-\frac{(t-\beta(1)t_2)^2}{4L_1}}\nonumber\\
    &\times f_0(t-\beta(2)t_2)e^{-\frac{(t-t_2)^2}{4\epsilon}}\d t_2\, \d L_1,
\end{align}
Any term with a factor of $e^{-(t+t_2)^2/4L_1}$ or  $e^{-(t+t_2)^2/4\epsilon}$ will vanish upon taking $\epsilon\to 0$ and $L\to 0$ because $t>0$ and $e^{-t^2/4T}$ vanishes to all orders in $T$ as $T\to 0$.
Hence, we are left to consider $\beta(1)=1=\beta(2)$:
\begin{align}
    O^{q}_{\beta=1,blk}(\epsilon,L, f)&\nonumber\\
    = \int_\epsilon^L\frac{1}{\sqrt{(4\pi)^2 \epsilon L_1^3}}\int_0^\infty& f(t_2) f_0^2(t-t_2)(t-t_2)e^{-\left(\frac{1}{4L_1}+\frac{1}{4\epsilon}\right)(t-t_2)^2}\d t_2 \, \d L_1
\end{align}
We may Taylor expand $f(t_2)f_0(t-t_2)$ in $t_2$ about $t_2=t$.
It suffices to consider only the zeroth order term in the Taylor expansion because any higher orders may be bounded by crude bounds by integrals that vanish as $\epsilon \to 0$ and $L\to 0$.
Hence, we would like to consider the integral
\begin{align}
    \int_\epsilon^L&\frac{1}{\sqrt{(4\pi)^2 \epsilon L_1^3}}\int_0^\infty (t-t_2) e^{-\left(\frac{1}{4L_1}+\frac{1}{4\epsilon}\right)(t-t_2)^2}\d t_2 \, \d L_1\nonumber\\
    = \int_\epsilon^L&\frac{1}{\sqrt{(4\pi)^2 \epsilon L_1^3}}\int_{-t}^\infty (-t'_2) e^{-\left(\frac{1}{4L_1}+\frac{1}{4\epsilon}\right)(t'_2)^2}\d t'_2 \, \d L_1.
\end{align}
We may divide the integral over $t_2$ into two regions: the first, where $t'_2\in [-t,t]$, for which the symmetry of the integrand gives 0; and the second, where $t'_2\in [t,\infty)$.
Again, because $t>0$, one may use the fact that the exponential factor will go to zero faster than any power of $\epsilon$ or $L_1$ to show that the integral vanishes upon taking $\epsilon\to 0$ and $L\to 0$.

Finally, let us consider the second new quantum diagram.
The Lie algebraic factor manifestly gives
\begin{equation}
  \frac{1}{2}f^{ab}_c \del_a\del_b \del^c;
\end{equation}
it suffices to show that the analytic integral gives unity.
As for the other quantum diagrams, the only integral which will not vanish in the $\epsilon\to 0$, $L\to 0$ limit involves $t-t_2$ (and not $t+t_2$).
The only possibly non-vanishing contribution in the $\epsilon\to 0$, $L\to 0$ limit therefore comes from the integral
\begin{equation}
    O^{q,2}_{blk}(\epsilon,L)=\frac{1}{(4\pi)^{3/2}}\int_\epsilon^L\int_\epsilon^L \int_0^\infty (t-t_2)^2 e^{ -(t-t_2)^2\left(\frac{1}{4\epsilon}+\frac{1}{4L_1}+\frac{1}{4L_2}\right)}\frac{1}{\sqrt{\epsilon (L_1)^3(L_2)^3}}\d t_2 \, \d L_2\, \d L_1
\end{equation}
We may integrate over $t_2\in (-\infty,\infty)$, since the contribution from $t_2\in (-\infty,0)$ will give zero under the limit $L\to 0$, $\epsilon\to 0$ for the same reason that the contributions with $(t+t_2)$ in them do when we integrate over $t_2\in(0,\infty)$.
Then, after performing the Gaussian integral, we find that
\begin{align}
    O^{q,2}_{blk}(\epsilon,L) & \sim \frac{1}{16\pi} \int_\epsilon^L\int_\epsilon^L \left(\frac{1}{4\epsilon}+\frac{1}{4L_1}+\frac{1}{4L_2}\right)^{-3/2}\frac{1}{\sqrt{\epsilon (L_1)^3(L_2)^3}} \d L_2\, \d L_1\nonumber\\
    & = \frac{1}{2\pi\sqrt{\epsilon}}\int_\epsilon^L \int_\epsilon^L \frac{1}{\left(L_1+L_2+L_1L_2/\epsilon\right)^{3/2}}\d L_2\, \d L_1,
\end{align}
where the symbol $\sim$ means ``equals upon taking $\lim_{L\to 0}\lim_{\epsilon\to 0}$.''
Let us now compute the integral:
\begin{align}
     &\int_\epsilon^L \frac{1}{\left(L_1+L_2+L_1L_2/\epsilon\right)^{3/2}}\d L_2\, \d L_1\nonumber\\
     &= \int_{\epsilon}^L\frac{1}{(L_1)^{3/2}}\int_{\epsilon}^L \frac{1}{\left(1+L_2\left( \frac{1}{\epsilon}+\frac{1}{L_1}\right)\right)^{3/2}}\d L_2\, \d L_1\nonumber\\
     &= -2\int_{\epsilon}^L \frac{1}{(L_1)^{3/2}}\left( \frac{1}{\epsilon}+\frac{1}{L_1}\right)^{-1}\left( \frac{1}{\sqrt{1+\frac{L}{\epsilon}+\frac{L}{L_1}}}-\frac{1}{\sqrt{1+1+\frac{\epsilon}{L_1}}}\right)\d L_1\nonumber\\
     &= -2\epsilon \int_\epsilon^L \frac{1}{\epsilon+L_1}\left( \frac{1}{\sqrt{L_1+\frac{L L_1}{\epsilon}+ L}}-\frac{1}{\sqrt{2L_1+\epsilon}}\right)\d L_1.
     \label{eq: someintermediatecomputations}
\end{align}
Let us compute an indefinite integral of the form 
\begin{equation}
    \int \frac{1}{(a+x)\sqrt{bx+c}}\d x,
\end{equation}
where $a,b,$ and $c$ are all greater than zero.
The integral in Equation \eqref{eq: someintermediatecomputations} can then be computed by taking $(a,b,c)=(\epsilon,1+L/\epsilon,L)$ or $(a,b,c)=(\epsilon, 2,\epsilon)$.
We manifestly have $ab-c=\epsilon>0$ in both cases, so we may further assume this in our computation.
Set $u=\sqrt{bx+c}/\sqrt{ab-c}$. Then, we find
\begin{equation}
    \int \frac{1}{(a+x)\sqrt{bx+c}}\d x =\int \frac{2}{\sqrt{ab-c}(1+u^2)} = \frac{2 \arctan(u)}{\sqrt{ab-c}},
\end{equation}
so the integral from Equation \eqref{eq: someintermediatecomputations} evaluates to
\begin{equation}
    -4\sqrt{\epsilon} \left( \arctan\left(\sqrt{\frac{L}{\epsilon}\left(2+\frac{L}{\epsilon}\right)}\right)-2\arctan \left(\sqrt{ 1+\frac{2L}{\epsilon}}\right) +\frac{\pi}{3}\right).
\end{equation}
Therefore, we find that
\begin{equation}
    O^{q,2}_{blk}(\epsilon,L) \sim \frac{1}{3}.
\end{equation}
Once we take into account that there are three choices of ``red edge'' amongst the quantum diagrams of this class, this completes the proof.
\end{proof}

%% file: Appendices/AppendixA.tex
\chapter{Functional Analysis (Going Under the Hood)}
\label{chap: appendix}
In the body of the text, we have taken a number of tensor products and inner homs of convenient vector spaces.
The goal of this appendix is to give an explicit characterization of these tensor products and inner homs.

\section{Bornological tensor products in the presence of boundary conditions}
\label{sec: tensorproduct}

How to find the correct ``natural'' tensor product of topological vector spaces is a notoriously subtle question.
In some situations there are options that are appealing for several reasons.
For instance, given two vector bundles $V_1\to M_1$ and $V_2\to M_2$, 
let $\sV_1$ and $\sV_2$ denote the locally convex topological vector spaces consisting of the smooth global sections of $V_1$ and $V_2$, respectively. 
There is a standard isomorphism (of topological vector spaces)
\begin{equation}
\sV_{1}\hotimes_\pi \sV_2 \cong \cinfty(M_1\times M_2, V_1\boxtimes V_2),
\end{equation}
where $V_1\boxtimes V_2$ is the external tensor product of the bundles $V_1$ and $V_2$
and $\hotimes_\pi$ is the completed projective tensor product of locally convex topological vector spaces.
In this case the geometrically attractive answer matches a completion that is natural from functional analysis.
Its main technical role is in the proof that the observables form a factorization algebra.

When we work with free bulk-boundary theories, we would like a similar geometric understanding of the completed bornological tensor product $\hotimes_\beta$ for spaces of sections with boundary conditions imposed.
From the point of view of the paper, 
this appendix is devoted to proving that $(\condfieldscs[1](U))^{\hotimes_\beta k}$
is isomorphic to the space of compactly-supported sections of $E^{\boxtimes k}$ over $U^{\times k}$ whose $j$th tensor factor lies in $L\oplus \Eb \,\d t$ when the corresponding $M$ coordinate lies on~$\del M$. 
But it is natural to treat several generalizations and variants of this fact. 

To state these generalizations, let $M_1,\cdots, M_k$ be manifolds with boundary; $V_1\to M_1, \cdots, V_k\to M_k$ be vector bundles on the $M_i$; and $W_1\subset V_1\big|_{\partial M_1},\cdots, W_k\subset V_k\big|_{\partial M_k}$ be subbundles of the indicated bundles. 
Breaking slightly with our usual notation, we will let $\sV_i:= \cinfty(M, V_i)$ and $\sW_i:= \cinfty(\del M, W_i)$,
i.e. we use the script letters to denote the spaces of global sections of vector bundles instead of the corresponding sheaves of sections.
\begin{notation}
Define 
\begin{equation}
(\sV_i)_{W_i}:= \{ \sigma \in \sV_i \mid \sigma|_{\del M_i} \in \cinfty (M_i,W_i)\}.
\end{equation}
The space $(\sV_i)_{W_i}$ is a closed subspace of $\sV_i$; since the latter space is nuclear Fr\'{e}chet, the former is as well. More categorically, $(\sV_i)_{W_i}$ is the pullback
\begin{equation}
\begin{tikzcd}
(\sV_i)_{W_i} \ar[r]\ar[d]& \sV_i\ar[d]\\
\sW_i \ar[r]& \cinfty(\del M, V_i\mid_{\del M})
\end{tikzcd}
\end{equation}
in any of the categories $\mathrm{LCTVS}$ (of locally convex vector spaces), $\CVS$, $\DVS$.
\end{notation}

\begin{notation}
Define
\begin{equation}
\tp:= \cinfty(M_1\times\cdots\times M_k, V_1\boxtimes \cdots \boxtimes V_k),
\end{equation}
and
\begin{align}
\tpbc&:=\\
\big \{ \sigma \in &\tp \mid \sigma (x_1,\cdots, x_k)\in (V_1)_{x_1}\otimes \cdots \otimes (W_i)_{x_i}\otimes \cdots \otimes (V_k)_{x_k} \text{ when } x_i\in \partial M_i \big\};
\end{align}
in other words, $\tpbc$ consists of sections of $V_1\boxtimes \cdots \boxtimes V_k$ whose $i$-th tensor factor belongs to $W_i$ whenever the corresponding coordinate lies in $\del M_i$. We endow $\tpbc$ with the topology which it inherits as a subspace of $\tp$. The resulting locally convex topological vector space $\tpbc$ is nuclear Fr\'{e}chet, since it is a closed subspace of $\tp$. $\tpbc$ can be described as a limit in the category of topological vector spaces. More precisely, it is the simultaneous limit of all the diagrams of the form 
\begin{equation}
\begin{tikzcd}
&\cinfty(M_1\times \cdots \times\del M_i\times\cdots \times M_k, V_1\boxtimes\cdots \boxtimes W_i\boxtimes \cdots \boxtimes V_k)\ar[d]\\
\tp\ar[r]& \cinfty(M_1\times \cdots \times \del M_i \times \cdots\times M_k, V_1\boxtimes \cdots\boxtimes (V_i)\mid_{\del M_i}\boxtimes \cdots \boxtimes V_k).
\end{tikzcd}
\end{equation}
as $i$ ranges from 1 to $k$. 
\end{notation}

Note that the continuous multilinear map 
\begin{equation}
\sV_1\times\cdots\times \sV_k \to \tp,
\end{equation}
when restricted to $(\sV_1)_{W_1}\times \cdots \times (\sV_k)_{W_k}$, has image in $\tpbc$, so there is a natural map 
\begin{equation}
\fS : (\sV_1)_{W_1}\widehat{\otimes}_\pi \cdots \widehat{\otimes}_\pi (\sV_k)_{W_k}\to \tpbc.
\end{equation}

We can establish similar notations when we require compact support for sections of the $V_i$. Let us choose compact subsets $\cK_i\subset M_i$. We choose to use a calligraphic font for the $\cK_i$ because the symbols $\cK_i$ and $W_i$ will both appear in subscripts in our notation, and we want to make clear that the two subscripts serve different purposes. 

\begin{notation}
Let 
\begin{enumerate}
\item  $(\sV_i)_{\cK_i}$ denote the space of sections of $V_i$ with compact support on $\cK_i$;
\item $(\sV_i)_{\cK_i,W_i}$ denote the space
\begin{equation}
(\sV_i)_{\cK_i}\cap (\sV_i)_{W_i},
\end{equation}
i.e. $(\sV_i)_{\cK_i,W_i}$ is the space of sections of $V_i$ satisfying both a boundary condition and a compact support condition;
\item $\tpcs$ denote the subspace of $\tp$ consisting of sections with compact support on $\cK_1\times \cdots \times \cK_k$; and
\item $\tpbccs$ denote the space
\begin{equation}
\tpcs \cap \tpbc.
\end{equation}
As with $(\sV_i)_{\cK_i,W_i}$, the sections in $\tpbccs$ satisfy both a boundary condition and a compact support condition.
\end{enumerate}
All four spaces are nuclear Fr\'{e}chet spaces.
\end{notation}

There is a map 
\begin{equation}
\fS_{c.s.} : (\sV_1)_{W_1,\cK_1}\widehat{\otimes}_\pi \cdots \widehat{\otimes}_\pi (\sV_k)_{W_k,\cK_k}\to \tpbccs.
\end{equation}

The aim of this appendix is to prove the following result.

\begin{theorem}
\label{thm: tensorofdirichlet}
The maps $\fS$ and $\fS_{c.s.}$ are isomorphisms for the topological vector space structures.
\end{theorem}

\begin{proof}
The completed projective tensor product commutes with limits separately in each variable. Hence, 
\begin{equation}
(\sV_1)_{W_1}\widehat{\otimes}_\pi \cdots \widehat{\otimes}_\pi (\sV_k)_{W_k}
\end{equation}
can be identified with the simultaneous limit of diagrams of the form 
\begin{equation}
\begin{tikzcd}
&\sV_1\hotimes_\pi \cdots \hotimes_\pi \sW_i \hotimes_\pi \cdots \hotimes_\pi\sV_k  \ar[d]\\
\sV_1\hotimes_\pi \cdots\hotimes_\pi \sV_k\ar[r]& \sV_1\hotimes_{\pi}\cdots\hotimes_\pi \cinfty(\del M, V_i\mid_{\del M}) \hotimes_\pi \cdots \hotimes_\pi \sV_k.
\end{tikzcd}
\end{equation}
as $i$ ranges from $1$ to $k$. As we have seen, $\tpbc$ is a similar limit. The isomorphism $\sV_1 \hotimes_\pi \cdots \hotimes_\pi \sV_k \to \tp$ and its analogs for the other entries of the diagrams induces an isomorphism between the diagram defining $\sV_1\hotimes_\pi \cdots \hotimes_\pi \sW_i \hotimes_\pi \cdots \hotimes_\pi\sV_k$ and the one defining $\tpbc$. $\fS$ is induced from this isomorphism of diagrams, so is an isomorphism. The same exact argument applies for~$\fS_{c.s}$.
\end{proof}

We now describe a consequence of Theorem~\ref{thm: tensorofdirichlet} that is of more direct relevance to the present context. Let us momentarily suppress the $i$ subscripts from our notation, letting $V\to M$ be a vector bundle and $W$ a subbundle of $V\mid_{\del M}$. 
We define $(\sV)_{W,c}$ to be the colimit (in $\CVS$)
\begin{equation}
\text{colim}\left(  (\sV)_{W,\cK_{1}}\to (\sV)_{W,\cK_2}\to\cdots\right),
\end{equation}
with $\cK_{j}\subset \cK_{(j+1)}$ and $\cup_{j} \cK_{j} = M$, i.e. the $\cK_j$ form a sequence of compact subsets of $M$ exhausting it. Equivalently, we can define $(\sV)_{W,c}$ via the pullback diagram
\begin{equation}
\begin{tikzcd}
(\sV)_{W,c}\arrow[r]\arrow[d] \arrow[rd, phantom, "\lrcorner", at start] & (\sV)_c\arrow[d]\\
(\sV)_{W}\arrow[r,hook] & \sV
\end{tikzcd};
\end{equation}
here $\sV_c$ is the space of compactly-supported sections of $V$ endowed with the inductive limit topology (when $\sV_c$ is endowed with this topology, the arrow on the right-hand side of the above diagram is not an embedding). The completed projective tensor product does not commute with colimits; hence Theorem \ref{thm: tensorofdirichlet} does not help us to compute completed projective tensor products of spaces of the form $(\sV)_{W,c}$. We may, however, forget the topology of all spaces involved, remembering only the bounded subsets. In other words, we remember only the underlying bornological vector spaces. Once we do, a new tensor product becomes available to us, namely the completed \emph{bornological} tensor product. The completed bornological tensor product \emph{does} commute with colimits. For nuclear Fr\'{e}chet spaces, it coincides with the completed projective tensor product. See \S B.4-5 of \cite{CG1} for details. 

In the main body of the text, we always use the completed bornological tensor product. Hence, we need to use Theorem \ref{thm: tensorofdirichlet} to infer statements about the bornological tensor products of interest to us. This task is undertaken in the following corollary:

\begin{ucorollary}
\label{crl: borntensorofdirichlet}
There are isomorphisms of bornological vector spaces
\begin{equation}
(\sV_1)_{W_1}\,\hotimes_\beta\, \cdots \,\hotimes_\beta\,(\sV_k)_{W_k}\cong \tpbc,
\end{equation}
\begin{equation}
(\sV_1)_{W_1,c}\,\hotimes_\beta\,\cdots\, \widehat{\otimes}_\beta\, (\sV_k)_{W_k,c}\cong \tpbccslf
\end{equation}
Here, $\tpbccslf$ is defined analogously to~$(\sV_i)_{W_i,c}$.
\end{ucorollary}

\begin{proof}[Proof of Corollary]
The isomorphism 
\begin{equation}
(\sV_1)_{W_1}\,\hotimes_\beta\, \cdots \,\hotimes_\beta\,(\sV_k)_{W_k}\cong \tpbc
\end{equation}
is a direct consequence of Theorem \ref{thm: tensorofdirichlet}, since the $(\sV_i)_{W_i}$ are nuclear Fr\'{e}chet spaces and the completed bornological tensor product coincides with the completed projective tensor product of such spaces, by Corollary 7.1.2 of~\cite{CG1}.

For the second isomorphism, the same argument as for $\tpbc$ gives that 
\begin{equation}
(\sV_1)_{W_1,\cK_1}\,\hotimes_\beta\, \cdots \,\hotimes_\beta\, (\sV_k)_{W_k,\cK_k}\cong \tpbccs;
\end{equation}
since the completed bornological tensor product commutes with colimits, the isomorphism 
\begin{equation}
(\sV_1)_{W_1,c}\,\hotimes_\beta\, \cdots \,\hotimes_\beta\, (\sV_k)_{W_k,c}\cong \tpbccslf
\end{equation}
follows.
\end{proof}

\section{The duals of function spaces with boundary conditions}
\label{sec: innerhom}
In the previous section, we studied the tensor product $(\sV_1)_{L_1}\widehat{\otimes}_\pi(\sV_2)_{L_2}$ .
We found that the tensor product was a very explicit subspace of $\sV_1\widehat{\otimes}_\pi \sV_2$. 
In this section, with the end of understanding the inner hom spaces
\begin{equation}
\label{eq: innerhom1}
\innerhom{\sV_L}{\RR}
\end{equation}
and 
\begin{equation}
\label{eq: innerhom2}
 \innerhom{\sV_{L,c}}{\RR}
\end{equation}
and their generalizations, 
we will provide splittings
\begin{align}
    \sV&\cong \sV_L \oplus \sC\\
    \sV_c &\cong \sV_{L,c}\oplus \sC_{c},
\end{align}
where $\sC$ is the space of global sections of a bundle on $\bdyM$ and $\sC_c$ its corresponding space of compactly-supported sections.
We will derive these isomorphisms in the category $\CVS$, but the functor from $\CVS$ to $\DVS$ preserves limits, so we obtain the isomorphism also in the category $\DVS$.
Note that these splittings will allow us to characterize the inner hom spaces in Equations \eqref{eq: innerhom1} and \eqref{eq: innerhom2} as quotients of the inner hom spaces
\begin{equation}
    \innerhom{\sV}{\RR}
\end{equation}
and 
\begin{equation}
    \innerhom{\sV_c}{\RR}
\end{equation}
respectively, in both $\CVS$ and $\DVS$.
(Recall that the functor $\CVS\to \DVS$ does not preserve colimits in general and hence a computation of a colimit in $\CVS$ does not suffice in general to give the colimit in $\DVS$.
In the present case, however, the comptuation in $\CVS$ does suffice.)

Let $C$ denote the bundle $(V\mid_{\bdyM})/L$ on $\bdyM$ and $\sC$ its space of global sections. 
There are natural surjective maps 
\begin{align}
P&:\sV\to \sC\\
P_c&:\sV_c\to \sC_c;
\end{align}
In the next proposition, we construct a map $I: \sC\to \sV$ which establishes $\sC$ as a complement to to $\sV_L$ in $\sV$.
\begin{proposition}
\label{prop: dualsofcondfields}
There are (non-canonical) isomorphisms
\begin{align}
    \sV&\cong \sV_L\oplus \sC\\
    \sV_c&\cong \sV_{L,c}\oplus \sC_c
\end{align}
of convenient vector spaces.
\end{proposition}

\begin{proof}
Let us focus on the first isomorphism first.
We will construct a splitting $I$ of $P$ with the property that $\im(1-IP)\subset \sV_L$. 
To see that this suffices to prove the Proposition, first note that $\sV_L= \ker P$, and consider the map 
\begin{align}
T&:\sV \to \sV_{L}\oplus \sC\\
T(v)&= ((1-IP)v,Pv).
\end{align}
$T$ is a continuous linear isomorphism, by standard arguments. $T$ has an inverse $S$ given by $i\oplus I$, where $i$ is the inclusion of $\sV_L$ as a closed subspace of $\sV$. In other words, we will have 
\begin{equation}
\label{eq: isoquot}
\sV\cong \sV_L\oplus \sC.
\end{equation}

Let us now construct the splitting $\sC\to \sV$. To this end, choose a tubular neighborhood $\tubnhd\cong \bdyM\times[0,\epsilon)$ of $\bdyM$ in $M$. 
Let $\pi$ be the projection $\tubnhd\to [0,\epsilon)$.
By the homotopy invariance of bundles, we may assume that $V|_T\, \cong \pi^* V\mid_{\bdyM}$.
In other words, we may assume that $V$ is trivial in the normal direction.
Let $\chi$ be a compactly-supported function on $[0,\epsilon)$ which is 1 in a neighborhood of $0$ and with support contained $[0,\epsilon/2]$. Let $c\in \sC$ be a section of $C$. 
Choose a splitting of bundles 
\begin{equation}
    \Psi: C\to (V|_{\bdyM}).
\end{equation}
Then, we set $I(c) = \chi \Psi (c)$.
(This is where we have used the isomorphism $V|_T \, \cong \pi^*V|_{\bdyM}$.)
It is straightforward to verify that $I$ is continuous and satisfies the equations $PI=id$ and~$\im(1-IP)\subset \sV_L$.

We now construct the second isomorphism of the lemma. 
Cover $M$ by a countable collection $
\cK_1\subset \cK_2\subset\cdots$ of compact subsets. 
We may assume, by replacing $\cK_i$ with $\cK_i\cup (\cK_i\cap \bdyM)\times [0,\epsilon/2]$ for all $i$, that $\cK_i$ contains $(\cK_i\cap \bdyM)\times [0,\epsilon/2]$. Then, the formulas for $I$ and $P$ send $\sC_{\cK_i\cap \bdyM}\to \sV_{\cK_i}$ and $\sV_{\cK_i}\to \sC_{\cK_i\cap \bdyM}$. We therefore obtain an isomorphism
\begin{equation}
\sV_{\cK_i}\cong \sV_{\cK_i,L}\oplus \sC_{\cK_i\cap \bdyM}
\end{equation}
for each $i$.
The isomorphism respects the maps induced from the inclusions $\cK_i\subset \cK_{i+1}$ and $\cK_i\cap \bdyM\subset \cK_{i+1}\cap \bdyM$, and so we have also an isomorphism 
\begin{equation}
\sV_c=\mathrm{colim}_i \sV_{\cK_i}\cong \mathrm{colim}_i(\sV_{L,\cK_i})\oplus\mathrm{colim}_i \sC_{\cK_i\cap \bdyM}= \sV_{\sL,c}\oplus \sC_c.
\end{equation}
This completes the proof.
\end{proof}

The most useful application of Proposition \ref{prop: dualsofcondfields} is the following corollary.
To make the notation more compact, we make the following definitions, valid for any convenient vector space $W$ and any $k\geq 0$:
\begin{align}
    W^{-k} &:= \innerhom{W^{\hotimes_\beta k}}{\RR}\\
    \Sym^{-k} W &:=\innerhomsym{W^{\hotimes_\beta k}}{\RR}{k}
\end{align}

\begin{corA}
There are canonical isomorphisms
\begin{align}
&\Sym^{-k}(\sV_L) \cong \frac{\Sym^{-k}(\sV)}{\innerhomsym{\sC\hotimes_\beta(\sV)^{\hotimes_\beta (k-1)}}{\RR}{k-1}}\\
&\Sym^{-k}(\sV_{L,c}) \cong \frac{\Sym^{-k}(\sV_c)}{\innerhomsym{\sC_c\hotimes_\beta(\sV_c)^{\hotimes_\beta (k-1)}}{\RR}{k-1}}
\end{align}
of convenient vector spaces.
\end{corA}

Note that Corollary A tells us the following about the relationship between $\Sym^{-k}(\sV)$ and $\Sym^{-k}(\sV_L)$.
First, there is a surjective map $\Sym^{-k}(\sV)\to \Sym^{-k}(\sV_L)$.
Second, the kernel of this map consists precisely of the $S_k$-orbits of maps of the form
\begin{equation}
    \Phi(v_1,\ldots, v_k) = \Phi'(Pv_1,v_2,\ldots, v_k), 
\end{equation}
where 
\[
\Phi' \in \innerhom{\sC\hotimes_\beta \sV^{\hotimes_\beta (k-1)}}{\RR}.
\]
Finally, $\Sym^{-k}(\sV_L)$ is identified with the quotient of the space $\Sym^{-k}(\sV)$ by this kernel, and this identification as a quotient is true both in $\CVS$ and $\DVS$.
This description of $\Sym^{-k}(\sV_L)$ will be useful to us in the body of the dissertation.

\section{A lemma of Atiyah-Bott type}
\label{sec: atiyahbott}

For an elliptic complex on a manifold without boundary, 
the complex of compactly supported smooth sections embeds into the complex of compactly supported distributional sections.
The {\it Atiyah-Bott lemma} is that this embedding is a continuous quasi-isomorphism 
(see Appendix D of \cite{CG1}),
and it plays a role in constructing the observables of free BV theories.
We wish to prove an analog relevant to free bulk-boundary theories.

Let $(\sE,\sL)$ be a free bulk-boundary system. 
We use the pairing $\ip$ to view $\condfieldscs[1]$ as a space of linear functionals on $\condfields$: 
each section $e_1\in \condfieldscs[1]$ gives a linear functional $\Phi_{e_1}$ by the formula
\begin{equation}
\Phi_{e_1}(e_2)=\ip[e_1,e_2].
\end{equation}
This inclusion has the following property.

\begin{proposition}
\label{prop: atiyahbott}
The map $\Phi_\cdot$ induces a quasi-isomorphism of complexes of cosheaves
\begin{equation}
\condfieldscs[1]\to \condfields^\vee,
\end{equation}
where $\condfields^\vee$ is the cosheaf which assigns to the open $U$, the strong topological dual to~$\condfields(U)$.
More precisely, on each open $U$, this map is a continuous linear map of topological vector spaces and a quasi-isomorphism.
\end{proposition}

\begin{proof}
The map is continuous because it is the composite 
\begin{equation}
\condfieldscs[1]\hookrightarrow \sE_c[1]\to \sE^\vee \to \condfields^\vee.
\end{equation}
The map preserves the differential $Q$ because 
\begin{equation}
\Phi_{Q e_1}(e_2)=\ip[Qe_1,e_2]=\pm \ip[e_1,Qe_2]=\Phi_{e_1}(Qe_2);
\end{equation}
this is only true because we have imposed the boundary condition $\sL$. It manifestly respects the extension maps of cosheaves. It remains only to check that it is a quasi-isomorphism. In the proof of Theorem \ref{thm: maingenlcl}, we show that $\condfieldscs[1]$ is a homotopy cosheaf; an almost identical argument shows that $\condfields^\vee$ is also a homotopy cosheaf. Hence, given any open $U\subset \del M$, and any (locally finite) cover $\fU$ of $U$, we have the following commutative diagram
\begin{equation}
\begin{tikzcd}
\check{C}(\condfieldscs[1],\fU)\ar[r,"\sim"]\ar[d,]&\condfieldscs(U)[1]\ar[d]\\
\check{C}(\condfields^\vee,\fU)\ar[r,"\sim"]&\condfields^\vee(U)
\end{tikzcd}.
\end{equation}
We will show that the left-hand downward pointing map is a quasi-isomorphism. 

Fix a tubular neighborhood $N\cong \del  M\times [0,T)$ of $\del M$. 
Let us assume that the cover $\fU$ is ``somewhat nice:'' 
it consists of open subsets $U_\alpha$ such that either $\overline{U}_\alpha\cap \del M=\emptyset$ or $V_\alpha\subset N$ of the form $V_\alpha\cong V'_\alpha\times [0,T')$ where $V'_\alpha$ is an open set in $\del M$. 
All finite intersections of somewhat nice sets are also somewhat nice, 
so all the summands in the \v{C}ech complexes will be of the form $\condfieldscs[1](U')$ or $\condfields^\vee(U')$ for $U'$ somewhat nice. 
If we prove that the map $\condfieldscs[1](U')\to \condfields^\vee(U')$ is a quasi-isomorphism for $U'$ somewhat nice, 
then the proposition follows, since the \v{C}ech complex has a filtration by degree of intersection 
(which is preserved by the map $\check{C}(\condfieldscs[1],\fU)\to \check{C}(\condfields^\vee,\fU)$) 
and the induced map on the associated graded spaces is a sum of maps $\condfieldscs[1](U')\to \condfields^\vee(U')$ for $U'$ somewhat nice. 

If $\overline{U'} \cap \del M=\emptyset$, then the map $\condfieldscs[1](U')\to \condfields^\vee(U')$ is a quasi-isomorphism, by the Atiyah-Bott lemma (see Appendix D of \cite{CG1}). Otherwise, suppose $U'=V\times [0,T')$, and let $L':=\Eb/L$. Denote by $\sL^\perp$ the sheaf of sections of $L'$. We saw in the proof of Theorem \ref{thm: maingenlcl} that there is a deformation retraction of $\condfieldscs[1](U')$ onto $\sL^\perp(V)$. Similarly, there is a deformation retraction of $\condfields(U')$ onto $\sL(V)$, and hence of $\sL^\vee(V)$ onto $\condfields^\vee(U')$. 
The map $\condfields^\vee(U')\to \sL^\vee(V)$ in this deformation retraction is dual to the inclusion $\sL(V)\to \condfields(U')$ of the $\sL$ fields as constants in the normal direction. From the characterization of the map $\sL^\perp_c(V)\to \condfieldscs[1](U')$ in Theorem \ref{thm: maingenlcl}, it follows that the composite 
\begin{equation}
\label{eq: abcomposite}
\sL^\perp_c(V)\to \condfieldscs[1](U')\to \condfields^\vee(U')\to \sL^\vee(V)
\end{equation}
is the Atiyah-Bott quasi-isomorphism (using the pairing $\ip_\partial$ to identify $L'$ with $L^!$). It follows that the map $\condfieldscs(U')\to \condfields^\vee(U')$ is a quasi-isomorphism, whence the proposition. 
\end{proof}
The following Corollary implies that the two models of classical observables for a free bulk-boundary system---namely the ones of Chapter \ref{chap: classical} and \ref{chap: freequantum}---are equivalent.

\begin{ucorollary}

\end{ucorollary}
The natural map
\begin{equation}
( \condfieldscs(U)[1]^{\hotimes_\beta k})_{S_k}\to \innerhomsym{(\condfields(U))^{\hotimes_\beta k}}{\RR}{k}
\end{equation}
is a quasi-isomorphism of differentiable vector spaces. 
\begin{proof}
It suffices to show this fact for sufficiently small neighborhoods of any point $x\in M$.
Let us first note that, since $\condfields(U)$ is nuclear Fr\'echet, we may make a number of 
If $x$ is in the interior $\mathring M$, then by choosing $U$ contained entirely in the interior, we note that $\cinfty(U^k, E^{\boxtimes k})$ is an elliptic complex.
Moreover, taking coinvariants with respect to $S_k$ is an exact functor.
Therefore, the classical Atiyah-Bott lemma \autocite{atiyahbott} applies to give the quasi-isomorphism of the corollary.
If $x\in \bdyM$, we may choose $U\cong U'\times [0,\delta)$ for $\delta<\epsilon$ and $U'$ open in $\bdyM$.
We obtain, by the usual formulas for the extension of a deformation retraction to symmetric algebras, quasi-isomorphisms
\begin{align}
    \Sym^k(\sL^\perp_c(U'))&\to \Sym^k(\condfieldscs[1](U))\\
    \innerhomsym{\condfields(U)^{\hotimes_\beta k}}{\RR}{k}&\to \innerhomsym{\sL(U')^{\hotimes_\beta k}}{\RR}{k}
\end{align}
Pre- and postcomposing the the map described in the statement of the corollary with these two quasi-isomorphisms, we obtain the Atiyah-Bott quasi-isomorphism for the $k$-th tensor power of $\sL$.
More precisely, we obtain a composite as in Equation \eqref{eq: abcomposite}.
Then we proceed just as in the discussion following that Equation.
The Corollary follows.
\end{proof}